\documentclass{article}
\usepackage[english]{babel}
\usepackage{geometry,amsmath,amssymb,bbm,stmaryrd,enumerate,hyperref,latexsym,xcolor,theorem,slashed,centernot,cancel}

\newcommand{\TeXmacs}{T\kern-.1667em\lower.5ex\hbox{E}\kern-.125emX\kern-.1em\lower.5ex\hbox{\textsc{m\kern-.05ema\kern-.125emc\kern-.05ems}}}
\newcommand{\assign}{:=}
\newcommand{\asterisk}{\mathord{*}}

\newcommand{\comma}{{,}}
\newcommand{\equallim}{\mathop{=}\limits}
\newcommand{\mathd}{\mathrm{d}}
\newcommand{\nobracket}{}

\newcommand{\tmcolor}[2]{{\color{#1}{#2}}}
\newcommand{\tmmathbf}[1]{\ensuremath{\boldsymbol{#1}}}

\newcommand{\tmop}[1]{\ensuremath{\operatorname{#1}}}
\newcommand{\tmstrong}[1]{\textbf{#1}}
\newcommand{\tmtextit}[1]{{\itshape{#1}}}
\newenvironment{proof}{\noindent\textbf{Proof\ }}{\hspace*{\fill}$\Box$\medskip}
\newenvironment{proof*}[1]{\noindent\textbf{#1\ }}{\hspace*{\fill}$\Box$\medskip}
\newtheorem{theorem}{Theorem}[section]
\newtheorem{axiom}{Axiom}
\newtheorem{corollary}[theorem]{Corollary}
\newtheorem{definition}[theorem]{Definition}
\newtheorem{lemma}[theorem]{Lemma}
\newtheorem{proposition}[theorem]{Proposition}
{\theorembodyfont{\rmfamily}\newtheorem{remark}[theorem]{Remark}}
\newcommand{\tmkeywords}{\textbf{Keywords:} }
\newcommand{\tmmsc}{\textbf{A.M.S. subject classification:} }

\begin{document}

\title{ Elliptic stochastic quantization of Sinh-Gordon QFT}

\author{ Nikolay Barashkov\thanks{Max Planck Institute for Mathematics in the Sciences, Leipzig, 
  \tmtextit{nikolay.barashkov@mis.mpg.de}} and Francesco C. De Vecchi
  \thanks{Department of Mathematics ``Felice Casorati'', University of Pavia, 
  \tmtextit{francescocarlo.devecchi@unipv.it}}}
\date{}

\maketitle

\begin{abstract}
  The (elliptic) stochastic quantization equation for the (massive) $\cosh
  (\beta \varphi)_2$ model, for the charged parameter in the $L^2$ regime
  (i.e. $\beta^2 < 4 \pi$), is studied. We prove the existence, uniqueness and
  the properties of the invariant measure of the solution to this equation.
  The proof is obtained through a priori estimates and a lattice approximation
  of the equation. For implementing this strategy we generalize some
  properties of Besov spaces in the continuum to analogous results for Besov
  spaces on the lattice. As a final result we show how to use the stochastic
  quantization equation to verify the Osterwalder-Schrader axioms for the
  $\cosh (\beta \varphi)_2$ quantum field theory, including the exponential
  decay of correlation functions.
\end{abstract}

\tmmsc{60H17, 81T08, 81T40}

\tmkeywords{stochastic quantization, Gaussian multiplicative chaos, elliptic
stochastic partial differential equations, Besov spaces (on lattice),
Euclidean quantum field theory, dimensional reduction }


\section{Introduction}

One of the first steps in Euclidean Quantum Field Theory (EQFT) is the (more
or less explicit) construction of quantum models satisfying one of the
equivalent formulations of Osterwalder--Schrader axioms (see, e.g,
{\cite{Glimm_Jaffe_book,Simon_phi2}}). In the case of scalar bosonic quantum
fields, this problem is equivalent to the definition of a probability measures
$\nu$ on the space of tempered distributions $\mathcal{S}' (\mathbb{R}^d)$
formally given by
\begin{equation}
  \mathd \nu = \frac{1}{Z} \exp \left( - \lambda \int_{\mathbb{R}^d} V
  (\varphi (x)) \mathd x \right) \mu (\mathd \varphi), \label{eq:eqft}
\end{equation}
where $\lambda \geqslant 0$, $\mu$ is the Gaussian Free Field with mass $m
\geqslant 0$, which is the Gaussian measure with covariance
\[ \int \langle f, \varphi \rangle_{L^2 (\mathbb{R}^2)} \langle g, \varphi
   \rangle_{L^2 (\mathbb{R}^2)} \mathd \mu = \langle f, (m^2 - \Delta)^{- 1} g
   \rangle, \]
$V : \mathbb{R} \rightarrow \mathbb{R}$ is a regular function describing the
interaction of the field $\varphi$, and $Z \in \mathbb{R}_+$ is a
normalization constant. Important examples, of measures of the form
{\eqref{eq:eqft}}, which have been extensively studied in the literature
include, are:
\begin{itemize}
  \item $\varphi^4_{2, 3}$ models, where $V (\varphi) = \varphi^4$, and $d =
  2, 3$ (see, e.g.
  {\cite{benfatto_ultraviolet_1980,brydges-new-1983,magnen_infinite_1976,park_lambda_1975}}),
  
  \item Sine-Gordon model, where $V (\varphi) = \cos (\beta \varphi)$, $d = 2$
  and $\beta^2 < 8 \pi$ (see, e.g.,
  {\cite{Albeverio-HK-sine-gordon-1979,lacoin_sine-gordon}}),
  
  \item H{\o}egh-Krohn models, where
  \begin{equation}
    V (\varphi) = \int^{\sqrt{8 \pi}}_{- \sqrt{8 \pi}} \exp (\hat{\beta}
    \varphi) \eta (\mathd \hat{\beta}), \label{eq:Hoergh}
  \end{equation}
  where $\eta$ is a positive measure supported in $ \left( - \sqrt{8 \pi},
  \sqrt{8 \pi} \right)$, and $d = 2$ (see, e.g. {\cite{Albeverio-1974}} and
  {\cite{FrohlichPark1977}}). 
\end{itemize}
Since the Gaussian measure $\mu$, in the formal equation {\eqref{eq:eqft}}, is
supported on generic distributions of negative regularity, it is not clear how
to define the composition of the field $\varphi$ with the nonlinear function
$V (\varphi)$. In order to solve this problem we need to exploit
renormalization techniques: We approximate the measure $\mu$ by a sequence of
more regular probability laws $\mu_{\varepsilon}$ (depending on some real
parameter $\varepsilon > 0$) such that $\mu_{\varepsilon} \rightarrow \mu$ as
$\varepsilon \rightarrow 0$. We also replace the potential $V$ by a sequence
of functions $V_{\varepsilon}$, depending on the regularization
$\mu_{\varepsilon}$ chosen, for which the limit of the sequence
$V_{\varepsilon} (\varphi_{\varepsilon})$ (as $\varepsilon \rightarrow 0$) is
well defined and nontrivial. A second problem that we face in the definition
of measure {\eqref{eq:eqft}}, is that, being $\mu$ is translation-invariant,
and thus one cannot expect the integral
\begin{equation}
  \lambda \int_{\mathbb{R}^d} V (\varphi (x)) \mathd x \label{eq:integral}
\end{equation}
to converge when integrated over the whole $\mathbb{R}^d$, even if one is able to give a precise meaning to the composition $V (\varphi)$. To deal with
this difficulty, one can approximate {\eqref{eq:integral}} by
\[ \lambda \int_{\mathbb{R}^d} \rho (x) V (\varphi (x)) \mathd x, \]
where $\rho : \mathbb{R}^d \rightarrow [0, 1]$ is a smooth function decaying
sufficiently quickly at infinity, and then send $\rho \rightarrow 1$. To
summarize to construct a rigorously meaningful version of the measure
{\eqref{eq:eqft}} we have to consider the weak limit
\begin{equation}
  \lim_{\rho \rightarrow 1, \varepsilon \rightarrow 0} \frac{1}{Z_{\rho,
  \varepsilon}} \exp \left( - \lambda \int_{\mathbb{R}^2} \rho V_{\varepsilon}
  (\varphi) \right) \mathd \mu_{\varepsilon} . \label{eq:eqft-approx}
\end{equation}
Many different methods have been used to prove the existence and nontriviality
of the limit {\eqref{eq:eqft-approx}}. Many of them analyze the measure in
{\eqref{eq:eqft-approx}} directly, proving tightness, and obtaining the
existence of some limit measure, at least up to passing to suitable
subsequences.

In the recent decade the technique of stochastic quantization (SQ), first
proposed by Parisi and Wu {\cite{parisi-perturbation-1981}}, for studying
{\eqref{eq:eqft}} has become increasingly popular. In this approach one
considers the measure {\eqref{eq:eqft}} as an invariant measure of some
suitable stochastic partial differential equations (SPDEs). An important
example, of such SPDEs, is the Langevin dynamics associated with the measure
{\eqref{eq:eqft}}, which is described by the equation
\begin{equation}
  \partial_t \phi (t, x) + (m^2 - \Delta) \phi (t, x) + \lambda V' (\phi (t,
  x)) = \xi (t, x) \label{eq:SQE-general}, \quad (t, x) \in \mathbb{R}_+
  \times \mathbb{R}^d
\end{equation}
where $\xi$ is a space time Gaussian white noise on $\mathbb{R}_+ \times
\mathbb{R}^d$. The difficulties in solving {\eqref{eq:SQE-general}} are
analogous to the ones mentioned before for the direct construction of the
measure {\eqref{eq:eqft}}: The expected regularity of the solution $\phi$ is
not enough to provide a natural sense to the composition with the nonlinearity
$V'$. Furthermore the noise $\xi$ is expected to grow at infinity, and, thus,
one can only expect to solve equation {\eqref{eq:SQE-general}} in a weighted
space.

One of the advantages of SQ methods is that we can apply many SPDEs and PDEs
techniques in solving equation {\eqref{eq:SQE-general}} which are, in general,
not available for the direct construction of the measure {\eqref{eq:eqft}}.
The first well posedness result for {\eqref{eq:SQE-general}}, in the case of
$\Phi_2^4$ measure on the torus $\mathbb{T}^2$, was obtained by Da Prato and
Debussche in {\cite{DaPrato_Debussche}} (see also {\cite{Mourrat_Weber2017}}
by Mourrat and Weber about the $\Phi_2^4$ stochastic quantization SPDE on the
plane $\mathbb{R}^2$). The paracontrolled distribution theory of Gubinelli,
Imkeller, and Perkowski {\cite{gubinelli_paracontrolled_2015}} (see also
{\cite{catellier_paracontrolled_2013}}), as well as Hairer's regularity
structures {\cite{haiere_theory_2014}} and Kupiainen's approach
{\cite{kupiainen_renormalization_2016}}, based on renormalization group
techniques, \ allowed to prove well posedness results in the more singular
case of $\Phi_3^4$ measure on $\mathbb{T}^3$. Later these methods were
improved to obtain a priori bounds on the solutions to
{\eqref{eq:SQE-general}}. These bounds were strong enough to control the
invariant measure and thus prove the existence of the limit
{\eqref{eq:eqft-approx}} in $\Phi_2^4$ and $\Phi_3^4$
{\cite{albeverio_invariant_2017,albeverio2021construction,GuHof2018,Mourrat_Weber20172}}.

Another approach using an infinite dimensional stochastic control problem was
used by Gubinelli and one of the authors for constructing $\Phi_3^4$ measure
on $\mathbb{T}^3$ in {\cite{Barashkov2020,Barashkov2021}}, and by Bringmann in
{\cite{bringmann2020}} for the Gibbs measure of Hartree hyperbolic equation.

\

Equation {\eqref{eq:SQE-general}} is not the only possibility of pursuing
stochastic quantization. In {\cite{Parisi_Sourlas_1979,Parisi-Sourlas-1982}}
Parisi and Sourlas proposed an elliptic approach to this problem. They
conjectured that if one solves the equation
\begin{equation}
  (m^2 - \Delta) \phi (z, x) + \lambda V' (\phi (z, x)) = \xi (z, x)
  \label{eq:SQE-elliptic}
\end{equation}
on $(z, x) \in \mathbb{R}^2 \times \mathbb{R}^d$, and where $\xi$ is a
$\mathbb{R}^2 \times \mathbb{R}^d$ Gaussian white noise, then, for any $z \in
\mathbb{R}^2$, the following equivalence should hold
\begin{equation}
  \phi (z, \cdot) \sim \frac{1}{Z} \exp (- \lambda V (\varphi)) \mathd \mu
  \label{eq:law}.
\end{equation}
In the mathematics literature this kind of stochastic quantization was
investigated in {\cite{Klein_supersymmetry}} and proved in {\cite{AlDeGu2018}}
for a class of potentials $V$ including convex potentials.

Let us remark that, although both the elliptic and parabolic SQ methods so
far often fall short of providing enough information on the decay of
correlation functions of $\nu$, they are capable of proving pathwise
properties even when absolute continuity is lost, since
{\eqref{eq:SQE-general}} and {\eqref{eq:SQE-elliptic}} provide a coupling
between the Gaussian Free Field (having law $\mu$) and the interacting fields
(having law $\nu$). Similar results have also been obtained using the
Polchinski equation, see, e.g., {\cite{bauerschmidt_2020}}.

\

The main goal of this paper is to study the elliptic stochastic quantization
equation for the $\tmop{Sinh}$-Gordon model in dimension $d = 2$ (hereafter
called $\cosh (\beta \varphi)_2$ model) in the $L^2$ regime of the charge
parameter $\beta$, namely when $| \beta | < \sqrt{4 \pi}$. In this case the SQ
elliptic equation is formally given by \
\begin{equation}
  (m^2 - \Delta) \phi (z, x) + \lambda \alpha \sinh (\alpha \phi (z, x)) = \xi (z,
  x), \quad (z, x) \in \mathbb{R}^2 \times \mathbb{R}^2
  \label{elliptic-sinh-formal}
\end{equation}
where $\xi$ is a Gaussian white noise on $\mathbb{R}^2 \times \mathbb{R}^2$,
and $\alpha^2 = 4 \pi \beta^2$. \ The $\cosh (\beta \varphi)_2$ model is a
particular case of the H{\o}egh-Krohn model {\eqref{eq:Hoergh}} cited before,
i.e. when the measure $\eta (\mathd \hat{\beta}) = \frac{1}{2} (\delta_{-
\beta} (\mathd \hat{\beta}) + \delta_{\beta} (\mathd \hat{\beta}))$ (where
$\delta_{\pm \beta}$ is the Dirac delta with unitary mass in $\pm \beta$) and
studied in
{\cite{Albeverio_Liang,Albeverio-1974,Albeverio_Song2004,FrohlichPark1977}}
(see also the related model in {\cite{Albeverio_Yoshida2002}}) using
constructive techniques. The Sinh-Gordon model (in the limit case of zero mass
$m \rightarrow 0$) is an Integrable Quantum Field Theory see, e.g,
{\cite{Mussardo2020,Smirnov1992}} and has received considerable attention in
the physics literature
{\cite{Pillin1998,Zamolodchikov2006,Zamolodchikov1995}}. We remark also that a
probabilistic approach developed to the Lioville model ($V (\varphi) = \exp
(\beta \varphi), m = 0$) developed by Kupianen, Rhodes, Vargas and their
coauthors has found remarkable success
{\cite{Guillarmou2020,Kupiainen2020-stress,Kupiainen2018,Kupiainen2020}}. As
of this moment we are not aware of a probabilistic construction of the
Sinh-Gordon in the case $m = 0$, and giving such a construction appears to be
an interesting and challenging problem. \

\

In order to solve equation {\eqref{elliptic-sinh-formal}} it can be useful to
rewrite it introducing a change of variables (introduced by Da Prato and
Debussche in {\cite{DaPrato_Debussche}} for the $\Phi^4_2$ model) in the
following way
\begin{equation}
 (m^2 - \Delta) \bar{\phi} (z, x) +\frac{\lambda \alpha e^{\alpha \bar{\phi}}}{2}
   \; \; ``e^{\alpha W}"\; \;- \frac{\lambda \alpha e^{-\alpha \bar{\phi}}}{2}
   \; \; ``e^{-\alpha W}"\; \; = 0 ,\label{eq:formal2}
\end{equation}
where $\bar{\phi} = \phi - W$ and $W$ is the solution to the linear equation
$(- \Delta + m^2) W = \xi$. It is possible to give a precise meaning to the
expressions $e^{\pm \alpha W}$ (when $\alpha^2 = 4 \pi \beta^2 < 2 (4
\pi)^2$) thanks to Gaussian Multiplicative Chaos (see, e.g.,
{\cite{Rhodes2014}}) through the limit
\begin{equation}
  \lim_{\varepsilon \rightarrow 0} e^{(\pm \alpha W_{\varepsilon} -
  c_{\varepsilon}  (\alpha))} \mathd x = \mu^{\pm \alpha} (\mathd x),
  \label{eq:mualpha}
\end{equation}
where $W_{\varepsilon}$ is a smooth (translation invariant) approximation of
the Gaussian Free Field (in $\mathbb{R}^4$) $W$, converging almost surely to
$W$ as $\varepsilon \rightarrow 0$, and $c_{\varepsilon} (\alpha) \in
\mathbb{R}_+$ is given by the relation
\[ c_{\varepsilon} (\alpha) = \frac{\alpha^2}{2} \mathbb{E}
   [W_{\varepsilon}^2], \]
and thus $c_{\varepsilon} (\alpha) \rightarrow + \infty$ as $\varepsilon
\rightarrow 0$. The limit {\eqref{eq:mualpha}} is known to exist in the
space of tempered distributions $\mathcal{S}' (\mathbb{R}^4)$ on
$\mathbb{R}^4$. There are two main properties that make the random
distribution $\mu^{\pm \alpha}$ different from, for example, Wick powers $: W^n :$
and the Wick complex exponential $: e^{i \alpha W} :$. The first one is the
fact that $\mu^{\pm \alpha}$ is a positive distribution, i.e. a
$\sigma$-finite measure (this permits to exploit the signed structure of the
noise). The second property is that $\mu^{\pm \alpha}$ exhibits
multifractality, i.e. the Besov regularity $s$ of $\mu^{\pm \alpha}$ depends
on the parameter $\alpha$ and on the integrability $p \in \left[ 1, \frac{2 (4
\pi)^2}{\alpha^2} \right)$ in the following way, for any $\delta > 0$ we have
\[ \mu^{\pm \alpha} \in B^{-\frac{\alpha^2}{(4 \pi)^2} (p - 1) - \delta}_{p, p,
   \tmop{loc}} (\mathbb{R}^4) = B^{-\frac{\beta^2}{(4 \pi)} (p - 1) -
   \delta}_{p, p, \tmop{loc}} (\mathbb{R}^4) . \]
These two properties are essential in the study of SPDEs involving the random
distribution {\eqref{eq:mualpha}}. Although, to the best of our knowledge, if
we exclude a brief mention in Remark 1.3 and Remark 1.18 in
{\cite{oh2019parabolic}} (on $\mathbb{T}^2$ and for $\beta^2 < \frac{8 \pi}{3
+ 2 \sqrt{2}} \simeq 1.37 \pi$), there are no studies of singular SPDEs
involving a $\sinh$ nonlinearity, SPDEs driven by a random forcing of the form
{\eqref{eq:mualpha}} has been subject of a certain number of works in
particular concerning the stochastic quantization of \ $V (\varphi) = \exp
(\beta \varphi)_2$ models corresponding to the aforementioned Liouville
quantum gravity: see {\cite{GARBAN2020}} by Garban on the parabolic equation
on $\mathbb{T}^2$ and $\mathbb{S}^2$ for $\beta^2 < 8 \left( 3 - \sqrt{2}
\right) \pi$, {\cite{hoshino2020stochastic,Kusuoka2021,hoshino2023stochastic}} by Hoshino, Kawabi,
and Kusuoka addressing the parabolic problem on $\mathbb{T}^2$ in the whole
subcritical regime $\beta^2 < 8 \pi$, {\cite{oh2019parabolic}} by Oh, Robert,
and Wang about the parabolic and hyperbolic equation for $\beta^2 < 4 \pi$,
{\cite{oh2020stochastic}} by Oh, Robert, Tzvetkov, and Wang about the
parabolic equation in the $L^2$ regime $\beta^2 \leqslant 4 \pi$ on a general
compact manifold, and {\cite{AlDeGu2019}} by Albeverio, Gubinelli and one of
the authors on the elliptic quantization on $\mathbb{R}^2$ with charged
parameter $\beta^2 < 4 \left( 8 - 4 \sqrt{3} \right) \pi$ (see also
{\cite{hoshino2020stochastic}} by Albeverio, Kawabi, Mihalache and Roeckner on
the SQ of exponential and trigonometric problem using a Dirichlet form
approach).

All the previous works use in an essential way the fact that the exponential
$\exp (\alpha \psi)$ is a positive function, in this way the stochastic
quantization equation is equivalent to an equation with bounded coefficients.
In the case of $\sinh (\alpha \phi)$ nonlinearity, we cannot directly use the
sign of the coefficients and we are forced to deal with an SPDE with unbounded
coefficients. For this reason we decided to use a sort of energy estimate for
$\bar{\phi}$ (see Theorem \ref{theorem_apriori1}), inspired by the works on
$\Phi^4_3$ model in
{\cite{albeverio_invariant_2017,albeverio2021construction,GuHof2018}}. This
idea permits us to obtain some a priori estimates for the weighted $H^1$ norm
of $\bar{\phi}$ and the integrals $\int e^{\pm \alpha \bar{\phi} (x)} \mu^{\pm
\alpha} (\mathd x)$ (here is where positivity of the distributions $\mu^{\pm
\alpha}$ plays a central role). The use of this energy method, and so the
choice of $H^1$ norm on $\bar{\phi}$ and not the weaker norms like the Sobolev
$W^{1, p}$ for $p < 2$, is the main reason for the limitation of our
techniques to the $L^2$ regime for the charged parameter $\beta$, i.e.
$\alpha^2 = 4 \pi \beta^2 \leqslant (4 \pi)^2$. These a priori estimates allow
us to prove the existence and uniqueness of the solution to equation
{\eqref{elliptic-sinh-formal}} using some regular approximation (for example a
Galerkin approximation). Unfortunately Theorem \ref{theorem_apriori1} below is
not enough to prove the stochastic quantization for the Sinh-Gordon model,
i.e. to rigorously establish the heuristic relation {\eqref{eq:law}}. For
obtaining this result we need some improved a priori estimates (see Theorem
\ref{theorem:improvedapriori}) obtained by a bootstrap procedure. In order to
implement this bootstrap argument a ``local'' approximation of equation
{\eqref{elliptic-sinh-formal}} is needed, and the Galerkin or other nonlocal
smoothing methods, exploited for studying the exponential models in the
articles mentioned before, cannot be used.

\

For all these reasons, we approximate the SPDE {\eqref{elliptic-sinh-formal}}
by the following equation on the lattice for $(x, z) \in \mathbb{R}^2 \times
\varepsilon \mathbb{Z}^2$:
\begin{equation}
  (m^2 - \Delta_{\mathbb{R}^2} - \Delta_{\varepsilon \mathbb{Z}^2})
  \phi^{\varepsilon} (x, z) + \lambda \varepsilon^{\frac{\alpha^2}{(4 \pi)^2}}
  \sinh (\alpha \phi^{\epsilon} (x, z)) = \xi_{\varepsilon} (x, z),
  \label{eq:elliptic-lattice}
\end{equation}
where $\Delta_{\varepsilon \mathbb{Z}^2}$ is the discrete Laplacian on
$\varepsilon \mathbb{Z}^2$, $\Delta_{\mathbb{R}^2}$ is the standard Laplacian
on $\mathbb{R}^2$, and $\xi_{\varepsilon}$ is an appropriate white noise
having Cameron-Martin space $L^2 (\mathbb{R}^2 \times \varepsilon
\mathbb{Z}^2)$. The factor $\varepsilon^{\alpha^2 / (4 \pi)^2}$ is the necessary
renormalization correction for obtaining the nontrivial limit
{\eqref{eq:mualpha}}, for $\varepsilon \rightarrow 0$.

\

In order to study equation {\eqref{eq:elliptic-lattice}}, we use Besov spaces
on $\mathbb{R}^2 \times \varepsilon \mathbb{Z}^2$, which are very similar to
the ones developed on $\varepsilon \mathbb{Z}^2$ by Martin and Perkowski in
{\cite{Perkowski2019}} and by Gubinelli and Hofmanova in {\cite{GuHof2018}}.
Other works involving the discretization of singular SPDEs are
{\cite{Chouk_Anderson_discrete,Gubinelli_KPZ,Mourrat_discrete,Shen_discrete,Zhu-Zhu-discrete}}
on $\mathbb{T}^d$, using an extension operator by interpolation with
trigonometric polynomials, and
{\cite{Hairer_Erhard,erhard2021scaling,Hairer_Matetski,hairer2021phi34}}
introducing a discrete version of regularity structures. About this framework
our main contribution is the following: the extension operators
$\mathcal{E}^{\varepsilon}$ and $\overline{\mathcal{E}}^{\varepsilon}$ used in
the current paper differ from the ones employed in
{\cite{GuHof2018,Perkowski2019}}. Indeed, in the cited papers, the extension
operators, associating to any function on the lattice a distribution on the
continuum, are nonlocal, and thus, it is not clear how they behave with
respect the composition with the nonlinear coefficients. This problem would
cause technical difficulties in the identification of the limiting equation
for {\eqref{eq:elliptic-lattice}}. Instead the extension operator
$\overline{\mathcal{E}}^{\varepsilon}$, proposed in the present work, is the
same as that employed in {\cite{park_lambda_1975}} (see also the more recent
{\cite{hairer2021phi34}}), and it takes values into linear superpositions of
step functions. Although using $\overline{\mathcal{E}}^{\varepsilon}$ we loose
some regularity, due to the limited smoothness of step functions, \ it
commutes with the nonlinearities and we can prove that, in the $\varepsilon
\rightarrow 0$ limit, it is possible to recover the original regularity (see
Theorem \ref{theorem:extension2} for a precise statement of this fact). Let us
mention also Theorem \ref{theorem:sobolev_besov} where we prove a sort of
equivalent definition of discrete Besov norm using the discrete differences
which could be of interest in the comparison between the discrete Besov
methods used in {\cite{GuHof2018,Perkowski2019}} and the ones based on
discrete regularity structures in
{\cite{Hairer_Erhard,erhard2021scaling,Hairer_Matetski,hairer2021phi34}}.

\

A final contribution of our work is about the convergence of the discrete
Gaussian Multiplicative chaos $\mu^{\alpha}_{\varepsilon}$ to its continuum
counterpart $\mu^{\alpha}$. This convergence problem was already addressed in
other works (see, e.g., Section 5.3 in {\cite{Rhodes2014}} and references
therein). The novelty of our result is that the convergence of
$\overline{\mathcal{E}}^{\varepsilon} (\mu_{\varepsilon}^{\alpha})$ to
$\mu^{\alpha}$ is not only weakly$^{\ast}$ in the space of positive
$\sigma$-finite measures, but also in the Besov spaces of the form $B^{-
\frac{\alpha^2}{_{(4 \pi)^2}} (p - 1) - \varepsilon}_{p, p, \tmop{loc}}
(\mathbb{R}^4)$ for suitable $p \in \left[ 1, \frac{2 (4 \pi)^2}{\alpha^2}
\right)$, $\alpha^2 \leqslant (4 \pi)^2$ and $\varepsilon > 0$ (see Section
\ref{sec:stochastic-estimates}). We conclude the introduction with the main
theorem proved in the paper.

\begin{theorem}
  \label{theorem_main1}Equation {\eqref{elliptic-sinh-formal}} admits a unique
  solution for $| \alpha | < 4 \pi$ in \[\left(B^{-\delta}_{\infty,\infty,\tmop{loc}}(\mathbb{R}^4)+B^{-\frac{\alpha^2}{(4 \pi)^2} (p - 1) - \delta+2}_{p, p,
   \tmop{loc}} (\mathbb{R}^4)\right) \] (for suitable $\delta(\alpha)>0$ and $p(\alpha)>1$) with at most polynomial growth at infinity. Furthermore we have that $\phi (x_0,
  \cdot) \in \mathcal{S}' (\mathbb{R}^2)$ has probability distribution
  $\nu^{\cosh, \beta}_m$ associated with the action
  \[ S (\varphi) = \frac{1}{2} \int_{\mathbb{R}^2} (| \nabla \varphi (z) |^2 +
     m^2 \varphi (z)^2) \mathd z + \lambda \int_{\mathbb{R}^2} : \cosh (\beta
     \varphi (z)) : \mathd z, \]
  where $\beta = \frac{\alpha}{\sqrt{4 \pi}}$ (see Definition
  \ref{definition:coshmodel} for the precise definition of $\nu^{\cosh,
  \beta}_m$). 
\end{theorem}

\

{\noindent}{\tmstrong{Structure of the Paper}}: In Section \ref{sec:setting}
we recall some standard results on Besov spaces on $\mathbb{R}^d$ and some of
the results proved in {\cite{GuHof2018,Perkowski2019}} on Besov spaces on
$\varepsilon \mathbb{Z}^d .$ We also extend these definitions to $\mathbb{R}^2
\times \varepsilon \mathbb{Z}^2$. Furthermore we introduce difference spaces
on $\varepsilon \mathbb{Z}^d$, introduce our extension operator and prove
bounds for it. Finally we discuss regularity of positive distributions
multiplied by a density on the lattice, which we will need to take advantage
of the positivity of the Gaussian Multiplicative Chaos. In Section
\ref{sec:stochastic-estimates} we will prove the estimates on Gaussian
Multiplicative Chaos which we require for the a priori estimates on
{\eqref{eq:elliptic-lattice}}. To achieve this we will establish estimates on
the Green's function of $(m^2 - \Delta_{\mathbb{R}^2 \times \varepsilon
\mathbb{Z}^2})^2$ on $\mathbb{R}^2 \times \varepsilon \mathbb{Z}^2$. In
Section \ref{sec:stochastic-quantization} we prove a priori estimates on
equation {\eqref{eq:elliptic-lattice}} which are strong enough to find a
convergent subsequence of the solution to {\eqref{eq:elliptic-lattice}} (after
applying the extension operator). The limit will be shown to be in $H^1 
(\mathbb{R}^4)$ almost surely and to satisfy equation
{\eqref{elliptic-sinh-formal}}. Furthermore we will show that solutions to
equation {\eqref{elliptic-sinh-formal}} are unique. Finally we will show that
the solutions have enough regularity to be restricted to a two dimensional
subspace or $\mathbb{R}^2$ thus establishing the main theorem. Finally, in
Appendix \ref{appendix:axioms}, we prove the measure $\nu^{\cosh, \beta}_m$
constructed satisfies the Osterwalder-Schrader axioms, exploiting the methods
of {\cite{albeverio2021construction,GuHof2018,hairer2021phi34}} for
$\varphi^4_3$ measures. To our knowledge this is the first selfcontained proof
of the Osterwalder-Schrader axioms including mass-gap by stochastic
quantization, however it relies heavily on the convexity of the renormalized
interaction. \\

{\noindent}\textbf{Acknowledgements.} The authors would like to thank
Massimiliano Gubinelli and Sergio Albeverio for the comments and suggestions
on an earlier version of the paper. N.B gratefully acknowledges financial
support from \ ERC Advanced Grant 74148 ``Quantum Fields and Probability'' and
the Deutsche Forschungsgemeinschaft (DFG, German Research Foundation) through
CRC 1060. F.C.D. is supported by the DFG under Germany's Excellence Strategy -
GZ770 2047/1, project-id 390685813.

\section{The setting}\label{sec:setting}

For all the rest of the paper we employ the following notations: If $H, K : S
\rightarrow \mathbb{R}_+$ are two functions defined on the same set $S$ we say
that $H \lesssim K$ (or with an abuse of notation that $H (f) \lesssim K (f)$
(where $f \in S$)), if there is some constant $C$ such that for any $f \in S$
we have $H (f) \leqslant C K (f)$. We write also $H \sim K$ (or with an abuse
of notation that $H (f) \sim K (f)$ (where $f \in S$)) \ if $H \lesssim K$ and
$K \lesssim H$.

\subsection{Besov spaces on $\mathbb{R}^d$}\label{section:besov:rd}

Here we report some standard definitions and properties of Besov spaces on
$\mathbb{R}^d$. The main definitions and theorems hold also for the cases of
Besov spaces on $\mathbb{T}^d$ or $\mathbb{R}^{d_1} \times \mathbb{T}^{d_2}$.

\

We fix some notations: For any $\ell \in \mathbb{R}$ we define the (weight)
function $\rho_{\ell}^{(d)} : \mathbb{R}^d \rightarrow \mathbb{R}_+$ as
\[ \rho_{\ell}^{(d)} (y) = \frac{1}{(1 + | y |^2)^{\ell / 2}}, \quad y \in
   \mathbb{R}^d . \]
If $p \in [1, + \infty)$ and $\ell \in \mathbb{R}$ we define
\[ \| f \|^p_{L^p_{\ell} (\mathbb{R}^d)} = \int_{\mathbb{R}^d} | f (y)
   \rho_{\ell}^{(d)} (y) |^p \mathd y, \]
where $\mathd y$ is the standard Lebesgue measure on $\mathbb{R}^d$, and
\[ \| f \|_{L^{\infty}_{\ell} (\mathbb{R}^d)} = \sup_{x \in \mathbb{R}^d} | f
   (x) \rho^{(d)}_{\ell} (x) | . \]
In the following we write $L^p (\mathbb{R}^d)$ for $L^p_0 (\mathbb{R}^d)$
(where $L_0 (\mathbb{R}^d)$ is the space $L^p_{\ell} (\mathbb{R}^d)$ for $\ell
= 0$). We denote by $\mathcal{S} (\mathbb{R}^d)$ the Fr{\'e}chet space of
Schwartz test functions and by $\mathcal{S}' (\mathbb{R}^d)$ its strong dual,
i.e. the nuclear space of tempered distributions. We denote by $\mathcal{F} :
\mathcal{S} (\mathbb{R}^d) \rightarrow \mathcal{S} (\mathbb{R}^d)$ the Fourier
transform on $\mathbb{R}^d$, i.e.
\[ \mathcal{F} (f) (k) = \int_{\mathbb{R}^d} e^{i (k \cdot y)} f (y) \mathd y,
   \quad k \in \mathbb{R}^d, f \in \mathcal{S} (\mathbb{R}^d) \]
and by $\mathcal{F}^{- 1}$ the inverse Fourier transform namely
\[ \mathcal{F}^{- 1} (\hat{f}) (y) = \frac{1}{(2 \pi)^d} \int_{\mathbb{R}^d}
   e^{- i (k \cdot y)} \hat{f} (k) \mathd k, \quad y \in \mathbb{R}^d, \hat{f}
   \in \mathcal{S} (\mathbb{R}^d) . \]
We can extend the Fourier transform on the space of tempered distributions by
duality.

\

We consider a smooth dyadic partition of unity $(\varphi^{(d)}_j)_{j
\geqslant - 1}$ defined on $\mathbb{R}^d$ such that $\varphi_{- 1}^{(d)}$ is
supported in a ball around 0 of radius $\frac{1}{2}$, $\varphi_0^{(d)}$ is
supported in an annulus, $\varphi_j^{(d)} (\cdot) = \varphi_0^{(d)} (2^{- j}
\cdot)$ for $j > 0$ and if $| i - j | > 1$ then $\tmop{supp} \varphi_j^{(d)}
\cap \tmop{supp} \varphi_i^{(d)} = \emptyset$. We set:
\[ K_j^{(d)} = \mathcal{F}^{- 1} (\varphi_j^{(d)}) . \]
\begin{remark}
  \label{remark:partition:conditions}In all the rest of the paper we consider
  an additional condition on the dyadic partition of the unity function,
  which, although not being completely standard, can always be assumed without
  loss of generality (see, e.g., Section 3 and Appendix A in
  {\cite{Mourrat_Weber2017}}): The functions $\varphi_{- 1}^{(d)}$ and
  $\varphi_0^{(d)}$ (and thus all the functions $\{ \varphi_j^{(d)} \}_{j
  \geqslant - 1}$) have to belong to a Gevrey class of index $\theta$ for some
  $\theta > 1$, i.e. there is a constant $C > 0$ for which, for any $\alpha =
  (\alpha_1, \cdots, \alpha_d) \in \mathbb{N}^d_0$, we have
  \[ \| \partial^{\alpha} \varphi^{(d)}_{- 1} \|_{L^{\infty} (\mathbb{R}^d)}
     \leqslant C^{| \alpha |} (\alpha !)^{\theta}, \quad \| \partial^{\alpha}
     \varphi^{(d)}_0 \|_{L^{\infty} (\mathbb{R}^d)} \leqslant C^{| \alpha |}
     (\alpha !)^{\theta}, \]
  where $| \alpha | = \sum \alpha_i$ and $\alpha ! = \prod \alpha_i !$. Under
  the previous conditions we get that there are some constants $\bar{a} > 0$
  and $\bar{\beta} = \frac{1}{\theta}$ for which
  \[ | K_{- 1}^{(d)} (y) | \lesssim \exp (- \bar{a}  (1 + | y
     |^2)^{\bar{\beta} / 2}), \quad | K_0^{(d)} (y) | \lesssim \exp (- \bar{a}
     (1 + | y |^2)^{\bar{\beta} / 2}), \quad y \in \mathbb{R}^d . \]
  By the definition of $\varphi_j^{(d)}$ in terms of $\varphi_0^{(d)}$ we also
  get, for any $j \geqslant 0$,
  \[ | K_j^{(d)} (y) | \lesssim 2^{j d} \exp (- \bar{a}  (1 + 2^{2 j} | y
     |^2)^{\bar{\beta} / 2}), \quad y \in \mathbb{R}^d, \]
  where the constants hidden in the symbol $\lesssim$ do not depend on $j
  \geqslant 0$.
\end{remark}

We introduce also the following function
\begin{equation}
  \upsilon^{(d)} (x) : = 1 - \varphi_{- 1}^{(d)} (x) = \sum_{j \geqslant 0}
  \varphi_0^{(d)} (2^{- j} x), \label{eq:greeku}
\end{equation}
where we use that $\varphi_j^{(d)}$ is a partition of unity and thus
$\varphi_{- 1}^{(d)} (x) + \sum_{j \geqslant 0} \varphi_0^{(d)} (2^{- j} x)
\equallim 1$. The previous relations imply that
\begin{equation}
  \upsilon^{(d)} (2^{- k} x) : = 1 - \varphi_{- 1}^{(d)} (2^{- k} x) = \sum_{j
  \geqslant 0} \varphi_0^{(d)} (2^{- k - j} x) = \sum_{j \geqslant k} \varphi_0^{(d)}
  (2^{ - j} x) . \label{eq:greeku2}
\end{equation}

If $f \in \mathcal{S}' (\mathbb{R}^d)$ we denote by $\Delta_j f \in
C^{\infty} (\mathbb{R}^d)$ the function, having at most polynomial growth at infinity, being the $i$-th Littlewood-Paley block of the tempered
distribution $f$, namely $\Delta_j f$ is given by
\[ \Delta_j f = \mathcal{F}^{- 1} (\varphi_j^{(d)} \cdot \mathcal{F} (f)) =
   K_j \ast f, \]
where $\cdot$ is the natural product between Schwartz functions and tempered
distributions.

\begin{definition}
  If $p, q \in [1, + \infty]$, $s \in \mathbb{R}$ and $\ell \in \mathbb{R}$ we
  define \ the Besov space $B^s_{p, q, \ell} (\mathbb{R}^d)$ as the subset of
  tempered distributions $\mathcal{S}' (\mathbb{R}^d)$ for which $f \in
  B^s_{p, q, \ell} (\mathbb{R}^d)$ if the norm
  \[ \| f \|_{B^s_{p, q, \ell} (\mathbb{R}^d)} \assign \| \{ 2^{(j - 1) s} \|
     \Delta_{j - 1} f \|_{L^p_{\ell} (\mathbb{R}^d)} \}_{j \in \mathbb{N}_0}
     \|_{\ell^q (\mathbb{N}_0)}, \]
  is finite. Here $\| \cdot \|_{\ell^q (\mathbb{N}_0)}$ is the $\ell^q$ norm
  on the space of sequences starting at $0$.
\end{definition}

We report here some results useful in what follows.

\begin{proposition}
  \label{propsition:equivalent}For any $s, \ell, m \in \mathbb{R}$, $p, q \in
  [1, + \infty]$ we have that for any $f \in B^s_{p, q, \ell} (\mathbb{R}^d)$
  \begin{align*} \| f \|_{B^s_{p, q, \ell} (\mathbb{R}^d)} \sim &\| f \cdot
     \rho_{\ell}^{(d)} \|_{B^s_{p, q, 0} (\mathbb{R}^d)} \\ 
   \| f \|_{B^s_{p, q, \ell} (\mathbb{R}^d)} \sim & \| (1 - \Delta)^{m / 2}
     (f) \|_{B^{s - m}_{p, q, \ell} (\mathbb{R}^d)} . \end{align*}
\end{proposition}

\begin{proof}
  The proof can be found in Theorem 6.5 in {\cite{Triebel2006}}.
\end{proof}

\begin{proposition}
  \label{proposition:embedding}Consider $p_1, p_2, q_1, q_2 \in [1, +
  \infty]$, $p_1 \leq p_2$, $s_1 > s_2$ and $\ell_1, \ell_2 \in \mathbb{R}$ such that
  \[ \ell_1 \leqslant \ell_2, \quad s_1 - \frac{d}{p_1} \geqslant s_2 -
     \frac{d}{p_2}, \]
  then $B^{s_1}_{p_1, q_1, \ell_1} (\mathbb{R}^d)$ is continuously embedded in
  $B^{s_2}_{p_2, q_2, \ell_2} (\mathbb{R}^d)$, and if $\ell_1 < \ell_2$ and
  $s_1 - \frac{d}{p_1} > s_2 - \frac{d}{p_2}$ the embedding is compact.
\end{proposition}

\begin{proof}
  The proof can be found in Theorem 6.7 in {\cite{Triebel2006}}.
\end{proof}

\begin{proposition}
  \label{propisition_product}Consider $p_1, p_2, p_3, q_1, q_2, q_3 \in [1, +
  \infty]$ such that $\frac{1}{p_1} + \frac{1}{p_2} = \frac{1}{p_3}$ and $q_1
  = q_3$ and $q_2 = \infty$. Moreover consider $\ell_{1,} \ell_2, \ell_3 \in
  \mathbb{R}$ with $\ell_1 + \ell_2 = \ell_3$ and consider $s_1 < 0$, $s_2 >
  0$ and $s_3 = s_1 + s_2 > 0$. Finally, consider the bilinear functional \
  $\Pi (f, g) = f \cdot g$ defined on $\mathcal{S}' (\mathbb{R}^d) \times
  \mathcal{S} (\mathbb{R}^d)$ taking values in $\mathcal{S} '(\mathbb{R}^d)$.
  Then there exists a unique continuous extension of $\Pi$ as the map
  \[ \Pi : B^{s_1}_{p_1, q_1, \ell_1} \times B^{s_2}_{p_2, q_2, \ell_2}
     \rightarrow B^{s_1}_{p_3, q_3, \ell_3}, \]
  and we have, for any $f, g$ for which the norms are defined:
  \[ \| \Pi (f, g) \|_{B^{s_1}_{p_3, q_3, \ell_3}} \lesssim \| f
     \|_{B^{s_1}_{p_1, q_1, \ell_1}} \| g \|_{B^{s_2}_{p_2, q_2, \ell_2}} . \]
\end{proposition}

\begin{proof}
  The proof can be found in {\cite{Mourrat_Weber2017}} in Section 3.3 for
  Besov spaces with exponential weights. The proof for polynomial weights is
  similar.
\end{proof}

\subsection{Besov spaces on lattice}\label{subsection:lattice}

\

Here we want to introduce a modification of Besov spaces when the base space
is a square lattice (see {\cite{GuHof2018,Perkowski2019}}). First we consider
the set $\varepsilon \mathbb{Z}^2$ for some $\varepsilon > 0$, i.e. $z \in
\varepsilon \mathbb{Z}^2$ if $z = (\varepsilon n_1, \varepsilon n_2)$ where
$n_1, n_2 \in \mathbb{Z}$. We endow $\varepsilon \mathbb{Z}^2$ with the
discrete topology and with the measure $\mathd z$ defined, for any integrable
function $f : \varepsilon \mathbb{Z}^2 \rightarrow \mathbb{R}$, as
\[ \int_{\varepsilon \mathbb{Z}^2} f (z) \mathd z \assign \varepsilon^2
   \sum_{z \in \varepsilon \mathbb{Z}^2} f (z) . \]
We can define also the weighted Lebesgue space $L^p_{\ell} (\varepsilon
\mathbb{Z}^2)$, where $p \in [1, + \infty)$ and $\ell \in \mathbb{R}$, with
the norm
\begin{equation}
  \| f \|_{L^p_{\ell}}^p = \int_{\varepsilon \mathbb{Z}^2} (\rho_{\ell} (z))^p
  | f (z) |^p \mathd z \label{eq:Lplattice},
\end{equation}
where $f \in L^p_{\ell} (\varepsilon \mathbb{Z}^2)$, where hereafter we write
$\rho_{\ell} (x) \assign \rho_{\ell}^{(2)} (x)$, $x \in \mathbb{R}^2$ and
$\ell \in \mathbb{R}$. We can define $L_{\ell}^{\infty} (\varepsilon
\mathbb{Z}^2)$ with the norm
\[ \| f \|_{L^{\infty}_{\ell} (\varepsilon \mathbb{Z}^2)} = \sup_{z \in
   \varepsilon \mathbb{Z}^2} | f (z) \rho_{\ell} (z) |, \]
where $f \in L^{\infty}_{\ell} (\varepsilon \mathbb{Z}^2)$ if the norm is
finite. We define the space $\mathcal{S} (\varepsilon \mathbb{Z}^2) =
\bigcap_{\ell < 0} L^{\infty}_{\ell} (\varepsilon \mathbb{Z}^2)$ as the space
of Schwartz test functions defined on $\varepsilon \mathbb{Z}^2$, which forms
a Fr{\'e}chet space with the set of seminorms $\{ \| \cdot
\|_{L^{\infty}_{\ell}} \}_{\ell < 0}$. We define also $\mathcal{S}'
(\varepsilon \mathbb{Z}^2)$ as the topological dual of $\mathcal{S}
(\varepsilon \mathbb{Z}^2)$ with respect to the Fr{\'e}chet space structure of
$\mathcal{S} (\varepsilon \mathbb{Z}^2)$. It is possible to define the Fourier
transform $\mathcal{F}_{\varepsilon} : \mathcal{S} (\varepsilon \mathbb{Z}^2)
\rightarrow C^{\infty} \left( \mathbb{T}_{\frac{1}{\varepsilon}}^2 \right)$,
where $\mathbb{T}_{\frac{1}{\varepsilon}}^2$ is the torus of length $\frac{2
\pi}{\varepsilon}$, i.e. $\mathbb{T}^2_{\frac{1}{\varepsilon}} = \left[ -
\frac{\pi}{\varepsilon}, \frac{\pi}{\varepsilon} \right]^2$, which has the
following form
\[ \mathcal{F}_{\varepsilon} (f) (h) \assign \hat{f} (h) = \int_{\varepsilon
   \mathbb{Z}^2} e^{i z \cdot h} f (z) \mathd z, \]
where $h \in \mathbb{T}^2_{\frac{1}{\varepsilon}}$. In the usual way it is
possible to define the inverse Fourier transform $\mathcal{F}^{- 1} :
C^{\infty} \left( \mathbb{T}_{\frac{1}{\varepsilon}}^2 \right) \rightarrow
\mathcal{S} (\varepsilon \mathbb{Z}^2)$
\[ \mathcal{F}^{- 1}_{\varepsilon} (\hat{f}) (z) \assign \frac{1}{(2 \pi)^2}
   \int_{\mathbb{T}^2_{\frac{1}{\varepsilon}}} e^{- i z \cdot h} \hat{f} (h)
   \mathd h, \]
from $\mathcal{F}^{- 1}_{\varepsilon} : C^{\infty} \left(
\mathbb{T}_{\frac{1}{\varepsilon}}^2 \right) \rightarrow \mathcal{S}
(\varepsilon \mathbb{Z}^2)$. We can extend the Fourier transform
$\mathcal{F}_{\varepsilon}$, and the inverse Fourier transform
$\mathcal{F}_{\varepsilon}^{- 1}$ from the space $\mathcal{S}' (\varepsilon
\mathbb{Z}^2)$ into $\mathcal{D}' \left( \mathbb{T}^2_{\frac{1}{\varepsilon}}
\right)$ (where $\mathcal{D}' \left( \mathbb{T}^2_{\frac{1}{\varepsilon}}
\right)$ is the space of distributions, i.e. the topological dual of
$C^{\infty} \left( \mathbb{T}_{\frac{1}{\varepsilon}}^2 \right)$), and from
$\mathcal{D}' \left( \mathbb{T}^2_{\frac{1}{\varepsilon}} \right)$ into
$\mathcal{S}' (\varepsilon \mathbb{Z}^2)$ respectively.

\

We can extend the concept of weighted Besov spaces $B^s_{p, q, \ell}
(\varepsilon \mathbb{Z}^2)$, where $s \in \mathbb{R}$, $p, q \in [1, +
\infty]$, $\ell \in \mathbb{R}$, of tempered distributions defined on
$\varepsilon \mathbb{Z}^2$. In the following we consider a fixed dyadic
partition of the unity $(\varphi_j)_{j \geqslant - 1} \assign
(\varphi_j^{(2)})_{j \geqslant - 1}$ (recall the notation preceding Remark
\ref{remark:partition:conditions}) defined on $\mathbb{R}^2$ having the
properties required in Section \ref{section:besov:rd} (in particular Remark
\ref{remark:partition:conditions}). For the definition of Besov spaces on the
lattice $\varepsilon \mathbb{Z}^2$ with $\varepsilon = 2^{- N}$, we introduce
a suitable partition of unity $(\varphi_j^{\varepsilon})_{- 1 \leqslant j
\leqslant J_{\varepsilon}}$ for $\mathbb{T}^2_{\frac{1}{\varepsilon}}$ as
follows
\begin{equation}
  \varphi_j^{\varepsilon} (k) = \left\{\begin{array}{l}
    \varphi_j (k) \hspace{6em} \tmop{if} j < J_{\varepsilon}\\
    1 - \sum_{j = - 1}^{J_{\varepsilon} - 1} \varphi_j (k) \quad \tmop{if} j =
    J_{\varepsilon}
  \end{array}\right.  \label{eq:dydadictorus1},
\end{equation}
where
\begin{equation}
  J_{\varepsilon} = \min \left\{ j \geqslant - 1, \tmop{supp} (\varphi_j)
  \centernot{\subset} \left[ - \frac{\pi}{\varepsilon}, \frac{\pi}{\varepsilon}
  \right]^2 \right\} .
\end{equation}
If $f \in \mathcal{S}' (\varepsilon \mathbb{Z}^2)$, for any $- 1 \leqslant i
\leqslant J_{\varepsilon}$ we define
\[ \Delta_i^{\epsilon} f = \mathcal{F}^{- 1}_{\varepsilon} (\varphi_i^{\varepsilon}
   \cdot \mathcal{F}_{\varepsilon} (f)) = K_i^{\varepsilon} \ast_{\varepsilon}
   f, \]
where $K_i^{\varepsilon} = \mathcal{F}^{- 1}_{\varepsilon}
(\varphi_i^{\varepsilon})$.

\begin{remark}
  \label{remark:k:epsilon}It is important to note that when $j <
  J_{\varepsilon}$ the function $K^{\varepsilon}_j$ does not depend on
  $\varepsilon$. More precisely if $K_j = \mathcal{F}^{- 1} (\varphi_j)$
  (where $\varphi_j$ is the function in the dyadic partition of $\mathbb{R}^2$
  introduced above) then, for any $z \in \varepsilon \mathbb{Z}^2$, we have
  \[ K_j^{\varepsilon} (z) = K_j (z), \]
  (the only difference is that in principle the function $K_j^{\varepsilon}$
  is not defined for points outside $\varepsilon \mathbb{Z}^2$). Thanks to
  Remark \ref{remark:partition:conditions}, we have that there are $\bar{a} >
  0$ and $\bar{\beta} \in (0, 1)$ such that for any $z \in \varepsilon
  \mathbb{Z}^2$
  \[ | K^{\varepsilon}_j (z) | \lesssim 2^{2 j} \exp (- \bar{a} (1 + 2^{2 j} |
     z |^2)^{\bar{\beta} / 2}), \]
  where the constants hidden in the symbol $\lesssim$ do not depend both on $j
  \geqslant - 1$ and $0 < \varepsilon \leqslant 1$.
\end{remark}

\begin{remark}
  \label{remark:epsilon2}
  In some sense also the function
  $K^{\varepsilon}_{J_{\varepsilon}}$ does not depend on $\varepsilon$. Indeed, let us write $\varepsilon=c 2^{-J_{\varepsilon}}$, for  a suitable constant $c>0$ .  
 If we denote by $\bar{K} $ the function such that $\bar{K}= K_0^c$ (that is $K^{\varepsilon}_0$ when $\varepsilon = c$, i.e. $K^c_0 (z) = \frac{1}{(2 \pi)^2} \int_{\mathbb{T}^2_1} e^{- i (z
  \cdot q)} (1 - \varphi_{- 1} \left(\frac{q}{c}\right)) \mathd q$), we have that
  \begin{eqnarray}
    |z|^2K^{\varepsilon}_{J_{\varepsilon}} (z) & = & \mathcal{F}_{\varepsilon}^{-
    1} \left( \Delta_x\left(1 - \sum_{- 1 \leqslant j \leqslant J_{\varepsilon}-1}
    \varphi_j^{\varepsilon} (x) \right) \right) (z) \nonumber\\
    & = &2^{-2J_{\varepsilon}} \mathcal{F}_{\varepsilon}^{- 1} (\Delta\upsilon^{(2)} (2^{-
    J_{\varepsilon}} \cdot)) (z) \nonumber\\
   & = &2^{-2J_{\varepsilon}} \mathcal{F}^{-1} \left( \Delta\upsilon^{(2)} (2^{- J_{\varepsilon}} \cdot)
   \right) (z) \nonumber\\
    & = & \varepsilon^{-2}2^{-2J_{\varepsilon}}\mathcal{F}^{-1} \left( \Delta\upsilon^{(2)} \left(c^{-1}\cdot\right)
   \right) \left(\frac{z}{\varepsilon}\right) \nonumber\\
    & = & 
    \frac{2^{-2J_{\varepsilon}}}{\varepsilon^2}  \left|\frac{ z}{\varepsilon c}\right|^2 \bar{K} \left( \frac{z}{\varepsilon} \right)=  \frac{|z|^2}{\varepsilon^2}  \bar{K} \left( \frac{z}{\varepsilon} \right),
    \nonumber
  \end{eqnarray}
  where we used the function $\upsilon^{(2)}$ introduced in
  {\eqref{eq:greeku}} and the relation {\eqref{eq:greeku2}}.  In particular, since $\Delta\upsilon^{(2)} (x) =  - \Delta\varphi_{- 1} (x)$ is in
  Gevrey class for $\theta > 1$ with compact support, as a function on $\mathbb{R}^2$, we have that
  there is $\bar{a} > 0$ and $0 < \bar{\beta} < 1$ such that
  \begin{equation}\label{eq:Jepsilon}
     | K^{\varepsilon}_{J_{\varepsilon}} (z) | \lesssim
     \frac{1}{\varepsilon^2} \exp \left( - \bar{a} \left( 1 + \frac{| z
     |^2}{\varepsilon^2} \right)^{\beta / 2} \right)\quad \forall z\not= 0,
     \end{equation}
  where the constants hidden in the symbol $\lesssim$ do not depend on
  $\varepsilon = c 2^{- k}$ for any $k \in \mathbb{N}_0$.\\
  For $z=0$, we note that 
  \[|K^{\varepsilon}_{J_{\varepsilon}} (z)|=\left|\mathcal{F}_{\varepsilon}^{-
    1} \left( 1 - \sum_{- 1 \leqslant j \leqslant J_{\varepsilon}-1}
    \varphi_j^{\varepsilon} (x) \right) (0)\right|=\left|\int_{\mathbb{T}^2_{\frac{1}{\varepsilon}}}\left(1 - \sum_{- 1 \leqslant j \leqslant J_{\varepsilon}-1}
    \varphi_j^{\varepsilon}(x) \right)\mathd x\right|\leqslant\frac{1}{\epsilon^2}.
  \]
This means that, if we change the constants involved if necessary, inequality \eqref{eq:Jepsilon} holds also for $z=0$, and $|K^{\varepsilon}_{J_{\varepsilon}}|$ is bounded from above by an exponential function rescaling with $\varepsilon$ and not depending on it. 
\end{remark}

\begin{definition}
  For any $p, q \in [1, + \infty]$, $s, \ell \in \mathbb{R}$ we define the
  Besov space $B^s_{p, q, \ell} (\varepsilon \mathbb{Z}^2)$ to be the Banach space,
  subspace of $ \mathcal{S}' (\varepsilon \mathbb{Z}^2)$, for which $f \in
  B^s_{p, q, \ell} (\varepsilon \mathbb{Z}^2)$ if $f \in \mathcal{S}'
  (\varepsilon \mathbb{Z}^2)$
  \begin{equation}
    \| f \|_{B^s_{p, q, \ell} (\varepsilon \mathbb{Z}^2)} = \| \{ 2^{j s} \|
    \Delta_j ^{\varepsilon}f \|_{L^p_{\ell} (\varepsilon \mathbb{Z}^2)} \}_{- 1 \leqslant j
    \leqslant J_{\varepsilon}} \|_{\ell^q (\{ - 1, \cdots, J_{\varepsilon}
    \})} \label{eq:norm:besov:descrete},
  \end{equation}
  where $\| \cdot \|_{\ell^q (\{ - 1, \cdots, J_{\varepsilon} \})}$ is the
  natural norm of the small $\ell$ space $\ell^q (\{ - 1, \cdots,
  J_{\varepsilon} \})$, is finite.
\end{definition}

Many properties of standard Besov spaces can be extended to the case of Besov
spaces on the lattice $\varepsilon \mathbb{Z}^2$. We report here some of them
which will be useful in what follows (see {\cite{GuHof2018,Perkowski2019}} for
a discussion of other properties of Besov spaces on lattices).

\begin{remark}
  From now on when the symbols $\lesssim$ and $\sim$ appear in expressions
  involving norms, or other functionals defined on Besov spaces $B^s_{p, q,
  \ell} (\varepsilon \mathbb{Z}^2)$ or one of their generalization in Section
  \ref{section:besov:rz}, we assume that the constants hidden in the symbols
  $\lesssim$ and $\sim$ can be chosen uniformly with respect to $0 <
  \varepsilon \leqslant 1$. Without this assumption the following results are
  trivial.
\end{remark}

First we introduce the discrete Laplacian $\Delta_{\varepsilon \mathbb{Z}^2} :
L^p_{\ell} (\varepsilon \mathbb{Z}^2) \rightarrow L^p_{\ell} (\varepsilon
\mathbb{Z}^2)$ defined as
\[ \Delta_{\varepsilon \mathbb{Z}^2} f (z) = \sum_{i = 1, 2} \frac{f (z +
   \varepsilon e_i) + f (z - \varepsilon e_i) - 2 f (z)}{\varepsilon^2}, \]
where $\{ e_1, e_2 \}$ is the standard basis of $\mathbb{R}^2$. It is simple
to compute the Fourier transform of $\Delta_{\varepsilon \mathbb{Z}^2}$,
indeed
\begin{eqnarray}
  \mathcal{F}_{\varepsilon} (\Delta_{\varepsilon \mathbb{Z}^2} f) & = &
  \int_{\varepsilon \mathbb{Z}^2} \Delta_{\varepsilon \mathbb{Z}^2} f (z) e^{i
  q \cdot z} \mathd z = \sum_{j = 1, 2} \int_{\varepsilon \mathbb{Z}^2}
  \frac{f (z + \varepsilon e_j) + f (z - \varepsilon e_j) - 2 f (z)}{
  \varepsilon^2} e^{i q \cdot z} \mathd z \nonumber\\
  & = & \sum_{j = 1, 2} \frac{1}{\varepsilon^2} \int_{\varepsilon
  \mathbb{Z}^2} (e^{- i \varepsilon q_j} + e^{+ i \varepsilon q_j} - 2) f
  (z) e^{i q \cdot z} \mathd z \nonumber\\
  & = & - \sum_{j = 1, 2} \frac{4}{\varepsilon^2} \sin^2 \left(
  \frac{\varepsilon q_j}{2} \right) \mathcal{F}_{\varepsilon} (f) (q),
  \nonumber
\end{eqnarray}
where $q = (q_1, q_2) \in \mathbb{T}^2_{\frac{1}{\varepsilon}}$. Writing
\[ \sigma_{-\Delta_{\varepsilon \mathbb{Z}^2}} (q) = \sum_{j = 1, 2}
   \frac{4}{\varepsilon^2} \sin^2 \left( \frac{\varepsilon q_j}{2} \right), \]
and taking any $c > 0$, we get that the operator $(c - \Delta_{\varepsilon
\mathbb{Z}^2})$ is a self-adjoint, positive and invertible on $\mathcal{S}'
(\varepsilon \mathbb{Z}^2)$ with the power $s \in \mathbb{R}$ given by
\[ (c - \Delta_{\varepsilon \mathbb{Z}^2})^s (f) = \mathcal{F}^{-
   1}_{\varepsilon} \left( \left( c + \sigma_{-\Delta_{\varepsilon \mathbb{Z}^2}}
   (q) \right)^s \mathcal{F}_{\varepsilon} (f) (q) \right) . \]
Hereafter, whenever $f,g:\varepsilon \mathbb{Z}^{2} \rightarrow \mathbb{R}$ are two functions in $L^{p_1}_{\ell_1}(\varepsilon \mathbb{Z}^{2})$ and $L^{p_2}_{\ell_2}(\varepsilon \mathbb{Z}^{2})$, for some $p_1,p_2 \geq 1$ and $\ell_1,\ell_2\in\mathbb{R}$, respectively, we write 
\[(f \ast_{\varepsilon} g)(z)= \int_{\epsilon \mathbb{Z}^2} f(z-z') g(z') \mathd z'. \]

\begin{lemma}
  \label{lemma-fm-reg}Consider $s, \ell \in \mathbb{R}$, $p, q \in [1, +
  \infty]$. Furthermore assume that $\sigma_{\varepsilon} =
  \sigma_{\varepsilon} (q) : \mathbb{T}^2_{\frac{1}{\varepsilon}} \rightarrow
  \mathbb{R}$ is a differentiable function such that there is an $m \in
  \mathbb{R}$ such that for any $\beta = (\beta_1, \beta_2) \in
  \mathbb{N}^2_0$, for which $| \beta | = \beta_1 + \beta_2 \leqslant 4$, we
  have
  \[ | \partial_q^{\beta} \sigma_{\varepsilon} (q) | \leqslant C_{\beta} (1 +
     | q |^2)^{(m - | \beta |) / 2}, \]
  where $| \cdot | : \mathbb{T}^2_{\frac{1}{\varepsilon}} \simeq \left[ -
  \frac{\pi}{ \varepsilon}, \frac{\pi}{ \varepsilon} \right]^2 \rightarrow
  \mathbb{R}_+$ defined as $| q | = \sqrt{q_1^2 + q^2_2}$, for some constant
  $C_{\beta} > 0$ independent on $0 < \varepsilon \leqslant 1$. Then the
  operator
  \[ T^{\sigma_{\varepsilon}} f = \mathcal{F}^{- 1}_{\varepsilon}
     (\sigma_{\varepsilon} (q) (\mathcal{F}_{\varepsilon} f) (q)), \quad f \in
     B^s_{p, q, \ell} (\varepsilon \mathbb{Z}^2) \]
  is well defined from $B^s_{p, q, \ell} (\varepsilon \mathbb{Z}^2)$ into
  $B^{s - m}_{p, q, \ell} (\varepsilon \mathbb{Z}^2)$, and we have
  \[ \| T^{\sigma_{\varepsilon}} f \|_{B_{p, q, \ell}^{s - m} (\varepsilon
     \mathbb{Z}^2)} \lesssim \| f \|_{B_{p, q, \ell}^s (\varepsilon
     \mathbb{Z}^2)} . \]
\end{lemma}

\begin{proof}
  Our goal is to show that if $m \in \mathbb{R}$ then
  \[ \| \mathcal{F}_{\varepsilon}^{- 1} (\varphi_j^{\varepsilon}
     \sigma_{\varepsilon} \mathcal{F}_{\varepsilon} f) \|_{L_{\ell}^p
     (\varepsilon \mathbb{Z}^2)} \lesssim 2^{m j} \| \Delta^{\varepsilon}_j f
     \|_{L_{\ell}^p (\varepsilon \mathbb{Z}^2)} . \]
  We write
  \[ \mathcal{F}_{\varepsilon}^{- 1} (\varphi_j^{\varepsilon} \sigma_{\varepsilon}
     \mathcal{F}_{\varepsilon} f) = K^{\sigma_{\varepsilon}}_j
     \ast_{\varepsilon} \Delta_j^{\varepsilon} f ,\]
  where
  \[ K_j^{\sigma_{\varepsilon}} = \sum_{| i - j | < 2} \mathcal{F}^{-
     1}_{\varepsilon} (\varphi_i^{\varepsilon} \sigma_{\varepsilon}) . \]
  We want to prove that for any $r \in \mathbb{N}$, $\|
  K_j^{\sigma_{\varepsilon}} \|_{L_{- r}^1 (\varepsilon \mathbb{Z}^2)}
  \leqslant C (r) 2^{m j}$, which implies the statement in Lemma
  \ref{lemma-fm-reg} by the weighted Young's inequality. \\
  Let us first consider
  the case $j = - 1$. We \ can choose a smooth function $\psi : \mathbb{R}^2
  \rightarrow [0, 1]$ supported in the ball $B (0, 2)$ and such that $\psi (x)
  = 1$ if $| x | \leqslant 1$ and, writing
  \begin{equation} K_{- 1}^{\sigma_{\varepsilon}} (z) = \frac{1}{(2 \pi)^2}
     \int_{\mathbb{T}_{\frac{1}{\varepsilon}}^2} e^{i (z \cdot q)} \psi (q)
     \sigma_{\varepsilon} (q) \mathd q \label{eq:Ksigma}, \end{equation}
  since $\tmop{supp} (\mathcal{F}_{\varepsilon} (\Delta^{\varepsilon}_{- 1} f)) \subset B
  (0, 1)$, we have
  \[ \mathcal{F}_{\varepsilon}^{- 1} (\sigma_{\varepsilon}
     (\mathcal{F}_{\varepsilon} \Delta_{- 1}^{\varepsilon} f) (k)) = K_{-
     1}^{\sigma_{\varepsilon}} \ast_{\varepsilon} \Delta_{- 1}^{\varepsilon} f. \]
  Considering $\beta = (\beta_1, \beta_2), \alpha = (\alpha_1, \alpha_2) \in
  \mathbb{N}^2_0$, from equation {\eqref{eq:Ksigma}} we get that
  \begin{eqnarray}
    | (1 + | z |^2)^M K^{\sigma_{\varepsilon}}_{- 1} (z) | & = & \frac{1}{(2
    \pi)^2} \left| \int_{\mathbb{T}_{\frac{1}{\varepsilon}}^2} \left( \left( 1
    - \Delta_{\mathbb{T}_{\frac{1}{\varepsilon}}^2}^q \right)^M e^{i (z \cdot
    q)} \right) \psi (q)_{} \sigma_{\varepsilon} (q) \mathd q \right|
    \nonumber\\
    & = & \frac{1}{(2 \pi)^2} \left|
    \int_{\mathbb{T}_{\frac{1}{\varepsilon}}^2} e^{i (z \cdot q)} \left( 1 -
    \Delta_{\mathbb{T}_{\frac{1}{\varepsilon}}^2}^q \right)^M (\psi (q)_{}
    \sigma_{\varepsilon} (q)) \mathd q \right| \nonumber\\
    & \leqslant & \sum_{| \alpha | + | \beta | \leqslant 2 M} c_{\alpha, \beta}
    \int_{\mathbb{T}_{\frac{1}{\varepsilon}}^2} e^{i (z \cdot q)} |
    \partial^{\alpha}_q \psi (q) (\partial^{\beta}_q \sigma_{\varepsilon} (q))
    | \mathd q \leqslant C_M, \label{eq:boundsKsigma1} 
  \end{eqnarray}
  where $\Delta^q_{\mathbb{T}_{\frac{1}{\varepsilon}}^2}$ is the standard
  Laplacian on $\mathbb{T}_{\frac{1}{\varepsilon}}^2$ with respect to the $q =
  (q_1, q_2)$ variables, and $c_{\alpha, \beta}, C_M > 0$ are suitable
  constants independent on $0 < \varepsilon \leqslant 1$. Choosing $M = 2 +
  \ell$ \ we deduce that $\| K \|_{L_{\ell}^1 (\varepsilon \mathbb{Z}^2)}
  \leqslant \tilde{C}_3$ for some constant $\tilde{C}_3$ not depending on $0 <
  \varepsilon \leqslant 1$.\\
  
  For $0 \leqslant j < J_{\varepsilon}$ we proceed similarly after a
  rescaling argument: We can choose a smooth function $\tilde{\psi} :
  \mathbb{R}^2 \rightarrow [0, 1]$ supported in an annulus strictly contained in $\mathbb{T}_1^2$, which is
  identically $1$ on the support of $\sum_{| i - 1 | < 2} \varphi^{\varepsilon}_i
  = \sum_{| i - 1 | < 2} \varphi_i$ (when $\varepsilon \leqslant 1$), and
  writing,
  \[ K^{\sigma_{\varepsilon},\lambda} (z) = \frac{1}{(2 \pi)^2}
     \int_{\mathbb{T}^2_{\frac{1}{\lambda \varepsilon}}} e^{i (z \cdot q)}
     \tilde{\psi}  (q) \sigma_{\varepsilon} (\lambda q) \mathd q = \frac{1}{(2
     \pi)^2} \int_{\mathbb{R}^2} e^{i (z \cdot q)} \tilde{\psi}  (q)
     \sigma_{\varepsilon} (\lambda q) \mathd q, \]
  where $\lambda = 2^j$, since $\tmop{supp} (\tilde{\psi}) \subset
  \mathbb{T}^2_{\frac{1}{\lambda \varepsilon}}$, we get
  \[ \mathcal{F}_{\varepsilon}^{- 1} (\sigma_{\varepsilon} (q)
     (\mathcal{F}_{\varepsilon} \Delta_j^{\varepsilon} f) (q)) = \lambda^2
     K^{\sigma_{\varepsilon},\lambda} (\lambda z) \ast_{\varepsilon}
     \Delta_j^{\varepsilon} f, \]
     where we used the fact that 
     \[K^{\sigma_{\varepsilon},\lambda}(\lambda z)=\frac{1}{(2\pi\lambda)^2}\int_{\mathbb{R}^2}e^{i (z\cdot q)}\tilde{\psi}\left(\frac{q}{\lambda} \right)\sigma_{\varepsilon}(q)\mathd q, \]
     which is $\frac{1}{\lambda^2}$ times the Fourier transform of $\sigma_{\epsilon}$ projected on the set $\lambda \tmop{supp} (\tilde{\psi}) \supseteq \tmop{supp} (\varphi_j))$. Then, using the same argument in equation {\eqref{eq:boundsKsigma1}}, we
  obtain
  \begin{eqnarray}
    | (1 + | z |^2)^M K^{\sigma_{\varepsilon},\lambda} (z) | & \leqslant & \frac{1}{(2
    \pi)^2} \left| (1 + | z |^2)^M \int_{\mathbb{R}^2} e^{i (z \cdot q)}
    \tilde{\psi}  (q) \sigma_{\varepsilon} (\lambda q) \mathd q \right|
    \nonumber\\
    & \leqslant & \frac{1}{(2 \pi)^2} \left| \int_{\mathbb{R}^2} e^{i (z
    \cdot q)} (1 - \Delta_q)^M (\tilde{\psi}  (q) \sigma_{\varepsilon}
    (\lambda q)) \mathd q \right| \nonumber\\
    & \leqslant & \left| \sum_{| \alpha | + | \beta | = 2 M} c_{\alpha,
    \beta} \lambda^{| \beta |} \int_{\mathbb{R}^2} e^{i (z \cdot q)}
    \partial^{\alpha}_q \tilde{\psi}  (q) \partial_q^{\beta}
    \sigma_{\varepsilon} (\lambda q) \mathd q \right| \nonumber\\
    & \leqslant & \left| \sum_{| \alpha | + | \beta | = 2 M} C_{\beta}
    c_{\alpha, \beta} \lambda^{| \beta |} \int_{\mathbb{R}^2} e^{i (z \cdot
    q)} \partial^{\alpha}_q \tilde{\psi}  (q) (1 + | \lambda q |^2)^{(m - |
    \beta |) / 2} \mathd q \right| . \nonumber
  \end{eqnarray}
  Since $\tilde{\psi}$ is supported in an annulus $B (0, R_1) \backslash B (0,
  R_2)$ for some constant radius $R_1 > R_2 > 0$ we can estimate $(1 + |
  \lambda q |^2)^{(m - | \beta |) / 2}$ on $B (0, R_1) \backslash B (0, R_2)$
  getting
  \[ \lambda^{| \beta |} (1 + | \lambda q |^2)^{(m - | \beta |) / 2} \leqslant
     R_2^{m - | \beta |} \lambda^{| \beta |} | \lambda |^{m - | \beta |}
     \leqslant R_2^{m - | \beta |} | \lambda |^m \leqslant R_2^{m - | \beta |}
     2^{m j} . \]
  Finally for $j = J_{\varepsilon}$ we recall that, setting $\upsilon^{(2)} =
  1 - \varphi_{- 1}$ we have $\varphi_{J_{\varepsilon}}^{\varepsilon} (\cdot)
  = 1 - \sum_{j < J_{\varepsilon}} \varphi_i^{\varepsilon} = \upsilon^{(2)}
  (2^{- J_{\varepsilon}} \cdot)$. Now we again proceed as it was done in Remark \ref{remark:epsilon2} (for the proof of inequality \eqref{eq:Jepsilon}) and in the previous parts of the present proof with the
  scaling $\lambda = 2^{J_{\varepsilon}}$
  \[ T^{\sigma_{\varepsilon}} \Delta_{J_{\varepsilon}} f = 2^{J_{\varepsilon}}
     \bar{K}^{\sigma_{\varepsilon},\lambda} (2^{J_{\varepsilon}} \cdot)
     \ast_{\varepsilon} \Delta_{J_{\varepsilon}} f, \]
  where
  \[ \bar{K}^{\sigma_{\varepsilon},\lambda} (x) = \int_{\mathbb{T}^2_2}
     e^{i (z \cdot q)} v^{(2)}  (q) \sigma_{\varepsilon} (\lambda q) \mathd q
     = \int_{\mathbb{R}^2} e^{i (z \cdot q)} \upsilon^{(2)}  (q)
     \sigma_{\varepsilon} (\lambda q) \mathd q, \]
  where we used that $2^{J_{\varepsilon}} \varepsilon \leqslant 2$ and that
  $\Delta\upsilon^{(2)}$ is supported in an annulus $\mathcal{A}$ satisfying
  $\mathcal{A} \subseteq \mathbb{T}_2^2$. We can proceed as in the previous
  case in order to obtain a uniform bound on $\| \bar{K}_{\lambda}
  \|_{L^1_{r} (\varepsilon \mathbb{Z}^2)}$, for any $r\in \mathbb{N}$, from which we obtain the claim.
\end{proof}

\begin{remark}
  \label{remark:besov:lattice:equivalence1}Consider a positive function
  $\sigma_{\varepsilon}$ for which both $\sigma_{\varepsilon}$ and the
  $\frac{1}{\sigma_{\varepsilon}}$ satisfy the hypotheses of Lemma
  \ref{lemma-fm-reg} for some $m \in \mathbb{R}$ and $- m$ respectively, then,
  $s, \ell \in \mathbb{R}$, $p, q \in [1, + \infty]$ and $f \in B^s_{p, q,
  \ell} (\varepsilon \mathbb{Z}^2)$, by using the fact that
  $T^{\sigma_{\varepsilon}} = T^{\frac{1}{\sigma_{\varepsilon}},- 1}$ and
  applying Lemma \ref{lemma-fm-reg} to both $T^{\sigma_{\varepsilon}}$ and
  $T^{\frac{1}{\sigma_{\varepsilon}}}$, we get
  \[ \| T^{\sigma_{\varepsilon}} f \|_{B^{s - m}_{p, q, \ell} (\varepsilon
     \mathbb{Z}^2)} \lesssim \| f \|_{B^s_{p, q, \ell} (\varepsilon
     \mathbb{Z}^2)} = \left\| T^{\frac{1}{\sigma_{\varepsilon}}}
     (T^{\sigma_{\varepsilon}} f) \right\|_{B^s_{p, q, \ell} (\varepsilon
     \mathbb{Z}^2)} \lesssim \| T^{\sigma_{\varepsilon}} (f) \|_{B^{s - m}_{p,
     q, \ell} (\varepsilon \mathbb{Z}^2)} . \]
  This implies that
  \[ \| T^{\sigma_{\varepsilon}} f \|_{B^{s - m}_{p, q, \ell} (\varepsilon
     \mathbb{Z}^2)} \sim \| f \|_{B^s_{p, q, \ell} (\varepsilon \mathbb{Z}^2)}
     . \]
\end{remark}

\begin{theorem}
  \label{theorem:besov:weight}For any $p, q \in [1, + \infty]$, $s, \ell \in
  \mathbb{R}$ we have that
  \begin{equation}
    \| f \|_{B^s_{p, q, \ell} (\varepsilon \mathbb{Z}^2)} \sim \| f \cdot
    \rho_{\ell} \|_{B^s_{p, q, 0} (\varepsilon \mathbb{Z}^2)}, \quad f \in
    B^s_{p, q, \ell} (\varepsilon \mathbb{Z}^2).
    \label{eq:besov:lattice:equivalence1}
  \end{equation}
  Furthermore for any $m \in \mathbb{R}$ we get
  \begin{equation}
    \| (1 - \Delta_{\varepsilon \mathbb{Z}^2})^{m / 2} f \|_{B^{s - m}_{p, q,
    \ell} (\varepsilon \mathbb{Z}^2)} \sim \| f \|_{B^s_{p, q, \ell}
    (\varepsilon \mathbb{Z}^2)} .\label{eq:besov:lattice:equivalence2}
  \end{equation}
\end{theorem}

\begin{proof}
  The proof of the equivalence {\eqref{eq:besov:lattice:equivalence1}} follows
  the same line of the analogous statement on $\mathbb{R}^d$ (see, e.g.,
  Theorem 6.5 of {\cite{Triebel2006}}) and it is consequence of the fact that,
  for any $\ell \in \mathbb{R}$, $\alpha \in \mathbb{N}_0^2$, there are some
  constants $C_{\alpha, \ell}, C_{\ell} > 0$ for which
  \[ | \partial^{\alpha} \rho_{\ell} (x) | \leqslant C_{\alpha, \ell}
     \rho_{\ell} (x), \quad x \in \mathbb{R}^2 \]
  and
  \[ \rho_{\ell} (x) \leqslant C_{\ell} \rho_{\ell} (y) (1 + | x - y |)^{|
     \ell |} . \]
  The equivalence {\eqref{eq:besov:lattice:equivalence2}} follows from Lemma
  \ref{lemma-fm-reg} and Remark \ref{remark:besov:lattice:equivalence1}, by
  noting that
  \[ \sigma_{\varepsilon} (q) = \mathcal{F}_{\varepsilon} ((1 -
     \Delta_{\varepsilon \mathbb{Z}^2})^m) (q) = \left( 1 + \sigma_{-
     \Delta_{\varepsilon \mathbb{Z}^2}} (q) \right) \leqslant (1 +
     C_{\tmop{sign} (m)} | q |^2)^{m / 2}, \]
  where
  \[ C_{+ 1} = 1, \quad C_{- 1} = \frac{2}{\pi}, \]
  and similarly for the derivatives of $\sigma_{\varepsilon} (q)$ with respect
  to $q$.
\end{proof}

We now extend the standard estimates on products between a function and a
distribution with suitable Besov regularity. More precisely if $f \in
\mathcal{S} (\varepsilon \mathbb{Z}^2)$ and $g \in B^s_{p, q, \ell}
(\varepsilon \mathbb{Z}^2)$ we can define the product
\begin{equation}
  f g = f \cdot g = \sum_{- 1 \leqslant i \leqslant J_{\varepsilon}, - 1
  \leqslant j \leqslant J_{\varepsilon}} \Delta_i^{\varepsilon} f \Delta_j^{\varepsilon} g,
  \label{eq:product}
\end{equation}
which is a well defined element of $B^s_{p, q, \ell} (\varepsilon
\mathbb{Z}^2)$.

\begin{theorem}
  \label{theorem:besov:product}Consider $s_1 \leqslant 0, s_2 > 0$ such that
  $s_1 + s_2 > 0$, $p_1, p_2, q_1, q_2 \in [1, + \infty]$ , $\frac{1}{p_1} +
  \frac{1}{p_2} = \frac{1}{p_3} \leqslant 1$, and $\frac{1}{q_1} +
  \frac{1}{q_2} = \frac{1}{q_3} \leqslant 1$ then the product $\Pi (f, g) = f
  \cdot g$ defined on $\mathcal{S}' (\varepsilon \mathbb{Z}^2) \times
  \mathcal{S} (\varepsilon \mathbb{Z}^2)$ into $\mathcal{S}' (\varepsilon
  \mathbb{Z}^2)$ can be extended in a unique continuous way to a map
  \[ \Pi \; : B^{s_1}_{p_1, q_1, \ell_1} (\varepsilon \mathbb{Z}^2) \times
     B^{s_2}_{p_2, q_2, \ell_2} (\varepsilon \mathbb{Z}^2) \rightarrow
     B^{s_1}_{p_3, q_3, \ell_1 + \ell_2} (\varepsilon \mathbb{Z}^2) . \]
  Furthermore for any $f \in B^{s_1}_{p_1, q_1, \ell_1} (\varepsilon
  \mathbb{Z}^2)$ and $g \in B^{s_2}_{p_2, q_2, \ell_2} (\varepsilon
  \mathbb{Z}^2)$ we have
  \[ \| \Pi (f, g) \| = \| f \cdot g \|_{B^{s_1}_{p_3, q_3, \ell_1 + \ell_2}
     (\varepsilon \mathbb{Z}^2)} \lesssim \| f \|_{B^{s_1}_{p_1, q_1, \ell_1}
     (\varepsilon \mathbb{Z}^2)} \| g \|_{\tmmathbf{} B^{s_2}_{p_2, q_2,
     \ell_2} (\varepsilon \mathbb{Z}^2)} . \]
\end{theorem}

\begin{proof}
  The proof of the present theorem is given in Lemma 4.2 of
  {\cite{Perkowski2019}} in the unweighted (i.e. $\ell_1 = \ell_2 = 0$) and
  infinite integrability (i.e. $p_1 = q_1 = p_2 = q_2 = \infty$) case. The
  proof in {\cite{Perkowski2019}}, is based on the fact that the Fourier
  transform of the product $\left( \sum_{i < j} \Delta_i f \right) \Delta_j g$
  is supported in the annulus $2^j \mathcal{A} \cap
  \mathbb{T}^2_{\frac{1}{\varepsilon}}$ (where $\mathcal{A} \subset B (0, R_1)
  \backslash B (0, R_1)$ for suitable $R_1 > R_2 > 0$ real constants), and the
  Fourier transform of the product $\Delta_i f \Delta_j g$, where $| i - j | < 2$,
  is supported in $B (0, 2^j R_3) \cap \mathbb{T}^2_{\frac{1}{\varepsilon}}$,
  for a suitable constant $R_3 > 0$ (see Lemma A.6 in {\cite{Perkowski2019}}).
  Exploiting these observations the proof follows the same lines of the
  analogous result for Besov spaces on $\mathbb{R}^d$ (see, e.g., Chapter 2 of
  {\cite{Bahouri2011}} and the proof of Theorem 3.17 of
  {\cite{Mourrat_Weber2017}}).
\end{proof}

\begin{theorem}
  \label{theorem:besov:lattice:embedding}Consider $p_1, p_2, q_1, q_2 \in [1,
  + \infty]$, $p_1 \leq p_2$, $s_1 > s_2$ and $\ell_1, \ell_2 \in \mathbb{R}$ if
  \[ \ell_1 \leqslant \ell_2, \quad s_1 - \frac{d}{p_1} \geqslant s_2 -
     \frac{d}{p_2}, \]
  then $B^{s_1}_{p_1, q_1, \ell_1} (\varepsilon \mathbb{Z}^2)$ is continuously
  embedded in $B^{s_2}_{p_2, q_2, \ell_2} (\varepsilon \mathbb{Z}^2)$, and the
  norm of the embedding is uniformly bounded in $0 < \varepsilon \leqslant 1$.
\end{theorem}

\begin{proof}
  By Theorem \ref{theorem:besov:weight}, we have that $I_{\ell_i} :
  B^{s_i}_{p_i, q_i, \ell_i} (\varepsilon \mathbb{Z}^2) \rightarrow
  B^{s_i}_{p_i, q_i, 0} (\varepsilon \mathbb{Z}^2)$, defined as
  \[ I_{\ell_i} (f) = \rho_{\ell_i} f ,\]
  is an isomorphism. This means that proving that $B^{s_1}_{p_1, q_1, \ell_1}
  (\varepsilon \mathbb{Z}^2) \hookrightarrow B^{s_2}_{p_2, q_2, \ell_2}
  (\varepsilon \mathbb{Z}^2)$ continuously is equivalent to prove that the map
  $I_{\ell_2} \circ I_{\ell_1}^{- 1}$ is a continuous embedding of
  $B^{s_1}_{p_1, q_1, 0} (\varepsilon \mathbb{Z}^2)$ into $B^{s_2}_{p_2, q_2,
  0} (\varepsilon \mathbb{Z}^2)$. On the other hand we have that
  \begin{equation}
    I_{\ell_2} \circ I_{\ell_1}^{- 1} (g) = \rho_{\ell_2 - \ell_1} g, \quad g
    \in B^{s_1}_{p_1, q_1, 0} (\varepsilon \mathbb{Z}^2).
    \label{eq:besov:lattice:weight1}
  \end{equation}
  Since $\ell_2 - \ell_1 \geqslant 0$ and so we have that $\rho_{\ell_2 -
  \ell_1} \in B^k_{\infty, \infty, 0} (\varepsilon \mathbb{Z}^2)$, for any $k
  \in \mathbb{N}_0$, by Theorem \ref{theorem:besov:product} $I_{\ell_2} \circ
  I_{\ell_1}^{- 1}$ is a continuous operator from $B^{s_1}_{p_1, q_1, 0}
  (\varepsilon \mathbb{Z}^2)$ into itself and with norm
  \[ \| I_{\ell_1} \circ I_{\ell_2}^{- 1} \|_{\mathcal{L} (B^{s_1}_{p_1, q_1,
     0} (\varepsilon \mathbb{Z}^2), B^{s_1}_{p_1, q_1, 0} (\varepsilon
     \mathbb{Z}^2))} \lesssim \| \rho_{\ell_2 - \ell_1} \|_{B^k_{\infty,
     \infty, 0} (\varepsilon \mathbb{Z}^2)} \lesssim \| \rho_{\ell_2 - \ell_1}
     \|_{B^k_{\infty, \infty, 0} (\mathbb{R}^2)}, \]
  which is uniformly bounded in $0 < \varepsilon \leqslant 1$. This implies
  that we can reduce the proof of the present theorem to the proof of the fact
  that $B^{s_1}_{p_1, q_1, 0} (\varepsilon \mathbb{Z}^2) \hookrightarrow
  B^{s_2}_{p_2, q_2, 0} (\varepsilon \mathbb{Z}^2)$ and the inclusion is
  continuous.
  
  This last statement can be proved following the argument for the analogous
  statement on $\mathbb{R}^d$ (see, e.g., Proposition 2.20 in
  {\cite{Bahouri2011}}).
\end{proof}

\subsection{Difference spaces and Besov spaces on
lattice}\label{section:besov:difference}

In this subsection we discuss a different space of functions with fractional
positive regularity on $\varepsilon \mathbb{Z}^2$. More precisely we introduce
the following Banach function spaces:

\begin{definition}
  \label{definition:difference}Consider $0 \leqslant s \leqslant 1$, $p \in
  [1, + \infty]$ and $\ell \in \mathbb{R}$ we define the (difference)
  fractional Sobolev space $W^{s, p}_{\ell} (\varepsilon \mathbb{Z}^2)$ the
  Banach space of measurable functions $f : \varepsilon \mathbb{Z}^2
  \rightarrow \mathbb{R}$ such that
  \begin{equation}
    \| f \|_{W^{s, p}_{\ell} (\varepsilon \mathbb{Z}^2)} = \| f \|_{L^p_{\ell}
    (\varepsilon \mathbb{Z}^2)} + \sup_{k \in \varepsilon \mathbb{Z}^2, 0 < |
    k | \leqslant 1} | k |^{- s} \left( \varepsilon^2 \sum_{x \in \varepsilon
    \mathbb{Z}^2} \rho^p_{\ell} (x)^{} | \nobracket f (x + k_{}) - f (x
    \nobracket) |^p \right)^{1 / p} \label{eq:no} 
  \end{equation}
  is finite.
\end{definition}

\begin{remark}
  \label{remark:equivalence:difference}In the case $s = 1$ an equivalent norm
  on $W^{1, p}_{\ell}$ is given by the following expression
  \begin{equation}
    \| f \|_{W^{s, p}_{\ell} (\varepsilon \mathbb{Z}^2)} \sim \| f
    \|_{L^p_{\ell} (\varepsilon \mathbb{Z}^2)} + \left( \varepsilon^2 \sum_{x
    \in \varepsilon \mathbb{Z}^2, i = 1, 2} (\rho_{\ell} (x))^p \varepsilon^{-
    p s} | f (x + \varepsilon e_i) - f (x) |^p \right)^{1 / p},
    \label{eq:difference:norm:equivalence}
  \end{equation}
  where $\{ e_i \}_{i = 1, 2}$ is the standard basis of $\mathbb{Z}^2$. Indeed
  we have that for any $k \in \varepsilon \mathbb{Z}^2$ of the form $k = (k_1,
  0)$, $\varepsilon < k_1 \leqslant 1$ we have
  \[ \begin{array}{rl}
       &\left( \varepsilon^2 \sum_{x \in \varepsilon \mathbb{Z}^2}
       (\rho_{\ell} (x))^p | k |^{- p} | f (x + k) - f (x) |^p \right)^{1
       / p} = \frac{1}{| k_1 |} \| f (k_1 e_i + \cdot) - f (\cdot)
       \|_{L^p_{\ell} (\varepsilon \mathbb{Z}^2)} =\\
       \leqslant & \frac{1}{| k_1 |} \sum_{1 \leqslant h \leqslant
       \frac{k_1}{\varepsilon}} \| f (\varepsilon h e_i + \cdot) - f
       (\varepsilon (h - 1) e_i + \cdot) \|_{L^p_{\ell} (\varepsilon \mathbb{Z}^2)} \sim
       \frac{k_1}{k_1 \varepsilon} \| f (\varepsilon e_i + \cdot) - f (\cdot)
       \|_{L^p_{\ell} (\varepsilon \mathbb{Z}^2)}\\
       \sim & \left( \varepsilon^2 \sum_{x \in \varepsilon \mathbb{Z}^2, i = 1,
       2} (\rho_{\ell} (x))^p \varepsilon^{- p s} | f (x + \varepsilon e_i) -
       f (x) |^p \right)^{1 / p} .
     \end{array} \]
  A similar argument can be done for $k = (0, k_2)$ and then, combining these
  two results, we get the equivalence
  {\eqref{eq:difference:norm:equivalence}}, since the converse inequality is
  obvious. 
\end{remark}

\begin{remark}
  \label{remark:Lipschitz}Let $F : \mathbb{R} \rightarrow \mathbb{R}$ be a
  globally Lipschitz function such that $F (0) = 0$, then it is easy to see
  that if $f \in W^{s, p}_{\ell} (\varepsilon \mathbb{Z}^2)$ then also $F (f)
  \in W^{s, p}_{\ell} (\varepsilon \mathbb{Z}^2)$ and
  \[ \| F (f) \|_{W^{s, p}_{\ell} (\varepsilon \mathbb{Z}^2)} \lesssim \| f
     \|_{W^{s, p}_{\ell} (\varepsilon \mathbb{Z}^2)} . \]
\end{remark}

Hereafter we denote by $\tau_z : \mathcal{S}' (\varepsilon \mathbb{Z}^2)
\rightarrow \mathcal{S}' (\varepsilon \mathbb{Z}^2)$ the translation with
respect to the vector $z \in \varepsilon \mathbb{Z}^2$
\[ \tau_z (f) (z') = f (z + z'), \quad z, z' \in \varepsilon \mathbb{Z}^2, f
   \in \mathcal{S}' (\varepsilon \mathbb{Z}^2). \]
Furthermore if $k \in \varepsilon \mathbb{Z}^2$ we introduce the notation
\[ D_k f = (\tau_k f - f) . \]
Finally we write
\[ \nabla^s_{\varepsilon} f (x) = \left( \begin{array}{c}
     \varepsilon^{- s} (\tau_{\varepsilon e_1} f (x) - f (x))\\
     \varepsilon^{- s} (\tau_{\varepsilon e_2} f (x) - f (x))
   \end{array} \right), \]
where $\{ e_i \}_{i = 1, 2}$ is the standard basis of $\mathbb{Z}^2$. We also
use the notation $\nabla_{\varepsilon} f = \nabla^1_{\varepsilon} f$.

\

We want to establish a relation between the space $W^{s, p}_{\ell}
(\varepsilon \mathbb{Z}^2)$ and the analogous Besov space $B^{s \pm
\delta}_{p, p, \ell} (\varepsilon \mathbb{Z}^2)$ (where $\delta > 0$).

\begin{theorem}
  \label{theorem:sobolev_besov}Consider $0 \leqslant s \leqslant 1$, $p, q \in
  [1, + \infty]$, $\ell \in \mathbb{R}$, and $\delta_1 > 0$ then for any $f
  \in B^s_{p, q, \ell} (\varepsilon \mathbb{Z}^2)$ we have
  \begin{equation}
    \| f \|_{B^{s - \delta_1}_{p, p, \ell} (\varepsilon \mathbb{Z}^2)}
    \lesssim \| f \|_{W^{s, p}_{\ell} (\varepsilon \mathbb{Z}^2)} \lesssim \|
    f \|_{B^{s + \delta_1}_{p, p, \ell} (\varepsilon \mathbb{Z}^2)} .
    \label{eq:inequality:difference}
  \end{equation}
\end{theorem}

In order to prove Theorem \ref{theorem:sobolev_besov} we establish the
following lemma.

\begin{lemma}
  \label{lemma:first:inequality}Consider $0 \leqslant s \leqslant 1$, $p \in
  [1, + \infty]$, $\ell \in \mathbb{R}$, and $\delta_1 > 0$ then for any $f
  \in W^{s, p}_{\ell} (\varepsilon \mathbb{Z}^2)$ we have
  \[ \| f \|_{B^{s - \delta_1}_{p, p, \ell} (\varepsilon \mathbb{Z}^2)}
     \lesssim \| f \|_{W^{s, p}_{\ell} (\varepsilon \mathbb{Z}^2)} . \]
\end{lemma}

\begin{proof}
  Let $\tilde{\varphi}_i^{\varepsilon} : \mathbb{R} \rightarrow [0, 1]$ be a
  dyadic partition of unity of $\mathbb{T}_{\frac{1}{\varepsilon}}^1$ as the
  one introduced in equation {\eqref{eq:dydadictorus1}} for
  $\mathbb{T}^2_{\frac{1}{\varepsilon}}$.  
Since for $i<J_{\varepsilon}$ the support of $\tilde{\varphi}_i^{\varepsilon}$ is strictly contained in $\left[-\frac{1}{\epsilon},\frac{1}{\epsilon} \right]$, we can suppose that the functions $\tilde{\varphi}_i^{\varepsilon} : \mathbb{R} \rightarrow [0, 1]$ do not depend on $\varepsilon$ when $i < J_{\varepsilon}$  and also
  \[ \tmop{supp} (\tilde{\varphi}_i^{\varepsilon} (q_1)
     \tilde{\varphi}_j^{\varepsilon} (q_2)) \subset \tmop{supp} \left( \sum_{|
     k - \max (i, j) | < 2} \varphi_k^{\varepsilon} \right), \quad \tmop{supp}
     (\varphi_k^{\varepsilon}) \subset \tmop{supp} \left( \sum_{| \max (i, j)
     - k | < 2} \tilde{\varphi}_i (q_1) \tilde{\varphi}_j^{\varepsilon} (q_2)
     \right), \]
  where $\varphi_k^{\varepsilon}$ is the dyadic partition of unity of
  $\mathbb{T}^2_{\frac{1}{\varepsilon}}$ introduced in equation
  {\eqref{eq:dydadictorus1}}. Consider also some constants $k_{\ell, j} =
  (k_{\ell, j, 1}, k_{\ell, j, 2}) \in \varepsilon \mathbb{Z}^2$ forming a
  sequence for which $\varepsilon \leqslant | k_{\ell, j} | \leqslant 1$ and
  there is a constant $C_1 > 0$ independent of $0 < \varepsilon \leqslant 1$
  for which
  \begin{equation}
    \left| \sin \left( \frac{k_{\ell, j} \cdot q_{}}{2} \right) \right|
    \geqslant C_1, \label{eq:sin:inequality}
  \end{equation}
  for any $q = (q_1, q_2) \in \tmop{supp} (\tilde{\varphi}_j (q_1)
  \tilde{\varphi}_{\ell} (q_2))$ for $j, \ell \centernot{=} - 1$. It is important to
  note that the inequality {\eqref{eq:sin:inequality}} implies that there is a
  $C_2 > 0$, independent of $\varepsilon > 0$, for which
  \begin{equation}
    | k_{\ell, j} | > \frac{C_2}{| q |}, \label{eq:sin:inequality2}
  \end{equation}
  for $q = (q_1, q_2) \in \tmop{supp} (\tilde{\varphi}_j (q_1)
  \tilde{\varphi}_{\ell} (q_2))$, $q \centernot{=} 0$. Let $\sigma_{\varepsilon} :
  \mathbb{T}^2_{\frac{1}{\varepsilon}} \rightarrow \mathbb{R}$ be a smooth
  function defined as
  \[ \sigma_{\varepsilon} (q) = \sum_{- 1 \leqslant \ell \leqslant
     J_{\varepsilon}} \sum_{- 1 \leqslant j \leqslant J_{\varepsilon}}
     \tilde{\varphi}_j^{\varepsilon} (q_1)
     \tilde{\varphi}_{\ell}^{\varepsilon} (q_2) \frac{(e^{ i k_{\ell, j} \cdot
     q} - 1)}{| k_{\ell, j} |^s \left( 1 + 4 | k_{\ell, j} |^{- 2 s} \sin^2
     \left( \frac{k_{\ell, j} \cdot q}{2} \right) \right)} \]
  We have that
  \[ \mathcal{F}^{- 1}_{\varepsilon} (\sigma_{\varepsilon}
     \tilde{\varphi}_j^{\varepsilon} \tilde{\varphi}_{\ell}^{\varepsilon} 
     \mathcal{F}_{\varepsilon} (| k_{\ell, j} |^{- s} D_{k_{\ell, j}} f)) =
     \mathcal{F}^{- 1}_{\varepsilon} (\tilde{\varphi}_{\ell}^{\varepsilon}
     \tilde{\varphi}_j^{\varepsilon}  \mathcal{F}_{\varepsilon} (f)) -
     \mathcal{F}^{- 1}_{\varepsilon} (\bar{\sigma}_{\varepsilon} 
     \tilde{\varphi}_j^{\varepsilon} \tilde{\varphi}_{\ell}^{\varepsilon}
     \mathcal{F}_{\varepsilon} (f)), \]
  where
  \[ \bar{\sigma}_{\varepsilon} (q) = \sum_{- 1 \leqslant \ell \leqslant
     J_{\varepsilon}} \sum_{- 1 \leqslant j \leqslant J_{\varepsilon}}
     \tilde{\varphi}_j^{\varepsilon} (q_1)
     \tilde{\varphi}_{\ell}^{\varepsilon} (q_2) \frac{1}{\left( 1 + 4 |
     k_{\ell, j} |^{- 2 s} \sin^2 \left( \frac{k_{\ell, j} \cdot q}{2} \right)
     \right)} . \]
  By inequalities {\eqref{eq:sin:inequality}} and
  {\eqref{eq:sin:inequality2}}, the functions $\sigma_{\varepsilon} (q)$ and
  $\bar{\sigma}_{\varepsilon} (q)$ are built in such a way that for any $\beta
  \in \mathbb{N}_0^2$, such that $| \beta | \leqslant 4$, there are some
  constants $C_{\beta} > 0$ (independent of $0 < \varepsilon \leqslant 1$) for
  which
  \[ | \partial^{\beta}_q \sigma_{\varepsilon} (q) | \leqslant C_{\beta} (1 +
     | q |)^{- s - | \beta |}, \quad | \partial^{\beta}_q
     \bar{\sigma}_{\varepsilon} (q) | \leqslant C_{\beta} (1 + | q |)^{- 2 s -
     | \beta |} . \]
  By Lemma \ref{lemma-fm-reg}, this implies that
  \begin{eqnarray}
    \| f \|_{B^{s - \delta - \delta'}_{p, p, \ell} (\varepsilon \mathbb{Z}^2)}
    & \leqslant & \sum_{r, j} \| T_{\sigma_{\varepsilon}}
    (\mathcal{F}_{\varepsilon}^{- 1} (\tilde{\varphi}_j^{\varepsilon}
    \tilde{\varphi}_r^{\varepsilon}  \mathcal{F}_{\varepsilon} (| k_{r, j}
    |^{- s}_{} D_{k_{r, j}} f)))_{} \|_{B^{s - \delta - \delta'}_{p, p, \ell}
    (\varepsilon \mathbb{Z}^2)} \nonumber\\
    &  & + \sum_{r, j} \| T_{\bar{\sigma}_{\varepsilon}}
    (\mathcal{F}_{\varepsilon}^{- 1} (\tilde{\varphi}_j^{\varepsilon}
    \tilde{\varphi}_r^{\varepsilon}  \mathcal{F}_{\varepsilon} (f)))_{}
    \|_{B^{s - \delta - \delta'}_{p, p, \ell} (\varepsilon \mathbb{Z}^2)}
    \nonumber\\
    & \lesssim & \sum_{r, j} \| \mathcal{F}_{\varepsilon}^{- 1}
    (\tilde{\varphi}_j^{\varepsilon} \tilde{\varphi}_r^{\varepsilon} 
    \mathcal{F}_{\varepsilon} (| k_{r, j} |^{- s}_{} D_{k_{r, j}} f))_{}
    \|_{B^{- \delta - \delta'}_{p, p, \ell} (\varepsilon \mathbb{Z}^2)} 
    \nonumber\\
    &  & + \sum_{r, j} \| \mathcal{F}_{\varepsilon}^{- 1}
    (\tilde{\varphi}_j^{\varepsilon} \tilde{\varphi}_r^{\varepsilon} 
    \mathcal{F}_{\varepsilon} (f))_{} \|_{B^{- s - \delta - \delta}_{p, p,
    \ell} (\varepsilon \mathbb{Z}^2)} \nonumber\\
    & \lesssim & \sum_m \sum_{| \max (r, j) - m | < 2} 2^{- m (\delta +
    \delta')} \| \mathcal{F}^{- 1}_{\varepsilon} (\varphi_m
    \mathcal{F}_{\varepsilon} (| k_{r, j} |^{- s}_{} D_{k_{r, j}} f))
    \|_{L^p_{\ell} (\varepsilon \mathbb{Z}^2)} \nonumber\\
    &  & + \sum_m \sum_{| \max (r, j) - m | < 2} 2^{- m (\delta + \delta')}
    \| \mathcal{F}_{\varepsilon}^{- 1} (\varphi_m  \mathcal{F}_{\varepsilon}
    (f))_{} \|_{L^p_{\ell} (\varepsilon \mathbb{Z}^2)} \nonumber\\
    & \lesssim & \sum_m 2^{- \delta' m} \sum_{| \max (r, j) - m | < 2} \| |
    k_{r, j} |^{- s}_{} D_{k_{r, j}} f \|_{B^{- \delta}_{p, p, \ell}
    (\varepsilon \mathbb{Z}^2)}  \nonumber\\
    &  & + \sum_m 2^{- \delta' m} m^2 \| f \|_{B^{- \delta}_{p, p, \ell}
    (\varepsilon \mathbb{Z}^2)} \nonumber\\
    & \lesssim & \sum_m 2^{- \delta' m} m^2 (\max_{\varepsilon < | k |
    \leqslant 1} \| | k_{r, j} |^{- s}_{} D_{k_{r, j}} f \|_{B^{- \delta}_{p,
    p, \ell} (\varepsilon \mathbb{Z}^2)}) \nonumber\\
    &  & + \| f \|_{B^{- \delta}_{p, p, \ell} (\varepsilon \mathbb{Z}^2)}
    \nonumber\\
    & \lesssim & \| f \|_{W^{s, p}_{\ell} (\varepsilon \mathbb{Z}^2)},
    \nonumber
  \end{eqnarray}
  where we used that $\| g \|_{B^{- \delta}_{p, p, \ell} (\varepsilon
  \mathbb{Z}^2)} \lesssim \| g \|_{L^p_{\ell} (\varepsilon \mathbb{Z}^2)}$ and
  the constants hidden in the symbol $\lesssim$ do not depend on $0 <
  \varepsilon \leqslant 1$. If we choose $\delta, \delta' > 0$ such that
  $\delta + \delta' = \delta_1$ we get the thesis.
\end{proof}

\begin{remark}
  \label{remark:first:inequality}The result of Lemma
  \ref{lemma:first:inequality} can be extended also to the case of $s
  \geqslant 1$ in the following way: for any $h \in \varepsilon \mathbb{Z}^2$,
  $n \in \mathbb{N}$, and $f \in L^p_{\ell} (\varepsilon \mathbb{Z}^2)$ define
  the operator
  \[ D_h^n f (z) = \sum_{m = 0}^n \left( \begin{array}{c}
       n\\
       m
     \end{array} \right) (- 1)^{m + 1} f \left( x + \left( \frac{n}{2} - m
     \right) h \right), \]
  when $n$ is even and $D^n_h f = D^{n - 1}_h (D_h f)$ when $n$ is odd. Then
  the norm of $W^{s, p}_{\ell} (\varepsilon \mathbb{Z}^2)$ when $n - 1 < s
  \leqslant n$ (where $n \in \mathbb{N}$) is given by, see e.g. Section 1.2.2
  equation (3) of {\cite{Triebel1992}},
  \begin{equation}
    \tmcolor{black}{\| f \|_{W^{s, p}_{\ell} (\varepsilon \mathbb{Z}^2)} = \|
    f \|_{L^p_{\ell} (\varepsilon \mathbb{Z}^2)} + \sum_{m = 1}^n \sup_{h \in
    \varepsilon \mathbb{Z}^2, \varepsilon \leqslant | h | \leqslant 1}
    \frac{\| D^m_h f \|_{L^p_{\ell}}}{| h |^{m \wedge s}} .
    \label{eq:higher:difference}}
  \end{equation}

  By noting that
  \[ \mathcal{F}_{\varepsilon} (D^n_h) \sim (\cos (h \cdot q) -
     1)^{\frac{n}{2}}, \]
  when $n$ is even, and
  \[ \mathcal{F}_{\varepsilon} (D^n_h) \sim (\cos (h \cdot q) - 1)^{\frac{n -
     1}{2}} (e^{i (h \cdot q)} - 1), \]
  when $n$ is odd, by using Lemma \ref{lemma-fm-reg} and a suitable extension
  of the proof of Lemma \ref{lemma:first:inequality}, we get that, for any $s
  \geqslant 0$, $p \in [1, + \infty]$ and $\delta > 0$
  \[ \| f \|_{B^{s - \delta}_{p, p, \ell} (\varepsilon \mathbb{Z}^2)}
     \lesssim \| f \|_{W^{s, p}_{\ell} (\varepsilon \mathbb{Z}^2)}, \]
  where the norm $\| \cdot \|_{W^{s, p}_{\ell} (\varepsilon \mathbb{Z}^2)}$ is
  defined in equation {\eqref{eq:higher:difference}}.
\end{remark}

\begin{lemma}
  \label{lemma:second:inequality}Consider $0 < s < 1$, $p \in [1, + \infty]$,
  and $\ell \in \mathbb{R}$ then for $f \in B^s_{p, p, \ell} (\varepsilon
  \mathbb{Z}^2)$ we have
  \[ \| f \|_{W^{s, p}_{\ell} (\varepsilon \mathbb{Z}^2)} \lesssim \| f
     \|_{B^s_{p, p, \ell} (\varepsilon \mathbb{Z}^2)} . \]
\end{lemma}

\begin{proof}
  Consider $f \in B^s_{p, p, \ell} (\varepsilon \mathbb{Z}^2) \subset B^s_{p,
  \infty, \ell} (\varepsilon \mathbb{Z}^2)$. First we compute $\| \tau_y
  \Delta^{\varepsilon}_j f - \Delta^{\varepsilon}_j f \|_{L^p_{\ell} (\varepsilon \mathbb{Z}^2)}$, for
  $\varepsilon \leqslant y \leqslant 1$. For $j = - 1$ we get
  \begin{eqnarray}
    \| \tau_y \Delta^{\varepsilon}_{- 1} f - \Delta_{- 1}^{\varepsilon} f \|_{L^p_{\ell} (\varepsilon
    \mathbb{Z}^2)} & = & \sum_{| i + 1 | < 2} \| (\tau_y - I) K_{- 1}^{\varepsilon}
    \ast_{\varepsilon} \Delta_i^{\varepsilon} f \|_{L^p_{\ell} (\varepsilon \mathbb{Z}^2)}
    \nonumber\\
    & \lesssim & \sum_{| i + 1 | < 2} \| \tau_y K_{- 1}^{\varepsilon} - K_{- 1}^{\varepsilon}
    \|_{L^1_{\ell} (\varepsilon \mathbb{Z}^2)} \| \Delta_i^{\varepsilon} f \|_{L^p_{\ell}
    (\varepsilon \mathbb{Z}^2)} \nonumber\\
    & \lesssim & 2^{- s} \| \tau_y K_{- 1}^{\varepsilon} - K_{- 1}^{\varepsilon} \|_{L^1_{\ell}
    (\varepsilon \mathbb{Z}^2)} \| f \|_{B^s_{p, \infty, \ell} (\varepsilon
    \mathbb{Z}^2)} . \nonumber
  \end{eqnarray}
  Furthermore we get
  \[ | \tau_y K_{- 1}^{\varepsilon} - K_{- 1}^{\varepsilon} | (x) = \left|\int_0^{|y|} \nabla_y K_{- 1}^{\varepsilon} \left( x +
     t \frac{y}{| y |} \right) \mathd t\right|, \]
  which implies that
  \[ \| \tau_y K_{- 1}^{\varepsilon} - K_{- 1}^{\varepsilon}\|_{L^1_{\ell} (\varepsilon \mathbb{Z}^2)}
     \leq \| \nabla K_{- 1}^{\varepsilon} \|_{L^1_{\ell} (\varepsilon \mathbb{Z}^2)} | y |,
  \]
  and so $\| \tau_y \Delta_{- 1}^{\varepsilon} f - \Delta_{- 1}^{\varepsilon}
  f \|_{L^p_{\ell} (\varepsilon \mathbb{Z}^2)} \lesssim 2^{- s} | y | \| f
  \|_{B^s_{p, \infty, \ell} (\varepsilon \mathbb{Z}^2)}$.\\
  
  If $- 1 < j < J_{\varepsilon}$, using a similar argument we obtain
  \[ \| \tau_y \Delta_j^{\varepsilon} f - \Delta_j^{\varepsilon} f
     \|_{L^p_{\ell} (\varepsilon \mathbb{Z}^2)} \lesssim 2^{- s j} | y | \|
     \nabla K_j^{\varepsilon} \|_{L^1_{\ell} (\varepsilon \mathbb{Z}^2)} \| f \|_{B^s_{p,
     \infty, \ell}} \lesssim 2^{- s j + j} | y | \| f \|_{B^s_{p, \infty,
     \ell} (\varepsilon \mathbb{Z}^2)}, \]
  and for $j = J_{\varepsilon}$ we have
  \[ \| \tau_y \Delta_{J_{\varepsilon}}^{\varepsilon} f -
     \Delta_{J_{\varepsilon}}^{\varepsilon} f \|_{L^p_{\ell} (\varepsilon
     \mathbb{Z}^2)} \lesssim 2^{- s J_{\varepsilon}} | y | \| \nabla
     K_{J_{\varepsilon}}^{\varepsilon} \|_{L^1_{\ell} (\varepsilon
     \mathbb{Z}^2)} \| f \|_{B^s_{p, \infty, \ell} (\varepsilon \mathbb{Z}^2)}
     \lesssim 2^{- s J_{\varepsilon} + J_{\varepsilon}} | y | \| f \|_{B^s_{p, \infty,
     \ell} (\varepsilon \mathbb{Z}^2)} . \]
  Finally we can use the following trivial inequality
  \[ \| \tau_y \Delta_i^{\varepsilon} f - \Delta_i^{\varepsilon} f
     \|_{L^p_{\ell} (\varepsilon \mathbb{Z}^2)} \leqslant 2 \|
     \Delta_i^{\varepsilon} f \|_{L^p_{\ell} (\varepsilon \mathbb{Z}^2)}
     \lesssim 2^{- j s} \| f \|_{B^s_{p, \infty, \ell} (\varepsilon
     \mathbb{Z}^2)} . \]
  Exploiting the previous inequalities we get
  \begin{equation}
    \| \tau_y f - f \|_{L^p_{\ell} (\varepsilon \mathbb{Z}^2)} \lesssim \| f
    \|_{B^s_{p, \infty, \ell} (\varepsilon \mathbb{Z}^2)} \left( | y | \sum_{j
    \leqslant j_y} 2^{(1 - s) j} + \sum_{j_y < j \leqslant J_{\varepsilon}}^{}
    2^{- s j} \right). \label{eq:normY}
  \end{equation}
  If we choose $j_y \in \mathbb{N}$ such that $\frac{1}{| y |} \leqslant
  2^{j_y} \leqslant \frac{2}{| y |}$, inequality {\eqref{eq:normY}} implies
  \begin{equation}
    \| \tau_y f - f \|_{L^p_{\ell} (\varepsilon \mathbb{Z}^2)} \lesssim | y
    |^s \| f \|_{B^s_{p, \infty, \ell} (\varepsilon \mathbb{Z}^2)} \lesssim |
    y |^s \| f \|_{B^s_{p, p, \ell} (\varepsilon \mathbb{Z}^2)}
    \label{eq:normY2} .
  \end{equation}
  Dividing both sides of {\eqref{eq:normY2}} by $| y |^s$ and taking the sup
  over $\varepsilon \leqslant | y | \leqslant 1$ we get the thesis, since $\|
  f \|_{L^p} \lesssim \| f \|_{B_{p, p}^s (\varepsilon \mathbb{Z}^2)}$.
\end{proof}

\begin{proof*}{Proof of Theorem \ref{theorem:sobolev_besov}}
For $s=0$, $W^{0,p}_{\ell}=L^p_{\ell}$, and the inequality \eqref{eq:inequality:difference} follows from Besov embedding theorem.\\

For $0<s\leqslant 1$, the first part of inequalities {\eqref{eq:inequality:difference}} was proved
  in Lemma \ref{lemma:first:inequality} and the second one was proved in Lemma
  \ref{lemma:second:inequality} under the condition $s < 1$, so we need to
  prove only the second part of {\eqref{eq:inequality:difference}} when $s =
  1$. In this case, by Remark \ref{remark:equivalence:difference}, we can use
  the equivalent expression of the norm given in equation
  {\eqref{eq:difference:norm:equivalence}}.
  
  First we observe that
  \[ \varepsilon^{- s} \mathcal{F} (\tau_{\varepsilon e_i} f - f) (q) =
     \varepsilon^{- s} (e^{i \varepsilon (e_i \cdot q)} - 1) \mathcal{F} f
     (q), \]
  where $\{ e_i \}_{i = 1.2}$ is the standard basis of $\mathbb{Z}^2$. For any
  $\beta \in \mathbb{N}^2$ we get
  \begin{equation}
    | \partial_q^{\beta} \varepsilon^{- s} (e^{i \varepsilon (e_i \cdot q)} -
    1) | \leqslant \varepsilon^{| \beta | - s} . \label{eq:difference10}
  \end{equation}
  In addition $| \varepsilon^{- s} (e^{i \varepsilon (e_i \cdot q)} - 1) |
  \leqslant | q |^s$ which, together with {\eqref{eq:difference10}}, gives
  \begin{equation}
    | \partial_q^{\alpha} \varepsilon^{- s} (e^{i \varepsilon (e_i \cdot q)} -
    1) | \leqslant C (1 + | q |)^{s - | \alpha |},
    \label{estm:symbol-discre-grad}
  \end{equation}
  for some constant $C > 0$ independent of $0 < \varepsilon \leqslant 1$.\\
  
  \tmcolor{red}{\tmcolor{red}{\tmcolor{black}{By Lemma \ref{lemma-fm-reg} and
  the embedding $\| f \|_{L^p} \lesssim \| f \|_{B_{p, p}^{\delta}
  (\varepsilon \mathbb{Z}^2)}$, inequality {\eqref{estm:symbol-discre-grad}} \
  implies that
  \[ \varepsilon^2 \sum_{z \in \varepsilon
     \mathbb{Z}^2} \rho_{\ell} (z)^p | \nabla_{\varepsilon} f(z) |^p
     \lesssim \| f \|_{B^{1 + \delta}_{p, p, \ell} (\varepsilon
     \mathbb{Z}^2)}^p ,\]}}}
  
  which in turn gives the statement.
\end{proof*}

\subsection{Extension operator on lattice}\label{section:extension}

\

In this section we want to study the problem of extending a function (or a
distribution) on the lattice $\varepsilon \mathbb{Z}^2$ to a function (or
distribution) on $\mathbb{R}^2$. Here in the definition of Besov space
$B^s_{p, p, \ell} (\mathbb{R}^2)$ we use the same dyadic partition of unity
$\{ \varphi_j \}_{j \geqslant - 1}$ introduced in Section
\ref{subsection:lattice}.

\

We introduce the reconstruction operator. Let $w_0 (x) =\mathbb{I}_{\left( -
\frac{1}{2}, \frac{1}{2} \right]^2} (x)$ and $w_{0, \varepsilon} (x) =
\varepsilon^{- 2} w_0 (\varepsilon^{- 1} x)$. We define an operator
\[ \mathcal{E}^{\varepsilon} : \mathcal{S}' (\varepsilon \mathbb{Z}^2)
   \rightarrow \mathcal{S}' (\mathbb{R}^2) \]
as follows: if $f \in \mathcal{S}' (\varepsilon \mathbb{Z}^2)$ we
defined the measurable function $\mathcal{E}^{\varepsilon} (f)$ on
$\mathbb{R}^2$ as
\begin{equation}
  \mathcal{E}^{\varepsilon} (f) (z) = w_{0, \varepsilon} \ast_{\varepsilon} f
  (z) = \sum_{z' \in \varepsilon \mathbb{Z}^2} f (z') \mathbb{I}_{\left( -
  \frac{\varepsilon}{2}, \frac{\varepsilon}{2} \right]^2} (z - z'), \quad f
  \in \mathcal{S}' (\varepsilon \mathbb{Z}^2).\label{eq:definition:extension1}
\end{equation}
In other words, for any $f \in \mathcal{S}' (\mathbb{R})$, the function
$\mathcal{E}^{\varepsilon} (f)$ is constant on the squares
\begin{equation}
  Q_{\varepsilon} (z') \assign \left( - \frac{\varepsilon}{2} + z'_1,
  \frac{\varepsilon}{2} + z'_1 \right] \times \left( - \frac{\varepsilon}{2} +
  z'_2, \frac{\varepsilon}{2} + z'_2 \right] \subset \mathbb{R}^2,
  \label{eq:defsquares}
\end{equation}
where $z' = (z'_1, z'_2) \in \varepsilon \mathbb{Z}^2$. More precisely for any
$z \in Q_{\varepsilon} (z')$ the function $\mathcal{E}^{\varepsilon} (f)$ on
$\mathbb{R}^2$ is equal to $f (z')$ which is the value of $f$ on the
intersection $Q_{\varepsilon} (z') \cap \varepsilon \mathbb{Z}^2$.

\begin{remark}
  \label{remark:extension}Since, for any $0 < \varepsilon \leqslant 1$,
  \begin{equation}
    \sup_{| x - y | < \varepsilon} \rho_{\ell} (x) \sim \rho_{\ell} (y)
    \label{eq:weight:sim}, \quad x, y \in \mathbb{R}^2,
  \end{equation}
  then for any $s \in \mathbb{R}$ and $p \in [1, + \infty]$ if $f \in B^s_{p,
  p, \ell} (\varepsilon \mathbb{Z}^2)$ then function
  $\mathcal{E}^{\varepsilon} (f) \in L^p_{\ell} (\mathbb{R}^2)$. Furthermore,
  if $s > 0$, we have
  \[ \| \mathcal{E}^{\varepsilon} (f) \|_{L^p_{\ell} (\mathbb{R}^2)} \sim \| f
     \|_{L^p_{\ell} (\varepsilon \mathbb{Z}^2)} \lesssim \| f \|_{B^s_{p, p,
     \ell} (\varepsilon \mathbb{Z}^2)} . \]
\end{remark}

\begin{theorem}
  \label{theorem:extension1}For any $p \in [1, + \infty)$ and $0 < s \leqslant 1$, the operator $\mathcal{E}^{\varepsilon}$ is continuous from
  $W^{s, p}_{\ell} (\varepsilon \mathbb{Z}^2)$ into $B^{s \wedge \frac{1}{p}}_{p, \infty, \ell}
  (\mathbb{R}^2)$. Furthermore we have
  \[ \sup_{0 < \varepsilon \leqslant 1} \| \mathcal{E}^{\varepsilon}
     \|_{\mathcal{L} (W^{s, p}_{\ell} (\varepsilon \mathbb{Z}^2), B^{s \wedge \frac{1}{p}}_{p,
     \infty, \ell} (\mathbb{R}^2))} < + \infty . \]
  In the case $s = 1$ we have that $\mathcal{E}^{\varepsilon}$ is continuous
  from $W^{1, p}_{\ell} (\varepsilon \mathbb{Z}^2)$ into $B^{\frac{1}{p}}_{p,
  p, \ell} (\mathbb{R}^2)$ with the operator norm being uniformly bounded in
  $0 < \varepsilon \leqslant 1$.
\end{theorem}

\begin{proof}
  In order to prove the theorem we consider an equivalent formulation of the
  norm $B^s_{p, \infty, \ell} (\mathbb{R}^2)$ (see, e.g., Theorem 6.9 Chapter
  6 of {\cite{Triebel2006}}) namely, for any $g \in B^s_{p, \infty, \ell}
  (\mathbb{R}^2)$,
  \begin{equation}
    \| g \|_{B^s_{p, \infty, \ell} (\mathbb{R}^2)} \sim \| g \|_{L^p_{\ell}
    (\mathbb{R}^2)} + \sup_{0 < | h | \leqslant 1, h \in \mathbb{R}^2} | h
    |^{- s} \| \tau_h g - g \|_{L^p_{\ell} (\mathbb{R}^2)} .
    \label{eq:extension1}
  \end{equation}
  By Definition \ref{definition:difference}, using the fact that $\| f
  \|_{L^p_{\ell} (\varepsilon \mathbb{Z}^2)} \sim \| \mathcal{E}^{\varepsilon}
  (f) \|_{L^p_{\ell} (\mathbb{R}^2)}$ and writing $g =
  \mathcal{E}^{\varepsilon} (f)$, we have
  \[ \| g \|_{L^p_{\ell} (\mathbb{R}^2)} + \sup_{\varepsilon < | h |
     \leqslant 1} | h |^{- s} \| \tau_h g - g \|_{L^p_{\ell} (\mathbb{R}^2)}
     \lesssim \| f \|_{W^{s, p}_{\ell} (\varepsilon \mathbb{Z}^2)} . \]
  When $| h | \leqslant \varepsilon$ the function $| \tau_h g - g |$ could be
  nonzero only in a rectangle of area $\varepsilon h$, and when it is nonzero we
  get
  \[ | \tau_h g (z') - g (z') | \leqslant \varepsilon^s |
     \nabla_{\varepsilon}^s f (z) |, \]
  where $z \in \varepsilon \mathbb{Z}^2$ such that $| z' - z | \leqslant
  \varepsilon$. Since $| h | \leqslant \varepsilon$ and $s \leqslant \frac{1}{p}$ we
  get
  \[ | h |^{- s} \| \tau_h g - g \|_{L^p_{\ell} (\mathbb{R}^2)} \lesssim | h
     |^{- s} | h |^{\frac{1}{p}} \varepsilon^{s - \frac{1}{p}} \|
     \nabla_{\varepsilon}^s f \|_{L^p_{\ell} (\varepsilon \mathbb{Z}^2)}
     \lesssim \left( \frac{| h |}{\varepsilon} \right)^{\frac{1}{p} - s} \|
     \nabla_{\varepsilon}^s f \|_{L^p_{\ell} (\varepsilon \mathbb{Z}^2)}
     \lesssim \| \nabla_{\varepsilon}^s f \|_{L^p_{\ell} (\varepsilon
     \mathbb{Z}^2)} . \]\end{proof}

\begin{corollary}
  \label{corollary:extension1}For any $p \in [1, + \infty)$, $\delta>0$, $1\leqslant q \leqslant +\infty$, and $0 < s
  \leqslant \frac{1}{p}$, the operator $\mathcal{E}^{\varepsilon}$ is
  continuous from $B^s_{p, q, \ell} (\varepsilon \mathbb{Z}^2)$ into $B^s_{p,
  \infty, \ell} (\mathbb{R}^2)$. Furthermore the norm
  \[ \sup_{0 < \varepsilon \leqslant 1} \| \mathcal{E}^{\varepsilon}
     \|_{\mathcal{L} (B^s_{p, q, \ell} (\varepsilon \mathbb{Z}^2), B^{s-\delta}_{p,
     \infty, \ell} (\mathbb{R}^2))} < + \infty . \]
\end{corollary}
\begin{proof}
  The corollary is a simple consequence of the continuous immersion $B^s_{p,
  p, \ell} (\varepsilon \mathbb{Z}^2) \hookrightarrow W^{s, p}_{\ell}
  (\varepsilon \mathbb{Z}^2)$ proved in Theorem \ref{theorem:sobolev_besov}.
\end{proof}

\begin{theorem}
  \label{theorem:extension2}Consider $0 < s \leqslant 1$ and let
  $f_{\varepsilon}$ be a sequence of functions in $W^{s, p}_{\ell}
  (\varepsilon \mathbb{Z}^2)$ such that $\sup_{0 < \varepsilon \leqslant 1} \|
  f_{\varepsilon} \|_{W^{s, p}_{\ell} (\varepsilon \mathbb{Z}^2)} < + \infty$,
  and $\mathcal{E}^{\varepsilon} (f_{\varepsilon}) \rightarrow f$ in
  $L^p_{\ell} (\mathbb{R}^2)$, then $f \in B^s_{p, \infty, \ell}
  (\mathbb{R}^2)$. Furthermore if $s = 1$ and $1 < p < + \infty$ then $f \in
  W^{1, p}_{\ell} (\mathbb{R}^2)$ (where $W^{1, p}_{\ell} (\mathbb{R}^2)$ is
  the weighted Sobolev space of one times weakly differentiable functions with
  weak derivatives in $L_{\ell}^p (\mathbb{R}^2)$).
\end{theorem}

\begin{proof}
  We have that
  \begin{eqnarray}
    \| f \|_{B^s_{p, \infty, \ell} (\mathbb{R}^2)} & \sim & \| f
    \|_{L^p_{\ell} (\mathbb{R}^2)} + \sup_{| h | < 1} | h |^{- s} \| \tau_h f
    - f \|_{L^p_{\ell} (\mathbb{R}^2)} \nonumber\\
    & \sim & \| f \|_{L^p_{\ell} (\mathbb{R}^2)} + \sup_{\delta > 0}
    \sup_{\delta < | h | < 1} | h |^{- s} \| \tau_h f - f \|_{L^p_{\ell}
    (\mathbb{R}^2)} \nonumber\\
    & \lesssim & \| f \|_{L^p_{\ell} (\mathbb{R}^2)} + \sup_{\delta > 0}
    \limsup_{\varepsilon \rightarrow 0} \sup_{\delta < | h | < 1} | h |^{- s}
    \| \tau_h \mathcal{E}^{\varepsilon} (f_{\varepsilon}) -
    \mathcal{E}^{\varepsilon} (f_{\varepsilon}) \|_{L^p_{\ell}} \nonumber\\
    & \lesssim & \| f \|_{L^p_{\ell} (\mathbb{R}^2)} + \sup_{\delta > 0}
    \limsup_{\varepsilon \rightarrow 0} \| f_{\varepsilon} \|_{W^{s, p}_{\ell}
    (\varepsilon \mathbb{Z}^2)} \nonumber\\
    & \lesssim & \sup_{0 < \varepsilon \leqslant 1} \| f_{\varepsilon}
    \|_{W^{s, p}_{\ell} (\varepsilon \mathbb{Z}^2)} < + \infty ,\nonumber
  \end{eqnarray}
  where we used that, for any $g\in L^p_{\ell}(\varepsilon \mathbb{Z}^2)$,  $\|\mathcal{E}^{\varepsilon}(g) \|_{L^p_{\ell}(\mathbb{R}^2)} \sim \|g \|_{L^p(\varepsilon \mathbb{Z}^2)}$,  and that, for $\delta \geq \varepsilon$, $\sup_{\delta<|h|<1}|h|^{-s} \|\tau_h g-g \|_{L^p_{\ell}(\varepsilon\mathbb{Z}^2)} \leq \sup_{\varepsilon <|h|<1} |h|^{-s}\|\tau_h g-g \|_{L^p_{\ell}(\varepsilon\mathbb{Z}^2)}  $. \\
  
   For the case $s = 1$ we recall the following properties of Sobolev spaces:
  Define the operator $\tilde{D}_h : L^p_{\ell} (\mathbb{R}^2) \rightarrow
  L^p_{\ell} (\mathbb{R}^2)$ as
  \[ \tilde{D}_h (g) = \frac{\tau_h g - g}{| h |}, \quad g \in L^p
     (\mathbb{R}^2), h \in \mathbb{R}^2 \]
  then if
  \[ \sup_{| h | \leqslant 1} (\| \tilde{D}_h g \|_{L^p_{\ell}
     (\mathbb{R}^2)}) < C, \]
  where $C \in \mathbb{R}_+$ is a suitable constant, then $g \in W^{s,
  p}_{\ell} (\mathbb{R}^2)$ and
  \[ \| \nabla g \|_{L^p_{\ell} (\mathbb{R}^2)} \leqslant 2 C, \]
  see, e.g., Theorem 3 in Chapter 5, Section 5.8 of {\cite{Evans1998}} for the
  proof in unweighted case, the case with weights is similar. Furthermore we observe that, considering $w \in L^p_{\ell}(\mathbb{R}^2)$ and $\varepsilon < | h |$, if we write
$n_{\varepsilon} = \left\lfloor \frac{| h |}{\varepsilon} \right\rfloor$ and
$e_h = \frac{h}{| h |}$ we get
\begin{eqnarray}
  \left\| \frac{\tau_h w - w}{| h |} \right\|_{L^p_{\ell}(\mathbb{R}^2)} & = &
  \frac{\varepsilon}{| h |}  \frac{\left\| \tau_h (w) - \tau_{n_{\varepsilon}
  e_h} (w) + \sum_{j = 1}^{n_{\varepsilon}} \tau_{j \varepsilon e_h} (w) -
  \tau_{(j - 1) \varepsilon e_h} (w) \right\|_{L^p_{\ell}(\mathbb{R}^2)}}{\varepsilon}
  \nonumber\\
  & \leqslant & \frac{1}{| h |} \| \tau_h (w) - \tau_{n_{\varepsilon} e_h}
  (w) \|_{L^p_{\ell}(\mathbb{R}^2)} + \frac{\varepsilon}{| h |} \sum_{j =
  1}^{n_{\varepsilon}} \| \tau_{(j - 1) \varepsilon e_h}
  (\tilde{D}_{\varepsilon e_h} w) \|_{L^p_{\ell}(\mathbb{R}^2)} \nonumber\\
  & \lesssim & \frac{1}{| h |} \| \tau_h (w) - \tau_{n_{\varepsilon} e_h} (w)
  \|_{L^p_{\ell}} + \frac{\varepsilon n_{\varepsilon}}{| h |} \|
  \tilde{D}_{\varepsilon e_h} w  \|_{L^p_{\ell}(\mathbb{R}^2)} \nonumber\\
  & \lesssim & \frac{1}{| h |} \| \tau_h (w) - \tau_{n_{\varepsilon} e_h} (w)
  \|_{L^p_{\ell}} + \| \tilde{D}_{\varepsilon e_h} w  \|_{L^p_{\ell}(\mathbb{R}^2)}
  \nonumber,
\end{eqnarray}
where we use that $\| g \|_{L^p_{\ell}} \sim \| \tau_k (g) \|_{L^p_{\ell}}$
when $| k | < 1$, and that $\frac{\varepsilon n_{\varepsilon}}{| h |} < 1$. By
the lower semi-continuity of norm of $L^p_{\ell}(\mathbb{R}^2)$ with respect to the weak convergence, we have that
\begin{eqnarray}
  \| \tilde{D}_h f \|_{L^p_{\ell} (\mathbb{R}^2)}= \left\| \frac{\tau_h f - f}{| h |} \right\|_{L^p_{\ell} (\mathbb{R}^2)} &
  \leqslant & \limsup_{\varepsilon \rightarrow 0} \left\| \frac{\tau_h
  (\mathcal{E}^{\varepsilon} (f_{\varepsilon})) - \mathcal{E}^{\varepsilon}
  (f_{\varepsilon})}{| h |} \right\|_{L^p_{\ell} (\mathbb{R}^2)} \nonumber\\
  & \leqslant & \limsup_{\varepsilon \rightarrow 0} \| \tilde{D}_{\varepsilon
  e_h} \mathcal{E}^{\varepsilon} (f_{\varepsilon})  \|_{L^p_{\ell}
  (\mathbb{R}^2)} \nonumber\\
  & \leqslant & \sup_{0<\varepsilon \leqslant 1} \| \nabla_{\varepsilon} f_{\varepsilon} \|_{L^p_{\ell}
  (\varepsilon \mathbb{Z}^2)} \nonumber,
\end{eqnarray}
where we used that $\| \tau_h (\mathcal{E}^{\varepsilon} (f_{\varepsilon}))-\tau_{n_{\varepsilon} e_h} (\mathcal{E}^{\varepsilon} (f_{\varepsilon}))
\|_{L^p_{\ell}} = 0$.
  This implies that $\sup_{| h | \leqslant 1} \| D_h f \|_{L^p_{\ell}
  (\mathbb{R}^2)} \leqslant \sup_{0 < \varepsilon \leqslant 1} \|
  \nabla_{\varepsilon} f_{\varepsilon} \|_{L^p_{\ell} (\varepsilon
  \mathbb{Z}^2)}$ and so
  \[ \| \nabla f \|_{L^p_{\ell} (\mathbb{R}^2)^{}} \leqslant 2 \sup_{0 <
     \varepsilon \leqslant 1} \| \nabla_{\varepsilon} f_{\varepsilon}
     \|_{L^p_{\ell} (\varepsilon \mathbb{Z}^2)} . \]
\end{proof}

We prove now a result about the extension operator and Besov spaces with
negative regularity $s < 0$ which, although probably not being optimal, is
enough for proving the main results in the current paper.

\begin{theorem}
  \label{theorem:extension:lattice:negative}Consider $p \in (1, + \infty)$, $s
  \leqslant 0$, $\delta > 0$ and $\ell \in \mathbb{R}$ then the extension
  operator $\mathcal{E}^{\varepsilon}$ is continuous from $B^s_{p, p, \ell}
  (\varepsilon \mathbb{Z}^2)$ into $B^{s - \delta}_{p, p, \ell}
  (\mathbb{R}^2)$. Furthermore the norm
  \[ \sup_{0 < \varepsilon \leqslant 1} \| \mathcal{E}^{\varepsilon}
     \|_{\mathcal{L} (B^s_{p, p, \ell} (\varepsilon \mathbb{Z}^2), B^{s -
     \delta}_{p, p, \ell} (\mathbb{R}^2))} < + \infty . \]
\end{theorem}

Before proving Theorem \ref{theorem:extension:lattice:negative}, we introduce
the notion of discretization operator $\mathcal{D}^{\varepsilon}$ from the
space $L^p_{\ell} (\mathbb{R}^2)$, where $p \in [1, + \infty]$, into
$\mathcal{S}' (\varepsilon \mathbb{Z}^2)$ defined as
\begin{equation}
  \mathcal{D}^{\varepsilon} (\varphi) (z) = \frac{1}{\varepsilon^2}
  \int_{Q_{\varepsilon} (z)} \varphi (x) \mathd x, \quad \varphi \in
  L^p_{\ell} (\mathbb{R}^2), \quad z \in \varepsilon \mathbb{Z}^2,
  \label{eq:operatorD}
\end{equation}
where $Q_{\varepsilon} (z)$ as defined in equation {\eqref{eq:defsquares}}.

\begin{lemma}
  \label{lemma:discretization:operator}Consider $p \in (1, + \infty)$ then the
  linear operator $\mathcal{D}^{\varepsilon} : L^p_{\ell} (\mathbb{R}^2)
  \rightarrow L^p_{\ell} (\varepsilon \mathbb{Z}^2)$ is continuous with a norm
  uniformly bounded in $0 < \varepsilon \leqslant 1$, and also, writing $q \in
  (1, + \infty)$ such that $\frac{1}{p} + \frac{1}{q} = 1$, for any $f \in
  L^p_{\ell} (\mathbb{R}^2)$ and $g \in L^q_{\ell} (\varepsilon \mathbb{Z}^2)$
  we have
  \begin{equation}
    \int_{\mathbb{R}^2} f (x) \mathcal{E}^{\varepsilon} (g) (x) \mathd x =
    \int_{\varepsilon \mathbb{Z}^2} \mathcal{D}^{\varepsilon} (f) (z) g (z)
    \mathd z. \label{eq:Depsilon1}
  \end{equation}
  Furthermore for any $s > 0$, $\delta > 0$ and $f \in B^s_{p, \infty, \ell}
  (\mathbb{R}^2)$ we have
  \[ \| \mathcal{D}^{\varepsilon} f \|_{B^{s - \delta}_{p, p, \ell}
     (\varepsilon \mathbb{Z}^2)} \lesssim \| f \|_{B^s_{p, \infty, \ell}
     (\mathbb{R}^2)} . \]
\end{lemma}

\begin{proof}
  Consider $f \in L^p_{\ell} (\mathbb{R}^2)$ then, by Jensen
  inequality, we have
  \begin{eqnarray}
    \| \mathcal{D}^{\varepsilon} ( f) \|^p_{L^p_{\ell} (\varepsilon
    \mathbb{Z}^2)} & = & \varepsilon^2 \sum_{z \in \varepsilon \mathbb{Z}^2}
    \left| \frac{1}{\varepsilon^2} \int_{Q_{\varepsilon} (z)} \rho_{\ell} (z)
    f (x) \mathd x \right|^p \nonumber\\
    & \lesssim & \varepsilon^2 \sum_{z \in \varepsilon \mathbb{Z}^2}
    \frac{1}{\varepsilon^2} \int_{Q_{\varepsilon} (z)} | \rho_{\ell} (x) f (x)
    |^p \mathd x \nonumber\\
    & \lesssim & \int | \rho_{\ell} (x) f (x) |^p \mathd x = \| f
    \|^p_{L^p_{\ell} (\mathbb{R}^2)} , \nonumber
  \end{eqnarray}
  where the constant in the symbol $\lesssim$ can be chosen uniformly in $\varepsilon \leqslant 1$.  Equality {\eqref{eq:Depsilon1}} follows from the definition of
  $\mathcal{E}^{\varepsilon} (g)$.\\
  
  Consider now $f \in B^s_{p, \infty, \ell} (\mathbb{R}^2)$, if $s - \delta <
  s' < s$ such $n - 1 < s' < s \leqslant n$ for some $n \in \mathbb{N}$, we
  have that an equivalent definition of $B^{s'}_{p, \infty, \ell}
  (\mathbb{R}^2)$ is given by
  \[ \| f \|_{B^{s'}_{p, \infty, \ell}(\mathbb{R}^{2})} \sim \| f \|_{L^p_{\ell} (\mathbb{R}^2)} + \sum_{m = 1}^n \sup_{0 <
     | h | \leqslant 1} \frac{\| D^m_h f \|_{L^p_{\ell} (\mathbb{R}^2)}}{| h |^{m \wedge s'}},\]
  where $D^m_h$ is the standard extension to $\mathbb{R}^2$ of the operator
  having the same name introduced in {\eqref{eq:higher:difference}}.\\
  
  When $h \in \varepsilon \mathbb{Z}^2$ the operator $D^m_h$ commutes with the
  operator $\mathcal{D}^{\varepsilon}$ in the following sense: for any $f \in
  L^p_{\ell} (\mathbb{R}^2)$ and for even $n$
  \begin{eqnarray}
    D^m_h (\mathcal{D}^{\varepsilon} (f)) (z) & = & \sum_{m = 0}^n \left(
    \begin{array}{c}
      n\\
      m
    \end{array} \right) (- 1)^{m + 1} \tau_{\left( \frac{n}{2} - m \right) h}
    \mathcal{D}^{\varepsilon} (f) (z) \nonumber\\
    & = & \frac{1}{\varepsilon^2} \sum_{m = 0}^n \left( \begin{array}{c}
      n\\
      m
    \end{array} \right) (- 1)^{m + 1} \int_{Q_{\varepsilon} \left( z + \left(
    \frac{n}{2} - m \right) h \right)} f (x) \mathd x \nonumber\\
    & = & \frac{1}{\varepsilon^2} \sum_{m = 0}^n \left( \begin{array}{c}
      n\\
      m
    \end{array} \right) (- 1)^{m + 1} \int_{Q_{\varepsilon} (z)} f \left( x -
    \left( \frac{n}{2} - m \right) h \right) \mathd x \nonumber\\
    & = & \mathcal{D}^{\varepsilon} (D^m_{- h} (f)) (z), \nonumber
  \end{eqnarray}
  and similarly for $n$ odd. Then, using the result of the first part of the
  proof, we get
  \begin{eqnarray}
    \| \mathcal{D}^{\varepsilon} (f) \|_{W^{s', p}_{\ell} (\varepsilon
    \mathbb{Z}^2)} & = & \| \mathcal{D}^{\varepsilon} (f) \|_{L^p_{\ell}
    (\varepsilon \mathbb{Z}^2)} + \sum_{m = 1}^n \sup_{h \in \varepsilon
    \mathbb{Z}^2, \varepsilon < | h | \leqslant 1} \frac{\| D^m_h
    (\mathcal{D}^{\varepsilon} (f)) \|_{L^p_{\ell} (\varepsilon
    \mathbb{Z}^2)}}{| h |^{m \wedge s'}} \nonumber\\
    & = & \| \mathcal{D}^{\varepsilon} (f) \|_{L^p_{\ell} (\varepsilon
    \mathbb{Z}^2)} + \sum_{m = 1}^n \sup_{h \in \varepsilon \mathbb{Z}^2,
    \varepsilon < | h | \leqslant 1} \frac{\| \mathcal{D}^{\varepsilon} (D^m_h
    f) \|_{L^p_{\ell} (\varepsilon \mathbb{Z}^2)}}{| h |^{m \wedge s'}}
    \nonumber\\
    & \lesssim & \| f \|_{L^p_{\ell} (\mathbb{R}^2)} + \sum_{m = 1}^n \sup_{h
    \in \varepsilon \mathbb{Z}^2, \varepsilon < | h | \leqslant 1} \frac{\|
    D^m_h f \|_{L^p_{\ell} (\mathbb{R}^2)}}{| h |^{m \wedge s'}} \nonumber\\
    & \lesssim & \| f \|_{B^s_{p, \infty, \ell} (\mathbb{R}^2)} . \nonumber
  \end{eqnarray}
  Since, by Lemma \ref{lemma:first:inequality} and Remark
  \ref{remark:first:inequality}, $\| \mathcal{D}^{\varepsilon} (f) \|_{B^{s -
  \delta}_{p, p, \ell} (\varepsilon \mathbb{Z}^2)} \lesssim \|
  \mathcal{D}^{\varepsilon} (f) \|_{W^{s, p}_{\ell} (\varepsilon
  \mathbb{Z}^2)}$ the lemma is proved.
\end{proof}

\begin{proof*}{Proof of Theorem \ref{theorem:extension:lattice:negative}}
  Consider $f \in B^s_{p, p, \ell} (\varepsilon \mathbb{Z}^2)$ then we have
  \begin{equation}
    \| \mathcal{E}^{\varepsilon} (f) \|_{B^{s - \delta}_{p, p, \ell}
    (\mathbb{R}^2)} = \sup_{g \in B^{- s + \delta}_{q, q, - \ell}
    (\mathbb{R}^2), \| g \|_{B^{- s + \delta}_{q, q, - \ell} (\mathbb{R}^2)} =
    1} \left| \int_{\mathbb{R}^2} \mathcal{E}^{\varepsilon} (f) (x) g (x)
    \mathd x \right| , \label{eq:extension:lattice:negative1}
  \end{equation}
  where $\frac{1}{p}+\frac{1}{q}=1$.
  On the other hand, by Lemma \ref{lemma:discretization:operator}, we have
  \[ \int_{\mathbb{R}^2} \mathcal{E}^{\varepsilon} (f) (x) g (x) \mathd x =
     \int_{\varepsilon \mathbb{Z}^2} f (z) \mathcal{D}^{\varepsilon} (g) (z)
     \mathd z. \]
  Furthermore, since $(1 - \Delta_{\varepsilon \mathbb{Z}^2})$ is
  self-adjoint, by H{\"o}lder's inequality, we get
  \[ \begin{array}{rl}
       &\left| \int_{\varepsilon \mathbb{Z}^2} f (z) \mathcal{D}^{\varepsilon}
       (g) (z) \mathd z \right| \\
       =& \left| \int_{\varepsilon \mathbb{Z}^2} \rho_{\ell} (z) (-
       \Delta_{\varepsilon \mathbb{Z}^2} + 1)^{s - \delta / 2} (f) (z) \rho_{-
       \ell} (z) (- \Delta_{\varepsilon \mathbb{Z}^2} + 1)^{- s + \delta / 2}
       (\mathcal{D}^{\varepsilon} (g)) (z) \mathd z \right|\\
       \leqslant& \| (- \Delta_{\varepsilon \mathbb{Z}^2} + 1)^{s - \delta / 2}
       (f) \|_{L^p_{\ell} (\varepsilon \mathbb{Z}^2)} \| (-
       \Delta_{\varepsilon \mathbb{Z}^2} + 1)^{- s + \delta / 2}
       (\mathcal{D}^{\varepsilon} (g)) \|_{L^q_{- \ell} (\varepsilon
       \mathbb{Z}^2)} .
     \end{array} \]
  By Theorem \ref{theorem:besov:weight} and Theorem
  \ref{theorem:besov:lattice:embedding} we obtain
  \[ \| (- \Delta_{\varepsilon \mathbb{Z}^2} + 1)^{s - \delta / 2} (f)
     \|_{L^p_{\ell} (\varepsilon \mathbb{Z}^2)} \lesssim \| f \|_{B^s_{p, p,
     \ell} (\varepsilon \mathbb{Z}^2)} \]
  \[ \| (- \Delta_{\varepsilon \mathbb{Z}^2} + 1)^{- s + \delta / 2}
     (\mathcal{D}^{\varepsilon} (g)) \|_{L^q_{- \ell} (\varepsilon
     \mathbb{Z}^2)} \lesssim \| \mathcal{D}^{\varepsilon} (g) \|_{B^{- s + 3
     \delta / 4}_{q, q, - \ell} (\varepsilon \mathbb{Z}^2)} . \]
  Thus by Lemma \ref{lemma:discretization:operator} we have
  \[ \| \mathcal{D}^{\varepsilon} (g) \|_{B^{- s + 3 \delta / 4}_{q, q, \ell}
     (\varepsilon \mathbb{Z}^2)} \lesssim \| g \|_{B^{- s + \delta}_{q, q, \ell} (\mathbb{R}^2)}, \]
  and so by equality {\eqref{eq:extension:lattice:negative1}} the thesis
  holds.
\end{proof*}

\subsection{Besov spaces on $\mathbb{R}^2 \times \varepsilon
\mathbb{Z}^2$}\label{section:besov:rz}

In this section we discuss the definition and properties of Besov spaces
defined on a four dimensional space where the first two dimensions are
continuous and the second two are discrete. More precisely we want to describe
the following Besov spaces
\begin{equation}
  B^{s_1, s_2}_{p, \ell_1, \ell_2} (\mathbb{R}^2 \times \varepsilon
  \mathbb{Z}^2) \assign B^{s_1}_{p, p, \ell_1} (\mathbb{R}^2, B^{s_2}_{p, p,
  \ell_2} (\varepsilon \mathbb{Z}^2)) \label{eq:besov:r2}
\end{equation}
\begin{equation}
  B^s_{p, r, \ell} (\mathbb{R}^2 \times \varepsilon \mathbb{Z}^2),
  \label{eq:besov:total}
\end{equation}
where $p, r \in [1, + \infty]$, $\ell, \ell_1, \ell_2 \in \mathbb{R}$ and
$s_1, s_2 \in \mathbb{R}$.

\

Before giving a precise definition of the spaces {\eqref{eq:besov:r2}} and
{\eqref{eq:besov:total}}, we introduce some sets of functions on the
topological space \ $\mathbb{R}^2 \times \varepsilon \mathbb{Z}^2$ equipped
with the natural product topology. If $p \in [1, + \infty)$ and $\ell, \ell_1,
\ell_2 \in \mathbb{R}$ we define the (weighted) Lebesgue space $L^p_{\ell_1,
\ell_2} (\mathbb{R}^2 \times \varepsilon \mathbb{Z}^2)$ and $L^p_{\ell}
(\mathbb{R}^2 \times \varepsilon \mathbb{Z}^2)$ as the subset of (Borel)
measurable functions $f : \mathbb{R}^2 \times \varepsilon \mathbb{Z}^2
\rightarrow \mathbb{R}$ for which the norms
\[ \| f \|^p_{L^p_{\ell_1, \ell_2} (\mathbb{R}^2 \times \varepsilon
   \mathbb{Z}^2)} \assign \int_{\mathbb{R}^2 \times \varepsilon \mathbb{Z}^2}
   | f (x, z) \rho_{\ell_1} (x) \rho_{\ell_2} (z) |^p \mathd x \mathd z \]
\[ \| f \|^p_{L^p_{\ell} (\mathbb{R}^2 \times \varepsilon \mathbb{Z}^2)}
   \assign \int_{\mathbb{R}^2 \times \varepsilon \mathbb{Z}^2} | f (x, z)
   \rho_{\ell}^{(4)} (x, z) |^p \mathd x \mathd z, \]
respectively, are finite. First we consider the space of continuous functions
$C^0 (\mathbb{R}^2 \times \varepsilon \mathbb{Z}^2)$ and also the space of
differentiable function $C^k (\mathbb{R}^2 \times \varepsilon \mathbb{Z}^2)$,
i.e. the subset of $C^0 (\mathbb{R}^2 \times \varepsilon \mathbb{Z}^2)$ for
which $f \in C^k (\mathbb{R}^2 \times \varepsilon \mathbb{Z}^2)$ if and only
if for any $\alpha = (\alpha_1, \alpha_2) \in \mathbb{N}^2_0$ with $| \alpha |
= \alpha_1 + \alpha_2 \leqslant k$ we have $\partial^{\alpha}_x f \in C^0
(\mathbb{R}^2 \times \varepsilon \mathbb{Z}^2)$ where
\[ \partial^{\alpha}_x f (x, z) = \partial^{\alpha_1}_{x_1}
   \partial^{\alpha_2}_{x_2} f (x, z), \]
where $x = (x_1, x_2) \in \mathbb{R}^2$. It is important to note that there is a
natural identification between the continuous functions $C^0 (\mathbb{R}^2
\times \varepsilon \mathbb{Z}^2)$ and the Fr{\'e}chet space $C^0
(\mathbb{R}^2, C^0 (\varepsilon \mathbb{Z}^2))$ of continuous functions from
$\mathbb{R}^2$ into the Fr{\'e}chet space $C^0 (\varepsilon \mathbb{Z}^2)$
(which since $\varepsilon \mathbb{Z}^2$ has the discrete topology is
equivalent to the set of functions from $\varepsilon \mathbb{Z}^2$ into
$\mathbb{R}^2$). In a similar way we can identify the space $C^k (\mathbb{R}^2
\times \varepsilon \mathbb{Z}^2)$ with the Fr{\'e}chet space $C^k
(\mathbb{R}^2, C^0 (\varepsilon \mathbb{Z}^2))$ of $k$ differentiable
functions from $\mathbb{R}^2$ into $\varepsilon \mathbb{Z}^2$. We can also
define
\[ \mathcal{S} (\mathbb{R}^2 \times \varepsilon \mathbb{Z}^2) \assign
   \mathcal{S} (\mathbb{R}^2) \hat{\otimes} \mathcal{S} (\varepsilon
   \mathbb{Z}^2), \quad \mathcal{S}' (\mathbb{R}^2 \otimes \varepsilon
   \mathbb{Z}^2) \assign \mathcal{S}' (\mathbb{R}^2) \hat{\otimes}
   \mathcal{S}' (\varepsilon \mathbb{Z}^2), \]
where $\mathcal{S} (\mathbb{R}^2)$ and $\mathcal{S}' (\mathbb{R}^2)$ are the
space of Schwartz test functions and tempered distribution on $\mathbb{R}^2$
respectively, and $\hat{\otimes}$ is the (unique) topological completion of
the tensor product between the nuclear spaces $\mathcal{S} (\mathbb{R}^2)$ and
$\mathcal{S}' (\mathbb{R}^2)$ and the Fr{\'e}chet spaces $\mathcal{S}
(\varepsilon \mathbb{Z}^2)$ and $\mathcal{S}' (\varepsilon \mathbb{Z}^2)$. By
standard arguments we can identify $\mathcal{S} (\mathbb{R}^2 \times
\varepsilon \mathbb{Z}^2)$ and $\mathcal{S}' (\mathbb{R}^2 \times \varepsilon
\mathbb{Z}^2)$ with $\mathcal{S} (\mathbb{R}^2, \mathcal{S} (\varepsilon
\mathbb{Z}^2))$ (i.e. the space of Schwartz test functions taking values in
the the Fr{\'e}chet space $\mathcal{S} (\varepsilon \mathbb{Z}^2)$) and
$\mathcal{S}' (\mathbb{R}^2, \mathcal{S}' (\varepsilon \mathbb{Z}^2))$ (i.e.
the space of tempered distribution taking values in the Fr{\'e}chet space
$\mathcal{S}' (\varepsilon \mathbb{Z}^2)$). We can define a Fourier transform
$\mathcal{F}$ from $\mathcal{S} (\mathbb{R}^2 \times \varepsilon
\mathbb{Z}^2)$ into $\mathcal{S} \left( \mathbb{R}^2 \times
\mathbb{T}^2_{\frac{1}{\varepsilon}} \right) = \mathcal{S} (\mathbb{R}^2)
\hat{\otimes} C^{\infty} \left( \mathbb{T}^2_{\frac{1}{\varepsilon}} \right)$
as follows
\[ \mathcal{F}_{\varepsilon} (f) (k, h) = \int_{\mathbb{R}^2 \times
   \varepsilon \mathbb{Z}^2} e^{i (k \cdot x + h \cdot z) } f (x, z) \mathd x
   \mathd z = \mathcal{F}^x (\mathcal{F}^z_{\varepsilon} (f)), \quad k \in
   \mathbb{R}^2, h \in \mathbb{T}^2_{\frac{1}{\varepsilon}} ,\]
where $f \in \mathcal{S} (\mathbb{R}^2 \times \varepsilon \mathbb{Z}^2)$ and
$\mathcal{F}^x$ is the Fourier transform with respect the continuum variables
$x \in \mathbb{R}^2$ and $\mathcal{F}^z_{\varepsilon}$ is the Fourier
transform with respect the discrete variables $z \in \varepsilon
\mathbb{Z}^2$. We also consider the inverse Fourier transform $\mathcal{F}^{-
1}_{\varepsilon} : \mathcal{S} \left( \mathbb{R}^2 \times
\mathbb{T}^2_{\frac{1}{\varepsilon}} \right) \rightarrow \mathcal{S}
(\mathbb{R}^2 \times \varepsilon \mathbb{Z}^2)$ as
\[ \mathcal{F}_{\varepsilon}^{- 1} (\hat{f}) (x, z) = \frac{1}{(2 \pi)^4}
   \int_{\mathbb{R}^2 \times \mathbb{T}^2_{\frac{1}{\varepsilon}}} e^{- i (k
   \cdot x + h \cdot z) } \hat{f} (k, h) \mathd k \mathd h = \mathcal{F}^{- 1,
   x} (\mathcal{F}^{- 1, z}_{\varepsilon} (f)), \quad x \in \mathbb{R}^2, z
   \in \varepsilon \mathbb{Z}^2 . \]
In a standard way it is possible to extend $\mathcal{F}_{\varepsilon}$ and
$\mathcal{F}^{- 1}_{\varepsilon}$ on $\mathcal{S}' (\mathbb{R}^2 \times
\varepsilon \mathbb{Z}^2)$. Using the previous identification if $f \in
\mathcal{S} (\mathbb{R}^2 \times \varepsilon \mathbb{Z}^2)$ or $g \in
\mathcal{S}' (\mathbb{R}^2 \times \varepsilon \mathbb{Z}^2)$ we can define the
Schwartz test function $\Delta_i^x f \in \mathcal{S} (\mathbb{R}^2 \times
\varepsilon \mathbb{Z}^2)$ and tempered distribution $\Delta_i^x g \in \mathcal{S}'
(\mathbb{R}^2 \times \varepsilon \mathbb{Z}^2)$ as follows
\[ \Delta_i^x f = K_i \asterisk_x f = \mathcal{F}^{- 1, x} (\varphi_i
   \mathcal{F}^x (f)), \quad \Delta_i^x g = K_i \ast_x g = \mathcal{F}^{- 1,
   x} (\varphi_i \mathcal{F}^x (g)), \]
where the convolution $\ast_x$ is taken only with respect the real variables
$x \in \mathbb{R}^2$. We use a similar notation $\Delta^z_j$ for the
Littlewood-Paley blocks with respect the discrete variable $z \in \varepsilon
\mathbb{Z}^2$.

\begin{definition}
  Consider $p, q \in [1, + \infty]$, $s \in \mathbb{R}$ and $\ell \in
  \mathbb{R}$ we define the Besov space $B^s_{p, q, \ell} (\mathbb{R}^2 \times
  \varepsilon \mathbb{Z}^2)$ the subspace of $\mathcal{S}' (\mathbb{R}^2
  \times \varepsilon \mathbb{Z}^2)$ for which $f \in B^s_{p, q, \ell}
  (\mathbb{R}^2 \times \varepsilon \mathbb{Z}^2)$ if and only if
  \begin{equation}
    \| f \|_{B^s_{p, q, \ell} (\mathbb{R}^2 \times \varepsilon \mathbb{Z}^2)}
    \assign \left\| \left\{ 2^{s (r - 1)} \left\| \sum_{- 1 \leqslant i, - 1
    \leqslant j \leqslant J_{\varepsilon}, \max (i, j) = r - 1} \Delta_i^x
    \Delta_j^z f \right\|_{L^p_{\ell} (\mathbb{R}^2 \times \varepsilon
    \mathbb{Z}^2)} \right\}_{r \in \mathbb{N}_0} \right\|_{\ell^q
    (\mathbb{N}_0)},
  \end{equation}
  is finite.
\end{definition}

Hereafter we write, for any $f \in \mathcal{S}' (\mathbb{R}^2 \times
\varepsilon \mathbb{Z}^2)$ and $r \geqslant - 1$,
\begin{equation}
  \tilde{\Delta}_r f = \mathcal{F}_{\varepsilon}^{- 1} \left( \sum_{- 1
  \leqslant i, - 1 \leqslant j \leqslant J_{\varepsilon}, \max (i, j) = r}
  \varphi_i (k) \varphi_j^{\varepsilon} (q) \mathcal{F} (f) (k, q) \right) .
  \label{eq:tilde:delta}
\end{equation}
\begin{remark}
  \label{remark:k:epsilon2}Consider the functions $\tilde{K}_r^{\varepsilon} :
  \mathbb{R}^2 \times \varepsilon \mathbb{Z}^2 \rightarrow \mathbb{R}$, $r
  \geqslant - 1$, defined as
  \begin{eqnarray}
    \tilde{K}_r^{\varepsilon} (x, z) & = & \mathcal{F}_{\varepsilon}^{- 1}
    \left( \sum_{- 1 \leqslant i, - 1 \leqslant j \leqslant J_{\varepsilon},
    \max (i, j) = r} \varphi_i (k) \varphi_j^{\varepsilon} (q) \right)
    \nonumber\\
    & = & \sum_{- 1 \leqslant i, - 1 \leqslant j \leqslant J_{\varepsilon},
    \max (i, j) = r} K_i (x) K_j^{\varepsilon} (z) . \nonumber
  \end{eqnarray}
  When $r \leqslant J_{\varepsilon} - 1$ the functions
  $\tilde{K}^{\varepsilon}_r$ does not depend on $\varepsilon$ (in the sense
  explained in Remark \ref{remark:k:epsilon} and Remark \ref{remark:epsilon2})
  and we have
  \[ \tilde{K}_r^{\varepsilon} (x, z) = K_r (x) \sum_{- 1 \leqslant j
     \leqslant r - 1} K_j (z) + \left( \sum_{- 1 \leqslant i \leqslant r - 1}
     K_i (x) \right) K_r (z) + K_r (x) K_r (z) . \]
  Using the fact that $(\varphi_j)_{j \geqslant - 1}$ is a dyadic partition of
  unity it is simple to see that
  \[ 2^{2 r} K_{- 1} (2^r x) = \sum_{- 1 \leqslant j \leqslant r} K_j (x) .
  \]
  This means that, when $r \leqslant J_{\varepsilon}-1$, there are $\bar{a} \in
  \mathbb{R}_+$ and $\bar{\beta} \in (0, 1)$ such that
  \[ | \tilde{K}_r^{\varepsilon} (x, z) | \lesssim 2^{4 r} \exp (- \bar{a} (1
     + 2^{2 r} | x |^2 + 2^{2 r} | z |^2)^{\bar{\beta}}), \]
  where the constants hidden in the symbol $\lesssim$ do not depend on $0 <
  \varepsilon \leqslant 1$ and $- 1 \leqslant r \leqslant J_{\varepsilon} -
  1$.\\
  
  In the case $r = J_{\varepsilon}$, the following equality holds
  \[ \tilde{K}^{\varepsilon}_{J_{\varepsilon}} (x, z) = K_{J_{\varepsilon}}
     (x) + \left( \sum_{- 1 \leqslant i \leqslant J_{\varepsilon} - 1} K_i (x)
     \right) K_{J_{\varepsilon}}^{\varepsilon} (z) . \] Thus, by Remark \ref{remark:partition:conditions} and Remark
  \ref{remark:epsilon2}, we obtain
  \[ | \tilde{K}^{\varepsilon}_{J_{\varepsilon}} (x, z) | \lesssim (2^{2
     J_{\varepsilon}} \exp (- \bar{a} (1 + 2^{2 J_{\varepsilon}} | x
     |^2)^{\bar{\beta}}) + 2^{4 J_{\varepsilon}} \exp (- \bar{a} (1 + 2^{2 r}
     | x |^2 + 2^{2 r} | z |^2)^{\bar{\beta}})), \]
  where the constants hidden in the symbol $\lesssim$ do not depend on $0 <
  \varepsilon \leqslant 1$.\\
  
  Finally, when $r > J_{\varepsilon}$, the function $\tilde{K}_r (x, z) = K_r
  (x)$ does not depend on $z \in \varepsilon \mathbb{Z}^2$, and, by Remark
  \ref{remark:partition:conditions}, we get
  \[ | \tilde{K}_r^{\varepsilon} (x, z) | \lesssim 2^{2 r} \exp (- \bar{a} (1
     + 2^{2 r} | x |^2)^{\bar{\beta}}), \]
  where, as usual, the constants hidden in the symbol $\lesssim$ do not
  depend on $0 < \varepsilon \leqslant 1$.
\end{remark}

\begin{definition}
  Consider $p \in [1, + \infty]$, $s_1, s_2 \in \mathbb{R}$, $\ell_1, \ell_2
  \in \mathbb{R}$ we define the Besov space of functions from $\mathbb{R}^2$
  taking values in the Banach space $B^{s_2}_{p, \ell_1, \ell_2} (\varepsilon
  \mathbb{Z}^2)$ as the Banach space of $f \in \mathcal{S}' (\mathbb{R}^2
  \times \varepsilon \mathbb{Z}^2)$ such that
  \begin{equation}
    \| f \|_{B^{s_1, s_2}_{p, \ell_2, \ell_2} (\mathbb{R}^2 \times \varepsilon
    \mathbb{Z}^2)} \assign \left\| \left\{ 2^{s_1 (j - 1)} \left\| \|
    \Delta_{(j - 1)}^x (f) \|_{B^{s_2}_{p, p, \ell_2} (\varepsilon
    \mathbb{Z}^2)} \right\|_{L^p_{\ell_1} (\mathbb{R}^2)} \right\}_{j \in
    \mathbb{N}_0} \right\|_{\ell^p (\mathbb{N}_0)} 
    \label{eq:norm:besovproduct}
  \end{equation}
  is finite. 
\end{definition}

\begin{remark}
  \label{remark:besov:banach}It is important to note that the norm
  {\eqref{eq:norm:besovproduct}} coincides with the one of the Besov space
  $B^{s_1}_{p, p, \ell_1}$ of distributions from $\mathbb{R}^2$ taking values
  in the Banach space $B^{s_2}_{p, p, \ell_2} (\varepsilon \mathbb{Z}^2)$.
  More generally if $E$ is a Banach space, $p, q \in [1, + \infty]$ and $s,
  \ell \in \mathbb{R}$ it is possible to define the Besov space $B^s_{p, q,
  \ell} (\mathbb{R}^d, E)$ of the distributions on $\mathbb{R}^d$ taking
  values in $E$ as the subset of $\mathcal{S}' (\mathbb{R}^d, E) \assign
  \mathcal{S}' (\mathbb{R}^d) \hat{\otimes} E$ with the finite norm
  \[ \| g \|_{B^s_{p, q, \ell} (\mathbb{R}^d, E)} = \| \{ 2^{s (j - 1)} \|
     \| \Delta_{(j - 1)} (f) \|_E \|_{L^p_{\ell} (\mathbb{R}^2)} \}_{j \in
     \mathbb{N}_0} \|_{\ell^q (\mathbb{N}_0)}, \]
  where $g \in \mathcal{S}' (\mathbb{R}^d, E)$ (see Chapter 2 of
  {\cite{Amann2019}} \ for the details). We will use this notion and the
  notation $B^s_{p, q, \ell} (\mathbb{R}^2, E)$ in what follows.
\end{remark}

We prove a useful equivalent expression of the norm $\| \cdot \|_{B^{s_1,
s_2}_{p, \ell_2, \ell_2} (\mathbb{R}^2 \times \varepsilon \mathbb{Z}^2)}$.

\begin{theorem}
  \label{theorem:besov:rz:norm}Consider $p \in [1, + \infty)$, $s_1, s_2 \in
  \mathbb{R}$, $\ell_1, \ell_2 \in \mathbb{R}$ if $f \in B^{s_1, s_2}_{p,
  \ell_1, \ell_2} (\mathbb{R}^2 \times \varepsilon \mathbb{Z}^2)$ we have that
  \begin{eqnarray}
    \| f \|^p_{B^{s_1, s_2}_{p, \ell_2, \ell_2} (\mathbb{R}^2 \times
    \varepsilon \mathbb{Z}^2)} & = & \sum_{i, j \geqslant - 1, j \leqslant
    J_{\varepsilon}} 2^{s_1 i p} 2^{s_2 j p} \int_{\mathbb{R}^2 \times
    \varepsilon \mathbb{Z}^2} | \Delta_i^x \Delta_j^z f (x, z) |^p
    \rho_{\ell_1} (x)^p \rho_{\ell_2} (z)^p \mathd x \mathd z, \nonumber\\
    & = & \sum_{i, j \geqslant - 1, j \leqslant J_{\varepsilon}} 2^{s_1 i p}
    2^{s_2 j p} \| \Delta_i^x \Delta_j^z f \|_{L^p_{\ell_1, \ell_2}
    (\mathbb{R}^2 \times \varepsilon \mathbb{Z}^2)}^p,
    \label{eq:besov:alternative:norm} 
  \end{eqnarray}
  where
  \[ \Delta_i^x \Delta_j^z f = K_i \ast^x (K_j^{\varepsilon} \ast^z f), \]
  where $\ast^x$ is the convolution with respect the continuum variables $x
  \in \mathbb{R}^2$ and $\ast^z$ is the convolution with respect the discrete
  variables $z \in \varepsilon \mathbb{Z}^2$.
\end{theorem}

\begin{proof}
  The proof follows from the fact that we can identify $\mathcal{S}'
  (\mathbb{R}^2, \mathcal{S}' (\varepsilon \mathbb{Z}^2))$ with $\mathcal{S}'
  (\mathbb{R}^2 \times \varepsilon \mathbb{Z}^2)$, from the monotone
  convergence theorem and from the definition of the norm
  {\eqref{eq:norm:besov:descrete}} of Besov space $B^{s_2}_{p, p, \ell_2}
  (\varepsilon \mathbb{Z}^2)$.
\end{proof}

We now propose an extension to the Besov spaces on $\mathbb{R}^2 \times
\varepsilon \mathbb{Z}^2$ of the results of the previous sections.

\begin{lemma}
  \label{lemma-fm-reg2}Assume that the smooth function $\sigma = \sigma (k, q)
  : \mathbb{R}^2 \times \varepsilon^{- 1} \mathbb{T}^2 \rightarrow \mathbb{R}$
  satisfies
  \[ | \partial_k^{\alpha} \partial_q^{\beta} \sigma (k, q) | \leqslant C (1 +
     | k |^2 + | q |^2)^{(m - (| \alpha | + | \beta |)) / 2}, \]
  where $\alpha, \beta \in \mathbb{N}_0^2$, $| \alpha |, | \beta | \leqslant
  4$, and $m \in \mathbb{R}$. Consider $s \in \mathbb{R}$, $\ell \in
  \mathbb{R}$ and $p, r \in [1, + \infty]$, then, the operator
  \[ T^{\sigma} f = \mathcal{F}^{- 1}_{\varepsilon} (\sigma (k, q)
     (\mathcal{F}_{\varepsilon} f) (k, q)), \quad f \in B^s_{p, \ell}
     (\mathbb{R}^2 \times \varepsilon \mathbb{Z}^2) \]
  is well defined from $B^s_{p, r, \ell} (\mathbb{R}^2 \times \varepsilon
  \mathbb{Z}^2)$ into $B^{s - m}_{p, r, \ell} (\mathbb{R}^2 \times \varepsilon
  \mathbb{Z}^2)$ and
  \begin{equation}
    \| T^{\sigma} f \|_{B_{p, r, \ell}^{s - m} (\mathbb{R}^2\times\varepsilon \mathbb{Z}^2)}
    \lesssim \| f \|_{B_{p, r, \ell}^s (\mathbb{R}^2\times\varepsilon \mathbb{Z}^2)} .
    \label{eq:inequality:sigma:rz1}
  \end{equation}
  Furthermore if $m \leqslant 0$ and considering $s_1, s_2 \in \mathbb{R}$ and
  $\theta_1, \theta_2 \geqslant 0$ such that $0 < \theta_1 + \theta_2
  \leqslant 1$ we have that $T^{\sigma}$ is a well defined operator from
  $B^{s_1, s_2}_{p, \ell_1, \ell_2} (\mathbb{R}^2 \times \varepsilon
  \mathbb{Z}^2)$ into $B^{s_1 - \theta_1 m, s_2 - \theta_2 m}_{p, \ell_1,
  \ell_2} (\mathbb{R}^2 \times \varepsilon \mathbb{Z}^2)$ and we have
  \begin{equation}
    \| T^{\sigma} f \|_{B^{s_1 - \theta_1 m, s_2 - \theta_2 m}_{p, \ell_1,
    \ell_2} (\mathbb{R}^2 \times \varepsilon \mathbb{Z}^2)} \lesssim \| f
    \|_{B^{s_1, s_2}_{p, \ell_1, \ell_2} (\mathbb{R}^2 \times \varepsilon
    \mathbb{Z}^2)} . \label{eq:inequality:sigma:rz2}
  \end{equation}
\end{lemma}

\begin{proof}
  The proof follows the same arguments of Lemma \ref{lemma-fm-reg}, we report
  here only the main differences. If we define
  \begin{eqnarray*}
    &  & K_{\lambda_1, \lambda_2}^{\sigma_{\varepsilon}} (x, z)\\
    & = & \frac{1}{(2 \pi)^4} \int_{\mathbb{R}^2 \times
    \mathbb{T}^2_{\frac{1}{\lambda_2 \varepsilon}}} e^{i (z \cdot q) + (x
    \cdot k)} \tilde{\psi}  (k) \tilde{\psi}  (q) \sigma_{\varepsilon}
    (\lambda_1 k, \lambda_2 q) \mathd k \mathd q\\
    & = & \frac{1}{(2 \pi)^4} \int_{\mathbb{R}^4} e^{i (z \cdot q) + (x \cdot
    k)} \tilde{\psi}  (k) \tilde{\psi}  (q) \sigma_{\varepsilon} (\lambda_1 k,
    \lambda_2 q) \mathd k \mathd q,
  \end{eqnarray*}
  where $\lambda_1 = 2^i$, $\lambda_2 = 2^j \leqslant 2^{J_{\varepsilon}}$ and
  $\psi : \mathbb{R}^2 \rightarrow [0, 1]$ is a suitable function supported in
  an annulus $B_{R_1} (0) \backslash B_{R_2} (0)$ of constant radius $R_1 >
  R_2 > 0$ as in the proof of Lemma \ref{lemma-fm-reg}, then we have
  \begin{eqnarray}
    | (1 + | z |^2 + | x |^2)^M K^{\sigma}_{\lambda_1, \lambda_2} (x, z) | &
    \leqslant & \left. \left| \sum_{| \alpha_1 | + | a_2 | + | \beta_1 | + |
    \beta_2 | = 2 M} C_{| \beta | + | \alpha |} c_{\alpha_1, \alpha_2,
    \beta_1, \beta_2} \lambda_1^{| \alpha_1 |} \lambda_2^{| \beta_1 |} \times
    \right. \right. \nonumber\\
    &  & \int_{\mathbb{R}^2} e^{i (z \cdot q) + i(x \cdot k)}
    \partial^{\alpha_2}_q \tilde{\psi}  (k) \partial^{\beta_2}_q \tilde{\psi} 
    (q) \times \nonumber\\
    &  & \left. \left. \phantom{\int_{\mathbb{R}^2}} (1 + | \lambda_2 k |^2 +
    | \lambda_2 q |^2)^{(m - | \beta_1 | - | \alpha_1 |) / 2} \mathd k \mathd
    q \right| \right. . \nonumber
  \end{eqnarray}
  So, since $\tilde{\psi}$ is supported in an annulus of constant radius, we
  can estimate on the support of $\tilde{\psi} (k) \tilde{\psi} (q)$:
  \[ \begin{array}{rl}
       &\lambda_1^{| \alpha_1 |} \lambda^{| \beta_1 |}_2 (1 + | \lambda_1 k |^2
       + | \lambda_2 q |^2)^{- (| \alpha_1 | + | \beta_1 |) / 2}\\
       = & (\lambda_1^{| \alpha_1 |} (1 + | \lambda_1 k |^2 + | \lambda_2 q
       |^2)^{- | \alpha_1 | / 2}) (\lambda^{| \beta_1 |}_2 (1 + | \lambda_1 k
       |^2 + | \lambda_2 q |^2)^{- | \beta_1 | / 2})\\
       \leqslant& \lambda_1^{| \alpha_1 |} (1 + | \lambda_1 k |^2)^{- |
       \alpha_1 | / 2} \lambda^{| \beta_1 |}_2 (1 + | \lambda_2 q |^2)^{- |
       \beta_1 | / 2}\\
       \leqslant & R_1^{| \alpha_1 |} R_1^{| \beta_1 |} .
     \end{array} \]
  On the other hand on the support of $\tilde{\psi} (k) \tilde{\psi} (q)$
  \[ (1 + | \lambda_1 k |^2 + | \lambda_2 q |^2)^{m / 2} \leqslant
     \left\{\begin{array}{l}
       R_1^m (\max (\lambda_1, \lambda_2))^m \hspace{4.6em} \text{if } m
       \leqslant 0\\
       (1 + 2 R_2)^m (\max (\lambda_1, \lambda_2))^m \hspace{1.2em} \text{if } m > 0
     \end{array}\right. . \]
  This implies that
  \[ \| K_{\lambda_1, \lambda_2} \|_{L_l^1 (\mathbb{R}^2)} \lesssim \max
     (\lambda_1, \lambda_2)^m \]
  which, by following the same argument of the proof of Lemma
  \ref{lemma-fm-reg}, implies inequality {\eqref{eq:inequality:sigma:rz1}}. On
  the other hand, when $m \leqslant 0$, we have
  \[ (1 + | \lambda_1 k |^2 + | \lambda_2 q |^2)^{m / 2} \leqslant (1 + |
     \lambda_1 k |^2)^{\theta_1 m / 2} (1 + | \lambda_2 q |^2)^{\theta_2 m /
     2} \leqslant R_1 \lambda_1^{\theta_1 m} \lambda_2^{\theta_2 m}, \]
  and following the previous line of reasoning we can analogously also obtain
  the second claim. 
\end{proof}

Hereafter we write
\[ \Delta_{\mathbb{R}^2 \times \varepsilon \mathbb{Z}^2} =
   \Delta_{\mathbb{R}^2} + \Delta_{\varepsilon \mathbb{Z}^2} \]
where $\Delta_{\mathbb{R}^2}$ is the standard Laplacian on $\mathbb{R}^2$
acting on the continuum variables $x \in \mathbb{R}^2$, and
$\Delta_{\varepsilon \mathbb{Z}^2}$ is the discrete Laplacian on $\varepsilon
\mathbb{Z}^2$ acting on the discrete variables $z \in \varepsilon
\mathbb{Z}^2$.

\begin{theorem}
  \label{theorem:besov:rz:equivalent}Consider $s \in \mathbb{R}$, $\ell \in
  \mathbb{R}$ and $p, q \in [1, + \infty]$ then for any $f \in B^s_{p, q,
  \ell} (\mathbb{R}^2 \times \varepsilon \mathbb{Z}^2)$ and any $m \in
  \mathbb{R}$ we have
  \[ \| f \|_{B^s_{p, q, \ell} (\mathbb{R}^2 \times \varepsilon \mathbb{Z}^2)}
     \sim \| f \cdot \rho_{\ell}^{(4)} (x, z) \|_{B^s_{p, q, \ell}
     (\mathbb{R}^2 \times \varepsilon \mathbb{Z}^2)} \]
  \[ \| (- \Delta_{\mathbb{R}^2 \times \varepsilon \mathbb{Z}^2} + 1)^{m/2} (f)
     \|_{B^{s - m}_{p, q, \ell} (\mathbb{R}^2 \times \varepsilon
     \mathbb{Z}^2)} \sim \| f \|_{B^s_{p, q, \ell} (\mathbb{R}^2 \times
     \varepsilon \mathbb{Z}^2)} . \]
  Consider $p \in [1, + \infty]$, $s_1, s_2 \in \mathbb{R}$ and $\ell_1,
  \ell_2 \in \mathbb{R}$. For any $f \in B^{s_1, s_2}_{p, \ell_1, \ell_2}
  (\mathbb{R}^2 \times \varepsilon \mathbb{Z}^2)$,and $m \leqslant 0$,
  $\theta_1, \theta_2 \geqslant 0$ and $0 < \theta_1 + \theta_2 < 1$ \ we have
  \[ \| f \|_{B^{s_1, s_2}_{p, \ell_1, \ell_2} (\mathbb{R}^2 \times
     \varepsilon \mathbb{Z}^2)} \sim \| f \cdot \rho_{\ell_1} (x)
     \rho_{\ell_2} (z) \|_{B^{s_1, s_2}_{p, 0, 0} (\mathbb{R}^2 \times
     \varepsilon \mathbb{Z}^2)}, \]
  \[ \| (- \Delta_{\mathbb{R}^2 \times \varepsilon \mathbb{Z}^2} + 1)^{\frac{m}{2}} (f)
     \|_{B^{s_1 - \theta_1 m, s_2 - \theta_2 m}_{p, \ell_1, \ell_2}
     (\varepsilon \mathbb{Z}^2)} \lesssim \| f \|_{B^{s_1, s_2}_{p, \ell_1,
     \ell_2} (\varepsilon \mathbb{Z}^2)} . \]
\end{theorem}

\begin{proof}
  The proof is analogous to the one of Theorem \ref{theorem:besov:weight},
  where we use the expressions of the norm given in Theorem
  \ref{theorem:besov:rz:norm}, and the results of Lemma \ref{lemma-fm-reg} are
  now replaced by the analogous ones of Lemma \ref{lemma-fm-reg2}. Indeed we
  have that
  \begin{eqnarray*}
    \sigma_{(- \Delta_{\mathbb{R}^2 \times \varepsilon \mathbb{Z}^2} + 1)}^m &
    = & \mathcal{F}_{\varepsilon} ((- \Delta_{\mathbb{R}^2 \times \varepsilon
    \mathbb{Z}^2} + 1)^m)\\
    & = & \left( | k |^2 + \sum_{j = 1, 2} \frac{4}{\varepsilon^2} \sin^2
    \left( \frac{\varepsilon q_i}{2} \right) + 1 \right)^m\\
    & = & \left( | k |^2 + \sigma_{\Delta_{\varepsilon \mathbb{Z}^2}} (q) + 1
    \right)^m.
  \end{eqnarray*}
  Moreover, we have that
  \begin{equation}
    \left| \partial^{\alpha} \left( \sigma_{(- \Delta_{\mathbb{R}^2 \times
    \varepsilon \mathbb{Z}^2} + 1)} (k, q) \right)^m \right| \lesssim
    \left\{\begin{array}{l}
      (| k |^2 + | q |^2 + 1)^{\frac{m - | \alpha |}{2} |} \quad \text{ for } m \geqslant
      0\\
      \left( | k |^2 + \frac{2 | q |^2}{\pi} + 1 \right)^{\frac{m - | \alpha |}{2}} \;
      \text{ for } m < 0
    \end{array},\right. \label{eq:inequalitysigma}
  \end{equation}
  where $\alpha\in \mathbb{N}^4_0$, the differential operator  is defined as $\partial^{\alpha}=\partial_{k_1}^{\alpha_1}\partial_{k_2}^{\alpha_2}\partial_{q_1}^{\alpha_3}\partial_{q_2}^{\alpha_4}$, and the constants hidden in the symbol $\lesssim$ do not depend on $0 <
  \varepsilon \leqslant 1$. Inequality {\eqref{eq:inequalitysigma}} permits
  the application of Lemma \ref{lemma-fm-reg2} in a way analogous of the one
  of the proof of Theorem \ref{theorem:besov:weight}, getting the thesis. \end{proof}

\begin{theorem}
  \label{theorem:besov:rz:product}Consider $s_1 \leqslant 0, s_2 > 0$ such
  that $s_1 + s_2 > 0$, $p_1, p_2, q_1, q_2 \in [1, + \infty]$, $p_1 \leq p_2$,
  $\frac{1}{p_1} + \frac{1}{p_2} = \frac{1}{p_3} < 1$, and $\frac{1}{q_1} +
  \frac{1}{q_2} = \frac{1}{q_3} < 1$ then the product $\Pi (f, g) = f \cdot g$
  defined on $\mathcal{S}' (\varepsilon \mathbb{Z}^2) \times \mathcal{S}
  (\varepsilon \mathbb{Z}^2)$ into $\mathcal{S}' (\varepsilon \mathbb{Z}^2)$
  can be extended in a unique continuous way to a map
  \[ \Pi \; : B^{s_1}_{p_1, q_1, \ell_1} (\mathbb{R}^2 \times \varepsilon
     \mathbb{Z}^2) \times B^{s_2}_{p_2, q_2, \ell_2} (\mathbb{R}^2 \times
     \varepsilon \mathbb{Z}^2) \rightarrow B^{s_1}_{p_3, q_3, \ell_1 + \ell_2}
     (\mathbb{R}^2 \times \varepsilon \mathbb{Z}^2), \]
  with the estimate, for any $f \in B^{s_1}_{p_1, q_1, \ell_1} (\mathbb{R}^2
  \times \varepsilon \mathbb{Z}^2)$ and $g \in B^{s_2}_{p_2, q_2, \ell_2}
  (\mathbb{R}^2 \times \varepsilon \mathbb{Z}^2)$,
  \[ \| f \cdot g \|_{B^{s_1}_{p_3, q_3, \ell_1 \ell_2} (\mathbb{R}^2 \times
     \varepsilon \mathbb{Z}^2)} \lesssim \| f \|_{B^{s_1}_{p_1, q_1, \ell_1}
     (\mathbb{R}^2 \times \varepsilon \mathbb{Z}^2)} \| g \|_{B^{s_2}_{p_2,
     q_2, \ell_2} (\mathbb{R}^2 \times \varepsilon \mathbb{Z}^2)} . \]
\end{theorem}

\begin{proof}
  The proof is similar to the one sketched in the proof of Theorem
  \ref{theorem:besov:product}. Indeed the support of the Fourier transform of
  $\tilde{\Delta}_r (f)$, for $f \in \mathcal{S}' (\mathbb{R}^2 \times
  \varepsilon \mathbb{Z}^2)$ and $r \geqslant 0$, is supported in the set $(B
  (0, 2^r R_1) \times B (0, 2^r R_1)) \backslash (B (0, 2^r R_2) \times B (0,
  2^r R_2))$ (where $B (x, R) \subset \mathbb{R}^2$ is the ball of radius $R >
  0$ and center $x \in \mathbb{R}^2$) for suitable positive constants $R_1 >
  R_2 > 0$. From this observation the proof follows the same line of the one
  on $\mathbb{R}^d$ (see, e.g., Chapter 2 of {\cite{Bahouri2011}} and the
  proof of Theorem 3.17 of {\cite{Mourrat_Weber2017}}).
\end{proof}

\begin{theorem}
  \label{theorem:besov:rz:embedding}Consider $p_1, p_2, q_1, q_2 \in [1, +
  \infty]$, $p_1 \leqslant p_2$, $s_1 > s_2$ and $\ell_1, \ell_2 \in \mathbb{R}$ if
  \[ \ell_1 \leqslant \ell_2, \quad s_1 - \frac{4}{p_1} \geqslant s_2 -
     \frac{4}{p_2}, \]
  then $B^{s_1}_{p_1, q_1, \ell_1} (\mathbb{R}^2 \times \varepsilon
  \mathbb{Z}^2)$ is continuously embedded in $B^{s_2}_{p_2, q_2, \ell_2}
  (\mathbb{R}^2 \times \varepsilon \mathbb{Z}^2)$, and the norm of the
  embedding is uniformly bounded in $0 < \varepsilon \leqslant 1$. Furthermore if $p_1, p_2 \in [1, + \infty]$, $p_1 \leqslant p_2$, $s_1, s_2, s_3, s_4 \in
  \mathbb{R}$ such that $s_1 > s_3$, $s_2 > s_4$, and $\ell_1, \ell_2, \ell_3,
  \ell_4 \in \mathbb{R}$ if
  \[ \ell_1 \leqslant \ell_3, \quad \ell_2 \leqslant \ell_4, \quad s_1 -
     \frac{2}{p_1} \geqslant s_3 - \frac{2}{p_2}, \quad s_2 - \frac{2}{p_1}
     \geqslant s_4 - \frac{2}{p_2} \]
  then $B^{s_1, s_2}_{p_1, \ell_1, \ell_2} (\mathbb{R}^2 \times \varepsilon
  \mathbb{Z}^2)$ is continuously embedded in $B^{s_3, s_4}_{p_2, \ell_3,
  \ell_4} (\mathbb{R}^2 \times \varepsilon \mathbb{Z}^2)$, and the norm of the
  embedding is uniformly bounded in $0 < \varepsilon \leqslant 1$.
\end{theorem}

\begin{proof}
  The proof of the first part of the theorem can be reduced to the one of
  Theorem \ref{theorem:besov:lattice:embedding}, using the same observation as
  in the proof of Theorem \ref{theorem:besov:rz:product}.\\
  
  The second part of the theorem is a special case of Theorem 2.2.4 and
  Theorem 2.2.5 of {\cite{Amann2019}} about the embedding of Besov spaces
  formed by distributions taking values in Banach spaces. 
\end{proof}

\subsection{Difference spaces and extension operator on $\mathbb{R}^2 \times
\varepsilon \mathbb{Z}^2$}

\

First we introduce the analogous of the difference spaces introduced in
Section \ref{section:besov:difference}.

\begin{definition}
  Consider $s > 0$, $p \in (1, + \infty)$ and $\ell \in \mathbb{R}$ we define
  the space $W^{s, p}_{\ell} (\mathbb{R}^2 \times \varepsilon \mathbb{Z}^2)$
  as the linear subspace of $L^p_{\ell} (\mathbb{R}^2 \times \varepsilon
  \mathbb{Z}^2)$ whose elements $f \in W^{s, p}_{\ell} (\mathbb{R}^2 \times
  \varepsilon \mathbb{Z}^2)$ are such that the norm
  \[ \| f \|_{W^{s, p}_{\ell} (\mathbb{R}^2 \times \varepsilon \mathbb{Z}^2)}
     \assign \| f \|_{L^p_{\ell}(\mathbb{R}^2 \times \varepsilon \mathbb{Z}^2)} + \| f \|_{S W^{s, p}_{\ell} (\mathbb{R}^2
     \times \varepsilon \mathbb{Z}^2)}, \]
  where $\| \cdot \|_{S W^{s, p}_{\ell} (\mathbb{R}^2 \times \varepsilon
  \mathbb{Z}^2)}$ is the seminorm
  \[ \| f \|_{S W^{s, p}_{\ell} (\mathbb{R}^2 \times \varepsilon
     \mathbb{Z}^2)} = \sup_{0 < | h | \leqslant 1, h \in \mathbb{R}^2 \times
     \varepsilon \mathbb{Z}^2} \frac{\| \tau_h f - f \|_{L^p_{\ell}
     (\mathbb{R}^2 \tmcolor{red}{\tmcolor{black}{\times \varepsilon
     \mathbb{Z}^2}})}}{| h |^s}, \]
  where $\tau_h f (y) = f (h + y)$, is finite.
\end{definition}

For the space $W^{s, p}_{\ell} (\mathbb{R}^2 \times \varepsilon \mathbb{Z}^2)$
we can prove a result analogous to Theorem \ref{theorem:sobolev_besov}.

\begin{theorem}
  \label{theorem:differenceBesov} Consider $0 < s \leqslant 1$, $p \in (1, +
  \infty)$ and $\ell \in \mathbb{R}$ then for any $\delta_1 > 0$ and any $f
  \in B^s_{p, p, \ell} (\mathbb{R}^2 \times \varepsilon \mathbb{Z}^2)$ we have
  \begin{equation}
    \| f \|_{B^{s - \delta_1}_{p, p, \ell} (\mathbb{R}^2 \times \varepsilon
    \mathbb{Z}^2)} \lesssim \| f \|_{W^{s, p}_{\ell} (\mathbb{R}^2 \times
    \varepsilon \mathbb{Z}^2)} \lesssim \| f \|_{B^{s + \delta_1}_{p, p, \ell}
    (\mathbb{R}^2 \times \varepsilon \mathbb{Z}^2)}.
    \label{eq:inequality:norms:rz1}
  \end{equation}
\end{theorem}

\begin{proof}
  We note that for any $h \in \mathbb{R}^2 \times \varepsilon \mathbb{Z}^2$ we
  have
  \begin{equation}
    \mathcal{F}_{\varepsilon} (\tau_h f) (k, q) = e^{i (h \cdot (k, q))}
    \mathcal{F}_{\varepsilon} (f) . \label{eq:translation:rz1}
  \end{equation}
  Using identity {\eqref{eq:translation:rz1}} and Lemma \ref{lemma-fm-reg2},
  the arguments of Lemma \ref{lemma:first:inequality}, Lemma
  \ref{lemma:second:inequality} and Theorem \ref{theorem:sobolev_besov} can be
  easily adapted for proving inequalities {\eqref{eq:inequality:norms:rz1}}.
\end{proof}

We study the following extension operator
$\overline{\mathcal{E}}^{\varepsilon}$ from $\mathcal{S}' (\mathbb{R}^2 \times
\varepsilon \mathbb{Z}^2)$ into $\mathcal{S}' (\mathbb{R}^2 \times
\mathbb{R}^2) \simeq \mathcal{S}' (\mathbb{R}^2) \hat{\otimes} \mathcal{S}'
(\mathbb{R}^2)$ defined as
\[ \overline{\mathcal{E}}^{\varepsilon} \assign I_{\mathcal{S}'
   (\mathbb{R}^2)} \otimes \mathcal{E}^{\varepsilon}, \]
where $\mathcal{E}^{\varepsilon} : \mathcal{S}' (\varepsilon \mathbb{Z}^2)
\rightarrow \mathcal{S}' (\mathbb{R}^2)$ is defined in
{\eqref{eq:definition:extension1}}. In particular if $f$ is a measurable
function on $\mathbb{R}^2 \times \varepsilon \mathbb{Z}^2$ we can identify
$\overline{\mathcal{E}}^{\varepsilon} (f)$ with a measurable function on
$\mathbb{R}^4$ and we have
\[ \overline{\mathcal{E}}^{\varepsilon} (f) (x, z) = \sum_{z' \in \varepsilon
   \mathbb{Z}^2} f (x, z') \mathbb{I}_{Q_{\varepsilon} (z')} (z), \]
where, as usual, $Q_{\varepsilon} (z')$ is defined in equation
{\eqref{eq:defsquares}}.

\begin{theorem}
  \label{theorem:extension3}For any $p \in (1, + \infty)$, $0 < s \leqslant 1$
  and $\ell \in \mathbb{R}$ and, for any $f_{\varepsilon} \in W^{s, p}_{\ell}
  (\mathbb{R}^2 \times \varepsilon \mathbb{Z}^2)$ we have
  \[ \| \overline{\mathcal{E}}^{\varepsilon} (f_{\varepsilon}) \|_{B^{s
     \wedge \frac{1}{p}}_{p, p, \ell} (\mathbb{R}^4)} \lesssim \|
     f_{\varepsilon} \|_{W^{s, p}_{\ell} (\mathbb{R}^2 \times \varepsilon
     \mathbb{Z}^2)} . \]
  Furthermore if $\sup_{0 < \varepsilon \leqslant 1} \| f_{\varepsilon}
  \|_{W^{s, p}_{\ell} (\mathbb{R}^2 \times \varepsilon \mathbb{Z}^2)} < +
  \infty$ and $\overline{\mathcal{E}}^{\varepsilon} (f_{\varepsilon})
  \rightarrow f$ in $\mathcal{S}' (\mathbb{R}^4)$, then $f \in B^s_{p, \infty,
  \ell} (\mathbb{R}^4)$ and $\| f \|_{B^s_{p, \infty, \ell} (\mathbb{R}^4)}
  \lesssim \sup_{0 < \varepsilon \leqslant 1} \| f_{\varepsilon} \|_{W^{s,
  p}_{\ell} (\mathbb{R}^2 \times \varepsilon \mathbb{Z}^2)}$, furthermore when
  $s = 1$ $f \in W^{1, p}_{\ell} (\mathbb{R}^4)$ and $\| f \|_{W^{1, p}_{\ell}
  (\mathbb{R}^4)} \lesssim \sup_{0 < \varepsilon \leqslant 1} \|
  f_{\varepsilon} \|_{W^{1, p}_{\ell} (\mathbb{R}^2 \times \varepsilon
  \mathbb{Z}^2)}$. Finally if $s \leqslant 0$ and $\delta > 0$ we get
  \[ \| \overline{\mathcal{E}}^{\varepsilon} (f_{\varepsilon}) \|_{B^{s -
     \delta}_{p, p, \ell} (\mathbb{R}^4)} \lesssim \| f_{\varepsilon}
     \|_{B^s_{p, p, \ell} (\mathbb{R}^2 \times \varepsilon \mathbb{Z}^2)} . \]
\end{theorem}

\begin{proof}
  The proof is a suitable modification of the proof of Theorem
  \ref{theorem:extension1}, Theorem \ref{theorem:extension2} and Theorem
  \ref{theorem:extension:lattice:negative}.
\end{proof}

We conclude this section with an useful compactness result about differences
and Besov spaces on $\mathbb{R}^2 \times \varepsilon \mathbb{Z}^2$.

\begin{theorem}
  \label{theorem:compact}Consider $p \in [1, + \infty]$, $0 < s_1 < s_2
  \leqslant 1$ and $\ell_1 > \ell_2$ then the space $W^{s_2, p}_{\ell_1}
  (\mathbb{R}^2 \times \varepsilon \mathbb{Z}^2)$ is compactly embedded in
  $W^{s_1, p}_{\ell_2} (\mathbb{R}^2 \times \varepsilon \mathbb{Z}^2)$.
\end{theorem}

\begin{proof}
  We prove the theorem in the case $\ell_2 = 0$. The general case can be
  obtained as in exploiting the same argument of the proof of Theorem
  \ref{theorem:besov:lattice:embedding}.\\
  
  For $\varepsilon > 0$, since $W^{s, p}_0 (\varepsilon \mathbb{Z}^2)$ is
  isomorphic to $L^p (\varepsilon \mathbb{Z}^2)$, the space $W^{s_2, p}_0
  (\mathbb{R}^2 \times \varepsilon \mathbb{Z}^2)$ is isomorphic to
  $B^{s_2}_{p, \infty, 0} (\mathbb{R}^2, L^p (\varepsilon \mathbb{Z}^2))$.
  Since $L^p (\varepsilon \mathbb{Z}^2)$ is compactly embedded in
  $L^p_{\frac{\ell_1}{2}} (\varepsilon \mathbb{Z}^2)$, by Theorem 7.2.3 of
  {\cite{Amann2019}}, the space $B^{s_2}_{p, \infty, 0} (\mathbb{R}^2, L^p
  (\varepsilon \mathbb{Z}^2))$ is compactly embedded (with respect to the
  natural embedding) in the space $B^{s_1}_{p, \infty, \frac{\ell_1}{2}}
  \left( \mathbb{R}^2, L^p_{\frac{\ell_1}{2}} (\varepsilon \mathbb{Z}^2)
  \right)$. Using the isomorphism between $L^p_{\frac{\ell_1}{2}} (\varepsilon
  \mathbb{Z}^2)$ and $B^{s_1}_{p, \infty, \frac{\ell_1}{2}} (\varepsilon
  \mathbb{Z}^2)$, the space $B^{s_1}_{p, \infty, \frac{\ell_1}{2}} \left(
  \mathbb{R}^2, L^p_{\frac{\ell_1}{2}} (\varepsilon \mathbb{Z}^2) \right)$ is
  continuously embedded in $W^{s_1, p}_{\ell_1} (\mathbb{R}^2 \times
  \varepsilon \mathbb{Z}^2)$. The thesis, thus, follows from the fact that the
  composition of a compact embedding with a continuous embedding is compact.
\end{proof}

\begin{remark}
  The isomorphism between $W^{s, p}_0 (\varepsilon \mathbb{Z}^2)$ and $L^p
  (\varepsilon \mathbb{Z}^2)$ (equipped with their natural norm
  {\eqref{eq:no}} and {\eqref{eq:Lplattice}} respectively) is not an isometry,
  and the norm of this isomorphism diverges to $+ \infty$ as $\varepsilon
  \rightarrow 0$. This is the reason why we cannot use it in the proofs of the
  previous results. Since compactness is a topological property and not a
  metric one we can use the isomorphism between $W^{s, p}_0 (\varepsilon
  \mathbb{Z}^2)$ and $L^p (\varepsilon \mathbb{Z}^2)$ in the proof of Theorem
  \ref{theorem:compact}. 
\end{remark}

\begin{corollary}
  \label{corollary:compact}Consider $0 < s_1 < s_2 \leqslant 1$ and $\ell_1 >
  \ell_2$ then the space $B^{s_2}_{p, p, \ell_1} (\mathbb{R}^2 \times
  \varepsilon \mathbb{Z}^2)$ is compactly embedded in $B^{s_1}_{p, p, \ell_1}
  (\mathbb{R}^2 \times \varepsilon \mathbb{Z}^2)$.
\end{corollary}

\begin{proof}
  The corollary is a consequence of Theorem \ref{theorem:differenceBesov} and
  Theorem \ref{theorem:compact}.
\end{proof}

\subsection{Positive distributions and Besov
spaces}\label{section:positive:distribution}

\

In this section we want to study positive (tempered) distribution on
$\mathbb{R}^d$ or $\mathbb{R}^2 \times \varepsilon \mathbb{Z}^2$. Namely we
want to consider elements $\eta$ or $\eta_{\varepsilon}$ respectively in
$\mathcal{S}' (\mathbb{R}^4)$ or $\mathcal{S}' (\mathbb{R}^2 \times
\varepsilon \mathbb{Z}^2)$ such that, for any $f \in \mathcal{S}
(\mathbb{R}^4)$ or $f_{\varepsilon} \in \mathcal{S} (\mathbb{R}^2 \times
\varepsilon \mathbb{Z}^2)$ positive, i.e. for each $y \in \mathbb{R}^4$ and
$(x, z) \in \mathbb{R}^2 \times \varepsilon \mathbb{Z}^2$ we have $f (y),
f_{\varepsilon} (x, z) \geqslant 0$, we have
\[ \langle f, \eta \rangle, \langle f_{\varepsilon}, \eta_{\varepsilon}
   \rangle \geqslant 0. \]
For well known theorems on positive tempered distribution on $\mathbb{R}^d$
see, for example, {\cite{Treves}}. These can be easily extended to the case  
of underlying space $\mathbb{R}^2 \times \varepsilon \mathbb{Z}^2$. In
particular we can identify every positive tempered distribution on
$\mathbb{R}^2 \times \varepsilon \mathbb{Z}^2$ with a Radon ($\sigma$-finite)
positive measure (see Theorem 21.1 of {\cite{Treves}} for a proof in
$\mathbb{R}^d$ case). This implies that considering a positive distribution in
$B^s_{p, q, \ell} (\mathbb{R}^d)$, $B^{s_1, s_2}_{p, \ell_1, \ell_2}
(\mathbb{R}^2 \times \varepsilon \mathbb{Z}^2)$ etc. is equivalent to consider
a Radon measure that is also in the previous spaces.

\

We introduce the following functions
\begin{equation}
  E_{i, j}^{a, \beta} (x, z) = 2^{2 i + 2 j} \exp \left( - a (1 + 2^{2 i} | x
  |^2 + 2^{2 j} | z |^2)^{\frac{\beta}{2}} \right) ,\label{eq:definitionE}
\end{equation}
where $a \in \mathbb{R}_+$, $0 < \beta < 1$, $i, j \in \mathbb{N}$ and $0
\leqslant j \leqslant J_{\varepsilon}$, and
\begin{equation}
  \tilde{E}_i^{a, \beta} (x) = 2^{2 i} \exp \left( - a (1 + 2^{2 i} | x
  |^2)^{\frac{\beta}{2}} \right) .\label{eq:definitionE1}
\end{equation}

We also introduce the following positive functionals

\begin{eqnarray}
  \mathbb{M}_{p, \ell_1, \ell_2, \varepsilon}^{s, s'} (\eta, a, \beta) &
  \assign & \left( \sum_{i \in \mathbb{N}_0, j \leqslant J_{\varepsilon}} 2^{s
  p i+ s' p j} \int_{\mathbb{R}^2 \times \varepsilon \mathbb{Z}^2}(\rho_{\ell_1} (x) \rho_{\ell_2} (z))^p ((E^{a,
  \beta}_{i, j} \asterisk  \eta) (x, z))^p
  \mathd x \mathd z \right)^{1 / p}  \label{eq:Mss1}\\
  \mathbb{M}_{p, \ell_1, \ell_2}^{s, s'} (\eta, a, \beta) & \assign & \left(
  \sum_{i, j \in \mathbb{N}_0} 2^{s p i + s' p j} \int_{\mathbb{R}^4} (\rho_{\ell_1} (x) \rho_{\ell_2} (z))^p((E^{a,
  \beta}_{i, j} \asterisk  \eta) (x, z))^p
  \mathd x \mathd z \right)^{1 / p}  \label{eq:Mss2}\\
  \mathbb{N}^s_{p, \ell, \varepsilon} (\eta, a, \beta) & \assign & \left(
  \sum_{- 1 \leqslant r \leqslant J_{\varepsilon}} 2^{r s p}
  \int_{\mathbb{R}^2 \times \varepsilon \mathbb{Z}^2} (\rho_{\ell}^{(4)} (x, z))^p ((E^{a, \beta}_{r, r}
  \asterisk \eta) (x, z))^p \mathd x \mathd z
  \right)^{1 / p} \nonumber\\
  &  & + \left( \sum_{r \geqslant J_{\varepsilon}} 2^{r s p}
  \int_{\mathbb{R}^2 \times \varepsilon \mathbb{Z}^2} (\rho_{\ell}^{(4)} (x, z))^p((\tilde{E}^{a, \beta}_r
  \asterisk_x  \eta) (x, z))^p \mathd x \mathd z
  \right)^{1 / p} \nonumber\\
  &  & + \left( \int_{\mathbb{R}^2 \times \varepsilon \mathbb{Z}^2}(\rho_{\ell}^{(4)} (x,
  z))^p
  ((\tilde{E}^{a, \beta}_{J_{\varepsilon}} \asterisk_z  \eta) (x, z))^p \mathd x \mathd z \right)^{1 / p}  \label{eq:Ns1}\\
  \mathbb{N}^s_{p, \ell} (\eta, a, \beta) & \assign & \left( \sum_{r \geqslant
  - 1} 2^{r s p} \int_{\mathbb{R}^4} (\rho_{\ell}^{(4)} (x, z))^p ((E^{a, \beta}_{r, r} \asterisk
   \eta) (x, z))^p \mathd x \mathd z \right)^{1 / p} .
  \label{eq:Ns2}
\end{eqnarray}

\begin{remark}\label{remark:exchangeconvolution}
It is important to note that in the expressions \eqref{eq:Mss1}, \eqref{eq:Mss2}, \eqref{eq:Ns1} and \eqref{eq:Ns2}, we can "exchange the order" between the convolution operation $E^{a,\beta}_{i,j}\asterisk$ and the multiplication by the weights $\rho_{\ell_1}\rho_{\ell_2}$ and $\rho^{(4)}_{\ell}$. Indeed, by using the fact that $\rho_{\ell}(x) \lesssim (1+|x-y|^{\upsilon}) \rho_{\ell}(y)$ (for some $\upsilon>0$), and the fact that $\eta$ is a positive distribution, we get that, for any $\kappa>0$ small enough,  
\[ (E^{a+\kappa,\beta} \asterisk (\rho_{\ell_1} \rho_{\ell_2}  \eta)) (x,z) \lesssim  \rho_{\ell_1}(x) \rho_{\ell_2}(z) (E^{a,\beta} \asterisk  \eta) (x,z)  \lesssim (E^{a-\kappa,\beta} \asterisk (\rho_{\ell_1} \rho_{\ell_2}  \eta)) (x,z), \]
with a similar relation holding replacing $\rho_{\ell_1}(x) \rho_{\ell_2}(z)$  by $\rho_{\ell}^{(4)}$, $E^{a,\beta}_{i,j}$ by $\tilde{E}^{a,\beta}_{r}$, and $\mathbb{R}^2 \times \varepsilon\mathbb{Z}^2$ by $\mathbb{R}^4$. In this way we obtain, for example in the case of $\mathbb{N}_{p,\ell}^s(\eta,a,\beta)$, 
\begin{multline*} 
\left( \sum_{r \geqslant
  - 1} 2^{r s p} \int_{\mathbb{R}^4} (E^{a+\kappa, \beta}_{r, r} \asterisk
  ( \rho_{\ell}^{(4)} \eta )(x,z))^p \mathd x \mathd z \right)^{1 / p}  \lesssim \mathbb{N}^s_{p, \ell} (\eta, a, \beta) \lesssim \\
  \left( \sum_{r \geqslant
  - 1} 2^{r s p} \int_{\mathbb{R}^4} (E^{a-\kappa, \beta}_{r, r} \asterisk
  ( \rho_{\ell}^{(4)} \eta )(x,z))^p \mathd x \mathd z \right)^{1 / p},
  \end{multline*}
  with analogous inequalities holding for $\mathbb{M}^{s,s'}_{p,\ell_1,\ell_2,\varepsilon}$, $\mathbb{M}^{s,s'}_{p,\ell_1,\ell_2}$, and $\mathbb{N}^s_{p, \ell,\varepsilon} (\eta, a, \beta)$.
\end{remark}

\begin{theorem}
\label{theorem:norminequalities}For any $p \in [1, + \infty]$, $s_1, s_2, s
  \in \mathbb{R}$ and $\ell_1, \ell_2 \in \mathbb{R}$, for any positive Radon
  measures $\eta_{\varepsilon} \in B^{s_1, s_2}_{p, p, \ell_1, \ell_2}
  (\mathbb{R}^2 \times \varepsilon \mathbb{Z}^2)$ and $\eta \in B^{s_1,
  s_2}_{p, p, \ell_1, \ell_2} (\mathbb{R}^4)$ we have
  \[ \| \eta_{\varepsilon} \|_{B^{s_1, s_2}_{p, \ell_1, \ell_2} (\mathbb{R}^2
     \times \varepsilon \mathbb{Z}^2)} \lesssim \mathbb{M}_{p, \ell_1, \ell_2,
     \varepsilon}^{s_1, s_2} (\eta_{\varepsilon}, \bar{a}, \bar{\beta}), \quad
     \| \eta \|_{B^{s_1, s_2}_{p, \ell_1, \ell_2} (\mathbb{R}^4)} \lesssim
     \mathbb{M}_{p, \ell_1, \ell_2}^{s_1, s_2} (\eta, \bar{a}, \bar{\beta}),
  \]
  where $\bar{a} \in \mathbb{R}_+$ and $\bar{\beta} \in (0, 1)$ are the
  constants appearing in Remark \ref{remark:partition:conditions}, Remark
  \ref{remark:k:epsilon} and Remark \ref{remark:epsilon2}. Furthermore for any
  $s \in \mathbb{R}$ and $\eta_{\varepsilon} \in B^s_{p, p, \ell}
  (\mathbb{R}^2 \times \varepsilon \mathbb{Z}^2)$ and $\eta \in B^s_{p, p,
  \ell} (\mathbb{R}^4)$ we have
  \[ \| \eta_{\varepsilon} \|_{B^s_{p, p, \ell} (\mathbb{R}^2 \times
     \varepsilon \mathbb{Z}^2)} \lesssim \mathbb{N}_{p, \ell, \varepsilon}^s
     (\eta_{\varepsilon}, \bar{a}, \bar{\beta}), \quad \| \eta \|_{B^s_{p, p,
     \ell} (\mathbb{R}^4)} \lesssim \mathbb{N}_{p, \ell}^s (\eta, \bar{a},
     \bar{\beta}) . \]
\end{theorem}

\begin{proof}
  The thesis is a consequence of the fact that $\eta_{\varepsilon}, \eta$ are
  positive measures (thus $\langle f_{\varepsilon}, \eta_{\varepsilon} \rangle
  \leqslant \langle g_{\varepsilon}, \eta_{\varepsilon} \rangle$, $\langle f,
  \eta \rangle \leqslant \langle g, \eta \rangle$ for any $f_{\varepsilon},
  g_{\varepsilon} \in \mathcal{S} (\mathbb{R}^2 \times \varepsilon
  \mathbb{Z}^2)$ and $f, g \in \mathcal{S} (\mathbb{R}^4)$ such that
  $f_{\varepsilon} \leqslant g_{\varepsilon}$ and $f \leqslant g$), and of
  Remark \ref{remark:partition:conditions}, Remark \ref{remark:k:epsilon} and
  Remark \ref{remark:epsilon2}.
\end{proof}

\begin{theorem}
  \label{thm:estimate-measure}Consider a positive measure $\eta$, suppose that
  $\nobracket \mathbb{M}_{r, \ell, \varepsilon}^{- s_1, - s_2} (\eta, a,
  \beta)) < + \infty$ for some $s_1, s_2 \geqslant 0$, $\ell \geqslant 0$ and
  $r \in [1, + \infty]$ and consider $g \in L^{p'}_{\ell} (\mathd \eta)$ for
  some $p' \in (1, + \infty)$. Then we have that $g \mathd \eta \in B^{-
  k_1}_{p, p, \ell} (\mathbb{R}^2, B^{- k_2}_{p, p, \ell} (\varepsilon
  \mathbb{Z}^2))$ with $p = r q'$, where $\frac{1}{p'} + \frac{1}{q'} = 1$,
  and $k_i = \frac{2}{p'} + \frac{s_i}{q'}$. Furthermore we have
  \[ \| g \mathd \eta \|_{B^{- k_1, - k_2}_{p, \ell, \ell} (\mathbb{R}^2
     \times \varepsilon \mathbb{Z}^2)} \lesssim \| g \|_{L^{p'}_{\ell}
     (\mathbb{R}^2 \times \varepsilon \mathbb{Z}^2, d \eta)} (\mathbb{M}_{r,
     \ell, \varepsilon}^{- s_1, - s_2} (\eta, q' \bar{a}, \bar{\beta}))^{r /
     p}, \]
  where $\bar{a} \in \mathbb{R}_+$ and $\bar{\beta} \in (0, 1)$ are the
  constants in Remark \ref{remark:partition:conditions}, Remark
  \ref{remark:k:epsilon} and Remark \ref{remark:epsilon2}.
\end{theorem}

\begin{proof}
  Let $K_j (\cdot)$ be the Littlewood-Paley function linked with the
  Littlewood-Paley block $\Delta_j$. By our assumptions on $K_j$ (see Remark
  \ref{remark:partition:conditions}, Remark \ref{remark:k:epsilon} and Remark
  \ref{remark:epsilon2}) there exists some $\bar{a} \in \mathbb{R}_+$ and
  $\bar{\beta} \in (0, 1)$ such that
  \[ | K_j (x) | \lesssim 2^{2 j} \exp (- \bar{a} 2^{j \bar{\beta}} | x
     |^{\bar{\beta}}), \]
  for $i \in \mathbb{N}_0$, $x \in \mathbb{R}^2$ and uniformly in
  $\varepsilon$. We have that by H{\"o}lder's inequality
  \begin{equation}
    \begin{array}{rl}
      &K_i \asterisk_x K_j \asterisk_z (\rho_{\ell'} (x') \rho_{\ell'} (z') g
      \mathd \eta) (x, z)\\
      \lesssim& \| g \|_{L^{p'}_{\ell'} (\mathd \eta)} \left(
      \int_{\mathbb{R}^2 \times \varepsilon \mathbb{Z}^2} | K_i (x - x')
      |^{q'} | K_j (z - z') |^{q'} \rho_{\ell'} (x') \rho_{\ell'} (z') \mathd
      \eta (x', z') \right)^{\frac{1}{q'}}\\
      \lesssim &\| g \|_{L^{p'}_{\ell'} (\mathd \eta)} \left( \int (E_{i,
      j}^{\bar{a}, \bar{\beta}} (x - x', z - z'))^{q'} \rho_{\ell'} (x')
      \rho_{\ell'} (z') \mathd \eta (x', z') \right)^{\frac{1}{q'}}\\
      \lesssim & \| g \|_{L^{p'}_{\ell'} (\mathd \eta)} 2^{\frac{2 j}{p'} +
      \frac{2 i}{p'}} \left( \int E_{i, j}^{q' \bar{a}, \bar{\beta}} (x - x',
      z - z') \rho_{\ell'} (x') \rho_{\ell'} (z') \mathd \eta (x', z')
      \right)^{\frac{1}{q'}},
    \end{array}
  \end{equation}
  where we used that
  \[ (E_{i, j}^{\bar{a}, \bar{\beta}} (x - x', z - z'))^{q'} = 2^{2 i (q' -
     1) + 2 j (q' - 1)} E_{i, j}^{q' \bar{a}, \bar{\beta}} (x - x', z - z') .
  \]
  On the other hand, by Theorem \ref{theorem:besov:rz:norm} and Remark \ref{remark:exchangeconvolution}, we have
  \begin{eqnarray}
    \| g \mathd \eta \|_{B^{- k_1, - k_2}_{p, \ell', \ell'} (\mathbb{R}^2
    \times \varepsilon \mathbb{Z}^2)}^p & \lesssim & \sum_{i \geqslant 0}
    \sum_{- 1 \leqslant j \leqslant J_{\varepsilon}} 2^{- k_1 p i - k_2 p j}
    \times \nonumber\\
    &  & \times \int_{\mathbb{R}^2 \times \varepsilon \mathbb{Z}^2} | K_i
    \ast_x K_j \ast_z (\rho_{\ell'} (x) \rho_{\ell'} (z) g \mathd \eta) |^p
    \mathd x \mathd z \nonumber\\
    & \lesssim & \| g \|_{L^{p'}_{\ell'} (\mathd \eta)}^p \sum_{i \geqslant
    0} \sum_{- 1 \leqslant j \leqslant J_{\varepsilon}} 2^{- \left( k_1 -
    \frac{2}{p'} \right) p i - \left( k_2 - \frac{2}{p'} \right) p j} \times
    \nonumber\\
    &  & \times \int_{\mathbb{R}^2 \times \varepsilon \mathbb{Z}^2} (E^{q'
    \bar{a}, \bar{\beta}}_{i, j} \asterisk (\rho_{\ell} (x) \rho_{\ell} (z)
    \eta (x, z)))^{\frac{p}{q'}} \mathd x \mathd z \nonumber\\
    & \lesssim & \| g \|_{L^{p'}_{\ell'} (\mathd \eta)}^p \sum_{i \geqslant
    0} \sum_{- 1 \leqslant j \leqslant J_{\varepsilon}} 2^{- s_1 r i - s_2 r
    j} \times \nonumber\\
    &  & \times \int_{\mathbb{R}^2 \times \varepsilon \mathbb{Z}^2} (E^{q'
    \bar{a}, \bar{\beta}}_{i, j} \asterisk (\rho_{\ell} (x, z) \eta (x, z)))^r
    \mathd x \mathd z \nonumber\\
    & \lesssim & \| g \|_{L^{p'}_{\ell'} (\mathd \eta)}^p (\mathbb{M}_{r,
    \ell, \varepsilon}^{- s_1, - s_2} (\eta, q' \bar{a}-\kappa, \bar{\beta}))^r .
    \nonumber
  \end{eqnarray} \end{proof}

\section{Stochastic estimates and regularity of the
noise}\label{sec:stochastic-estimates}

\

In this section we want to study the regularity of the stochastic terms in
equation {\eqref{eq:formal2}}.

\subsection{The setting}\label{section:setting}

\

Let $\xi$ be a white noise on $\mathbb{R}^4$. Using the noise $\xi$ we can
define a white noise $\xi_{\varepsilon}$ on $\mathbb{R}^2 \times \varepsilon
\mathbb{Z}^2$ in the following way
\[ \xi_{\varepsilon} (x, z) =\frac{1}{\varepsilon} \xi (\delta_x \otimes
   \mathbb{I}_{Q_{\varepsilon} (z)}), \]
where $Q_{\varepsilon} (z)$ is defined in equation {\eqref{eq:defsquares}},
$\delta_x$ is the Dirac delta with unitary mass in $x \in \mathbb{R}^2$ and
$\mathbb{I}_{Q_{\varepsilon} (z)}$ denotes the indicator function of the set
$Q_{\varepsilon} (z)$.\\

 We can define also the following natural free fields on $\mathbb{R}^4$ and
$\mathbb{R}^2 \times \varepsilon \mathbb{Z}^2$.
\[ W = (- \Delta_{\mathbb{R}^4} + m^2)^{- 1} (\xi), \]
\[ W_{\varepsilon} = (- \Delta_{\mathbb{R}^2 \times \varepsilon \mathbb{Z}^2}
   + m^2)^{- 1} (\xi_{\varepsilon}), \]
\[ \bar{W}_{\varepsilon} = \overline{\mathcal{E}}_{\varepsilon} ((-
   \Delta_{\mathbb{R}^2 \times \varepsilon \mathbb{Z}^2} + m^2)^{- 1}
   (\xi_{\varepsilon})) . \]
We have to define four Green functions:
\[ \mathcal{G}_{\varepsilon} : \mathbb{R}^2 \times \varepsilon \mathbb{Z}^2
   \rightarrow \mathbb{R}, \quad \overline{\mathcal{G}}_{\varepsilon} :
   \mathbb{R}^4 \rightarrow \mathbb{R}, \quad
   \widetilde{\mathcal{G}}_{\varepsilon} : \mathbb{R}^4 \rightarrow
   \mathbb{R}, \quad \mathcal{G} : \mathbb{R}^4 \rightarrow \mathbb{R}, \]
where
\begin{align*}
    \mathcal{G}_{\varepsilon} (x, z) \assign & \frac{1}{(2 \pi)^4}
   \int_{\mathbb{R}^2 \times \mathbb{T}^2_{\frac{1}{\varepsilon}}} \frac{e^{-
   i (x \cdot y + k \cdot z)}}{\left( | y |^2 + 4 \varepsilon^{- 2} \sin^2
   \left( \frac{\varepsilon k_1}{2} \right) + 4 \varepsilon^{- 2} \sin^2
   \left( \frac{\varepsilon k_2}{2} \right) + m^2 \right)^2} \mathd y \mathd k
   ; \\
\overline{\mathcal{G}}_{\varepsilon} (x, z) \assign & \frac{1}{(2 \pi)^4}
   \int_{\mathbb{R}^4} \frac{\left(16 \sin^2 \left( \frac{\varepsilon k_1}{2}
   \right) \sin^2 \left( \frac{\varepsilon k_2}{2} \right)
   \right)}{\varepsilon^4 k_1^2 k_2^2} \frac{e^{- i (x \cdot y + k \cdot
   z)}}{\left( | y |^2 + 4 \varepsilon^{- 2} \sin^2 \left( \frac{\varepsilon
   k_1}{2} \right) + 4 \varepsilon^{- 2} \sin^2 \left( \frac{\varepsilon
   k_2}{2} \right) + m^2 \right)^2} \mathd y \mathd k \\
   \widetilde{\mathcal{G}}_{\varepsilon} (x, z) \assign & \frac{1}{(2 \pi)^4}
   \int_{\mathbb{R}^4} \frac{\left(16 \sin^2 \left( \frac{\varepsilon k_1}{2}
   \right) \sin^2 \left( \frac{\varepsilon k_2}{2} \right)
   \right)}{\varepsilon^4 k_1^2 k_2^2} \frac{e^{- i (x \cdot y + k \cdot
   z)}}{(| y |^2 + | k |^2 + m^2)^2} \mathd y \mathd k 
   \end{align*}
and
\[ \mathcal{G} (x, z) \assign \frac{1}{(2 \pi)^4} \int_{\mathbb{R}^4}
   \frac{e^{- i (x \cdot y + k \cdot z)}}{(| y |^2 + | k |^2 + m^2)^2} \mathd
   y \mathd k. \]
Using the previous notation we have that
\[ \mathcal{G} (x - x', z - z') =\mathbb{E} [W (x, z) W (x', z')], \quad
   \mathcal{G}_{\varepsilon} (x - x', z - z') =\mathbb{E} [W_{\varepsilon} (x,
   z) W_{\varepsilon} (x', z')], \]
\[ \overline{\mathcal{G}}_{\varepsilon} (x, z) =\mathbb{E}
   [\bar{W}_{\varepsilon} (x, z) \bar{W}_{\varepsilon} (x', z')] . \]

\subsection{Wick exponential and multiplicative Gaussian chaos}

Giving a meaning to equation {\eqref{eq:formal2}} will require to us define
the exponential of $W$. However $W$ is almost surely not a function, so we
require renormalization. For $\alpha^2 \in [0, 2 (4 \pi)^2)$, we want to
define
\[ \mu^{\alpha} = : \exp (\alpha W) : ='' \exp \left( \alpha W -
   \frac{\alpha^2}{2} \mathbb{E} [W^2] \right) \;'' \]
which does not immediately make sense since $\mathbb{E} [W^2] = \infty$.
However we can approximate it in the following way: Let $\sigma_s$ be a smooth compactly supported function such
that, for any $T > 0$,
\[ \int^T_0 \sigma^2_s (n) \mathd s = \frac{\rho_T (n)}{(m^2 + | n |^2)^2}, \quad n\in \mathbb{R}^4, \]
with $\rho_T \in C^{\infty}_c (\mathbb{R}^4)$ depending smoothly on $T$ and
$\rho_T \rightarrow 1$ as $T \rightarrow \infty$. Then define
\[ W_T = \int^T_0 \sigma_{t} (D) \mathd X_t ,\]
where $X_t$ is a cylindrical Browning motion in $L^2 (\mathbb{R}^4)$. Since
$W_T$ has almost surely compact support in frequency it is a function and we
can compute by Ito's isometry
\[ \mathbb{E} [W_T (x) W_T (y)] = K_T (x - y), \]
where we have introduced $K_T = \mathcal{F}^{- 1} \left( \frac{\rho_T
(n)}{(m^2 + | n |^2)^2} \right)$. From this it follows that
\[ W_T \rightarrow W, \]
as $T \rightarrow \infty$ at least in law since $W_T$ is Gaussian. We can then
introduce the approximation
\[ \mu^{\alpha, T} = \exp \left( \alpha W_T - \frac{\alpha^2}{2} \mathbb{E}
   [W_T^2] \right) . \]
Recall that for $W_T$ we can also define the Wick powers $: W^n_T :$
inductively by $: W^1_T : = W_T$ and
\[ : W^{n + 1}_T : = n \int_0^T W_t \mathd : W^n_t : . \]
An alternative definition of $\mu^{\alpha, T}$ would be to write the
exponential as a series and Wick order the individual terms. The following
lemma shows that these two ways are equivalent and that $\mu^{\alpha, T}$
converges to a well defined limit as $T \rightarrow \infty$ which we will
denote by $\mu^{\alpha}$. Furthermore define $\mu^{- \alpha}$ as the limit of
\[ \mu^{- \alpha, T} = \exp \left( - \alpha W_T - \frac{\alpha^2}{2}
   \mathbb{E} [W_T^2] \right). \]

\begin{lemma}
  \label{lemma:series}For any $\alpha \in \mathbb{R}$,we have that
  \[ \exp \left( \alpha W_T - \frac{\alpha^2}{2} \mathbb{E} [W_T^2] \right) =
     \sum_{n\in \mathbb{N}_{0}} \frac{\alpha^n}{n!} : (W_T)^n : . \]
  Furthermore, when $\alpha^2 < (4 \pi)^2$, we get that
  \[ \lim_{T \rightarrow \infty} \sum_{n \in \mathbb{N}_0} \frac{\alpha^n}{n!} : W_T^n : = \sum_{n\in \mathbb{N}_{0}}
     \frac{\alpha^n}{n!} : W^n : \]
  almost surely.
\end{lemma}

\begin{proof}
  We define
  \[ M_T^1 = \exp \left( \alpha W_T - \frac{\alpha^2}{2} \mathbb{E} [W_T^2]
     \right) \]
  and
  \[ M_T^2 = \sum_{n \in \mathbb{N}} \frac{\alpha^n}{n!} : W_T^n : . \]
  We claim that $M_T^1 = M_T^2$. Obviously $M_0^1 = M_0^2$. By the well known
  calculation for the stochastic exponential \
  \[ \mathd M_T^1 = \alpha M^1_T \mathd W_T, \]
  from which it also follows that $M^{1}_T$ is a local martingale. Furthermore, since $\alpha^2 < (4 \pi)^2$ and, thus, the second moment of $M^{1}_T$ is finite, we
  have that \tmcolor{black}{$\sup_T \mathbb{E} [\| M^{1}_T \|^2_{B^{-
  s}_{2, 2, \ell}}] < \infty$ for some $s > 0$, so considered as a process on the
  Hilbert space $B^{- s}_{2, 2, \ell}$, $M^1_T$ has bounded quadratic variation }and
  converges to some $M^1_{\infty}$ in $B^{- s}_{2, 2, \ell}$ almost surely.\\
  
  On the other hand
  \begin{eqnarray*}
    \mathd M_T^2 & = & \mathd \left( \sum_{n \in \mathbb{N}}
    \frac{\alpha^n}{n!} : W_T^n : \right)\\
    & = & \left( \sum_{n \in \mathbb{N}} \frac{\alpha^n}{n!} \mathd : W_T^n :
    \right)\\
    & = & \left( \sum_{n \in \mathbb{N}} \frac{\alpha^n}{(n - 1) !} : W_T^{n
    - 1} : \mathd W_T \right)\\
    & = & \alpha M_T^2 \mathd W_T .
  \end{eqnarray*}
  Now the result follows by standard SDE theory.
\end{proof}

We also define the positive measures on $\mathbb{R}^4$ and $\mathbb{R}^2
\times \varepsilon \mathbb{Z}^2$
\[ \mu_{\varepsilon}^{\alpha} = : \exp (\alpha W_{\varepsilon}) : = \exp
   \left( \alpha W_{\varepsilon} - \frac{\alpha^2}{2} \mathbb{E}
   [W_{\varepsilon}^2] \right), \]
and
\[ \bar{\mu}^{\alpha}_{\varepsilon} = : \exp (\alpha \bar{W}_{\varepsilon}) :
   = \exp \left( \alpha \bar{W}_{\varepsilon} - \frac{\alpha^2}{2}
   \mathbb{E} [\bar{W}_{\varepsilon}^2] \right) . \]
We recall that, using the fact that $\overline{\mathcal{E}}_{\varepsilon}$
commutes with the composition with functions, we get
\[ \bar{\mu}_{\varepsilon}^{\alpha} (x, y) \assign
   \overline{\mathcal{E}}_{\varepsilon} (\mu^{\alpha}_{\varepsilon}) (x, y) .
\]
Furthermore we can show in analogy with Lemma \ref{lemma:series} that
\[ \mu^{\alpha}_{\varepsilon} = \sum \frac{\alpha^n}{n!} : W_{\varepsilon}^n
   : \quad, \bar{\mu}^{\alpha}_{\varepsilon} = \sum \frac{\alpha^n}{n!} :
   \bar{W}_{\varepsilon}^n : . \]
We recall a well know fact on the Green function $\mathcal{G}$:

\begin{proposition}
  \label{prop:estimate-green-classical}For any $m > 0$ there is a constant
  $C_1$ (depending only on $m$) such that, for any $(x, z) \in \mathbb{R}^4$,
  $(x, z) \centernot{=} 0$,
  \[ | \mathcal{G} (x, z) | \leqslant - \frac{2}{(4 \pi)^2} \log \left(
     \sqrt{x^2 + z^2} \wedge 1 \right) + C_1 . \]
\end{proposition}

\begin{proof}
  See Section 3, Chapter V of {\cite{Stein1970}}.
\end{proof}

More generally, in order to stress the dependence of $\mathcal{G}$ on the mass
$m > 0$, we use also the notation
\[ \mathfrak{G}^{m^2} (x, z) = \frac{1}{(2 \pi)^4} \int_{\mathbb{R}^4}
   \frac{e^{-i (y \cdot x + k \cdot z)}}{(| y |^2 + | k |^2 + m^2)^2} \mathd y
   \mathd k. \]

\subsection{Estimates on $\mathcal{G}_{\varepsilon}$}

In this section we prove the following useful estimate on the Green function
$\mathcal{G}_{\varepsilon}$. Our proof differs from proofs of the analogous
statement on $\varepsilon \mathbb{Z}^2$ (see, e.g., Section 4.2.2 of
{\cite{Itzyksonbook}} or Appendix A of {\cite{bauerschmidt_2019}}) and uses a
comparison with the Green's function of $(m^2 - \Delta_{\mathbb{R}^4})^2$ on
$\mathbb{R}^4$ which is well known.

\begin{theorem}
  \label{theorem:green1}For any $m > 0$, $0 < \varepsilon \leqslant 1$ and $C
  > \frac{2}{(4 \pi)^2}$ there exists a $D_C \in \mathbb{R}_+$ such that, for
  any $x \in \mathbb{R}^2$ and $z \in \varepsilon \mathbb{Z}^2$, we have
  \begin{equation}
    | \mathcal{G}_{\varepsilon} (x, z) | \leqslant C \log_+ \left(
    \frac{1}{\sqrt{x^2 + z^2} \vee \varepsilon} \right) + D_C,
    \label{eq:green1}
  \end{equation}
  where $\log_+ (x) = \log (x) \vee 0$.
\end{theorem}

\begin{proof}
  We distinguish the cases $\max (| z |, | x |) \geqslant \varepsilon$ and
  $\max (| z |, | x |) < \varepsilon$.
  
  Consider first $\max (| z |, | x |) \geqslant \varepsilon$, and the
  difference $\mathcal{G}_{\varepsilon} (x, z) - \mathcal{G} (x, z)$. Then we
  have
  \begin{eqnarray*}
    &  & | \mathcal{G}_{\varepsilon} (x, z) - \mathcal{G} (x, z) |\\
    & \leqslant & \left| \frac{1}{(2 \pi)^4} \int_{\mathbb{R}^2 \times
    \mathbb{T}_{\frac{1}{\varepsilon}}^2} \frac{e^{-i (y \cdot x + k \cdot
    z)}}{\left( | y |^2 + 4 \varepsilon^{- 2} \sin^2 \left( \frac{\varepsilon
    k_1}{2} \right) + 4 \varepsilon^{- 2} \sin^2 \left( \frac{\varepsilon
    k_2}{2} \right) + m^2 \right)^2} - \frac{e^{-i (y \cdot x + k \cdot z)}}{(|
    y |^2 + | k |^2 + m^2)^2} \mathd y \mathd k \right|\\
    &  & + \left| \frac{1}{(2 \pi)^4} \int_{\mathbb{R}^4
    \backslash\mathbb{R}^2 \times \mathbb{T}^2_{\frac{1}{\varepsilon}}} \frac{e^{-i (y \cdot x + k
    \cdot z)}}{(| y |^2 + | k |^2 + m^2)^2} \mathd y \mathd k \right|\\
    & = & J_1 + J_2.
  \end{eqnarray*}
  For $J_1$ we have
  \begin{eqnarray*}
    &  & (2 \pi)^4 J_1\\
    & = & \int_{\mathbb{R}^2 \times \mathbb{T}_{\frac{1}{\varepsilon}}^2}
    e^{-i (y \cdot x + k \cdot z)} \left( \frac{(| y |^2 + | k |^2 + m^2)^2 -
    \left( | y |^2 + 4 \varepsilon^{- 2} \sin^2 \left( \frac{\varepsilon
    k_1}{2} \right) + 4 \varepsilon^{- 2} \sin^2 \left( \frac{\varepsilon
    k_2}{2} \right) + m^2 \right)^2}{\left( | y |^2 + 4 \varepsilon^{- 2}
    \sin^2 \left( \frac{\varepsilon k_1}{2} \right) + 4 \varepsilon^{- 2}
    \sin^2 \left( \frac{\varepsilon k_2}{2} \right) + m^2 \right)^2 (| y |^2 +
    | k |^2 + m^2)^2} \right) \mathd y \mathd k\\
    & = & \int_{\mathbb{R}^2 \times \mathbb{T}_{\frac{1}{\varepsilon}}^2}
    e^{-i (y \cdot x + k \cdot z)} \left( \frac{(A + B) (A - B)}{A^2 B^2}
    \right) \mathd y \mathd k.
  \end{eqnarray*}

  Here we set \ $A = (| y |^2 + | k |^2 + m^2)$ and $B = \left( | y |^2 + 4
  \varepsilon^{- 2} \sin^2 \left( \frac{\varepsilon k_1}{2} \right) + 4
  \varepsilon^{- 2} \sin^2 \left( \frac{\varepsilon k_2}{2} \right) + m^2
  \right)$.
  
  Now it is simple to see that $| A | \lesssim B$ (where the constants hidden
  in the symbol $\lesssim$ are independent of $0 < \varepsilon \leqslant 1$)
  which implies $| A + B | \lesssim B$ and
  \[ \left| \int_{\mathbb{R}^2 \times \mathbb{T}_{\frac{1}{\varepsilon}}^2}
     e^{i (y \cdot x + k \cdot z)} \left( \frac{(A + B) (A - B)}{A^2 B^2}
     \right) \mathd k \mathd y \right| \lesssim 2 \int_{\mathbb{R}^2 \times
     \mathbb{T}_{\frac{1}{\varepsilon}}^2} \left( \frac{(A - B)}{A^2 B}
     \right) \mathd y \mathd k. \]
  If we apply Taylor expansion to $\varepsilon^{- 2} \sin^2 \left(
  \frac{\varepsilon | k_1 |}{2} \right)$ we get
  \[ 4 \varepsilon^{- 2} \sin^2 \left( \frac{\varepsilon | k_1 |}{2} \right) =
     | k_1 |^2 + \frac{1}{\varepsilon^2} \mathcal{O} ((\varepsilon | k_1 |)^3),
  \]
  and so
  \begin{eqnarray*}
    \left( \frac{(A - B)}{A^2 B} \right) & \lesssim & \frac{\varepsilon (| k_1
    |)^3}{(| y |^2 + \varepsilon^{- 2} \sin^2 (\varepsilon k_1) +
    \varepsilon^{- 2} \sin^2 (\varepsilon k_2) + m^2) (| y |^2 + | k |^2 +
    m^2)^2}\\
    & \lesssim & \frac{\varepsilon (| k_1 |)^3}{(| y |^2 + | k |^2 +
    m^2)^3}\\
    & \lesssim & \frac{\varepsilon}{(| y |^2 + | k |^2 + m^2)^{3 / 2}}.
  \end{eqnarray*}
  Since
  \[ \int_{\mathbb{R}^2 \times \mathbb{T}_{\frac{1}{\varepsilon}}^2}
     \frac{1}{(| y |^2 + | k |^2 + m^2)^{3 / 2}} \lesssim \varepsilon^{- 1}, \]
  we can conclude that
  \[ J_1 = \frac{1}{(2 \pi)^4} \left| \int_{\mathbb{R}^2 \times
     \mathbb{T}_{\frac{1}{\varepsilon}}^2} \frac{e^{-i (k, y) \cdot x}}{\left(
     | y |^2 + 4 \varepsilon^{- 2} \sin^2 \left( \frac{\varepsilon k_1}{2}
     \right) + 4 \varepsilon^{- 2} \sin^2 \left( \frac{\varepsilon k_2}{2}
     \right) + m^2 \right)^2} - \frac{e^{-i (k, y) \cdot x}}{(| y |^2 + | k |^2
     + m^2)^2} \mathd k \mathd y \right| \leqslant K_1, \]
  where the constant $K_1 > 0$ is independent of $0 < \varepsilon \leqslant
  1$. Let us now turn to $J_2$. Consider $R_{\varepsilon} (y, k) =
  \frac{1}{(\varepsilon m^2 + | y |^2 + | k |^2)^2}$ \
  \begin{align*}
      J_2 =& \frac{1}{(2 \pi)^4} \left| \int_{\mathbb{R}^4
     \backslash \mathbb{R}^2 \times \mathbb{T}^2} R_{\varepsilon} (y, k) e^{-i \left( y \cdot
     \frac{x}{\varepsilon} + k \cdot \frac{z}{\varepsilon} \right)} \mathd y
     \mathd k \right| \\
     \lesssim& \left| \int_{\mathbb{R}^4, | k_1 | > 1, | k_2 |
     < 1} R_{\varepsilon} (y, k) e^{-i \left( y \cdot \frac{x}{\varepsilon} + k
     \cdot \frac{z}{\varepsilon} \right)} \mathd y \mathd k \right|  
    + \left| \int_{\mathbb{R}^4, | k_2 | > 1, | k_1 | < 1} R_{\varepsilon}
     (y, k) e^{-i \left( y \cdot \frac{x}{\varepsilon} + k \cdot
     \frac{z}{\varepsilon} \right)} \mathd y \mathd k \right| \\& + \left|
     \int_{\mathbb{R}^4, | k_2 | > 1, | k_1 | > 1} R_{\varepsilon} (y, k) e^{-i
     \left( y \cdot \frac{x}{\varepsilon} + k \cdot \frac{z}{\varepsilon}
     \right)} \mathd y \mathd k \right| \\
   =& J_{2, 1} + J_{2, 2} + J_{2, 3} . \end{align*}
  First we focus on $J_{2, 1}$ obtaining
  \begin{eqnarray*}
    J_{2, 1} & = & \left| \int_{\mathbb{R}^2 \times \mathbb{R}^2} \frac{e^{-i
    \left( y \cdot \frac{x}{\varepsilon} + k \cdot \frac{z}{\varepsilon}
    \right)}}{(\varepsilon m^2 + | y |^2 + | k_1 |^2 + | k_2 |^2 + 1)^2}
    \mathd y \mathd k \right|\\
    &  & + \left| \int_{\mathbb{R}^4, | k_1 | > 1, | k_2 | < 1} \left(
    \frac{1}{(\varepsilon m^2 + | y |^2 + | k_1 |^2 + | k_2 |^2)^2} -
    \frac{1}{(\varepsilon m^2 + | y |^2 + | k_1 |^2 + | k_2 |^2 + 1)^2}
    \right) \mathd y \mathd k \right|\\
    &  & + K_2\\
    & \leqslant & \left| \mathfrak{G}^{\varepsilon m^2 + 1} \left(
    \frac{x}{\varepsilon}, \frac{z}{\varepsilon} \right)\right| + K_2 \\
    &  & + \left| \int_{\mathbb{R}^4, | k_1 | > 1, | k_2 | < 1} \left(
    \frac{2 (\varepsilon m^2 + | y |^2 + | k_1 |^2 + | k_2 |^2 +
    1)}{(\varepsilon m^2 + | y |^2 + | k_1 |^2 + | k_2 |^2 + 1)^2
    (\varepsilon m^2 + | y |^2 + | k_1 |^2 + | k_2 |^2)^2} \right) \mathd y
    \mathd k \right|\\
    & \leqslant & \left| \mathfrak{G}^{\varepsilon m^2 + 1} (1, 0) \right|+ K_3 \leqslant
    K_4,
  \end{eqnarray*}
  where we use the fact that $\max \left( \frac{| x |}{\varepsilon}, \frac{| z
  |}{\varepsilon} \right) \geqslant 1$, and \ $K_2, K_3, K_4 > 0$ are suitable
  constants independent of $0 < \varepsilon \leqslant 1$. In a similar way we
  get an estimate $J_{2, 2}, J_{2, 3} \leqslant K_5$ for some constant $K_5$
  independent of $0 < \varepsilon \leqslant 1$. Since $\mathcal{G} (x, z)$
  satisfies inequality {\eqref{eq:green1}} when $\max (| x |, | z |) \geqslant
  \varepsilon$ the thesis is proved in this case.\\
  
  Consider now $z = 0$ (that means $| z | < \varepsilon$ since $z \in
  \varepsilon \mathbb{Z}^2$) and $| x | < \varepsilon$, then we have
  \begin{align*}
       \mathcal{G}_{\varepsilon} (x, 0) &\leqslant \frac{1}{(4 \pi) (2 \pi)^2}
     \int_{\mathbb{T}^2} \frac{1}{\sin^2 (k_1) + \sin^2 (k_2) + \varepsilon^2
     m^2} \mathd k_1 \mathd k_2 \\
   &\leqslant \frac{1}{(4 \pi) (2 \pi)^2} \int_{| k | < H} \frac{1}{\left(
     \frac{\sin (H)}{H} \right)^2  | k |^2 + \varepsilon^2 m^2} \mathd k_1
     \mathd k_2 + \frac{1}{(4 \pi)^2} \int_{\mathbb{T}^2} \frac{1}{2 \sin
     (H)^2} \mathd k_1 \mathd k_2, \end{align*}
  where we can choose any $H$ satisfying $0 < H < \pi$. The integral
  $\frac{1}{(4 \pi)^2} \int_{\mathbb{T}^2} \frac{1}{2 \sin (H)^2} \mathd k_1
  \mathd k_2$ is bounded in $\varepsilon$. On the other hand
  \[ \frac{1}{(4 \pi) (2 \pi)^2} \int_{| k | < H} \frac{1}{\left( \frac{\sin
     (H)}{H} \right)^2  | k |^2 + \varepsilon^2 m^2} \mathd k_1 \mathd k_2
     \leqslant - \left( \frac{H}{\sin (H)} \right)^2 \frac{2}{(4 \pi)^2} \log
     (\varepsilon) + K_6 ,\]
  where $K_6$ can be chosen uniformly on $\varepsilon$. This means that
  \[ \mathcal{G}_{\varepsilon} (x, 0) \leqslant - \left( \frac{H}{\sin (H)}
     \right)^2 \frac{2}{(4 \pi)^2} \log (\varepsilon) + K_{7, H} ,\]
  where $K_{7, H}$ depends only on $H$, and it is finite when $H > 0$. Since
  $\frac{H}{\sin (H)} \rightarrow 1$ as $H \rightarrow 0$ the theorem is
  proved.
\end{proof}

\subsection{Estimates on $\overline{\mathcal{G}}_{\varepsilon}$}

In this section we focus our attention on the the Green functions
$\overline{\mathcal{G}}_{\varepsilon}$ and
$\widetilde{\mathcal{G}}_{\varepsilon}$.

\begin{theorem}
  \label{eq:eepsilon}For any $m > 0$ there is a constant $C_2 > 0$ (not
  depending on $0 < \varepsilon \leqslant 1$) such that
  \[ | \overline{\mathcal{G}}_{\varepsilon} (\bar{x}) -
     \widetilde{\mathcal{G}}_{\varepsilon} (\bar{x}) | \leqslant C_2 . \]
\end{theorem}

\begin{proof}
  We have that
  \begin{eqnarray}
    &  & | \overline{\mathcal{G}}_{\varepsilon} (\bar{x}) -
    \widetilde{\mathcal{G}}_{\varepsilon} (\bar{x}) | \nonumber\\
    & = & \frac{1}{(2 \pi)^4} \int_{\mathbb{R}^4} \frac{\left( \sin^2 \left(
    \frac{\varepsilon k_1}{2} \right) \sin^2 \left( \frac{\varepsilon k_2}{2}
    \right) \right)}{\varepsilon^4 k_1^2 k_2^2} \nonumber\\
    &  & \times \left( \frac{1}{\left( m^2 + \frac{4 \sin^2 \left(
    \frac{\varepsilon k_1}{2} \right)}{\varepsilon^2} + \frac{4 \sin^2 \left(
    \frac{\varepsilon k_2}{2} \right)}{\varepsilon^2} + k_3^2 + k_4^2
    \right)^2} - \frac{1}{(m^2 + | k |^2)^2} \right) \mathd k \nonumber.
  \end{eqnarray}
  Integrating with respect to the variables $k_3, k_4$ we obtain
  \[ K = \frac{1}{(2 \pi)^3} \int_{\mathbb{R}^2} \frac{\left( \sin^2 \left(
     \frac{\varepsilon k_1}{2} \right) \sin^2 \left( \frac{\varepsilon k_2}{2}
     \right) \right)}{\varepsilon^4 k_1^2 k_2^2} \left( \frac{1}{\left( m^2 +
     \frac{4 \sin^2 \left( \frac{\varepsilon k_1}{2} \right)}{\varepsilon^2} +
     \frac{4 \sin^2 \left( \frac{\varepsilon k_2}{2} \right)}{\varepsilon^2}
     \right)} - \frac{1}{(m^2 + | k |^2)} \right) \mathd k. \]
  We can decompose the integral in
  \[ K_{<} = \frac{1}{(2 \pi)^3} \int_{| k_1 | < \frac{\pi}{\varepsilon}, |
     k_2 | < \frac{\pi}{\varepsilon}} \frac{\left( \sin^2 \left(
     \frac{\varepsilon k_1}{2} \right) \sin^2 \left( \frac{\varepsilon k_2}{2}
     \right) \right) \left( | k |^2 - \frac{4 \sin^2 \left( \frac{\varepsilon
     k_1}{2} \right)}{\varepsilon^2} - \frac{4 \sin^2 \left( \frac{\varepsilon
     k_2}{2} \right)}{\varepsilon^2} \right)}{\varepsilon^4 k_1^2 k_2^2 \left(
     m^2 + \frac{4 \sin^2 \left( \frac{\varepsilon k_1}{2}
     \right)}{\varepsilon^2} + \frac{4 \sin^2 \left( \frac{\varepsilon k_2}{2}
     \right)}{\varepsilon^2} \right) (m^2 + | k |^2)} \mathd k \]
  and
  \begin{eqnarray*}
    &  & K - K_{<}\\
    & \leqslant & K_{>, 1} + K_{>, 2} + K_{>, 3}\\
    & = & \frac{1}{(2 \pi)^3} \int_{| k_1 | > \frac{\pi}{\varepsilon}, | k_2
    | > \frac{\pi}{\varepsilon}} \frac{\left( \sin^2 \left( \frac{\varepsilon
    k_1}{2} \right) \sin^2 \left( \frac{\varepsilon k_2}{2} \right) \right)
    \left( | k |^2 - \frac{4 \sin^2 \left( \frac{\varepsilon k_1}{2}
    \right)}{\varepsilon^2} - \frac{4 \sin^2 \left( \frac{\varepsilon k_2}{2}
    \right)}{\varepsilon^2} \right)}{\varepsilon^4 k_1^2 k_2^2 \left( m^2 +
    \frac{4 \sin^2 \left( \frac{\varepsilon k_1}{2} \right)}{\varepsilon^2} +
    \frac{4 \sin^2 \left( \frac{\varepsilon k_2}{2} \right)}{\varepsilon^2}
    \right) (m^2 + | k |^2)} \mathd k\\
    &  & + \frac{1}{(2 \pi)^3} \int_{| k_1 | < \frac{\pi}{\varepsilon}, | k_2
    | > \frac{\pi}{\varepsilon}} \frac{\left( \sin^2 \left( \frac{\varepsilon
    k_1}{2} \right) \sin^2 \left( \frac{\varepsilon k_2}{2} \right) \right)
    \left( | k |^2 - \frac{4 \sin^2 \left( \frac{\varepsilon k_1}{2}
    \right)}{\varepsilon^2} - \frac{4 \sin^2 \left( \frac{\varepsilon k_2}{2}
    \right)}{\varepsilon^2} \right)}{\varepsilon^4 k_1^2 k_2^2 \left( m^2 +
    \frac{4 \sin^2 \left( \frac{\varepsilon k_1}{2} \right)}{\varepsilon^2} +
    \frac{4 \sin^2 \left( \frac{\varepsilon k_2}{2} \right)}{\varepsilon^2}
    \right) (m^2 + | k |^2)} \mathd k\\
    &  & + \frac{1}{(2 \pi)^3} \int_{| k_1 | > \frac{\pi}{\varepsilon}, | k_2
    | < \frac{\pi}{\varepsilon}} \frac{\left( \sin^2 \left( \frac{\varepsilon
    k_1}{2} \right) \sin^2 \left( \frac{\varepsilon k_2}{2} \right) \right)
    \left( | k |^2 - \frac{4 \sin^2 \left( \frac{\varepsilon k_1}{2}
    \right)}{\varepsilon^2} - \frac{4 \sin^2 \left( \frac{\varepsilon k_2}{2}
    \right)}{\varepsilon^2} \right)}{\varepsilon^4 k_1^2 k_2^2 \left( m^2 +
    \frac{4 \sin^2 \left( \frac{\varepsilon k_1}{2} \right)}{\varepsilon^2} +
    \frac{4 \sin^2 \left( \frac{\varepsilon k_2}{2} \right)}{\varepsilon^2}
    \right) (m^2 + | k |^2)} \mathd k.
  \end{eqnarray*}

  Since $\left| \frac{\sin (x)}{x} \right| < 1$ and $\sin^2 \left(
  \frac{\varepsilon k_1}{2} \right) = \frac{\varepsilon^2 | k_1 |^2}{4} + \mathcal{O}
  (\varepsilon^3 | k_1 |^3)$ we have
  \[ K_{<} \lesssim \varepsilon \int_{| k_i | < \frac{\pi}{\varepsilon}}
     \frac{| k |^3}{(m^2 + | k |^2)^2} \mathd k \lesssim \varepsilon
     \frac{\pi}{\varepsilon} \lesssim 1. \]
  For $K_{>, 1}$ instead we have
  \begin{eqnarray}
    K_{>, 1} & \leqslant & \frac{1}{\varepsilon^2 (2 \pi)^3} \int_{| k_i | >
    \pi} \frac{\left( \sin^2 \left( \frac{k_1}{2} \right) \sin^2 \left(
    \frac{k_2}{2} \right) \right)}{k_1^2 k_2^2} \left( \frac{1}{\left( m^2 +
    \frac{4 \sin^2 \left( \frac{k_1}{2} \right)}{\varepsilon^2} + \frac{4
    \sin^2 \left( \frac{k_2}{2} \right)}{\varepsilon^2} \right)} +
    \frac{1}{\left( m^2 + \frac{| k |^2}{\varepsilon^2} \right)} \right)
    \mathd k \nonumber\\
    & \lesssim & \frac{1}{\varepsilon^2} \int_{| k_i | > \pi} \frac{\left(
    \sin^2 \left( \frac{k_1}{2} \right) \sin^2 \left( \frac{k_2}{2} \right)
    \right)}{k_1^2 k_2^2} \frac{1}{\left( 1 + \frac{4 \sin^2 \left(
    \frac{k_1}{2} \right)}{m^2 \varepsilon^2} + \frac{4 \sin^2 \left(
    \frac{k_2}{2} \right)}{m^2 \varepsilon^2} \right)} \mathd k + \nonumber\\
    &  & + \frac{1}{\varepsilon^2} \int_{| k_i | > \pi} \frac{\left( \sin^2
    \left( \frac{k_1}{2} \right) \sin^2 \left( \frac{k_2}{2} \right)
    \right)}{k_1^2 k_2^2} \frac{1}{\left( m^2 + \frac{| k |^2}{\varepsilon^2}
    \right)} \mathd k \nonumber\\
    & \lesssim & \frac{1}{\varepsilon^2} \sum_{z \in \mathbb{Z}^2, z_1
    \centernot{=} 0, z_2 \centernot{=} 0} \frac{1}{z_1^2 z_2^2} \int_{[- \pi, \pi]^2}
    \frac{\left( \sin^2 \left( \frac{k_1}{2} \right) \sin^2 \left(
    \frac{k_2}{2} \right) \right)}{\left( 1 + \frac{4 \sin^2 \left(
    \frac{k_1}{2} \right)}{m^2 \varepsilon^2} + \frac{4 \sin^2 \left(
    \frac{k_2}{2} \right)}{m^2 \varepsilon^2} \right)} \mathd k + \nonumber\\
    &  & + \frac{1}{\varepsilon^2} \sum_{z \in \mathbb{Z}^2, z_1 \centernot{=} 0,
    z_2 \centernot{=} 0} \frac{1}{z_1^2 z_2^2} \frac{1}{\frac{| z
    |^2}{\varepsilon^2}} \int_{[- \pi, \pi]^2} \left( \sin^2 \left(
    \frac{k_1}{2} \right) \sin^2 \left( \frac{k_2}{2} \right) \right) \mathd k
    \lesssim 1, \nonumber
  \end{eqnarray}
  since
  \begin{eqnarray}
    \int_{[- \pi, \pi]^2} \frac{\left( \sin^2 \left( \frac{k_1}{2} \right)
    \sin^2 \left( \frac{k_2}{2} \right) \right)}{\left( 1 + \frac{4 \sin^2
    \left( \frac{k_1}{2} \right)}{m^2 \varepsilon^2} + \frac{4 \sin^2 \left(
    \frac{k_2}{2} \right)}{m^2 \varepsilon^2} \right)} \mathd k & \lesssim &
    \int_{[- \pi, \pi]^2} \frac{k_1^2 k_2^2}{\left( 1 +
    \frac{k_1^2}{\varepsilon^2} + \frac{k_2^2}{\varepsilon^2} \right)} \mathd
    k \lesssim \varepsilon^6 \int_{\left[ - \frac{\pi}{\varepsilon},
    \frac{\pi}{\varepsilon} \right]^2} \frac{k_1^2 k_2^2}{(1 + k_1^2 + k_2^2)}
    \mathd k \nonumber\\
    & \lesssim & \varepsilon^6 \int_{\left[ - \frac{\pi}{\varepsilon},
    \frac{\pi}{\varepsilon} \right]^2} k_1 k_2 \mathd k \lesssim \varepsilon^6
    \cdot \frac{1}{\varepsilon^2} \cdot \frac{1}{\varepsilon^2} \lesssim
    \varepsilon^2  ,\label{eq:ellepsilon}
  \end{eqnarray}
  and also
  \[ \int_{[- 2 \pi, 2 \pi]^2} \left( \sin^2 \left( \frac{k_1}{2} \right)
     \sin^2 \left( \frac{k_2}{2} \right) \right) \mathd k \lesssim 1. \]
  For $K_{>, 2}$ we have
  \begin{eqnarray}
    K_{>, 2} & \lesssim & \frac{1}{\varepsilon} \int_{| k_1 | <
    \frac{\pi}{\varepsilon}, | k_2 | > \frac{\pi}{\varepsilon}}
    \frac{\varepsilon \sin^2 \left( \frac{\varepsilon k_2}{2} \right) | k_1
    |^3}{k_2^2 \left( m^2 + \frac{4 \sin^2 \left( \frac{\varepsilon k_1}{2}
    \right)}{\varepsilon^2} + \frac{4 \sin^2 \left( \frac{\varepsilon k_2}{2}
    \right)}{\varepsilon^2} \right) \left( 1 + \frac{| k |^2}{\varepsilon^2}
    \right)} \mathd k + \nonumber\\
    &  & + \frac{1}{\varepsilon} \int_{| k_2 | > \frac{\pi}{\varepsilon}}
    \frac{1}{\left( 1 + \frac{| k_2 |^2}{\varepsilon^2} \right)} \mathd k_2
    \nonumber\\
    & \lesssim & \frac{1}{\varepsilon} \cdot \varepsilon^3 \cdot \log
    (\varepsilon) + \frac{1}{\varepsilon} \cdot \varepsilon \lesssim 1,
    \nonumber
  \end{eqnarray}
  and a similar estimate can be proved for $K_{>, 3}$. Since all the previous
  constants do not depend on $0 < \varepsilon \leqslant 1$ the theorem is
  proved.
\end{proof}

\begin{corollary}
  \label{corollary:green2}For any $m > 0$ and $C > \frac{2}{(4 \pi)^2}$ there
  is a constant $D^a_C$ independent of $0 < \varepsilon \leqslant 1$ such that
  \[ | \overline{\mathcal{G}}_{\varepsilon} (x, z) | \leqslant C \log_+ \left(
     \frac{1}{\sqrt{x^2 + z^2} \vee \varepsilon} \right) + D^a_C .\]
\end{corollary}

\begin{proof}
  This is a consequence of Theorem \ref{eq:eepsilon}, and the fact that, using
  Proposition \ref{prop:estimate-green-classical}, we get that
  \[ \widetilde{\mathcal{G}}_{\varepsilon} (x, z) = (w_{0, \varepsilon} \ast
     w_{0, \varepsilon} \ast \mathcal{G}) (x, z) \leqslant - \frac{2}{(4
     \pi)^2} \log \left( \sqrt{x^2 + z^2} \vee \varepsilon \right) + K, \]
  where $w_{0, \varepsilon} (z) = \frac{1}{\varepsilon^2} \mathbb{I}_{\left( -
  \frac{1}{2}, \frac{1}{2} \right]^2} (\varepsilon^{- 1} \cdot)$, for some
  constant $K > 0$ independent of $0 < \varepsilon \leqslant 1$.
\end{proof}

\subsection{Stochastic estimates}

From the estimate of Theorem \ref{eq:eepsilon} we get the following results.

\begin{lemma}
  \label{Lemma:estimate-wick-green}For any $a > 0$, $\beta \in (0, 1)$, $C >
  \frac{2}{(4 \pi)^2}$, $i \geqslant 0$ and $0 \leqslant j \leqslant
  J_{\varepsilon}$ we have
  \begin{equation}
    \begin{array}{ll}
      & \frac{1}{n!} \mathbb{E} \left[ \left( \int_{\mathbb{R}^2 \times
      \varepsilon \mathbb{Z}^2} E^{a, \beta}_{i, j} (x, z) : W_{\varepsilon}
      (x, z)^n : \mathd x \mathd z \right)^2 \right]\\
      \lesssim & \int_{\mathbb{R}^4 \times \varepsilon \mathbb{Z}^4} E^{a,
      \beta}_{i, j} (x, z) E^{a, \beta}_{i, j} (x', z') \left( D_C + C \log_+
      \left( \frac{1}{\sqrt{| x - x' |^2 + | z - z' |^2} \vee \varepsilon}
      \right) \right)^n \mathd x \mathd z \mathd x' \mathd z'
    \end{array} \label{eq:EAB1}
  \end{equation}
  
  \begin{equation}
    \begin{array}{ll}
      & \frac{1}{n!} \mathbb{E} \left[ \left( \int_{\mathbb{R}^2 \times
      \varepsilon \mathbb{Z}^2} \tilde{E}^{a, \beta}_i (x) : W_{\varepsilon}
      (x, z)^n : \mathd x \right)^2 \right]\\
      \lesssim & \int_{\mathbb{R}^4} \tilde{E}^{a, \beta}_i (x) \tilde{E}^{a,
      \beta}_i (x') \left( D_C + C \log_+ \left( \frac{1}{\sqrt{| x - x' |^2}
      \vee \varepsilon} \right) \right)^n \mathd x \mathd x'
    \end{array} \label{eq:tildeE1}
  \end{equation}
  \begin{equation}
    \begin{array}{ll}
      & \frac{1}{n!} \mathbb{E} \left[ \left( \int_{\mathbb{R}^2 \times
      \varepsilon \mathbb{Z}^2} \tilde{E}^{a, \beta}_{J_{\varepsilon}} (z) :
      W_{\varepsilon} (x, z)^n : \mathd z \right)^2 \right]\\
      \lesssim & \int_{\varepsilon \mathbb{Z}^4} \tilde{E}^{a, \beta}_i (z) \tilde{E}^{a,
      \beta}_i (z') \left( D_C + C \log_+ \left( \frac{1}{\sqrt{| z - z' |^2}
      \vee \varepsilon} \right) \right)^n \mathd z \mathd z',
    \end{array} \label{eq:tildeE2}
  \end{equation}
  where $E^{a, \beta}_{i, j}$ and $\tilde{E}^{a, \beta}_i$ are defined in
  equations {\eqref{eq:definitionE}} and equation {\eqref{eq:definitionE1}}, and the constants hidden in the symbols $\lesssim$ do not depend on $n\in\mathbb{N}$.
\end{lemma}

\begin{proof}
  By definition and Fubini's theorem
  \begin{eqnarray*}
    &  & \mathbb{E} \left[ \left( \int_{\mathbb{R}^2 \times \varepsilon
    \mathbb{Z}^2} E^{a, \beta}_{i, j} (x, z) : W_{\varepsilon} (x, z)^n :
    \mathd x \mathd z \right)^2 \right]\\
    & = & \int_{\mathbb{R}^4 \times \varepsilon \mathbb{Z}^4} E^{a,
    \beta}_{i, j} (x, z) E^{a, \beta}_{i, j} (x', z') \mathbb{E} [:
    W_{\varepsilon} (x, z)^n : : W_{\varepsilon} (x', z')^n :] \mathd x \mathd
    z \mathd x' \mathd z' .
  \end{eqnarray*}
  Now by Wick's theorem we have
  \[ \mathbb{E} [: W_{\varepsilon} (x, z)^n : : W_{\varepsilon} (x', z')^n :]
     = n!\mathbb{E} [W_{\varepsilon} (x, z) W_{\varepsilon} (x', z')]^n = n!
     [\mathcal{G}_{\varepsilon} (x - x', z - z')]^n . \]
  Thus, inequality {\eqref{eq:EAB1}} follows from Theorem
  \ref{theorem:green1}. Using the fact that
  \[ \mathbb{E} \left[ \left( \int_{\mathbb{R}^2} \tilde{E}^{a, \beta}_i (x) :
     W_{\varepsilon} (x, z)^n : \mathd x \right)^2 \right] =
     \int_{\mathbb{R}^4} \tilde{E}^{a,
     \beta}_i (x) \tilde{E}^{a, \beta}_i (x') \mathbb{E} [: W_{\varepsilon}
     (x, z)^n : : W_{\varepsilon} (x', z)^n :] \mathd x \mathd x' \]
  \[ \mathbb{E} \left[ \left( \int_{\varepsilon\mathbb{Z}^2} \tilde{E}^{a,
     \beta}_{J_{\varepsilon}} (z) : W_{\varepsilon} (x, z)^n :  \mathd
     z \right)^2 \right] = \int_{\varepsilon \mathbb{Z}^4}
     \tilde{E}^{a, \beta}_{J_{\varepsilon}} (z) \tilde{E}^{a,
     \beta}_{J_{\varepsilon}} (z') \mathbb{E} [: W_{\varepsilon} (x, z)^n : :
     W_{\varepsilon} (x, z')^n :] \mathd z \mathd z', \]
  and again Theorem \ref{theorem:green1}, we get the thesis.
\end{proof}

\begin{lemma}
  \label{Lemma:estimate-wick-green2}Under the same hypothesis of Lemma
  \ref{Lemma:estimate-wick-green}, inequalities {\eqref{eq:EAB1}},
  {\eqref{eq:tildeE1}} and {\eqref{eq:tildeE2}} hold by replacing
  $W_{\varepsilon}$ by $\bar{W}_{\varepsilon}$ and the integration on
  $\mathbb{R}^2 \times \varepsilon \mathbb{Z}^2$ and by integration on
  $\mathbb{R}^4$.
\end{lemma}

\begin{proof}
  The proof is the similar to the one of Lemma \ref{Lemma:estimate-wick-green}
  where the role of Theorem \ref{theorem:green1} is replaced by Corollary
  \ref{corollary:green2}.
\end{proof}

\begin{theorem}
  \label{theorem:stochasticestimates1}Fix $\alpha \in \mathbb{R}$ such that
  $\frac{\alpha^2}{(4 \pi)^2} < 2$, then for any $0 < \varepsilon \leqslant
  1$, $a \in \mathbb{R}_+$ and $0 < \beta < 1$, $p \geqslant 2$ such that
  $\frac{\alpha^2}{(4 \pi)^2} (p - 1) < 2$, any $\delta \geqslant 0$ and $s_1,
  s_2 > \frac{\delta}{2}$ such that $s_1 + s_2 - \delta > \frac{\alpha^2}{(4
  \pi)^2} (p - 1)$ and $s_1\vee s_2 \leq \frac{\delta}{2}+2$ we have
  \begin{equation}
    \mathbb{E} [| E_{i, j}^{a, \beta} \asterisk \mu_{\varepsilon}^{\alpha} |^p
    (x)] =\mathbb{E} [| E_{i, j}^{a, \beta} \asterisk
    \mu_{\varepsilon}^{\alpha} |^p (0)] \leqslant C_{\alpha, a, \beta} 2^{p
    \left[ \left( s_1 - \frac{\delta}{2} \right) i + \left( s_2 -
    \frac{\delta}{2} \right) j \right]} \label{eq:stochasticmu1}
  \end{equation}
  \begin{equation}
    \mathbb{E} [| E_{i, j}^{a, \beta} \asterisk
    \bar{\mu}_{\varepsilon}^{\alpha} |^p (x)] =\mathbb{E} [| E_{i, j}^{a,
    \beta} \asterisk \bar{\mu}_{\varepsilon}^{\alpha} |^p (0)] \leqslant
    \bar{C}_{\alpha, a, \beta} 2^{p \left[ \left( s_1 - \frac{\delta}{2}
    \right) i + \left( s_2 - \frac{\delta}{2} \right) j \right]},
    \label{eq:stochasticmu2}
  \end{equation}
  where $C_{\alpha, a, \beta}, \bar{C}_{\alpha, a, \beta} > 0$ are some
  constants depending only on $\alpha, a$ and $\beta$ and not on $0 <
  \varepsilon \leqslant 1$.
\end{theorem}

\begin{proof}
  We provide a detailed proof only for inequality {\eqref{eq:stochasticmu1}}
  the proof of {\eqref{eq:stochasticmu2}} being completely analogous where the
  roles of Theorem \ref{theorem:green1} and Lemma
  \ref{Lemma:estimate-wick-green} are played by Corollary
  \ref{corollary:green2} and Lemma \ref{Lemma:estimate-wick-green2}.
  
  First we consider the case $p = 2$. We have that
  \begin{eqnarray}
    \mathbb{E} [| E_{i, j}^{a, \beta} \asterisk \mu_{\varepsilon}^{\alpha}
    |^2] & = & \mathbb{E} \left[ \left( \int_{\mathbb{R}^2 \times \varepsilon
    \mathbb{Z}^2} E^{a, \beta}_{i, j} (x, z) \mu^{\alpha}_{\varepsilon}
    (\mathd x, \mathd z) \right)^2 \right] \nonumber\\
    & = & \mathbb{E} \left[ \left( \sum_{n = 0}^{+ \infty} \int_{\mathbb{R}^2
    \times \varepsilon \mathbb{Z}^2} \frac{\alpha^n}{n!} E^{a, \beta}_{i, j}
    (x, z) : W (x, z)^n : \mathd x \mathd z \right)^2 \right] \nonumber\\
    & = & \sum_{n = 0}^{+ \infty} \frac{\alpha^{2 n}}{(n!)^2} \mathbb{E}
    \left[ \left( \int_{\mathbb{R}^2 \times \varepsilon \mathbb{Z}^2} E^{a,
    \beta}_{i, j} (x, z) : W (x, z)^n : \mathd x \mathd z \right)^2 \right]
    \nonumber\\
    \tmop{Lemma} \ref{Lemma:estimate-wick-green} & \leqslant &
    \int_{\mathbb{R}^4 \times \varepsilon \mathbb{Z}^4} E^{a, \beta}_{i, j}
    (x, z) E^{a, \beta}_{i, j} (x', z') \nonumber\\
    &  & \times \exp \left( \alpha^2 D_C + C \alpha^2 \log_+ \left(
    \frac{1}{\sqrt{| x - x' |^2 + | z - z' |^2} \vee \varepsilon} \right)
    \right) \mathd x \mathd z \mathd x' \mathd z' \nonumber\\
    & \lesssim & \int_{\mathbb{R}^8} E^{a, \beta}_{i, j} (x, z) E^{a,
    \beta}_{i, j} (x', z') \exp (\alpha^2 D_C) \nonumber\\
    &  & \times \left( 1 + \left( \frac{1}{\sqrt{| x - x' |^2 + | z - z' |^2}
    \vee \varepsilon} \right)^{(C + \kappa) \alpha^2} \right) \mathd x \mathd
    z \mathd x' \mathd z' \nonumber
,  \end{eqnarray}
  where $\kappa$ is any positive number, the constant in $\lesssim$ depends
  only on $a, \beta$ and $\kappa$, we used the orthogonality of the Wick
  monomials $: W_{\varepsilon} (x, z)^n :$ of different degrees, the fact that
  for any $g \in L^1 (\mathbb{R}^2 \times \varepsilon \mathbb{Z}^2) \cap C^0
  (\mathbb{R}^2 \times \varepsilon \mathbb{Z}^2)$ we have
  \[ \mathbb{E} \left[ \left( \int_{\mathbb{R}^2 \times \varepsilon
     \mathbb{Z}^2} g (x, z) : W_{\varepsilon} (x, z)^n : \mathd x \mathd z
     \right)^2 \right] = n! \int_{\mathbb{R}^4 \times \varepsilon
     \mathbb{Z}^4} g (x, z) g (x', z') (\mathcal{G}_{\varepsilon} (x - x', z -
     z'))^n \mathd x \mathd z \mathd x' \mathd z', \]
  and we used Theorem \ref{theorem:green1} in the last two steps.\\
  
  For any $0 < \theta < 1$ we have $| x |^{\theta} | z |^{1 - \theta}
  \leqslant \sqrt{| x |^2 + | z |^2} \leq  \sqrt{| x |^2 + | z |^2} \vee \varepsilon $ which implies that
  \[ \frac{1}{\sqrt{| x |^2 + | z |^2} \vee \varepsilon} \leq  \frac{1}{\sqrt{| x
     |^2 + | z |^2}} \leqslant \frac{1}{| x |^{\theta} | z |^{1 - \theta}} .
  \]
  From this it follows that
  \begin{align}\label{eq:estimate-diagonal}
    &\mathbb{E} [| E_{i, j}^{a, \beta} \asterisk \mu_{\varepsilon}^{\alpha}
    |^2] \\\lesssim& \int_{\mathbb{R}^8} E^{a, \beta}_{i, j} (x, z) E^{a,
    \beta}_{i, j} (x', z') e^{\alpha^2 D_C} (1 + (| x - x' |^{\theta} | z - z'
    |^{1 - \theta})^{- (C + \kappa) \alpha^2}) \mathd x \mathd z \mathd x'
    \mathd z' \nonumber\\
    \lesssim& 2^{(i \theta + j (1 - \theta)) (C + \kappa) \alpha^2}
    \int_{\mathbb{R}^8} E^{a, \beta}_{0, 0} (x, z) E^{a, \beta}_{0, 0} (x',
    z') e^{\alpha^2 D_C} (1 + (| x - x' |^{\theta} | z - z' |^{1 - \theta})^{-
    (C + \kappa) \alpha^2}) \mathd x \mathd z \mathd x' \mathd z' . \nonumber
  \end{align}
  
  If we choose $C > \frac{2}{(4 \pi)^2}$ and $\kappa > 0$ small enough such
  that $\frac{2 \alpha^2}{(4 \pi)^2} < (C + \kappa) \alpha^2 < 2(s_1 + s_2 -
  \delta)$ and if we choose $0 < \theta < 1$ such that
  \[ \theta (C + \kappa) \alpha^2 < 2\left(s_1 - \frac{\delta}{2}\right)<2 \quad (1 - \theta)
     (C + \kappa) \alpha^2 < 2\left(s_2 - \frac{\delta}{2}\right)<2 \]
     (observe that this also implies that the integrands in \eqref{eq:estimate-diagonal} are integrable on the diagonal),
  we get the thesis for $p = 2$.\\
  
  Let us consider $p > 2$, and consider any $\gamma > 1$, then, using
  hypercontractivity of Gaussian Wick monomials (see Chapter IV, Theorem 1 of
  {\cite{Ustunel1995}}), i.e.
  \[ \mathbb{E} [(: W^n_{\varepsilon} :)^p] \leqslant \left( \sqrt{p - 1}
     \right)^n (\mathbb{E} [: W^n_{\varepsilon} :^2])^{\frac{p}{2}}, \]
  we have
  \begin{eqnarray}
    (\mathbb{E} [| E_{i, j}^{a, \beta} \asterisk \mu_{\varepsilon}^{\alpha}
    |^p])^{\frac{1}{p}} & \leqslant & \sum_{n = 0}^{+ \infty} \frac{| \alpha
    |^n}{n!} (\mathbb{E} [| E_{i, j}^{a, \beta} \asterisk : W_{\varepsilon}^n
    : |^p])^{\frac{1}{p}} \nonumber\\
    & \lesssim & \sum_{n = 0}^{+ \infty} \frac{\left( | \alpha | \sqrt{p - 1}
    \right)^n}{n!} (\mathbb{E} [| E_{i, j}^{a, \beta} \asterisk :
    W_{\varepsilon}^n : |^2])^{\frac{1}{2}} \nonumber\\
    & \lesssim & \left( \frac{\gamma^2}{\gamma^2 - 1} \right)^{\frac{1}{2}}
    \left( \sum_{n = 0}^{+ \infty} \frac{\gamma^{2n} \alpha^{2 n} (p -
    1)^n}{(n!)^2} \mathbb{E} [| E_{i, j}^{a, \beta} \asterisk :
    W_{\varepsilon}^n : |^2] \right)^{\frac{1}{2}} \nonumber\\
    & \lesssim & \left( \frac{\gamma^2}{\gamma^2 - 1} \right)^{\frac{1}{2}}
    \left( \mathbb{E} \left[ \left| E_{i, j}^{a, \beta} \asterisk
    \mu_{\varepsilon}^{\gamma \alpha \sqrt{(p - 1)}} \right|^2 \right]
    \right)^{\frac{1}{2}} . \nonumber
  \end{eqnarray}
  If we choose $\gamma > 1$ small enough such that $s_1 + s_2 - \delta >
  \frac{2 \alpha^2 \gamma^2 (p - 1)}{(4 \pi)^2}$ we get the thesis from the
  first part of the proof where $p = 2$.
\end{proof}

\begin{theorem}
  \label{label:thm-estimate-chaos}Fix $\alpha \in \mathbb{R}$ such that
  $\frac{\alpha^2}{(4 \pi)^2} < 2$, then for any $0 < \varepsilon \leqslant
  1$, $a \in \mathbb{R}_+$ and $0 < \beta < 1$, $p \geqslant 2$ such that
  $\frac{\alpha^2}{(4 \pi)^2} (p - 1) < 2$, for any $s > \frac{\alpha^2}{(4
  \pi)^2} (p - 1)$ we have
  \begin{equation}
    \mathbb{E} [| E_{i, j}^{a, \beta} \asterisk \mu_{\varepsilon}^{\alpha}
    |^p] \leqslant C'_{\alpha, a, \beta} 2^{p  (i\wedge j) s}
    \label{eq:stochasticmumax1}
  \end{equation}
  \begin{equation}
    \mathbb{E} [| \tilde{E}_i^{a, \beta} \asterisk_x
    \mu_{\varepsilon}^{\alpha} |^p] \leqslant C'_{\alpha, a, \beta} 2^{p i s},
    \quad \mathbb{E} [| \tilde{E}_{J_{\varepsilon}}^{a, \beta} \asterisk_z
    \mu_{\varepsilon}^{\alpha} |^p] \leqslant C'_{\alpha, a, \beta} 2^{p
    J_{\varepsilon} s}
  \end{equation}
  \begin{equation}
    \mathbb{E} [| E_{i, j}^{a, \beta} \asterisk
    \bar{\mu}_{\varepsilon}^{\alpha} |^p] \leqslant \bar{C}'_{\alpha, a,
    \beta} 2^{p (i\wedge j) s} \label{eq:stochasticmumax2}
  \end{equation}
  \begin{equation}
    \mathbb{E} [| \tilde{E}_i^{a, \beta} \asterisk_x
    \bar{\mu}_{\varepsilon}^{\alpha} |^p] \leqslant C'_{\alpha, a, \beta} 2^{p
    i s}, \quad \mathbb{E} [| \tilde{E}_{J_{\varepsilon}}^{a, \beta}
    \asterisk_z \bar{\mu}_{\varepsilon}^{\alpha} |^p] \leqslant C'_{\alpha, a,
    \beta} 2^{p J_{\varepsilon} s},
  \end{equation}
  where $C'_{\alpha, a, \beta}, \bar{C}'_{\alpha, a, \beta} > 0$ are some
  universal constants depending only on $\alpha, a$ and $\beta$.
\end{theorem}

\begin{proof}
  The proof is analogous to the one of Theorem
  \ref{theorem:stochasticestimates1} where, using the notations of the proof
  of Theorem \ref{theorem:stochasticestimates1}, we exploit the fact that
  \[ \left( \sqrt{| 2^{- i} x |^2 + | 2^{- j} z |^2} \right)^{- (C + \kappa)
     \alpha^2} \leqslant 2^{(i\wedge j) (C + \kappa) \alpha^2} \left( \sqrt{|
     x |^2 + | z |^2} \right)^{- (C + \kappa) \alpha^2} . \]
\end{proof}

\begin{corollary}\label{corollary-3.10}
  For $\alpha \in \mathbb{R},$ $0 < \beta < 1$, $p \geqslant 2$ such that
  $\frac{\alpha^2}{(4 \pi)^2} (p - 1) < 2$, for any $s,s_1,s_2 \in \mathbb{R}$ such that $s,s_1+s_2 < -\frac{\alpha^2}{(4
  \pi)^2} (p - 1)$, we have
  \[ \sup_{\varepsilon} \mathbb{E} \left[ | \mathbb{M}^{s_1, s_2}_{p, \ell_1,
     \ell_2, \varepsilon} (\mu_{\varepsilon}^{\alpha}) |^p + |
     \mathbb{M}^{s_1, s_2}_{p, \ell_1, \ell_2}
     (\bar{\mu}_{\varepsilon}^{\alpha}) |^p + | \mathbb{N}^s_{p, \ell,
     \varepsilon} (\mu_{\varepsilon}^{\alpha}, a, \beta) |^p + |
     \mathbb{N}^s_{p, \ell} (\bar{\mu}_{\varepsilon}^{\alpha}, a, \beta) |^p
     \right] < \infty, \]
  provided that $\ell_1, \ell_2 > 2$, $\ell > 4$.
\end{corollary}

\begin{proof}
  This follows from Theorem \ref{label:thm-estimate-chaos} and the definition
  of $\mathbb{M}, \mathbb{N}$.
\end{proof}

\subsection{Convergence of $\bar{\mu}^{\alpha}_{\varepsilon}$}

In this section we want to prove that $\bar{\mu}_{\varepsilon}^{\alpha}$
converges (in suitable Besov spaces) in $L^p$ to $\mu^{\alpha}$. We also want
to discuss some important consequences that can be deduced from this
convergence.

\begin{theorem}
  \label{theorem:Lpconvergence}For any $f \in L^r (\mathbb{R}^4)$ (for some $1
  < r < 2$) we have that, for any $n \in \mathbb{N}$,
  \[ \langle \overline{\mathcal{E}}_{\varepsilon} (: W_{\varepsilon}^n :), f
     \rangle = \langle : \bar{W}_{\varepsilon}^n :, f \rangle \rightarrow
     \langle : W^n :, f \rangle \]
  in $L^p (\Omega)$ for any $p \geqslant 2$. In particular there is a
  subsequence $\varepsilon_k = 2^{- r_k} \rightarrow 0$, such that for any $n
  \in \mathbb{N}$ we have $: \bar{W}^n_{\varepsilon_k} : \rightarrow : W^n :$
  weakly in $L^{r'} (\mathbb{R}^4)$, where $\frac{1}{r} + \frac{1}{r'} = 1$,
  and almost surely.
\end{theorem}

In the following we write
\[ \mathfrak{K}_{\varepsilon} (k) = \frac{16\left( \sin \left( \frac{\varepsilon
   k_1}{2} \right) \sin \left( \frac{\varepsilon k_2}{2} \right)
   \right)}{\varepsilon^2 k_1 k_2 \left( m^2 + \frac{4 \sin^2 \left(
   \frac{\varepsilon k_1}{2} \right)}{\varepsilon^2} + \frac{4 \sin^2 \left(
   \frac{\varepsilon k_2}{2} \right)}{\varepsilon^2} + k_3^2 + k_4^2 \right)},
\]
where $(k_1, k_2, k_3, k_4) \in \mathbb{R}^4$, and
\[ \mathfrak{K} (k) = \frac{1}{(m^2 + | k |^2)}, \]
\begin{lemma}
  We have that $\mathfrak{K}_{\varepsilon} \rightarrow \mathfrak{K}$ in $L^{2
  p} (\mathbb{R}^4)$ for any $1 < p < + \infty$.
\end{lemma}

\begin{proof}
  We observe that
  \begin{eqnarray}
    | \mathfrak{K}_{\varepsilon} (k) |^{2 p} & \lesssim & \frac{1}{(1 + | k
    |^2)^p} + \frac{1}{1 + | k_1 |^p} \frac{1}{(1 + | k_2 |^2 + | k_3 |^2 + |
    k_4 |^2)^{\frac{3}{2} p}}  \nonumber\\
    &  & + \frac{1}{1 + | k_2 |^p} \frac{1}{(1 + | k_1 |^2 + | k_3 |^2 + |
    k_4 |^2)^{\frac{3}{2} p}} .  \label{eq:mathfrakG}
  \end{eqnarray}
  where the constant in implicit in the symbol $\lesssim$ do not depend on
  $\varepsilon$. Since the function in the right hand side of equation
  {\eqref{eq:mathfrakG}} is in $L^1 (\mathbb{R}^4)$ and since
  $\mathfrak{K}_{\varepsilon} \rightarrow \mathfrak{K}$ pointwise, we can
  apply dominated convergence theorem to $\int_{\mathbb{R}^4} |
  \mathfrak{K}_{\varepsilon} - \mathfrak{K} |^{2 p} \mathd k$ obtaining the
  thesis.
\end{proof}

\begin{proof*}{Proof of Theorem \ref{theorem:Lpconvergence}}
  Here we follow an argument similar to {\cite{Simon_phi2}} Theorem V.3. By
  hypercontractivity it is sufficient to prove the convergence in $L^2
  (\Omega)$. We have that
  \begin{eqnarray}
    \mathbb{E} | \langle : \bar{W}_{\varepsilon}^n :, f \rangle - \langle :
    W^n :, f \rangle |^2 & = & \int_{\mathbb{R}^{4 n}} | \hat{f} (k_1 + \cdots
    + k_n) |^2  \nonumber\\
    & \phantom{\int} & \times | \mathfrak{K}_{\varepsilon} (k_1) \cdots
    \mathfrak{K}_{\varepsilon} (k_n) - \mathfrak{K} (k_1) \cdots \mathfrak{K}
    (k_n) |^2 \mathd k_1 \cdots \mathd k_n \nonumber\\
    & \lesssim & \sum_{\ell = 1}^n \int_{\mathbb{R}^{4 n}} | \hat{f} (k_1 +
    \cdots + k_n) |^2 \prod_{a = 1}^{\ell - 1} | \mathfrak{K}_{\varepsilon}
    (k_a) |^2 \nonumber\\
    &  & \times | \mathfrak{K}_{\varepsilon} (k_{\ell}) - \mathfrak{K}
    (k_{\ell}) |^2 \prod_{a = \ell + 1}^n | \mathfrak{K} (k_a) |^2 \mathd k_1
    \cdots \mathd k_n . \label{eq:convergenceprob1}
  \end{eqnarray}
  We use the fact that, by Hausdorff-Young inequality, $| \hat{f} |^2 \in
  L^{r'} (\mathbb{R}^4)$ for some $r' > 1$. Since $|
  \mathfrak{K}_{\varepsilon} - \mathfrak{K} |^2$ converges to $0$ in $L^p
  (\mathbb{R}^4)$ for $p = \frac{n r'}{(r' - 1)}$ and $|
  \mathfrak{K}_{\varepsilon} |^2$ is uniformly bounded in $L^p
  (\mathbb{R}^4)$, by Young inequality for convolutions, the integrand in
  {\eqref{eq:convergenceprob1}} converges to $0$ in $L^1$, which implies the
  first thesis.
  
  Furthermore observe in view of the previous discussion that
  {\eqref{eq:convergenceprob1}} can reformulated as
  \[ \mathbb{E} | \langle : \bar{W}_{\varepsilon}^n :, f \rangle - \langle :
     W^n :, f \rangle |^2 \leqslant C_n o (\varepsilon) \| f \|^2_{L^r}, \]
  where $C_n > 0$ is some constant depending on $n$ but not on $\varepsilon$.
  Now take a sequence $f_m$ which is dense in the ball $B (0, 1) \subset L^r
  (\mathbb{R}^4)$, of center $0$ and radius $1$, and define the quantity
  \[ d (\varphi, \psi) = \sum_{m\in \mathbb{N}} \frac{1}{m^2} | \langle \varphi - \psi, f_m
     \rangle_{L^2 (\mathbb{R}^4)} | . \]
  Clearly $d (\varphi_k, \psi) \rightarrow 0$ implies $\varphi_k \rightarrow
  \psi$ weakly in $L^{r'}$. On the other hand
  \[ \mathbb{E}d \left( : \bar{W}_{\varepsilon}^n :, \quad : W^n : \right)
     \leqslant C^{1 / 2}_n o (\varepsilon) \sum \frac{1}{m^2} \rightarrow 0.
  \]
  Extracting a subsequence such that $d \left( : \bar{W}_{\varepsilon}^n :, \;
  : W^n : \right) \rightarrow 0$ almost surely implies the statement of
  Theorem \ref{theorem:Lpconvergence}.
\end{proof*}

We are now going to prove a theorem about the convergence of the functionals
$\mathbb{M}_{p, \ell_1, \ell_2, \varepsilon}^{s, s'}
(\mu^{\alpha}_{\varepsilon}, a, \beta)$, $\mathbb{M}_{p, \ell_1, \ell_2}^{s,
s'} (\bar{\mu}^{\alpha}_{\varepsilon}, a, \beta)$, $\mathbb{M}_{p, \ell_1,
\ell_2}^{s, s'} (\mu^{\alpha}, a, \beta)$, $\mathbb{N}^s_{p, \ell,
\varepsilon} (\mu^{\alpha}_{\varepsilon}, a, \beta)$, $\mathbb{N}^s_{p, \ell}
(\bar{\mu}^{\alpha}_{\varepsilon}, a, \beta)$ and $\mathbb{N}^s_{p, \ell}
(\mu^{\alpha}, a, \beta)$ defined in equations {\eqref{eq:Mss1}},
{\eqref{eq:Mss2}}, {\eqref{eq:Ns1}} and {\eqref{eq:Ns2}}.

\begin{theorem}
  \label{theorem:convergencebar}Fix $\alpha \in \mathbb{R}$ such that
  $\frac{\alpha^2}{(4 \pi)^2} < 2$, then for any $a \in \mathbb{R}_+$ and $0 <
  \beta < 1$, $p \geqslant 2$ such that $\frac{\alpha^2}{(4 \pi)^2} (p - 1) <
  2$, any $\delta \geqslant 0$, any $s_1, s_2 <- \frac{\delta}{2}$ such that
  $s_1 + s_2 + \delta <- \frac{\alpha^2}{(4 \pi)^2} (p - 1)$ and $s < - \frac{\alpha^2}{(4 \pi)^2} (p - 1)$, and for any $\ell_1, \ell_2 > 2$ and
  $\ell > 4$, we have
  \[ \mathbb{M}^{s_1, s_2}_{p, \ell_1, \ell_2}
     (\bar{\mu}_{\varepsilon}^{\alpha}, a, \beta) \rightarrow \mathbb{M}^{s_1,
     s_2}_{p, \ell_1, \ell_2} (\mu^{\alpha}, a, \beta), \quad \mathbb{N}^s_{p,
     \ell} (\bar{\mu}_{\varepsilon}^{\alpha}, a, \beta) \rightarrow
     \mathbb{N}^s_{p, \ell} (\mu^{\alpha}, a, \beta) \]
  in $L^p (\Omega)$.
\end{theorem}

\begin{proof}
  We only prove the assertion for $\mathbb{M}^{s_1, s_2}_{p, \ell_1, \ell_2}
  (\bar{\mu}_{\varepsilon}^{\alpha}, a, \beta)$, the proof for
  $\mathbb{N}^s_{p, \ell} (\bar{\mu}_{\varepsilon}^{\alpha}, a, \beta)$ is
  similar.
  
  First we want to prove that $E_{i, j}^{a, \beta} \ast
  \bar{\mu}^{\alpha}_{\varepsilon} \rightarrow E_{i, j}^{a, \beta} \ast
  \mu^{\alpha}$ in $L^1 (\Omega, L^p_{\ell_1,\ell_2} (\mathbb{R}^4))$. For this, we
  consider first the case $p = 2$. For any $N \in \mathbb{N}$, write
  \[ \bar{\mu}^{\alpha, N}_{\varepsilon} = \sum_{n = 0}^N \frac{\alpha^n}{n!} :
     \bar{W}^n_{\varepsilon} : . \]
  Then, by the properties of the expectation of a Wick product and Remark \ref{remark:exchangeconvolution}, we have
  \begin{eqnarray*}
    &  & \mathbb{E} \left[ \| E^{a, \beta}_{i, j} \ast (\bar{\mu}^{\alpha,
    N}_{\varepsilon}) - E^{a, \beta}_{i, j} \ast (\bar{\mu}^{\alpha}_{\varepsilon})
    \|_{L^2_{\ell_1,\ell_2} (\mathbb{R}^4)}^2 \right]\\
   & = & \sum_{n = N + 1}^{+ \infty} \frac{\alpha^{2n}}{(n!)^2} \mathbb{E}\| E_{i,j}^{a,\beta}\asterisk :W^n_{\varepsilon}:\|_{L^2_{\ell_1,\ell_2}(\mathbb{R}^4)}^2\\
    & \lesssim & 
    \sum_{n = N + 1}^{+ \infty} \frac{\alpha^{2n}}{(n!)^2} \mathbb{E}[| E^{a,\beta}_{i,
    j} \ast : \bar{W}^n_{\varepsilon} :(0,0)|^2]
    \left( \int_{\mathbb{R}^4} [\rho_{\ell_1} (x) \rho_{\ell_2} (z)]^2 \mathd
    x \mathd z \right)\\
    & \lesssim & \sum_{n = N + 1}^{+ \infty} \frac{\alpha^{2n}}{(n!)}\int_{\mathbb{R}^8} E^{a, \beta}_{i, j} (x, z) E^{a-\kappa, \beta}_{i, j} (x',
    z')\\
    &  & \times \left( D_C + C \log_+ \left( \frac{1}{\sqrt{| x - x' |^2 + |
    z - z' |^2} \vee \varepsilon} \right) \right)^n \mathd x \mathd z \mathd
    x' \mathd z'\\
    & \longrightarrow & 0, \quad \tmop{as} N \rightarrow \infty,
  \end{eqnarray*}
  if we choose $\ell_1, \ell_2 > 2$, and $C>\frac{2}{4\pi^2}$ close enough to $\frac{2}{4\pi^2}$. So for any $\delta > 0$ we can choose an
  $N$ large enough such that
  \[ \sup_{0 < \varepsilon \leqslant 1} \left\{ \mathbb{E} \left[\| E^{a, \beta}_{i, j} \ast \bar{\mu}^{\alpha,
    N}_{\varepsilon} - E^{a, \beta}_{i, j} \ast \bar{\mu}^{\alpha}_{\varepsilon}
    \|^2_{L^{2}_{\ell_1,\ell_2}(\mathbb{R}^4)}\right] \right\}
     \leqslant \delta . \]
  Now from Theorem \ref{theorem:Lpconvergence}
  $\overline{\mathcal{E}}^{\varepsilon} (: W^n_{\varepsilon} :) =
  \bar{W}^{n}_{\varepsilon} \rightarrow : W^n :$ weakly in $L^r_{\ell_1, \ell_2}
  (\mathbb{R}^4)$ $\mathbb{P}$-almost surely. Since $E_{i, j}$ are smooth, it
  is not hard to deduce from the previous convergence that
  \[ E_{i, j}^{a,\beta} \ast :\bar{W}_{\varepsilon}^n: \rightarrow E_{i, j}^{a,\beta} \ast  : W^n : \]
  in $L^2 (\mathbb{R}^4)$, and almost surely, after
  extracting a suitable subsequence. Since
  \[ \sup_{\varepsilon} \left| \mathbb{E} \left[ \| E_{i, j}^{a, \beta} \ast
     \bar{W}^{n}_{\varepsilon} \|^{p}_{L_{\ell_1, \ell_2}^2 (\mathbb{R}^4)} \right]
     \right| < \infty, \]
  we can deduce by uniform integrability that
  \[ \mathbb{E} \left[ \| E_{i, j}^{a,\beta} \ast : \bar{W}_{\varepsilon}^n : - E_{i,
     j}^{a,\beta} \ast : W^n : \|^{2}_{L^2_{\ell_1, \ell_2} (\mathbb{R}^4)} \right]
     \rightarrow 0. \]
  Summing everything up we get
  \begin{eqnarray*}
    \lim_{\varepsilon \rightarrow 0} \mathbb{E} \left[ \| E_{i, j}^{a, \beta}
    \ast \bar{\mu}^{\alpha}_{\varepsilon} - E_{i, j}^{a, \beta} \ast
    \mu^{\alpha} \|^{2}_{L^2_{\ell_1, \ell_2} (\mathbb{R}^4)} \right] & \leqslant
    & \lim_{\varepsilon \rightarrow 0} \sum_{n = 0}^N \frac{\alpha^{2n}}{(n!)^2}
    \mathbb{E} \| : \bar{W}^n_{\varepsilon} : - : W^n : \|^{2}_{L^2_{\ell_1, \ell_2}
    (\mathbb{R}^4)} + 2 \delta\\
    & = & 2 \delta .
  \end{eqnarray*}
  Since $\delta > 0$ is arbitrary, this gives the statement for $p = 2$. Now
  to prove the full statement for $p > 2$ we proceed similarly: Again from
  equation {\eqref{eq:stochasticmumax1}} we deduce that for any $\delta > 0$
  there exists an $N$ such that that
  \begin{equation}
    \sup_{\varepsilon} \mathbb{E} \left[ \left| \mathbb{M}^{s_1, s_2}_{p,
    \ell_1,\ell_2, \varepsilon} (\mu_{\varepsilon}^{\alpha}, a, \beta)^p - \sum_{i \in
    \mathbb{N}_0, j \leqslant N} 2^{s_1 p i + s_2 p j} \int_{\mathbb{R}^4} (\rho_{\ell_1} (x)\rho_{\ell_2}(z))^p ((E^{a,
    \beta}_{i, j} \asterisk 
    \bar{\mu}_{\varepsilon}^{\alpha}) (x, z))^p \mathd x \mathd z \right|
    \right] \leqslant \delta . \label{eq:sum:cutoff}
  \end{equation}
  By the first step we have that
  \begin{equation}
    \begin{array}{lll}
      &  & \lim_{\varepsilon \rightarrow 0} \sum_{i \in \mathbb{N}_0, j
      \leqslant N} 2^{s_1 p i + s_2 p j} \int_{\mathbb{R}^4}(\rho_{\ell_1} (x) \rho_{\ell_2} (z))^p ((E^{a, \beta}_{i, j}
      \asterisk 
\bar{\mu}_{\varepsilon}^{\alpha}) (x, z))^p \mathd x \mathd z\\
      & = & \sum_{i \in \mathbb{N}_0, j \leqslant N} 2^{s_1 p i + s_2 p j}
      \int_{\mathbb{R}^4} (\rho_{\ell_1} (x)
      \rho_{\ell_2} (z))^p((E^{a, \beta}_{i, j} \asterisk  \mu^{\alpha}) (x, z))^p \mathd x \mathd z.
    \end{array} \label{eq:sum:finite}
  \end{equation}
  Combining {\eqref{eq:sum:cutoff}} and {\eqref{eq:sum:finite}} it is not hard
  to deduce that
  \begin{eqnarray*}
    &  & \lim_{\varepsilon \rightarrow 0} \mathbb{E} \left| \sum_{i, j \in
    \mathbb{N}_0} 2^{s_1 p i + s_2 p j} \int_{\mathbb{R}^4}(\rho_{\ell_1} (x) \rho_{\ell_2} (z))^p ((E^{a, \beta}_{i, j}
    \asterisk \bar{\mu}_{\varepsilon}^{\alpha}) (x, z))^p \mathd x \mathd z \right.\\
    &  & \left. - \sum_{i, j \in \mathbb{N}_0} 2^{s_1 p i + s_2 p j}
    \int_{\mathbb{R}^4} (\rho_{\ell_1} (x)
    \rho_{\ell_2} (z))^p((E^{a, \beta}_{i, j} \asterisk \mu^{\alpha}) (x, z))^p \mathd x \mathd z \right|.\\
    & = & 0
  \end{eqnarray*}
  This proves that $|\mathbb{M}^{s_1, s_2}_{p, \ell_1, \ell_2}
     (\bar{\mu}_{\varepsilon}^{\alpha}, a, \beta)|^p \rightarrow |\mathbb{M}^{s_1,
     s_2}_{p, \ell_1, \ell_2} (\mu^{\alpha}, a, \beta)|^p$ in $L^1(\Omega)$, which means in particular that $\mathbb{M}^{s_1, s_2}_{p, \ell_1, \ell_2}
     (\bar{\mu}_{\varepsilon}^{\alpha}, a, \beta) \rightarrow \mathbb{M}^{s_1,
     s_2}_{p, \ell_1, \ell_2} (\mu^{\alpha}, a, \beta)$ in probability and that the sequence $|\mathbb{M}^{s_1, s_2}_{p, \ell_1, \ell_2}
     (\bar{\mu}_{\varepsilon}^{\alpha}, a, \beta)|^p $ is uniformly integrable. By Vitali theorem (see \cite[Theorem 4.5.2]{Bogachev}) it follows that $\mathbb{M}^{s_1, s_2}_{p, \ell_1, \ell_2}
     (\bar{\mu}_{\varepsilon}^{\alpha}, a, \beta) \rightarrow \mathbb{M}^{s_1,
     s_2}_{p, \ell_1, \ell_2} (\mu^{\alpha}, a, \beta)$ in $L^p(\Omega,L^p_{l_1,l_2}(\Omega))$.
  
\end{proof}

\begin{remark}
  \label{remark:convergenceintegral}A consequence of proof of Theorem
  \ref{theorem:Lpconvergence} and the proof of Theorem
  \ref{theorem:convergencebar} is the following: consider $\frac{\alpha^2}{(4
  \pi)^2} < 2$ and $p \geqslant 2$ such that $\frac{\alpha^2}{(4 \pi)^2} (p -
  1) < 2$, then for any continuous function $f \in C^0 (\mathbb{R}^4)$
  decreasing at infinity faster then $\rho_{\ell}$ (for some $\ell > 4$) we
  have that $\int_{\mathbb{R}^4} f (x, z) \mu^{\alpha}_{\varepsilon} (\mathd
  x, \mathd z)$ converges to $\int_{\mathbb{R}^4} f (x, z) \mu^{\alpha}
  (\mathd x, \mathd z)$ in $L^p (\Omega)$.
\end{remark}

In the following theorem we want to analyze the relationship between
$\bar{\mu}^{\alpha}_{\varepsilon}$ (which, we recall, is defined on
$\mathbb{R}^4$) and its discrete counterpart $\mu^{\alpha}_{\varepsilon}$
defined on $\mathbb{R}^2 \times \varepsilon \mathbb{Z}^2$.

\begin{theorem}
  \label{theorem_convergenceepsilon}Under the same hypotheses on $\alpha$,
  $s_1$, $s_2$, $s$, $\ell_1$, etc. of Theorem \ref{theorem:convergencebar} we
  have
  \begin{eqnarray*}
    \mathbb{E} [| \mathbb{M}^{s_1, s_2}_{p, \ell_1, \ell_2, \varepsilon}
    (\mu_{\varepsilon}^{\alpha}, a, \beta) -\mathbb{M}^{s_1, s_2}_{p, \ell_1,
    \ell_2} (\bar{\mu}_{\varepsilon}^{\alpha}, a, \beta) |^p] & \rightarrow &
    0\\
    \mathbb{E} [| \mathbb{N}^s_{p, \ell, \varepsilon}
    (\mu_{\varepsilon}^{\alpha}, a, \beta) -\mathbb{N}^s_{p, \ell}
    (\bar{\mu}_{\varepsilon}^{\alpha}, a, \beta) |^p] & \rightarrow & 0,
  \end{eqnarray*}
  and so
  \[ \mathbb{M}^{s_1, s_2}_{p, \ell_1, \ell_2, \varepsilon}
     (\mu_{\varepsilon}^{\alpha}, a, \beta) \rightarrow \mathbb{M}^{s_1,
     s_2}_{p, \ell_1, \ell_2} (\mu^{\alpha}, a, \beta), \quad \mathbb{N}^s_{p,
     \ell, \varepsilon} (\mu_{\varepsilon}^{\alpha}, a, \beta) \rightarrow
     \mathbb{N}^s_{p, \ell} (\mu^{\alpha}, a, \beta), \]
  in $L^p (\Omega)$.
\end{theorem}

\begin{proof}
  We only prove the statement for $\mathbb{M}^{s_1, s_2}_{p, \ell,
  \varepsilon}$, the proof for $\mathbb{N}^s_{p, \ell, \varepsilon}$ is
  analogous. By definition of $\mathbb{M}^{s_1, s_2}_{p, \ell, \varepsilon}$
  we have
  \begin{eqnarray*}
    &  & \mathbb{E} | \mathbb{M}^{s_1, s_2}_{p, \ell, \varepsilon}
    (\mu_{\varepsilon}^{\alpha}, a, \beta)^p -\mathbb{M}^{s_1, s_2}_{p, \ell}
    (\bar{\mu}_{\varepsilon}^{\alpha}, a, \beta)^p |\\
    & = & \sum_{i, j} 2^{s_1 p i + s_2 p j} \mathbb{E} \left[ \left|
    \int_{\mathbb{R}^4} (\rho_{\ell_1} \rho_{\ell_2})^p | E_{i, j}^{a,\beta} \ast
\bar{\mu}^{\alpha}_{\varepsilon} |^p \mathd z \mathd x -
\int_{\mathbb{R}^2 \times \varepsilon \mathbb{Z}^2} (\rho_{\ell_1}
    \rho_{\ell_2})^p | E_{i, j}^{a,\beta} \ast \mu^{\alpha}_{\varepsilon} |^p \mathd x
    \mathd z \right| \right]\\
    & \leqslant & \sum_{i, j} 2^{s_1 p i + s_2 p j} \left( \mathbb{E} \left[
    \left| \int_{\mathbb{R}^4} (\rho_{\ell_1} \rho_{\ell_2})^p | E_{i, j}^{a,\beta} \ast
    \bar{\mu}^{\alpha}_{\varepsilon} |^p \mathd z \mathd x \right| \right]
    +\mathbb{E} \left[ \left| \int_{\mathbb{R}^2 \times \varepsilon
    \mathbb{Z}^2} (\rho_{\ell_1} \rho_{\ell_2})^p | E_{i, j}^{a,\beta} \ast
    \mu^{\alpha}_{\varepsilon} |^p \mathd x \mathd z \right| \right] \right)\\
    & = & \mathbb{E} [| \mathbb{M}^{s_1, s_2}_{p, \ell_1, \ell_2}
    (\bar{\mu}_{\varepsilon}^{\alpha}, a, \beta) |^p] +\mathbb{E} [|
    \mathbb{M}^{s_1, s_2}_{p, \ell_1, \ell_2, \varepsilon}
    (\mu_{\varepsilon}^{\alpha}, a, \beta) |^p],
  \end{eqnarray*}
  so applying dominated convergence to the sum it is enough to show that
  \begin{equation}
    \mathbb{E} \left[ \left| \int_{\mathbb{R}^4} (\rho_{\ell_1} \rho_{\ell_2})^p |
    E_{i, j}^{a,\beta} \ast \bar{\mu}^{\alpha}_{\varepsilon} |^p \mathd z \mathd x -
    \int_{\mathbb{R}^2 \times \varepsilon \mathbb{Z}^2} (\rho_{\ell_1}
    \rho_{\ell_2})^p | E_{i, j}^{a,\beta} \ast \mu^{\alpha}_{\varepsilon} |^p \mathd x
    \mathd z \right| \right] \rightarrow 0 \label{eq:mumubar1}.
  \end{equation}
  To prove this observe that
  \[ \left| \int_{\mathbb{R}^4} (\rho_{\ell_1} \rho_{\ell_2})^p | E_{i, j}^{a,\beta} \ast
     \bar{\mu}^{\alpha}_{\varepsilon} |^p (x, z) \mathd x \mathd z -
     \varepsilon^2 \sum_{z \in \varepsilon \mathbb{Z}^2} \int_{\mathbb{R}^2}
     (\rho_{\ell_1} \rho_{\ell_2})^p | E_{i, j}^{a,\beta} \ast
     \bar{\mu}^{\alpha}_{\varepsilon} |^p (x, z) \mathd x \right| \rightarrow
     0 \quad \tmop{as} \varepsilon \rightarrow 0 \]
  almost surely since we have that
  \[ \| \rho_{\ell_2} \|_{\mathcal{C}^{\delta}_{- \ell_2} (\mathbb{R}^2)}, \|
     | E_{i, j}^{a,\beta} \ast \bar{\mu}^{\alpha}_{\varepsilon} (\omega) | (x, z)
     \|_{L^{\infty}_{- \ell_1} (\mathbb{R}^2, C^{\delta}_{- \ell_2}
     (\mathbb{R}^2))} \lesssim C_{i, j} (\omega) \]
  where $C_{i, j} (\omega)$ is an almost surely finite positive random
  variable.\\
  
  Indeed for any $x \in \mathbb{R}^2$ we get
  \begin{eqnarray*}
    &  & \left| \int_{\mathbb{R}^4} (\rho_{\ell_1} \rho_{\ell_2})^p | E_{i, j}^{a,\beta}
    \ast \bar{\mu}^{\alpha}_{\varepsilon} |^p (x, z) \mathd x \mathd z -
    \varepsilon^2 \sum_{z \in \varepsilon \mathbb{Z}^2} \int_{\mathbb{R}^2}
    (\rho_{\ell_1} \rho_{\ell_2})^p | E_{i, j}^{a,\beta} \ast \mu^{\alpha}_{\varepsilon} |^p
    (x, z) \mathd x \right|\\
    & = & \left| \int_{\mathbb{R}^2} \sum_{z \in \varepsilon \mathbb{Z}^2}
    \rho_{\ell_1}^p \left( \int_{y \in Q_{\varepsilon} (z)} \rho_{\ell_2}^p (y) |
    E_{i, j}^{a,\beta} \ast \bar{\mu}^{\alpha}_{\varepsilon} |^p \left( x, y \right)
    \mathd y \right) \mathd x - \varepsilon^2 \sum_{z \in \varepsilon
    \mathbb{Z}^2} \int_{\mathbb{R}^2} (\rho_{\ell_1} \rho_{\ell_2})^p | E_{i, j}^{a,\beta}
    \ast \mu^{\alpha}_{\varepsilon} |^p (x, z) \mathd x \right|\\
    & = & \left| \int_{\mathbb{R}^2} \sum_{z \in \varepsilon \mathbb{Z}^2}
    \rho_{\ell_1}^p (x) \left[ \int_{y \in Q_{\varepsilon} (z)} \left(
    \rho_{\ell_2}^p (y) \left| E_{i, j}^{a,\beta} \ast \bar{\mu}^{\alpha}_{\varepsilon}
    \right|^p (x, y) \mathd y - \rho_{\ell_2}^p (z) | E_{i, j} ^{a,\beta}\ast
    \mu^{\alpha}_{\varepsilon} |^p (x, z) \right) \mathd y \right] \mathd x
    \right|\\
    & = & \left| \int_{\mathbb{R}^2} \sum_{z \in \varepsilon \mathbb{Z}^2}
    \rho_{\ell_1}^p (x) \left[ \int_{y \in Q_{\varepsilon} (z)} (\rho_{\ell_2}^p
    (y) - \rho_{\ell_2}^p (z)) | E_{i, j}^{a,\beta} \ast \bar{\mu}^{\alpha}_{\varepsilon}
    |^p (x, y)) \mathd y + \right. \right.\\
    &  & \left. \left. + \int_{y \in Q_{\varepsilon} (z)} \rho_{\ell_2}^p (z)
    (| E_{i, j}^{a,\beta} \ast \bar{\mu}^{\alpha}_{\varepsilon} |^p (x, y) - | E_{i, j}^{a,\beta}
    \ast \mu^{\alpha}_{\varepsilon} |^p (x, z)) \mathd y \right] \mathd x
    \right|\\
    & \leqslant & \left| \int_{\mathbb{R}^2} \sum_{z \in \varepsilon
    \mathbb{Z}^2} \rho_{\ell_1}^p (x) \int_{y \in Q_{\varepsilon} (z)}
    (\rho_{\ell_2}^p (y) - \rho_{\ell_2}^p (z) | E_{i, j}^{a,\beta} \ast
    \bar{\mu}^{\alpha}_{\varepsilon} |^p (x, y) \mathd y) \mathd x \right| +\\
    &  &  \left| \sum_{z \in \varepsilon \mathbb{Z}^2} \int_{\mathbb{R}^2}
    (\rho_{\ell_1} \rho_{\ell_2})^p \int_{y \in Q_{\varepsilon} (z)} (| E_{i, j}^{a,\beta}
    \ast \bar{\mu}^{\alpha}_{\varepsilon} |^{p - 1} (x, y) + | E_{i, j}^{a,\beta} \ast
    \mu^{\alpha}_{\varepsilon} |^{p - 1} (x, z)) \times \right.\\
    &  & \left. \left. \times (| E_{i, j}^{a,\beta} \ast
    \bar{\mu}^{\alpha}_{\varepsilon} | (x, y) - | E_{i, j}^{a,\beta} \ast
    \bar{\mu}^{\alpha}_{\varepsilon} | (x, z)) \mathd y \mathd x 
    \phantom{\int} \right|. \right.
  \end{eqnarray*}
  Now the first term can be estimated by
  \begin{eqnarray*}
    &  & \varepsilon^{\delta} \| \rho_{\ell_1} \|_{\mathcal{C}_{-
    \ell_2}^{\delta} (\mathbb{R}^2)} \int_{\mathbb{R}^2} \sum_{z \in
    \varepsilon \mathbb{Z}^2} \int_{y \in Q_{\varepsilon} (z)} ((\rho_{\ell_1}
    (x) \rho_{\ell_2} (y))^p | E_{i, j}^{a,\beta} \ast \bar{\mu}^{\alpha}_{\varepsilon} |^p
    (x, y) \mathd y) \mathd x\\
    & = & \varepsilon^{\delta} \| \rho_{\ell_1} \|_{\mathcal{C}_{-
    \ell_2}^{\delta} (\mathbb{R}^2)} \int_{\mathbb{R}^4} \rho_{\ell_1} (x)
    \rho_{\ell_2} (z) | E_{i, j}^{a,\beta} \ast \bar{\mu}^{\alpha}_{\varepsilon} |^p (z,
    x)  \mathd x \mathrm{d}z.
  \end{eqnarray*}
  For the second term we get the bound
  \begin{eqnarray*}
    &  & \sum_{z \in \varepsilon \mathbb{Z}^2} \int_{y \in Q_{\varepsilon}
    (z)} \rho_{\ell_2}^p (z) (| E_{i, j} ^{a,\beta}\ast \bar{\mu}^{\alpha}_{\varepsilon}
    |^{p - 1} (x, y) + | E_{i, j}^{a,\beta} \ast \bar{\mu}^{\alpha}_{\varepsilon} |^{p -
    1} (x, z))\\
    &  & (| E_{i, j}^{a,\beta} \ast \bar{\mu}^{\alpha}_{\varepsilon} | (x, y) - | E_{i,
    j}^{a,\beta} \ast \bar{\mu}^{\alpha}_{\varepsilon} | (x, z)) \mathd y\\
    & \leqslant & \varepsilon^{\delta} \sum_{z \in \varepsilon \mathbb{Z}^2}
    \| | E_{i, j}^{a,\beta} \ast \bar{\mu}^{\alpha}_{\varepsilon}  | (z)
    \|_{L^{\infty} (\mathbb{R}^2, C^{\delta}_{- \ell_2} (\mathbb{R}^2))}
    \rho_{\ell_1}^{2p} (z) \int_{y \in Q_{\epsilon} (z)} (| E_{i, j}^{a,\beta} \ast
    \bar{\mu}^{\alpha}_{\varepsilon} |^{p - 1} (x, y) + | E_{i, j}^{a,\beta} \ast
    \bar{\mu}^{\alpha}_{\varepsilon} |^{p - 1} (x, z)) \mathd y\\
    & \leqslant & \varepsilon^{\delta}\sum_{z\in\varepsilon\mathbb{Z}^2} \int_{\mathbb{R}^2} \rho^{2p}_{\ell_1} (y) (|
    E_{i, j}^{a,\beta} \ast \bar{\mu}^{\alpha}_{\varepsilon} |^{p - 1} (y) + | E_{i, j}^{a,\beta}
    \ast \bar{\mu}^{\alpha}_{\varepsilon} |^{p - 1} (z)) \| | E_{i, j}^{a,\beta} \ast
    \bar{\mu}^{\alpha}_{\varepsilon}  | \|_{C^{\delta}_{- \ell_2}}\mathrm{d} y\\
    & = & \varepsilon^{\delta} \| E_{i, j}^{a,\beta} \ast
    \bar{\mu}^{\alpha}_{\varepsilon} \|^{p - 1}_{L^{p - 1}} \| E_{i, j}^{a,\beta} \ast
    \bar{\mu}^{\alpha}_{\varepsilon} \|_{C^{\delta}_{- \ell}},
  \end{eqnarray*}
  and, by H{\"o}lder inequality,
  \[ \mathbb{E} [\| E_{i, j}^{a,\beta} \ast \bar{\mu}^{\alpha}_{\varepsilon} \|^{p -
     1}_{L^{p - 1}} \| E_{i, j}^{a,\beta} \ast \bar{\mu}^{\alpha}_{\varepsilon} 
     (z) \|_{C^{\delta}_{- \ell}}] \leqslant \mathbb{E} [\| E_{i, j}^{a,\beta} \ast
     \bar{\mu}^{\alpha}_{\varepsilon} \|^p_{L^{p - 1}}]^{\frac{p - 1}{p}}
     \mathbb{E} [\| E_{i, j}^{a,\beta} \ast \bar{\mu}^{\alpha}_{\varepsilon} (\omega)
     (z) \|^p_{C^{\delta}_{- \ell}}]^{1 / p} . \]
\end{proof}

\subsection{Convergence of the periodic approximation}\label{section:periodic}

In this section we discuss a periodic approximation of the measure
$\mu^{\alpha}_{\varepsilon}$ (with respect to the $\varepsilon \mathbb{Z}^2$
variables) that we will use in Section \ref{sec:stochastic-quantization}.

First we introduce a periodized version of $\xi_{\varepsilon}$ in the
following way: for any $N \in \mathbb{N}$ even number, we denote by
$[z]_{\varepsilon, N} \in \left[ - \frac{N \varepsilon}{2 }, \frac{N \varepsilon}{2} \right)^2$ the following number: let  $r_{i, z}$ be the reminder of the integer  division
$\frac{z_i}{\varepsilon} = N q_{i,z} + r_{i,z}$ (where $0 \leqslant r_{i,z} < N$, $q_{i,z} \in
\mathbb{Z}$ and $i \in \{ 1, 2 \}$), then we set $[z]_{\varepsilon, N} = \left(
\varepsilon r_{1, z} - \frac{N\varepsilon}{2}, \varepsilon r_{2, z} - \frac{N
\varepsilon}{2} \right)$. Using this notation we can introduce the periodized noise
\[ \xi_{\varepsilon, N} (f \otimes \delta_z) = \xi_{\varepsilon} (f \otimes
   \delta_{[z]_{\varepsilon, N}}) . \]
It is simple to see that for $z \in \left[ - \frac{N \varepsilon}{2}, \frac{N
\varepsilon}{2} \right)^2 \cap \varepsilon \mathbb{Z}^2$ we have
\[ \xi_{\varepsilon, N} (f \otimes \delta_z) = \xi_{\varepsilon} (f \otimes
   \delta_z), \]
and $\xi_{\varepsilon, N}$ is $\varepsilon N$ periodic with respect to the
variables in $\varepsilon \mathbb{Z}^2$. We write
\[ W_{\varepsilon, N} = (- \Delta_{\mathbb{R}^2} - \Delta_{\varepsilon
   \mathbb{Z}^2} + m)^{- 1} (\xi_{\varepsilon, N}) \]
and then we introduce the positive measure
\[ \mu^{\alpha}_{\varepsilon, N} = \sum_{n = 0}^{+ \infty} \frac{\alpha^n}{n!} :
   W_{\varepsilon, N}^n : = \exp \left( \alpha W_{\varepsilon, N} -
   \frac{\alpha^2}{2} \mathbb{E} [(W_{\varepsilon, N})^2] \right) . \]
It is useful also to write
\[ \mathcal{G}_{\varepsilon, N} (x, z) = \sum_{k \in \left(
   \mathbb{T}_{\frac{1}{\varepsilon}} \cap \frac{\pi}{N}
   \mathbb{Z} \right)^2} \frac{1}{N^2 (2 \pi)^2} \int_{\mathbb{R}^2}
   \frac{e^{- i (x \cdot y + k \cdot z)}}{\left( | y |^2 + 4 \varepsilon^{- 2}
   \sin^2 \left( \frac{\varepsilon k_1}{2} \right) + 4 \varepsilon^{- 2}
   \sin^2 \left( \frac{\varepsilon k_2}{2} \right) + m^2 \right)^2} \mathd y.
\]
We remark that
\[ \mathbb{E} [W_{\varepsilon, N} (x, z) W_{\varepsilon, N} (0, 0)] =
   \mathcal{G}_{\varepsilon, N} (x, z) . \]
\begin{lemma}
  The function $\mathcal{G}_{\varepsilon, N} (x, z)$ is uniformly bounded in
  $N$ and, for any $\varepsilon > 0$, it converges pointwise to
  $\mathcal{G}_{\varepsilon}$.
\end{lemma}

\begin{proof}
  We have that
  \[ | \mathcal{G}_{\varepsilon, N} | \leqslant \frac{1}{\varepsilon^2 (2
     \pi)^2} \int_{\mathbb{R}^2} \frac{1}{(| y |^2 + m^2)^2} \mathd y. \]
  The convergence of $\mathcal{G}_{\varepsilon, N} (x, z)$ follows from the
  convergence of the Riemann sums to the Riemann integral and the fact that
  the integrands defining $\mathcal{G}_{\varepsilon, N}$ and
  $\mathcal{G}_{\varepsilon}$ are smooth and summable.
\end{proof}

\begin{proposition}
  \label{proposition:convergenceperiodic}For any $\varepsilon > 0$, $s < 0$,
  $s_1 \in \mathbb{R}$, $s_2 < 0$, $p \in [2, + \infty]$, $a \in
  \mathbb{R}_+$, $0 < \beta < 1$, and $\ell > 4$ we have that
  \[ \begin{array}{lll}
       \mathbb{E} [| \mathbb{M}^{s_1, s_2}_{p, \ell, \varepsilon}
       (\mu_{\varepsilon, N}^{\alpha}, a, \beta) -\mathbb{M}^{s_1, s_2}_{p,
       \ell, \varepsilon} (\mu_{\varepsilon}^{\alpha}, a, \beta) |^2] &
       \rightarrow & 0\\
       \mathbb{E} [| \mathbb{N}^s_{p, \ell, \varepsilon} (\mu_{\varepsilon,
       N}^{\alpha}, a, \beta) -\mathbb{N}^s_{p, \ell, \varepsilon}
       (\bar{\mu}_{\varepsilon}^{\alpha}, a, \beta) |^2] & \rightarrow & 0,
     \end{array} \]
  as $N \rightarrow + \infty$.
\end{proposition}

\begin{proof}
  Arguing as in the proof of Theorem \ref{theorem_convergenceepsilon} it is
  enough to show that
  \[ \mathbb{E} \| E_{i, j}^{a,\beta} \ast \mu_{\varepsilon, N}^{\alpha} - E_{i, j}^{a,\beta}
     \ast \mu_{\varepsilon}^{\alpha} \|^2_{L_{\ell_1, \ell_2}^p (\mathbb{R}^2
     \times \varepsilon \mathbb{Z}^2)} \rightarrow 0 \]
  as $N \rightarrow \infty$. Observe that for any $g \in L^1 (\mathbb{R}^2
  \times \varepsilon \mathbb{Z}^2)$ we have
  \begin{eqnarray*}
    &  & \mathbb{E} \left[ \left( \int_{\mathbb{R}^2 \times \varepsilon
    \mathbb{Z}^2} g (x, z) \mathd \mu^{\alpha}_{\varepsilon, N} \right)^2
    \right]\\
    & = & \int_{(\mathbb{R}^2 \times \varepsilon \mathbb{Z}^2)^2} g (x, z) g
    (x', z') e^{\alpha^2 \mathcal{G}_{\varepsilon, N} (x - x', z - z')} \mathd
    x \mathd x' \mathd z \mathd z'\\
    & \lesssim & \exp \left( \frac{1}{\varepsilon^2} \right) \| g \|_{L^1
    (\mathbb{R}^2 \times \varepsilon Z^2)}^2,
  \end{eqnarray*}
  uniformly in $N \in \mathbb{N}$. Furthermore by the binomial formula
  \begin{align*} &\mathbb{E} \left[ \left( \int_{\mathbb{R}^2 \times \varepsilon
     \mathbb{Z}^2} g (x, z) \mathd \mu^{\alpha}_{\varepsilon, N} -
     \int_{\mathbb{R}^2 \times \varepsilon \mathbb{Z}^2} g (x, z) \mathd
     \mu^{\alpha}_{\varepsilon} \right)^2 \right] \\
   =&\mathbb{E} \left[ \left( \int_{\mathbb{R}^2 \times \varepsilon
     \mathbb{Z}^2} g (x, z) \mathd \mu^{\alpha}_{\varepsilon, N} \right)^2
     \right] +\mathbb{E} \left[ \left( \int_{\mathbb{R}^2 \times \varepsilon
     \mathbb{Z}^2} g (x, z) \mathd \mu^{\alpha}_{\varepsilon} \right)^2
     \right] \\
  & - 2\mathbb{E} \left[ \int_{\mathbb{R}^2 \times \varepsilon \mathbb{Z}^2}
     g (x, z) \mathd \mu^{\alpha}_{\varepsilon, N} \int_{\mathbb{R}^2 \times
     \varepsilon \mathbb{Z}^2} g (x, z) \mathd \mu^{\alpha}_{\varepsilon}
     \right]. \end{align*}
  A consequence of Wick theorem is that
  \begin{eqnarray*}
    &  & \mathbb{E} \left[ \int_{\mathbb{R}^2 \times \varepsilon
    \mathbb{Z}^2} g (x, z) \mathd \mu^{\alpha}_{\varepsilon, N}
    \int_{\mathbb{R}^2 \times \varepsilon \mathbb{Z}^2} g (x, z) \mathd
    \mu^{\alpha}_{\varepsilon} \right]\\
    & = & \int_{(\mathbb{R}^2 \times \varepsilon \mathbb{Z}^2)^2} g (x, z) g
    (x', z') e^{\alpha^2 \mathbb{E} [W_{\varepsilon, N} (x, z) W_{\varepsilon}
    (x', z')]} \mathd x \mathd x' \mathd z \mathd z' .
  \end{eqnarray*}
  And since
  \[ W_{\varepsilon, N} (x, z) W_{\varepsilon} (x', z') \leqslant
     (W_{\varepsilon, N} (x, z))^2 + (W_{\varepsilon} (x', z'))^2 \]
  and clearly
  \[ W_{\varepsilon, N} (x, z) \rightarrow W_{\varepsilon} (x, z) \]
  almost surely as $N \rightarrow \infty$, we have by dominated convergence
  that
  \[ \mathbb{E} [W_{\varepsilon, N} (x, z) W_{\varepsilon} (x', z')]
     \rightarrow \mathcal{G}_{\varepsilon} (x - x', z - z') . \]
  So
  \begin{eqnarray*}
    &  & \lim_{N \rightarrow \infty} \mathbb{E} \left[ \left(
    \int_{\mathbb{R}^2 \times \varepsilon \mathbb{Z}^2} g (x, z) \mathd
    \mu^{\alpha}_{\varepsilon, N} \right)^2 \right] +\mathbb{E} \left[ \left(
    \int_{\mathbb{R}^2 \times \varepsilon \mathbb{Z}^2} g (x, z) \mathd
    \mu^{\alpha}_{\varepsilon} \right)^2 \right]\\
    &  & - \lim_{N \rightarrow \infty} 2\mathbb{E} \left[ \int_{\mathbb{R}^2
    \times \varepsilon \mathbb{Z}^2} g (x, z) \mathd
    \mu^{\alpha}_{\varepsilon, N} \int_{\mathbb{R}^2 \times \varepsilon
    \mathbb{Z}^2} g (x, z) \mathd \mu^{\alpha}_{\varepsilon} \right]\\
    & = & 2 \int_{(\mathbb{R}^2 \times \varepsilon \mathbb{Z}^2)^2} g (x, z)
    g (x', z') e^{\alpha^2 \mathcal{G}_{\varepsilon} (x - x', z - z')} \mathd
    x \mathd x' \mathd z \mathd z'\\
    &  & - 2 \int_{(\mathbb{R}^2 \times \varepsilon \mathbb{Z}^2)^2} g (x, z)
    g (x', z') e^{\alpha^2 \mathcal{G}_{\varepsilon} (x - x', z - z')} \mathd
    x \mathd x' \mathd z \mathd z'\\
    & = & 0.
  \end{eqnarray*}
  This implies that $\mathbb{E} [| E_{i, j}^{a,\beta} \ast \mu_{\varepsilon, N}^{\alpha}
  - E_{i, j}^{a,\beta} \ast \mu_{\varepsilon}^{\alpha} |^2 (x, z)] \rightarrow 0$ for
  $x, z \in \mathbb{R}^2 \times \varepsilon \mathbb{Z}^2$. Now since
  \[ \| E_{i, j}^{a,\beta} \ast \mu^{\alpha}_{\varepsilon} \|_{L_{\ell_1,
     \ell_2}^{\infty} (\mathbb{R}^2 \times \varepsilon \mathbb{Z}^2)} \lesssim
     \| E_{i, j}^{a,\beta} \|_{L^{\infty} (\mathbb{R}^2 \times \varepsilon
     \mathbb{Z}^2)} \| \rho_{\ell_1} (x) \rho_{\ell_2} (z) \mu^{\alpha}_{\varepsilon} \|_{L^1(\mathbb{R}^2\times \varepsilon \mathbb{Z}^2)},
  \]
  we have, by dominated convergence, that, for any \ $\ell_1, \ell_2 > 2$,
  \[ \mathbb{E} \left[ \int_{\mathbb{R}^2\times \varepsilon \mathbb{Z}^2} | E_{i, j}^{a,\beta} \ast \mu_{\varepsilon, N}^{\alpha} -
     E_{i, j}^{a,\beta} \ast \mu_{\varepsilon}^{\alpha} |^p (x, z) (\rho_{\ell_1} (x)
     \rho_{\ell_2} (z))^p \mathd x \mathd z \right] \rightarrow 0, \]
  which implies the statement.
\end{proof}

\section{Elliptic stochastic quantization of $\cosh (\beta \varphi)_2$ model
}\label{sec:stochastic-quantization}

\subsection{The $\cosh (\beta \varphi)_2$ model}\label{section:coshmodel}

\

In this section we want to recall briefly the definition and the construction
of Euclidean $\cosh (\beta \varphi)_2$ model.

\

First consider the Gaussian measure $\tilde{\nu}_{m, \varepsilon, N}$ on
$\mathcal{S}' (\varepsilon (\mathbb{Z}/ N\mathbb{Z})^2)$, where $m > 0$ is the
mass of the boson, $N \in \mathbb{N}$ \ is the size of the spatial cut-off,
and $\varepsilon > 0$ is the size of the ultra-violet cut-off, describing a
random field $\tilde{\varphi}_{\varepsilon, N} \in \mathcal{S}' (\varepsilon
(\mathbb{Z}/ N\mathbb{Z})^2)$ with covariance
\[ \int \tilde{\varphi}_{\varepsilon, N} (x) \tilde{\varphi}_{\varepsilon, N}
   (y) \tilde{\nu}_{m, \varepsilon, N} (\mathd \tilde{\varphi}_{\varepsilon, N}) = (-
   \Delta_{\varepsilon \mathbb{Z}^2, N} + m^2)^{- 1} (x - y), \quad x, y \in
   \varepsilon (\mathbb{Z}/ N\mathbb{Z})^2, \]
where $(- \Delta_{\varepsilon \mathbb{Z}^2, N} + m^2)^{- 1} (x - y)$ is the
Green function of the operator $- \Delta_{\varepsilon \mathbb{Z}^2} + m^2$
with periodic boundary condition, and $\Delta_{\varepsilon \mathbb{Z}^2, N}$
is the discrete Laplacian on $(\mathbb{Z}/ N\mathbb{Z})^2$. There is a
natural continuous map
\[ i_{\varepsilon, N} : \mathcal{S}' (\varepsilon (\mathbb{Z}/ N\mathbb{Z})^2)
   \rightarrow \mathcal{S}' (\varepsilon \mathbb{Z}^2), \]
given by the periodic extension, i.e. if $g \in \mathcal{S}' (\varepsilon
(\mathbb{Z}/ N\mathbb{Z})^2)$ we have
\[ i_{\varepsilon, N} [g] (z) = g ([z]_{\varepsilon, N}), \quad z \in
   \varepsilon \mathbb{Z}^2, \]
where $[z]_{\varepsilon, N} \in \varepsilon (\mathbb{Z}/ N\mathbb{Z})^2$ is
the equivalence class of the point $z \in \varepsilon \mathbb{Z}^2$. We recall
that the extension operator $\mathcal{E}_{\varepsilon} : \mathcal{S}'
(\varepsilon \mathbb{Z}^2) \rightarrow \mathcal{S}' (\mathbb{R}^2)$
(introduced in Section \ref{section:extension}) is a continuous map (and so
measurable) we then denote by
\[ \nu_{m, \varepsilon, N} = (\mathcal{E}_{\varepsilon} \circ i_{\varepsilon,
   N})_{\ast} (\tilde{\nu}_{m, \varepsilon, N}) . \]
The measure $\nu_{m, \varepsilon, N}$ is a Gaussian measure on $\mathcal{S}'
(\mathbb{R}^2)$ whose associated random field is constant in boxes of side
$\varepsilon$ centers on $\varepsilon \mathbb{Z}^2$ and periodic with period
$\varepsilon N$. We use the following notation $\hat{\mathbb{T}}^2_N = \left[ -
\frac{N \varepsilon}{2 }, \frac{N \varepsilon}{2 } \right]^2$ and we denote
\[ V^{\cosh, \beta}_N (\varphi) = \lambda
   \int_{\hat{\mathbb{T}}^2_N} (: e^{\beta \varphi (x)} : + : e^{- \beta
   \varphi (x)} :) \mathd x, \]
which is well defined $\nu_{m, \varepsilon, N}$ almost surely, and $: e^{\pm
\beta \varphi} :$ is defined
\[ : e^{\pm \beta \varphi} : = \sum_{n = 0}^{+ \infty} \frac{(\pm
   \beta)^n}{n!} : \varphi^n : \]
as the Wick exponential with of the random field $\varphi$ with respect the
measure $\nu_{m, \varepsilon, N}$. We also use the notation
\[:\cosh(\beta \varphi ):=\frac{1}{2}\left(:e^{\beta \varphi}: + :e^{-\beta \varphi}: \right). \]
We further write
\[ Z_{m, \varepsilon, N} = \int e^{- V_N^{\cosh, \beta} (\varphi)} \nu_{m,
   \varepsilon, N} (\mathd \varphi) . \]
\begin{definition}
  \label{definition:coshmodel}We say that the measure $\nu_m^{\cosh, \beta}$
  on $\mathcal{S}' (\mathbb{R}^2)$ is the measure related with the Euclidean
  quantum field theory having action
  \[ S (\varphi) = \frac{1}{2} \int_{\mathbb{R}^2} (| \nabla \varphi (z) |^2
     + m^2 \varphi (z)^2) \mathd z + 2 \lambda \int_{\mathbb{R}^2} : \cosh
     (\beta \varphi (z)) : \mathd z, \]
  if there are two sequences $\varepsilon_{n'} \rightarrow 0$ and $N_n
  \rightarrow + \infty$ (where $\varepsilon_{n'} > 0$ and $N_n \in
  \mathbb{N}$) such that
  \[ \nu_m^{\cosh, \beta} (\mathd \varphi) = \left( \lim_{\varepsilon_{n'}\rightarrow 0}
     \left( \lim_{N_n \rightarrow + \infty} \frac{e^{- V_{N_n}^{\cosh, \beta}
     (\varphi)}}{Z_{m, \varepsilon_{n'}, N_n}} \nu_{m, \varepsilon_{n'}, N_n}
     (\mathd \varphi) \right) \right), \]
  where the limits are taken in weak sense in the space of probability
  measures.
\end{definition}

\subsection{The equation and its approximations}\quad

Here we consider the following elliptic (heuristic) SPDE
\begin{equation} (- \Delta_{\mathbb{R}^4} + m^2) \phi + 2 \alpha \sinh (\alpha \phi -
   \infty) = \xi . \label{eq:heuristic1} \end{equation}
More precisely we say that a random field taking values in $\mathcal{S}'
(\mathbb{R}^4)$ is a solution to equation {\eqref{eq:heuristic1}}, if, writing
$\bar{\phi} = \phi - W$, where as usual $W = (- \Delta_{\mathbb{R}^4} +
m^2)^{- 1} (\xi)$, then the random field solves the following equation
\begin{equation}
  (- \Delta_{\mathbb{R}^4} + m^2) \bar{\phi} + \alpha e^{\alpha \bar{\phi}}
  \mu^{\alpha} - \alpha e^{- \alpha \bar{\phi}} \mu^{- \alpha} = 0,
  \label{eq:main5}
\end{equation}
where $\mu^{\alpha}$ and $\mu^{- \alpha}$ are the positive measure introduced
in Section \ref{sec:stochastic-estimates} and defined by $\mu^{\pm \alpha}
\assign \sum_{n = 0}^{+ \infty} \frac{(\pm \alpha)^n}{n!} : W^n :$.

\

We introduce some approximations to equation {\eqref{eq:main5}}, using the
discretization and the perioditization introduced in Section
\ref{sec:stochastic-estimates}. We define the following equations
\begin{equation}
  (- \Delta_{\mathbb{R}^2 \times \varepsilon \mathbb{Z}^2} + m^2)
  \bar{\phi}_{\varepsilon, N} + \alpha e^{\alpha \bar{\phi}_{\varepsilon, N}}
  \mu^{\alpha}_{\varepsilon, N} - \alpha e^{- \alpha \bar{\phi}_{\varepsilon,
  N}} \mu^{- \alpha}_{\varepsilon, N} = 0 \label{eq:mainfinitedimensional}
\end{equation}
\begin{equation}
  (- \Delta_{\mathbb{R}^2 \times \varepsilon \mathbb{Z}^2} + m^2)
  \bar{\phi}_{\varepsilon} + \alpha e^{\alpha \bar{\phi}_{\varepsilon}}
  \mu^{\alpha}_{\varepsilon} - \alpha e^{- \alpha \bar{\phi}_{\varepsilon}}
  \mu^{- \alpha}_{\varepsilon} = 0, \label{eq:maindiscrete}
\end{equation}
where $- \Delta_{\mathbb{R}^2 \times \varepsilon \mathbb{Z}^2} (g) = -
\Delta_{\mathbb{R}^2} g - \Delta_{\varepsilon \mathbb{Z}^2} g$, and
\[ \mu^{\pm \alpha}_{\varepsilon, N} \assign \sum_{n = 0}^{+ \infty}
   \frac{(\pm \alpha)^n}{n!} : W^n_{\varepsilon, N} :, \quad \mu^{\pm
   \alpha}_{\varepsilon} \assign \sum_{n = 0}^{+ \infty} \frac{(\pm
   \alpha)^n}{n!} : W^n_{\varepsilon} :, \]
and where as, usual, $W_{\varepsilon, N} = (- \Delta_{\mathbb{R}^2 \times
\varepsilon \mathbb{Z}^2} + m^2)^{- 1} (\xi_{\varepsilon, N})$ and
$W_{\varepsilon} = (- \Delta_{\mathbb{R}^2 \times \varepsilon \mathbb{Z}^2} +
m^2)^{- 1} (\xi_{\varepsilon})$ (see Section \ref{section:periodic} and
Section \ref{section:setting} for the definition of $\xi_{\varepsilon, N}$ and
$\xi_{\varepsilon}$).

\

In the rest of the present section we want to prove the existence and
uniqueness of solution to equation {\eqref{eq:main5}} using an approximation
argument starting from the solutions to equations
{\eqref{eq:mainfinitedimensional}} and {\eqref{eq:maindiscrete}}.

\

\begin{theorem}
There is a unique solution $\bar{\phi}$ to equation {\eqref{eq:main5}} such
  that $\bar{\phi}, e^{\pm \frac{\alpha}{2} \bar{\phi}} \in H^1_{\ell} (\mathbb{R}^4)$, $ e^{\pm 2\alpha
  \bar{\phi}} \in L^1_{2\ell} (\mathbb{R}^4, \mathd \mu^{\pm \alpha})$, for any $\ell \geqslant \ell_0$ (where $\ell_0 > 0$ is large enough). Furthermore, this solution $\bar\phi$ is such that $e^{\pm \beta \alpha \bar{\phi}} \in H^1_{\ell}
  (\mathbb{R}^4 )$ and $ e^{\pm (1 + 2
  \beta) \alpha \bar{\phi}} \in L^1_{\ell}(\mathbb{R}^4, \mathd
  \mu^{\pm \alpha})$ for any $0 \leqslant \beta < 1$. Here we introduce the norm given by 
  \[\|f\|_{L_\ell^1(\mathbb{R}^4,\mathd\mu^{\pm \alpha})}=\int \rho^{(4)}_{\ell}(x)|f(x)|\mathd \mu^{\pm \alpha}, \] and $L_\ell^1(\mathbb{R}^4,\mathd\mu^{\pm \alpha})$ is the space of functions where this norm is finite. 
\end{theorem}

\subsection{Some a priori estimates}

In order to prove the existence for equation {\eqref{eq:maindiscrete}} we need
to prove some a priori estimates for an equation of the form
\begin{equation}
  (m^2 - \Delta_{\mathbb{R}^2} - \Delta_{\varepsilon \mathbb{Z}^2}) \psi +
  e^{\alpha \psi} \mathd \eta_+ - e^{- \alpha \psi} \mathd \eta_- = 0,
  \label{eq:apriori}
\end{equation}
where
\begin{itemize}
  \item $\alpha, m > 0$,
  
  \item $\eta_+, \eta_- \in \mathcal{M} (\mathbb{R}^2 \times \varepsilon
  \mathbb{Z}^2) \cap B^{- 1 + \kappa_{\alpha}}_{2, 2, \ell} (\mathbb{R}^2
  \times \varepsilon \mathbb{Z}^2)$ where $\kappa_{\alpha}$ is a suitable
  constant depending on $\alpha$, and $\ell > 0$ big enough, and where we denote by $\mathcal{M}(\mathbb{R}^2 \times \varepsilon \mathbb{Z}^2)$ the space of $\sigma$-finite measures on $\mathbb{R}^2 \times \varepsilon \mathbb{Z}^2$.
\end{itemize}

Hereafter we use the following notation: if $\upsilon : \mathbb{R}^2 \times
\varepsilon \mathbb{Z}^2 \rightarrow \mathbb{R}$ is a function weakly
differentiable with respect the first two variable we write
\[ \nabla_{\mathbb{R}^2 \times \varepsilon \mathbb{Z}^2} \upsilon = \left(
   \begin{array}{c}
     \nabla_{\mathbb{R}^2} \upsilon\\
     \nabla_{\varepsilon} \upsilon
   \end{array} \right), \]
where the first component is the vector in $\mathbb{R}^2$ obtained as the
gradient with respect to the first two coordinates and the second is the
finite difference defined in Section \ref{section:besov:difference}.

\begin{theorem}
  \label{theorem_apriori1}Suppose that $\psi$ solves {\eqref{eq:apriori}}, and
  suppose that $\int \rho_{\ell}^{(4)} (x, z) \mathd \eta_+ (x, z), \int
  \rho_{\ell}^{(4)} (x, z) \mathd \eta_- (x, z) < + \infty$ and $\| \psi
  \|_{H^1_{\ell/2} (\mathbb{R}^2 \times \varepsilon \mathbb{Z}^2)} < + \infty$,
  then
  \begin{eqnarray*}
    &  & \| \rho^{(4)}_{\ell/2} \nabla_{\mathbb{R}^2 \times \varepsilon
    \mathbb{Z}^2} \psi \|^2_{L^2} + \| \rho^{(4)}_{\ell/2} \psi \|_{L^2}^2 +
    \int_{\mathbb{R}^2 \times \varepsilon \mathbb{Z}^2} \rho_{\ell}^{(4)} |
    \psi | e^{\alpha \psi} d \eta_+ + \int_{\mathbb{R}^2 \times \varepsilon
    \mathbb{Z}^2} \rho_{\ell}^{(4)} | \psi | e^{- \alpha \psi} \mathd \eta_-\\
    & \lesssim & (\rho_{\ell}^{(4)} \mathd \eta_+) (\mathbb{R}^2 \times
    \varepsilon \mathbb{Z}^2) + (\rho_{\ell}^{(4)} \mathd \eta_-)
    (\mathbb{R}^2 \times \varepsilon \mathbb{Z}^2).
  \end{eqnarray*}
\end{theorem}

\begin{proof}
  Multiplying equation {\eqref{eq:apriori}} by $\rho_{\ell, \lambda}^{(4)}
  \psi$, where
  \[ \rho_{\ell, \lambda}^{(4)} (x, z) = \rho^{(4)}_{\ell} (\lambda x,
     \lambda z), \]
  with $\lambda > 0$, and then taking the integral, we obtain
  \begin{multline}
    - \int_{\mathbb{R}^2 \times \varepsilon \mathbb{Z}^2} \rho_{\ell,
    \lambda}^{(4)} \psi (\Delta_{\mathbb{R}^2 \times \varepsilon \mathbb{Z}^2}
    \psi) \mathd x \mathd z + m^2 \int_{\mathbb{R}^2 \times \varepsilon
    \mathbb{Z}^2} \rho_{\ell, \lambda}^{(4)} (\psi)^2 \mathd x \mathd z
    \\
    + \int_{\mathbb{R}^2 \times \varepsilon \mathbb{Z}^2} \rho_{\ell,
    \lambda}^{(4)} \psi e^{\alpha \psi} d \eta_+ - \int_{\mathbb{R}^2 \times
    \varepsilon \mathbb{Z}^2} \rho_{\ell, \lambda}^{(4)} \psi e^{- \alpha
    \psi} d \eta_- = 0. \label{eq:esppart}
  \end{multline}
The term $ - \int_{\mathbb{R}^2 \times
    \varepsilon \mathbb{Z}^2} \rho_{\ell, \lambda}^{(4)} \psi e^{- \alpha
    \psi} d \eta_-$ could be bounded from below as follows 
\begin{align*} 
- \int \rho_{\ell, \lambda}^{(4)} \psi e^{- \alpha \psi} \mathd \eta_- =& \int \rho_{\ell, \lambda}^{(4)} \psi_- e^{\alpha \psi_-} \mathd \eta_- - \int \rho_{\ell, \lambda}^{(4)} \psi_+ e^{- \alpha \psi_+} \mathd \eta_- \\
=& \int \rho_{\ell, \lambda}^{(4)} \psi_-e^{\alpha \psi_-} \mathd \eta_- +\int \rho_{\ell, \lambda}^{(4)} \psi_+ e^{- \alpha \psi_+} \mathd \eta_- - 2 \int \rho_{\ell, \lambda}^{(4)} \psi_+ e^{- \alpha \psi_+} \mathd \eta_-\\
\geqslant& \int \rho_{\ell, \lambda}^{(4)} | \psi | e^{- \alpha \psi} \mathd \eta_- - 2 C \int \rho_{\ell, \lambda}^{(4)} \mathd \eta_-,
\end{align*}
where $\psi_{\pm}$ are respectively the positive and negative parts of the function $\psi=\psi_+-\psi_+$, and 
 $C=\sup_{x \geqslant 0} x e^{- \alpha x}$.
In an analogous way we can handle the term $\int_{\mathbb{R}^2 \times \varepsilon \mathbb{Z}^2} \rho_{\ell,
    \lambda}^{(4)} \psi e^{\alpha \psi} d \eta_+$ in \eqref{eq:esppart}.
For the other terms in equality \eqref{eq:esppart}, we start by integrating by parts
  \[ - \int_{\mathbb{R}^2 \times \varepsilon \mathbb{Z}^2} \rho_{\ell,
     \lambda}^{(4)} \psi (\Delta_{\mathbb{R}^2 \times \varepsilon
     \mathbb{Z}^2} \psi) \mathd x \mathd z = \int_{\mathbb{R}^2 \times
     \varepsilon \mathbb{Z}^2} \nabla_{\mathbb{R}^2 \times \varepsilon
     \mathbb{Z}^2} \psi \cdot (\nabla_{\mathbb{R}^2 \times \varepsilon
     \mathbb{Z}^2} (\rho_{\ell, \lambda}^{(4)} \psi)) \mathd x \mathd z \]
  and
\begin{equation} 
\nabla_{\mathbb{R}^2 \times \varepsilon \mathbb{Z}^2} (\rho_{\ell,
     \lambda}^{(4)} \psi) = \ll \rho_{\ell, \lambda}^{(4)},
     \nabla_{\mathbb{R}^2 \times \varepsilon \mathbb{Z}^2} \psi \gg + \psi
     \nabla_{\mathbb{R}^2 \times \varepsilon \mathbb{Z}^2} \rho_{\ell,
     \lambda}^{(4)} \label{eq:nabla1},
\end{equation}
  where $\ll \rho_{\ell, \lambda}^{(4)}, \nabla_{\mathbb{R}^2 \times
  \varepsilon \mathbb{Z}^2} \psi \gg$ is the vector with components
  \[ \ll \rho_{\ell, \lambda}^{(4)}, \nabla_{\mathbb{R}^2 \times \varepsilon
     \mathbb{Z}^2} \psi \gg (x, z) = \left( \begin{array}{c}
       \rho_{\ell, \lambda}^{(4)} (x, z) \partial_{x^1} \psi (x, i)\\
       \rho_{\ell, \lambda}^{(4)} (x, z) \partial_{x^2} \psi (x, i)\\
       \varepsilon^{- 1} \rho_{\ell, \lambda}^{(4)} (x, z + \varepsilon e_1)
       D_{\varepsilon e_1} \psi\\
       \varepsilon^{- 1} \rho_{\ell, \lambda}^{(4)} (x, z + \varepsilon e_2)
       D_{\varepsilon e_2} \psi
     \end{array} \right), \]
  where $\{ e_1, e_2 \}$ is the standard basis of $\varepsilon \mathbb{Z}^2
  \subset \mathbb{R}^2$, and $D_h (f) (x, z) = f (x, z + h) - f (x, z)$, for
  any $h \in \varepsilon \mathbb{Z}^2$. For any $\delta > 0$ we can choose
  $\ell$ large enough and $\lambda$ small enough such that, by computing
  \[ \nabla_{\mathbb{R}^4} \rho_{\ell, \lambda}^{(4)} (x, z) = -\ell \lambda (1
     + \lambda | (x, z) |^2)^{- (\ell + 2)/2}  (x, z), \]
  we get
  \begin{equation}
    | \nabla_{\mathbb{R}^4} \rho_{\ell, \lambda}^{(4)} (x,z) | \leqslant \delta |
    \rho_{\ell, \lambda}^{(4)} (x,z) | . \label{eq:rholambda}
  \end{equation}
  Using the fundamental theorem of calculus, we get
  \[ | \nabla_{\mathbb{R}^2 \times \varepsilon \mathbb{Z}^2} \rho_{\ell,
     \lambda}^{(4)} (x, z) | \leqslant \delta | \rho_{\ell, \lambda}^{(4)} (x, z) | . \]
Putting all the previous inequalities together we obtain 
  \begin{eqnarray*}
    \int_{\mathbb{R}^2 \times \varepsilon \mathbb{Z}^2} \nabla_{\mathbb{R}^2
    \times \varepsilon \mathbb{Z}^2} (\rho_{\ell, \lambda}^{(4)}) \psi
    \nabla_{\mathbb{R}^2 \times \varepsilon \mathbb{Z}^2} \psi \mathd x \mathd
    z & \leqslant & \delta \left( \int_{\mathbb{R}^2 \times \varepsilon
    \mathbb{Z}^2} \rho_{\ell, \lambda}^{(4)} (\psi^2 + | \nabla_{\mathbb{R}^2 \times
    \varepsilon \mathbb{Z}^2} \psi |^2) \mathd x \mathd z \right).
  \end{eqnarray*}      
For the second term in the sum \eqref{eq:nabla1} we get
  \begin{eqnarray*}
    \int \ll \rho_{\ell, \lambda}^{(4)}, \nabla_{\mathbb{R}^2 \times \varepsilon
    \mathbb{Z}^2} \psi \gg \nabla_{\mathbb{R}^2 \times \varepsilon
    \mathbb{Z}^2} \psi \mathrm{d}x \mathrm{d}z & = & \int_{\mathbb{R}^2 \times \varepsilon
    \mathbb{Z}^2} \rho_{\ell, \lambda}^{(4)} (x, z) (| \partial_{x^1} \psi (x, z)
    |^2 + | \partial_{x^2} \psi (x, z) |^2) \mathd x \mathd z\\
    &  & + \int_{\mathbb{R}^2 \times \varepsilon \mathbb{Z}^2} \rho_{\ell,
    \lambda}^{(4)} (x, z + \varepsilon e_1) | \varepsilon^{- 1} D_{\varepsilon e_1}
    \psi (x, z) |^2 \mathd x \mathd z\\
    &  & + \int_{\mathbb{R}^2 \times \varepsilon \mathbb{Z}^2} \rho_{\ell,
    \lambda}^{(4)} (x, z + \varepsilon e_2) | \varepsilon^{- 1} D_{\varepsilon e_2}
    \psi (x, z) |^2 \mathd x \mathd z\\
    & \geqslant & \frac{1}{2} \int_{\mathbb{R}^2 \times \varepsilon
    \mathbb{Z}^2} \rho_{\ell, \lambda}^{(4)} | \nabla_{\mathbb{R}^2 \times
    \varepsilon \mathbb{Z}^2} \psi |^2 \mathd x \mathd z.
  \end{eqnarray*}
Inserting the previous estimates in \eqref{eq:esppart}, choosing $\delta$ small enough, and using the fact that $\rho_{\ell}^{(4)} \lesssim_{\lambda} \rho^{(4)}_{\ell,\lambda} \lesssim_{\lambda} \rho^{(4)}_{\ell}$ (for any fixed $\lambda >0$), we can conclude. 
\end{proof}

\begin{theorem}
  \label{theorem_apriori2} Suppose that $\eta_{\pm}, \psi$ satisfy the
  hypotheses of Theorem \ref{theorem_apriori1} and suppose that $\eta_{\pm}
  \in B^{- 1 + \kappa}_{2, 2, \ell} (\mathbb{R}^2 \times \varepsilon
  \mathbb{Z}^2)$ where $\kappa \in (0, 1)$, then for any $\beta \in (- 1, 1)$
  we have
  \begin{align*}&\| \nabla_{\mathbb{R}^2 \times \varepsilon \mathbb{Z}^2} (e^{\pm \beta
     \alpha \psi}) \|^2_{L^2_{\ell / 2}} + \left\| \sqrt{| \psi |}_{} e^{\pm
     \beta \alpha \psi} \right\|_{L^2_{\ell / 2}}^2 + \int \rho_{\ell}^{(4)}
     e^{\pm (1 + 2 \beta) \alpha \psi} d \eta_{\pm} \\
   \lesssim& P_{\beta} \left( \int_{\mathbb{R}^2 \times \varepsilon
     \mathbb{Z}^2} \rho_{\ell}^{(4)} \mathd \eta_{\pm}, \| \eta_{\pm} \|_{B^{- 1 + \kappa}_{2, 2, \ell
     / 2} (\mathbb{R}^2 \times \varepsilon \mathbb{Z}^2)} \right), \end{align*}
     
  where $P_{\beta}$ is a suitable polynomial depending on $\beta$ and the
  constant implicit in the symbol $\lesssim$ are independent of $\varepsilon$.
\end{theorem}

Before proving the previous result we prove the following lemma.

\begin{lemma}
  \label{lemma:differenceexp}Consider a measurable function $\psi :
  \mathbb{R}^2 \times \varepsilon \mathbb{Z}^2 \rightarrow \mathbb{R}$, and
  consider $\gamma > 0$, then there exists a constant $C > 0$ independent of
  $\varepsilon$ such that \
  \begin{equation}
    | D_{\varepsilon e_i} e^{\pm (\gamma / 2)\psi} |^2 \leq \pm C (D_{\varepsilon
    e_i} \psi) D_{\varepsilon e_i} e^{\pm \gamma \psi},
    \label{inequality-grad-exp}
  \end{equation}
  where $\{ e_1, e_2 \}$ is the standard basis of $\mathbb{Z}^2$.
\end{lemma}

\begin{proof}
  We have to show that, for any $(x, z) \in \mathbb{R}^2 \times \varepsilon
  \mathbb{Z}^2$, we have, for a suitable $C > 0$,
  \begin{equation}
    \begin{array}{rl}
      & \left( e^{\pm \frac{\gamma}{2} (\psi (x, z + \varepsilon e_i) -
      \psi (x, z))} - 1 \right) ^2 e^{\pm \gamma \psi (x, z)}\\
      \leqslant& C e^{\pm \gamma \psi (x, z)} (e^{\pm \gamma (\psi (x, z +
      \varepsilon e_i) - \psi (x, z))} - 1) (\psi (x, z + \varepsilon e_i) -
      \psi (x, z)) .
    \end{array} \label{eq:diffexp}
  \end{equation}
  Dividing the inequality {\eqref{eq:diffexp}} by $e^{\pm \gamma \psi (x,
  z)}$, {\eqref{eq:diffexp}} is equivalent to \
  \[  \left( e^{\pm \frac{\gamma}{2} (\psi (x, z + \varepsilon e_i) -
     \psi (x, z))} - 1 \right) ^2 \leqslant C (e^{\pm \gamma (\psi (x,
     z + \varepsilon e_i) - \psi (x, z))} - 1) (\psi (x, z + \varepsilon e_i)
     - \psi (x, z)). \]
  So proving inequality {\eqref{eq:diffexp}} is equivalent to show that
  \begin{equation}
    \frac{1}{C} \leqslant \frac{(e^{\pm \gamma (\psi (x, z + \varepsilon e_i)
    - \psi (x, z))} - 1) (\psi (x, z + \varepsilon e_i) - \psi (x, z))}{\left(
    e^{\pm \frac{\gamma}{2} (\psi (x, z + \varepsilon e_i) - \psi (x, z))} - 1
    \right)^2} ,\label{eq:diffexp2}
  \end{equation}
  for some $C > 0$. In order to prove this, if we denote by $g = \psi (x, z +
  \varepsilon e_i) - \psi (x, z)$, inequality {\eqref{eq:diffexp2}} is
  equivalent to what follows \ \ \
  \[ \frac{1}{C} \leqslant \frac{(e^{\pm \gamma g} - 1) g}{\left( e^{\pm
     \frac{\gamma}{2} (g)} - 1 \right)^2} = \frac{\left( e^{\pm
     \frac{\gamma}{2} g} - 1 \right) \left( e^{\pm \frac{\gamma}{2} g} + 1
     \right) g}{\left( e^{\pm \frac{\gamma}{2} g} - 1 \right)^2} =
     \frac{\left( e^{\pm \frac{\gamma}{2} g} + 1 \right)}{\left( e^{\pm
     \frac{\gamma}{2} (g)} - 1 \right)} g = \coth \left( \frac{\gamma g}{4}
     \right) g = \frac{g}{\tanh \left( \frac{\gamma g}{4} \right)}. \]
  Now it is not hard to see that $\lim_{g \rightarrow \pm \infty} \coth \left(
  \frac{\gamma g}{4} \right) g = \infty$ and $\coth \left( \frac{\gamma g}{4}
  \right) g > 0$ if $g \neq 0$. By de L'Hopital rule we get
  \[ \lim_{g \rightarrow 0} \frac{g}{\tanh \left( \frac{\gamma g}{4} \right)}
     = \lim_{g \rightarrow 0} \frac{4}{\gamma \tanh' \left( \frac{\gamma g}{4}
     \right)} = \lim_{g \rightarrow 0} \frac{4}{\gamma \left( 1 - \tanh^2
     \left( \frac{\gamma g}{4} \right) \right)} = \frac{4}{\gamma}. \]
  Now by continuity of $\coth \left( \frac{\gamma g}{4} \right) g$ we can
  obtain inequality {\eqref{eq:diffexp2}} and thus the thesis.
\end{proof}

{\color{blue} }
\begin{proof*}{Proof of Theorem \ref{theorem_apriori2}}
  The proof can be done by induction on $|\beta| \leqslant \frac{2^k - 1}{2^k}$.
  For $k = 1$ (and $\beta = \frac{1}{2}$), using the notations of the proof of
  Theorem \ref{theorem_apriori1}, we multiply equation {\eqref{eq:apriori}} by
  $\pm \rho_{\ell, \lambda} e^{\pm \alpha \psi}$,  
  \[ \mp \int (\Delta_{\mathbb{R}^2 \times \varepsilon \mathbb{Z}^2} \psi)
     \rho_{\ell, \lambda}^{(4)} e^{\pm \alpha \psi} \mathd x \mathd z \pm m^2 \int
     \rho_{\ell, \lambda}^{(4)} \psi e^{\pm \alpha \psi} \mathd x \mathd z + \int
     \rho_{\ell, \lambda}^{(4)} e^{\pm 2\alpha \psi} \mathd \eta_{\pm} - \int
     \rho_{\ell, \lambda}^{(4)} \mathd \eta_{\mp} =0. \]
  Now integrating by parts on the first term we obtain
  \begin{eqnarray*}
    &  &\int (\Delta_{\mathbb{R}^2 \times \varepsilon \mathbb{Z}^2} \psi)
    \rho_{\ell, \lambda}^{(4)} e^{- \alpha \psi} \mathd x \mathd z\\
    & = & - \int \nabla_{\mathbb{R}^2 \times \varepsilon \mathbb{Z}^2} \psi
    \nabla_{\mathbb{R}^2 \times \varepsilon \mathbb{Z}^2} (\rho_{\ell,
    \lambda}^{(4)} e^{- \alpha \psi}) \mathd x \mathd z\\
    & = & -\int \nabla_{\mathbb{R}^2 \times \varepsilon \mathbb{Z}^2} (\psi)
    \ll \rho_{\ell, \lambda}^{(4)}, \nabla_{\mathbb{R}^2 \times \varepsilon
    \mathbb{Z}^2} e^{- \alpha \psi} \gg \mathd x \mathd z - \int
    \nabla_{\mathbb{R}^2 \times \varepsilon \mathbb{Z}^2} (\psi)
    \nabla_{\mathbb{R}^2 \times \varepsilon \mathbb{Z}^2} (\rho_{\ell,
    \lambda}^{(4)}) e^{- \alpha \psi} \mathd x \mathd z
  \end{eqnarray*}
  \[ \  \]
  with, similarly to the proof of Theorem \ref{theorem_apriori1},
  \[ \ll \rho_{\ell,\lambda}^{(4)}, \nabla_{\mathbb{R}^2 \times \varepsilon \mathbb{Z}^2} e^{-
     \alpha \psi} \gg (x, z) = \left( \begin{array}{c}
       \rho_{\ell, \lambda}^{(4)} (x, z) \partial_{x^1} e^{- \alpha \psi (x, z)}\\
       \rho_{\ell, \lambda}^{(4)} (x, z) \partial_{x^2} e^{- \alpha \psi (x, z)}\\
       \rho_{\ell, \lambda}^{(4)} (x, z + \varepsilon e_1) \varepsilon^{- 1}
       D_{\varepsilon e_1} e^{- \alpha \psi (x, z)}\\
       \rho_{\ell, \lambda}^{(4)} (x, z + \varepsilon e_2) \varepsilon^{- 1}
       D_{\varepsilon e_2} e^{- \alpha \psi (x, z)}
     \end{array} \right). \]
  Computing the scalar product we have
  \begin{align*}
&\nabla_{\mathbb{R}^2 \times \varepsilon \mathbb{Z}^2} \psi \ll
       \rho_{\ell, \lambda}^{(4)}, \nabla_{\mathbb{R}^2 \times \varepsilon
       \mathbb{Z}^2} e^{- \alpha \psi} \gg (x, z) 
\\
       =& \rho_{\ell, \lambda}^{(4)} (x, z) \partial_{x^1} e^{- \alpha \psi (x, z)}
       \partial_{x^1} \psi (x, z) + \rho_{\ell, \lambda}^{(4)} (x, z) \partial_{x^2}
       e^{- \alpha \psi (x, z)} \partial_{x^2} \psi (x, z)\\
       &+ \varepsilon^{- 2} \rho_{\ell, \lambda}^{(4)} (x, z + \varepsilon e_1)
       D_{\varepsilon e_1} e^{- \alpha \psi (x, z)} D_{\varepsilon e_1} \psi
       (x, z) \\
       &+ \varepsilon^{- 2} \rho_{\ell, \lambda} (x, z + \varepsilon e_2)
       D_{\varepsilon e_2} e^{- \alpha \psi (x, z)} D_{\varepsilon e_2} \psi
       (x, z) \\
       =& -\frac{4}{\alpha} \rho_{\ell, \lambda}^{(4)} (x, z) (\partial_{x^1} e^{-
       \alpha / 2 \psi (x, z)})^2 - \frac{4}{\alpha} \rho_{\ell, \lambda}^{(4)} (x,
       z) (\partial_{x^2} e^{- \alpha / 2 \psi (x, z)})^2 \\
       &+ \varepsilon^{- 2} \rho_{\ell, \lambda}^{(4)} (x, z + \varepsilon e_1)
       D_{\varepsilon e_1} e^{- \alpha \psi (x, z)} D_{\varepsilon e_1} \psi
       (x, z) \\
       &+ \varepsilon^{- 2} \rho_{\ell, \lambda}^{(4)}(x, z + \varepsilon e_2)
       D_{\varepsilon e_2} e^{- \alpha \psi (x, z)} D_{\varepsilon e_2} \psi
       (x, z) .
\end{align*}
  In analogous way as in the proof of Theorem \ref{theorem_apriori1} and the
  result of Lemma \ref{lemma:differenceexp} fixing $\delta > 0$ we can choose
  $\lambda$ small enough such that we have
  \[ -\nabla_{\mathbb{R}^2 \times \varepsilon \mathbb{Z}^2} \psi \ll
     \rho_{\ell, \lambda}^{(4)}, \nabla_{\mathbb{R}^2 \times \varepsilon
     \mathbb{Z}^2} e^{- \alpha \psi} \gg \geqslant \;  \frac{1}{C} \rho_{\ell,
     \lambda}^{(4)} \left| \nabla_{\mathbb{R}^2 \times \varepsilon \mathbb{Z}^2}
     e^{- \frac{\alpha}{2} \psi} \right|^2 . \]
  Now we have to estimate
  \begin{eqnarray*}
    \int \nabla_{\mathbb{R}^2 \times \varepsilon \mathbb{Z}^2} (\psi)
    (\nabla_{\mathbb{R}^2 \times \varepsilon \mathbb{Z}^2} \rho_{\ell,
    \lambda}^{(4)}) e^{- \alpha \psi} \mathd x \mathd z & = & \int \partial_{x^1}
    \rho_{\ell, \lambda}^{(4)} \partial_{x^1} \psi e^{- \alpha \psi} \mathd x \mathd
    z + \int \partial_{x^2} \rho_{\ell, \lambda}^{(4)} \partial_{x^2} \psi e^{-
    \alpha \psi} \mathd x \mathd z\\
    &  & + \varepsilon^{- 2} \int D_{\varepsilon e_1} \rho_{\ell, \lambda}^{(4)}
    D_{\varepsilon e_1} \psi e^{- \alpha \psi} \mathd x \mathd z \\
    &  & + \varepsilon^{- 2} \int D_{\varepsilon e_2} \rho_{\ell, \lambda}^{(4)}
    D_{\varepsilon e_2} \psi e^{- \alpha \psi} \mathd x \mathd z.
  \end{eqnarray*}
  For the first term we apply chain rule to get, for $\alpha > 0$,
  \begin{align*} \int | \partial_{x^i} \rho_{\ell, \lambda}^{(4)} \partial_{x^i} \psi e^{-
     \alpha \psi} | \mathd x \mathd z =& \frac{2}{\alpha} \int | \partial_{x^i}
 \rho_{\ell, \lambda}^{(4)}|\cdot| e^{- \alpha \psi / 2} | \cdot |\partial_{x^i} e^{- \alpha
     \psi / 2} | \mathd x \mathd z \\
     \leqslant& \frac{\delta}{\alpha} \left(\left\|
     \rho_{\ell/2, \lambda}^{(4)} \partial_{x^i} e^{- \frac{\alpha}{2} \psi}
     \right\|_{L^2}^2 + \left\|
     \rho_{\ell/2, \lambda}^{(4)}  e^{- \frac{\alpha}{2} \psi}
     \right\|_{L^2}^2 \right),  \end{align*}
  for $i \in \{ 1, 2 \}$. Now for the 3rd and 4th term we observe that
  denoting by $\varepsilon^{- 1} D_{\varepsilon e_1} \rho_{\ell, \lambda}^{(4)} (x,
  z) = \varepsilon^{- 1} (\rho_{\ell, \lambda}^{(4)} (x, z + \varepsilon e_1) -
  \rho_{\ell, \lambda}^{(4)} (x, z))$ we have that $\varepsilon^{- 1} D_{\varepsilon
  e_1} \rho_{\ell, \lambda} ^{(4)}\leqslant \delta \rho_{\ell, \lambda}^{(4)}$.\\
  
  Observing that for any $g \in \mathbb{R}$
  \[ | g | \lesssim | e^{- g} - 1 | | g | + | e^{- g / 2} - 1 |, \]
  we get, for some $C > 0$,
  \begin{align*}
       &| (\psi (x, z + \varepsilon e_i) - \psi (x, z)) e^{- \alpha \psi (x,
       z)} |\\
       \leqslant& C | \psi (x, z + \varepsilon e_i) - \psi (x, z) | | e^{-
       \alpha \psi (x, z)} (e^{- \alpha (\psi (x, z + \varepsilon e_i) - \psi
       (x, z))} - 1) | \\
       &+ C | e^{- \alpha \psi (x, z)} (e^{- \alpha (\psi (x, z + \varepsilon
       e_i) - \psi (x, z)) / 2} - 1) |\\
       \leqslant& C D_{\varepsilon e_i} \psi D_{\varepsilon e_i} e^{- \alpha
       \psi} + C e^{- \alpha \psi / 2} D_{\varepsilon e_i} e^{- \alpha \psi /
       2}.
\end{align*}
  From this we have
 \begin{align*}
       &\varepsilon^{- 2} \int | D_{\varepsilon e_i} (\rho_{\ell, \lambda}^{(4)}) (x,
       z) | | D_{\varepsilon e_i} (\psi) (x, z) | e^{- \alpha \psi (x, z)}
       \mathd x \mathd z \\
       \leqslant& C \delta \varepsilon^{- 2} \int \rho_{\ell, \lambda}^{(4)} (x, z +
       \varepsilon e_i) D_{\varepsilon e_i} \psi D_{\varepsilon e_i} e^{-
       \alpha \psi} + \varepsilon^{- 2} C \delta \int \rho_{\ell, \lambda}^{(4)} |
       D_{\varepsilon e_i} e^{- \alpha \psi / 2} |^2 \mathd x \mathd z +
       \delta \int \rho_{\ell, \lambda}^{(4)} C e^{- \alpha \psi} \mathd x \mathd z.
 \end{align*}
  By Lemma \ref{lemma:differenceexp}, the first two terms, on the right hand
  side of the previous inequality, can be absorbed into the positive term
  $\int \nabla_{\mathbb{R}^2 \times \varepsilon \mathbb{Z}^2} \psi \ll
  \rho_{\ell, \lambda}^{(4)}, \nabla_{\mathbb{R}^2 \times \varepsilon \mathbb{Z}^2}
  e^{- \alpha \bar{\phi}_n} \gg$. For the third term we estimate
  \[ C \int \rho_{\ell, \lambda}^{(4)} e^{- \alpha \psi} \mathd x \mathd z
     \leqslant C + \int \rho_{\ell, \lambda}^{(4)} | \psi | e^{- \alpha \psi} \mathd
     x \mathd z, \]
  which was estimated in Theorem \ref{theorem_apriori1} by $(\rho_{\ell}^{(4)}
  \mathd \eta_+) (\mathbb{R}^2 \times \varepsilon \mathbb{Z}^2) +
  (\rho_{\ell}^{(4)} \mathd \eta_-) (\mathbb{R}^2 \times \varepsilon
  \mathbb{Z}^2)$. This proves the theorem for $\beta \leq \frac{1}{2} .$\\
  
  Suppose that we prove the theorem for $\beta \leq \frac{2^k - 1}{2^k}$, $k
  \in \mathbb{N}$, and consider $\beta_{k + 1} = \frac{2^{k + 1} - 1}{2^{k +
  1}}$. \ Multiplying both sides by $\rho e^{\pm 2 \beta_{k + 1} \alpha
  \bar{\phi}}$ and repeating the previous estimates we obtain
  \begin{align*} &\left\| \rho_{\ell/2}^{(4)} \nabla_{\mathbb{R}^2 \times \varepsilon
     \mathbb{Z}^2} e^{\pm \beta_{k + 1} \alpha \psi} \right\|^2_{L^2} + \left\|
     \rho_{\ell/2}^{(4)} \sqrt{\psi} e^{\pm \beta_{k + 1} \alpha \psi}
     \right\|_{L^2}^2 + \int \rho_{\ell}^{(4)} e^{\pm (1 + 2 \beta_k) \alpha \psi} d
     \eta_{\pm}\\
   \lesssim &\int \rho^{(4)}_{\ell} e^{- \beta_k \alpha \psi} \mathd \eta_+ + \int \rho^{(4)}_{\ell}
     e^{+ \beta_k \alpha \psi} \mathd \eta_- . \end{align*}
  On the other hand by
  \[ \int \rho_{\ell}^{(4)} e^{\mp \beta_k \alpha \psi} \mathd \eta_{\pm} \leq \|
     e^{\mp \beta_k \alpha \psi} \|_{B^{1 - \kappa_{\alpha} + \varepsilon}_{2,
     2, \ell / 2} (\mathbb{R}^2 \times \varepsilon \mathbb{Z}^2)} \|
     \eta_{\pm} \|_{B^{- 1 + \kappa_{\alpha}}_{2, 2, \ell / 2} (\mathbb{R}^2
     \times \varepsilon \mathbb{Z}^2)}, \]
  and using the induction hypothesis we obtain the thesis.
\end{proof*}

\subsection{Existence and uniqueness of solutions to equation
{\eqref{eq:main5}}}

In this section we prove the existence and uniqueness of the solution to
equation {\eqref{eq:main5}}. We start with a results on the boundedness of the
measures $\mu^{\alpha}_{\varepsilon, N}$, $\mu^{\alpha}_{\varepsilon}$,
$\bar{\mu}^{\alpha}_{\varepsilon}$ and $\mu^{\alpha}$ with respect to
$\varepsilon$ and $N$ in suitable Besov spaces.

\begin{proposition}
  \label{proposition:convergencealpha}Fix $| \alpha | < 4 \pi$ and $0 < \delta
  < 1 - \frac{\alpha^2}{(4 \pi)^2}$, then there are two sequences
  $\varepsilon_n = 2^{- k_n} \rightarrow 0$ (for some $k_n \in
  \mathbb{N}$) and $N_r \rightarrow \infty$ such that for any $\ell \geqslant 4$ we
  have
  \begin{equation}\label{eq:supepsilon}
      \sup_{\varepsilon_n, N_r} (\| \mu_{\varepsilon_n, N_r}^{\pm \alpha}
     \|_{B^{- 1 + \delta}_{2, 2, \ell} (\mathbb{R}^2 \times \varepsilon_n
     \mathbb{Z}^2)}, \| \mu_{\varepsilon_n}^{\pm \alpha} \|_{B^{-1 +
     \delta}_{2, 2, \ell} (\mathbb{R}^2 \times \varepsilon_n \mathbb{Z}^2)},
     \| \bar{\mu}^{\pm \alpha}_{\varepsilon_n} \|_{B^{- 1 + \delta}_{2, 2,
     \ell} (\mathbb{R}^4)}, \| \mu^{\pm \alpha} \|_{B^{- 1 + \delta}_{2, 2,
     \ell} (\mathbb{R}^4)}) < + \infty 
    \end{equation} 
  almost surely. Furthermore $\mu_{\varepsilon_n, N_r}^{\pm \alpha}
  \rightarrow \mu^{\pm \alpha}_{\varepsilon_n}$ and $\bar{\mu}^{\pm
  \alpha}_{\varepsilon_n} = \overline{\mathcal{E}}_{\varepsilon_n} (\mu^{\pm
  \alpha}_{\varepsilon_n}) \rightarrow \mu^{\pm \alpha}$ almost surely in
  $B^{- 1 + \delta}_{2, 2, \ell} (\mathbb{R}^2 \times \varepsilon_n
  \mathbb{Z}^2)$ and in $B^{- 1 + \delta}_{2, 2, \ell} (\mathbb{R}^4)$
  respectively.
\end{proposition}

\begin{proof}
  From Theorem \ref{theorem:convergencebar} and Theorem
  \ref{theorem_convergenceepsilon} we have that, as $\varepsilon \rightarrow
  0$, $\mathbb{N}^s_{p, \ell} (\bar{\mu}_{\varepsilon}^{\alpha}, a, \beta)
  \rightarrow \mathbb{N}^s_{p, \ell} (\mu^{\alpha}, a, \beta)$ and
  $\mathbb{N}^s_{p, \ell, \varepsilon} (\mu_{\varepsilon}^{\alpha}, a, \beta)
  \rightarrow \mathbb{N}^s_{p, \ell} (\mu^{\alpha}, a, \beta)$ in $L^p
  (\Omega)$ for any $a \in \mathbb{R}_+$, $\beta \in (0, 1)$, $p \geqslant 2$,
  $\frac{\alpha^2}{(4 \pi)^2} (p - 1) < 2$ and $s < - \frac{\alpha^2}{(4
  \pi)^2} (p - 1)$. We can choose $\bar{\beta} \in (0, 1)$ and $\bar{a} \in
  \mathbb{R}_+$ such that $| K_i | \lesssim \exp (- \bar{a} (| x | + | z
  |)^{\bar{\beta}})$ for any $i \geqslant - 1$, and a subsequence
  $\varepsilon_n = 2^{- k_n}$ (where $k_n \in \mathbb{N}$) such that
  $\mathbb{N}^{- 1 + \delta}_{2, \ell} (\bar{\mu}_{\varepsilon_n}^{\pm
  \alpha}, \bar{a}, \bar{\beta}) \rightarrow \mathbb{N}^{- 1 + \delta}_{2,
  \ell} (\mu^{\pm \alpha}, \bar{a}, \bar{\beta})$ and $\mathbb{N}^{- 1 +
  \delta}_{2, \ell, \varepsilon_n} (\mu_{\varepsilon_n}^{\pm \alpha}, \bar{a},
  \bar{\beta}) \rightarrow \mathbb{N}^{- 1 + \delta}_{2, \ell} (\mu^{\pm
  \alpha}, \bar{a}, \bar{\beta})$ almost surely. Since, by Theorem
  \ref{theorem:norminequalities}, we have
  \begin{equation}
    \| \mu^{\pm \alpha}_{\varepsilon_n} \|_{B^{-1 + \delta}_{2, 2, \ell}
    (\mathbb{R}^2 \times \varepsilon_n \mathbb{Z}^2)} \lesssim \mathbb{N}^{- 1
    + \delta}_{2, \ell, \varepsilon_n} (\mu_{\varepsilon_n}^{\pm \alpha},
    \bar{a}, \bar{\beta}) \label{eq:Minequality1}
  \end{equation}
  \begin{equation}
    \| \bar{\mu}^{\pm \alpha}_{\varepsilon_n} \|_{B^{- 1 + \delta}_{2, 2,
    \ell} (\mathbb{R}^4)} \lesssim \mathbb{N}^{- 1 + \delta}_{2, \ell}
    (\bar{\mu}^{\pm \alpha}_{\varepsilon_n}, \bar{a}, \bar{\beta}), \quad \|
    \mu^{\pm \alpha} \|_{B^{- 1 + \delta}_{2, 2, \ell} (\mathbb{R}^4)}
    \lesssim \mathbb{N}^{- 1 + \delta}_{2, \ell} (\mu^{\pm \alpha}, \bar{a},
    \bar{\beta}), \label{eq:Minequality2}
  \end{equation}
  the sequence $\bar{\mu}^{\pm \alpha}_{\varepsilon_n}$ is a bounded in $L^2
  (\Omega, B^{-1 + \delta}_{2, 2, \ell} (\mathbb{R}^4))$, which, by the
  Banach-Alaoglu theorem is weakly precompact. Let $\kappa \in L^2 (\Omega,
  B^{-1 + \delta}_{2, 2, \ell} (\mathbb{R}^4))$ be a weak limit of some
  (convergent) subsequence $\bar{\mu}^{\alpha}_{\varepsilon_{n_k}}$ of
  $\bar{\mu}^{\alpha}_{\varepsilon_n}$. By Remark
  \ref{remark:convergenceintegral} we have that, for any $f \in \mathcal{S}
  (\mathbb{R}^4)$ and $A \in \mathcal{F}$ (where $\mathcal{F}$ is the
  $\sigma$-algebra in the underlying probability space $(\Omega, \mathcal{F},
  \mathbb{P})$), we have that
  \[ \mathbb{I}_A (\omega) \int_{\mathbb{R}^4} f (x, z)
     \bar{\mu}^{\alpha}_{\varepsilon_{n_k}} (\mathd x, \mathd z) \rightarrow
     \mathbb{I}_A (\omega) \int_{\mathbb{R}^4} f (x, z) \mu^{\alpha} (\mathd
     x, \mathd z) \]
  in $L^p (\Omega)$, which implies that $\mathbb{I}_A (\omega)
  \int_{\mathbb{R}^4} f (x, z) \mu^{\alpha} (\mathd x, \mathd z) =\mathbb{I}_A
  (\omega) \int_{\mathbb{R}^4} f (x, z) \kappa (\mathd x, \mathd z)$ for some
  $\kappa \in L^2 (\Omega, B^{-1 + \delta}_{2, 2, \ell} (\mathbb{R}^4))$. Since
  the linear combinations of the elements of the form $\mathbb{I}_A (\cdot)
  \int_{\mathbb{R}^4} f (x, z) \cdot \mathrm{d}x \mathrm{d}z  \in (L^2 (\Omega, B^{-1 + \delta}_{2, 2,
  \ell} (\mathbb{R}^4)))^{\ast}$, for generic $f \in \mathcal{S}
  (\mathbb{R}^4)$ and $A \in \mathcal{F}$, is dense in $(L^2 (\Omega, B^{-1 +
  \delta}_{2, 2, \ell} (\mathbb{R}^4)))^{\ast}$ we get that for the weak limit
  $\kappa$ the equality $\kappa = \mu^{\alpha}$ holds, and thus the whole
  sequence $\bar{\mu}^{\alpha}_{\varepsilon_n}$ weakly converges to
  $\mu^{\alpha}$. In a similar way it is possible to prove that $\bar{\mu}^{-
  \alpha}_{\varepsilon_n}$ converges to $\mu^{- \alpha}$.\\
  
  Inequalities {\eqref{eq:Minequality1}} and {\eqref{eq:Minequality2}} imply
  that $\sup_{\varepsilon_n} (\| \mu_{\varepsilon_n}^{\pm \alpha} \|_{B^{- 1 +
  \delta}_{2, 2, \ell} (\mathbb{R}^2 \times \varepsilon_n \mathbb{Z}^2)}, \|
  \bar{\mu}^{\pm \alpha}_{\varepsilon_n} \|_{B^{- 1 + \delta}_{2, 2, \ell}
  (\mathbb{R}^4)}, \| \mu^{\pm \alpha} \|_{B^{- 1 + \delta}_{2, 2, \ell}
  (\mathbb{R}^4)}) < + \infty$. Thus, by the extended Lebesgue dominated
  convergence theorem (see, e.g. Theorem 2.3.11 of {\cite{Athreya2006}}), we
  have that $\mathbb{E} [\| \bar{\mu}_{\varepsilon_n}^{\pm \alpha} \|_{B^{-1 +
  \delta}_{2, 2, \ell} (\mathbb{R}^4)}^2]$ converges to $\mathbb{E} [\|
  \mu^{\pm \alpha} \|_{B^{- 1 + \delta}_{2, 2, \ell} (\mathbb{R}^4)}^2]$.
  Since $L^2 (\Omega, B^{-1 + \delta}_{2, 2, \ell} (\mathbb{R}^4)^{})$ is a
  Hilbert space (more generically a uniformly convex space), and since
  $\bar{\mu}^{\pm \alpha}_{\varepsilon_n}$ weakly converges to $\mu^{\pm
  \alpha}$, we have that the sequence $\bar{\mu}_{\varepsilon_n}^{\pm \alpha}$
  converges to $\mu^{\pm \alpha}$ in $L^2 (\Omega, B^{-1 + \delta}_{2, 2, \ell}
  (\mathbb{R}^4))$. By passing to a suitable subsequence we have that
  $\bar{\mu}_{\varepsilon_n}^{\pm \alpha}$ converges to $\mu^{\pm \alpha}$
  almost surely in $B^{-1 + \delta}_{2, 2, \ell} (\mathbb{R}^4)$. This proves
  the statements of the theorem involving $\bar{\mu}_{\varepsilon_n}^{\pm
  \alpha}$ and $\mu^{\pm \alpha}$.\\
  
  By applying an analogous reasoning, combined with a diagonal argument applied to the double sequence $\varepsilon_{n} \rightarrow 0$ and $N \rightarrow +\infty$, to $\mu_{\varepsilon_n, N_r}^{\pm
  \alpha}$ and $\mu^{\pm \alpha}_{\varepsilon_n}$, replacing Theorem  \ref{theorem:convergencebar} and Theorem \ref{theorem_convergenceepsilon} by
  Proposition \ref{proposition:convergenceperiodic} the thesis is proved; namely we find two sequences $\varepsilon_{n} \rightarrow 0$ and $N_{r} \rightarrow +\infty$ (independent of each other) such that inequality \eqref{eq:supepsilon} holds.
\end{proof}

\begin{remark}
  \label{remark:bound1}Since $\rho_{\ell}^{(4)} \in (B^{- 1 + \delta}_{2, 2,
  \ell / 2} (\mathbb{R}^4))^{\ast}, (B^{-1 + \delta}_{2, 2, \ell / 2}
  (\mathbb{R}^2 \times \varepsilon \mathbb{Z}^2))^{\ast}$ an important
  consequence of Proposition \ref{proposition:convergencealpha}, is that
  \[ \sup_{\varepsilon_n, N_r} \left( \int_{\mathbb{R}^2 \times \varepsilon_n
     \mathbb{Z}^2} \rho^{(4)}_{\ell} (x, z) \mathd \mu_{\varepsilon_n,
     N_r}^{\pm \alpha}, \int_{\mathbb{R}^2 \times \varepsilon_n \mathbb{Z}^2}
     \rho^{(4)}_{\ell} (x, z) \mathd \mu_{\varepsilon_n}^{\pm \alpha},
     \int_{\mathbb{R}^4} \rho^{(4)}_{\ell} (x, z) \mathd \bar{\mu}^{\pm
     \alpha}_{\varepsilon_n} \right) < + \infty \]
  almost surely.
\end{remark}

We can now prove the uniqueness of solution to equation {\eqref{eq:main5}}.

\begin{proposition}
  \label{proposition:uniqueness}Suppose that $\bar{\phi}$ and $\bar{\psi}$ are
  two solutions to equation {\eqref{eq:main5}} such that, there exists $\ell >
  4$ for which $\bar{\phi}, \bar{\psi}, e^{\pm \frac{\alpha}{2} \bar{\psi}},
  e^{\pm \frac{\alpha}{2} \bar{\phi}} \in H^1_{\ell} (\mathbb{R}^4)$ and
  $e^{\pm 2 \alpha \bar{\phi}}, e^{\pm 2 \alpha \bar{\psi}} \in L^1_{2 \ell}
  (\mathd \mu^{\pm \alpha})$ almost surely, then $\bar{\phi} = \bar{\psi}$
  almost surely.
\end{proposition}

\begin{proof}
  By Proposition \ref{proposition:convergencealpha}, $\mu^{\pm \alpha} \in
  B^{-1 + \delta}_{2, 2, \ell} (\mathbb{R}^4)$ almost surely for any $\ell >
  4$. Since, by hypothesis, $e^{\pm \frac{\alpha}{2}
  \bar{\psi}}, e^{\pm \frac{\alpha}{2} \bar{\phi}} \in H^1_{\ell}
  (\mathbb{R}^4)$, and using Proposition \ref{propisition_product} about products in Besov spaces, we get that $\cosh \left( \frac{\alpha \bar{\psi}}{2} \right),
  \cosh \left( \frac{\alpha \bar{\phi}}{2} \right) \in L^1_{2 \ell} (\mathd
  \mu^{\pm \alpha})$, and thus $\bar{\phi},\bar{\psi} \in L^p_{2\ell}(\mathd \mu^{\pm \alpha})$.
  Then, by Young inequality, we obtain
  \begin{equation}
    | \bar{\psi} | e^{\pm \alpha \bar{\phi}}, | \bar{\phi} | e^{\pm \alpha
    \bar{\phi}}, | \bar{\psi} | e^{\pm \alpha \bar{\psi}}, | \bar{\phi} |
    e^{\pm \alpha \bar{\phi}} \in L^1_{2 \ell} (\mathd \mu^{\pm \alpha}) .
    \label{eq:L1mualpha}
  \end{equation}
  Since $\bar{\phi}, \bar{\psi}$ are solution to {\eqref{eq:main5}}, then
  $\bar{\phi} - \bar{\psi}$ satisfies
  \[ (m^2 - \Delta) (\bar{\phi} - \bar{\psi}) + \alpha [\exp (\alpha
     \bar{\phi}) - \exp (\alpha \bar{\psi})] \mathd \mu^{\alpha} - \alpha
     [\exp (- \alpha \bar{\phi}) - \exp (- \alpha \bar{\psi})] \mathd \mu^{-
     \alpha} = 0. \]
  By the property {\eqref{eq:L1mualpha}}, we can test the previous equation 
  with $(\bar{\phi} - \bar{\psi}) \rho_{2 \ell, \lambda}^{(4)}$ (where $\rho_{2
  \ell, \lambda}^{(4)} (x, z) = \rho_{2 \ell}^{(4)} (\lambda x, \lambda z)$) and
  integrating by parts (since we have assumed that $(\bar{\phi} - \bar{\psi})$
  in $H^1_{\ell} (\mathbb{R}^4)$), and choosing $\lambda > 0$ small enough, we
  obtain
  \begin{equation}
    \begin{array}{lll}
      0 & = & m^2 \int_{\mathbb{R}^4} \rho_{2 \ell, \lambda}^{(4)} (x, z) |
      \bar{\phi} - \bar{\psi} |^2 \mathd x \mathd z + \int_{\mathbb{R}^4}
      \rho_{2 \ell, \lambda}^{(4)} (x, z) | \nabla (\bar{\phi} - \bar{\psi}) |^2
      \mathd x \mathd z\\
      &  & + \int_{\mathbb{R}^4} (\bar{\phi} - \bar{\psi}) \nabla (\rho_{2
      \ell, \lambda}^{(4)}) (x, z) \nabla (\bar{\phi} - \bar{\psi}) \mathd x \mathd
      z\\
      &  & + \alpha \int_{\mathbb{R}^4} (\bar{\phi} - \bar{\psi}) \rho_{2
      \ell, \lambda}^{(4)} (x, z) (\exp (\alpha \bar{\phi}) - \exp (\alpha
      \bar{\psi})) \mathd \mu^{\alpha} (\mathd x, \mathd z)\\
      &  & - \alpha \int_{\mathbb{R}^4} (\bar{\phi} - \bar{\psi}) \rho_{2
      \ell, \lambda}^{(4)} (x, z) (\exp (- \alpha \bar{\phi}) - \exp (- \alpha
      \bar{\psi})) \mathd \mu^{- \alpha} (\mathd x, \mathd z)\\
      & \geqslant & \frac{1}{2} \left( \int_{\mathbb{R}^4} \rho_{2 \ell,
      \lambda}^{(4)} (x, z) | \bar{\phi} - \bar{\psi} |^2 \mathd x \mathd z +
      \int_{\mathbb{R}^4} \rho_{2 \ell, \lambda}^{(4)} (x, z) | \nabla (\bar{\phi} -
      \bar{\psi}) |^2 \mathd x \mathd z \right),
    \end{array} \label{eq:uniqueness1}
  \end{equation}
  where in the last line we have used that
  \[ (\bar{\phi} - \bar{\psi}) (\exp (\alpha \bar{\phi}) - \exp (\alpha
     \bar{\psi})), (\bar{\phi} - \bar{\psi}) (\exp (- \alpha \bar{\psi}) -
     \exp (- \alpha \bar{\phi})) \geqslant 0, \]
  due to convexity of $\exp (- \alpha x), \exp (\alpha x)$ and the fact that
  $\nabla \rho_{2 \ell, \lambda}^{(4)} \leqslant \delta \rho^{(4)}_{2 \ell,\lambda}$ when $\lambda > 0$ is
  small enough, and thus
  \[ \int_{\mathbb{R}^4} | (\bar{\phi} - \bar{\psi}) \nabla \rho_{2 \ell,
     \lambda}^{(4)} \nabla (\bar{\phi} - \bar{\psi}) | \mathd x \mathd z \leqslant \delta \left(
     \int_{\mathbb{R}^4} \rho_{2 \ell, \lambda}^{(4)} (x, z) | \bar{\phi} -
     \bar{\psi} |^2 \mathd x \mathd z + \int_{\mathbb{R}^4} \rho_{2 \ell,
     \lambda}^{(4)} (x, z) | \nabla (\bar{\phi} - \bar{\psi}) |^2 \mathd x \mathd z
     \right). \]
  Inequality {\eqref{eq:uniqueness1}} implies the thesis.
\end{proof}

\begin{remark}
  \label{remark:uniqueness}Using a reasoning similar to the one in the proof
  of Proposition \ref{proposition:uniqueness}, we get that also equations
  {\eqref{eq:mainfinitedimensional}} and {\eqref{eq:maindiscrete}} admit a
  unique solution. In particular if $\bar{\phi}_{\varepsilon, N}$, and
  $\bar{\phi}_{\varepsilon}$ are solutions to equation
  {\eqref{eq:maindiscrete}} and {\eqref{eq:mainfinitedimensional}}
  respectively such that $\bar{\phi}_{\varepsilon, N},
  \bar{\phi}_{\varepsilon} \in H^1_{\ell} (\mathbb{R}^2 \times \varepsilon
  \mathbb{Z}^2)$, by Theorem \ref{theorem_apriori1} and Theorem
  \ref{theorem_apriori2}, we have $e^{\pm \frac{\alpha}{2}
  \bar{\phi}_{\varepsilon, N}}, e^{\pm \frac{\alpha}{2}
  \bar{\phi}_{\varepsilon}} \in H^1_{\ell} (\mathbb{R}^2 \times \varepsilon
  \mathbb{Z}^2)$, $e^{\pm 2 \alpha \bar{\phi}_{\varepsilon, N}} \in L^1
  (\rho^{(4)}_{2 \ell} \mathd \mu^{\pm \alpha}_{\varepsilon, N})$ and $e^{\pm
  2 \alpha \bar{\phi}_{\varepsilon}} \in L^1 (\rho^{(4)}_{2 \ell} \mathd
  \mu^{\pm \alpha}_{\varepsilon})$. This implies that
  $\bar{\phi}_{\varepsilon, N}$, and $\bar{\phi}_{\varepsilon}$ are the unique
  solutions to {\eqref{eq:mainfinitedimensional}} and
  {\eqref{eq:maindiscrete}} in $H^1_{\ell} (\mathbb{R}^2 \times \varepsilon
  \mathbb{Z}^2)$.
\end{remark}

In order to prove existence for equation {\eqref{eq:main5}} we need the next
lemma.

\begin{lemma}
  \label{lemma:existencefinitedim}The solution $\bar{\phi}_{\varepsilon, N}$
  to equation {\eqref{eq:mainfinitedimensional}} exists and is unique in
  $H^1_{\ell} (\mathbb{R}^2 \times \varepsilon \mathbb{Z}^2)$ (where $\ell >
  \ell_0$ for some $\ell_0 \in \mathbb{R}_+$) almost surely. Furthermore, $e^{\pm \beta \alpha \bar{\phi}_{\varepsilon, N}} \in H^1_{\ell}
  (\mathbb{R}^2 \times \varepsilon \mathbb{Z}^2)$ and $e^{\pm (1 + 2
  \beta) \alpha \bar{\phi}_{\varepsilon, N}} \in L^1 (\rho_{\ell}^{(4)} \mathd
  \mu^{\pm \alpha}_{\varepsilon, N})$ for any $0 \leqslant \beta < 1$.
\end{lemma}

\begin{proof}
  We consider the following equation
  \begin{equation}
    (- \Delta_{\mathbb{R}^2 \times \varepsilon \mathbb{Z}^2} + m^2)
    \bar{\phi}_{\varepsilon, N, K} + \alpha e^{\alpha \bar{\phi}_{\varepsilon,
    N, K}} f_K (x) \mu^{\alpha}_{\varepsilon, N} - \alpha e^{- \alpha
    \bar{\phi}_{\varepsilon, N, K}} f_K (x) \mu^{- \alpha}_{\varepsilon, N} =
    0 ,\label{eq:spacialcutoff}
  \end{equation}
  where $f_K (x) = \exp \left( - \left( \frac{x^2}{K} + 1 \right)^{1 / 2} + 1
  \right)$. Equation {\eqref{eq:spacialcutoff}} admits a unique solution (see
  Lemma 7 in {\cite{AlDeGu2018}}) in $L^{\infty} (\mathbb{R}^2 \times
  \varepsilon \mathbb{Z}^2) \cap \mathcal{C}^{2 - \delta} (\mathbb{R}^2,
  L^{\infty} (\varepsilon \mathbb{Z}^2))$ for any $\delta > 0$.\\
  
  By Lemma 48 in {\cite{AlDeGu2018}} (see also the discussion of Section 6 of
  {\cite{AlDeGu2018}}) the solution $\bar{\phi}_{\varepsilon, N, K}$ converges
  to $\bar{\phi}_{\varepsilon, N}$ in $L^{\infty}_{\exp (\beta)} (\mathbb{R}^2
  \times \varepsilon \mathbb{Z}^2) \cap C^{2 - \delta}_{\tmop{loc}}
  (\mathbb{R}^2, L^{\infty} (\varepsilon \mathbb{Z}^2))$ (where
  $L^{\infty}_{\exp (\beta)} (\mathbb{R}^2 \times \varepsilon \mathbb{Z}^2)$
  is the space of $L^{\infty}$ functions with an exponential weight, more
  precisely the norm of $L^{\infty}_{\exp (\beta)} (\mathbb{R}^2 \times
  \varepsilon \mathbb{Z}^2)$ is
  \[ \| f \|_{L^{\infty}_{\exp (\beta)} (\mathbb{R}^2 \times \varepsilon
     \mathbb{Z}^2)} = \sup_{(x, z) \in \mathbb{R}^2 \times \varepsilon
     \mathbb{Z}^2} | \exp (- \beta (| x | + | z |)) f (x, z) |, \quad f \in
     L^{\infty}_{\exp (\beta)} (\mathbb{R}^2 \times \varepsilon \mathbb{Z}^2),
  \]
  where $\beta \in \mathbb{R}_+$). Furthermore, since
  $\bar{\phi}_{\varepsilon, N, K} \in L^{\infty} (\mathbb{R}^2 \times
  \varepsilon \mathbb{Z}^2) \cap \mathcal{C}^{2 - \delta} (\mathbb{R}^2,
  L^{\infty} (\varepsilon \mathbb{Z}^2))$, then $\bar{\phi}_{\varepsilon, N,
  K} \in H^1_{\ell} (\mathbb{R}^2 \times \varepsilon \mathbb{Z}^2)$ for any
  $\ell \geqslant 4$, and so by Theorem \ref{theorem_apriori1} and Theorem
  \ref{theorem_apriori2} and Proposition \ref{proposition:convergencealpha} we
  have that
  \[ \| \bar{\phi}_{\varepsilon, N, K} \|_{H^1_{\ell} (\mathbb{R}^2 \times
     \varepsilon \mathbb{Z}^2)}, \| e^{\pm \beta \alpha
     \bar{\phi}_{\varepsilon, N, K}} \|_{H^1_{\ell} (\mathbb{R}^2 \times
     \varepsilon \mathbb{Z}^2)}, \| e^{\pm (1 + 2 \beta)\alpha
     \bar{\phi}_{\varepsilon, N, K}} \|_{L^1 (\rho_{2 \ell}^{(4)} \mathd
     \mu^{\pm \alpha}_{\varepsilon, N})} \lesssim 1, \]
  where the constants, hidden in the symbol $\lesssim$, do not depend on $K$.
  Since $\bar{\phi}_{\varepsilon, N, K}$ converges strongly to
  $\bar{\phi}_{\varepsilon, N}$, this implies that $\bar{\phi}_{\varepsilon,
  N, K}$ converges weakly to $\bar{\phi}_{\varepsilon, N}$ in $H^1_{\ell}
  (\mathbb{R}^2 \times \varepsilon \mathbb{Z}^2)$. By lower semicontinuity of
  the norm $\| \cdot \|_{H^1_{\ell} (\mathbb{R}^2 \times \varepsilon
  \mathbb{Z}^2)}$ with respect to the weak convergence we have $\|
  \bar{\phi}_{\varepsilon, N} \|_{H^1_{\ell} (\mathbb{R}^2 \times \varepsilon
  \mathbb{Z}^2)} < + \infty$.
  
  The uniqueness of the solution to equation
  {\eqref{eq:mainfinitedimensional}} is proved in Remark
  \ref{remark:uniqueness}.
\end{proof}

\begin{lemma}
  \label{lemma:solutiondiscrete}For any $| \alpha | < (4 \pi)$ and
  $\varepsilon_n, N_k > 0$ as in the thesis of Proposition
  \ref{proposition:convergencealpha}, there exists a (unique) solution to
  equation {\eqref{eq:maindiscrete}} such that $\bar{\phi}_{\varepsilon_n} \in
  H^1_{\ell} (\mathbb{R}^2 \times \varepsilon_n \mathbb{Z}^2)$ almost surely
  (for any $\ell > \ell_0 \geqslant 0$ for some $\ell_0$ not depending on
  $\varepsilon_n$). Furthermore, $e^{\pm \beta \alpha \bar{\phi}_{\varepsilon_n}} \in H^1_{\ell}
  (\mathbb{R}^2 \times \varepsilon_n \mathbb{Z}^2)$ and $e^{\pm (1 \pm 2
  \beta) \alpha \bar{\phi}_{\varepsilon_n}} \in L^1 (\rho_{\ell}^{(4)} \mathd
  \mu^{\pm \alpha}_{\varepsilon_n})$ for any $0 \leqslant \beta < 1$.
\end{lemma}

\begin{proof}
  By Lemma \ref{lemma:existencefinitedim}, equation
  {\eqref{eq:mainfinitedimensional}} admits a (unique) solution
  $\bar{\phi}_{\varepsilon_n, N_k}$ in $H^1_{\ell} (\mathbb{R}^2 \times
  \varepsilon_n \mathbb{Z}^2)$ for $\ell$ big enough. By Proposition
  \ref{proposition:convergencealpha}, Theorem \ref{theorem_apriori1} and
  Theorem \ref{theorem_apriori2} we obtain
  \begin{equation}
    \| \bar{\phi}_{\varepsilon_n, N_k} \|_{H^1_{\ell} (\mathbb{R}^2 \times
    \varepsilon_n \mathbb{Z}^2)}, \left\| e^{\pm \beta \alpha
    \bar{\phi}_{\varepsilon_n, N_k}} \right\|_{H^1_{\ell} (\mathbb{R}^2 \times
    \varepsilon_n \mathbb{Z}^2)}, \left\| e^{\pm ( 1+ 2 \beta)\alpha
    \bar{\phi}_{\varepsilon_n, N_k}} \right\|_{L^1 (\rho_{2 \ell}^{(4)} \mathd
    \mu^{\pm \alpha}_{\varepsilon_n, N_k})} \lesssim 1
    \label{eq:inequalityepsilonN},
  \end{equation}
  where the constants hidden in the symbol $\lesssim$ are independent of
  $N_k$.
  
  \
  
  This means in particular that the sequences (in $N_k$)
  $\bar{\phi}_{\varepsilon_n, N_k}$ and $e^{\pm \beta \alpha
  \bar{\phi}_{\varepsilon_n, N_k}}$ are weakly precompact in $H^1_{\ell}
  (\mathbb{R}^2 \times \varepsilon_n \mathbb{Z}^2)$. Consider $\kappa,
  \kappa'_{\pm} \in H^1_{\ell} (\mathbb{R}^2 \times \varepsilon_n
  \mathbb{Z}^2)$ such that there is a subsequence $\bar{\phi}_{\varepsilon_n,
  N_{k_o}}$ and $e^{\pm \beta \alpha \bar{\phi}_{\varepsilon_n, N_{k_o}}}$
  converging to $\kappa$ and $\kappa'_{\pm}$ respectively weakly in
  $H^1_{\ell} (\mathbb{R}^2 \times \varepsilon_n \mathbb{Z}^2)$. Since, by
  Theorem \ref{theorem_apriori1} and Theorem \ref{theorem_apriori2},
  $\bar{\phi}_{\varepsilon_n, N_{k_o}}$ and $e^{\pm \beta \alpha
  \bar{\phi}_{\varepsilon_n, N_{k_o}}}$ converge strongly to $\kappa$ and
  $\kappa'$ in $W^{s, 2}_{\ell'} (\mathbb{R}^2 \times \varepsilon_n
  \mathbb{Z}^2)$, for any $\ell' > \ell$ and $0 < s < 1$, we have also
  $\kappa'_{\pm} = e^{\pm \beta \alpha \kappa}$ almost everywhere with respect
  to the measure $\mathd x \mathd z$. We now want to prove that $\kappa$ is a solution to the equation
  {\eqref{eq:maindiscrete}}. First we want to prove that $e^{\pm \alpha
  \kappa} \in L^1 (\rho^{(4)}_{\ell'} \mu^{\pm \alpha}_{\varepsilon_n} (\mathd x,
  \mathd z))$ almost surely, for any $\ell' > \ell \geqslant 4$.
  
  \
  
  Let $K > 0$ and consider the function $F_K : \mathbb{R} \rightarrow
  \mathbb{R}$ defined as $F_K (y) = y^2 \wedge K^2$, $y \in \mathbb{R}$. The
  function $F_K$ is (globally) Lipschitz and bounded, this implies, by Remark
  \ref{remark:Lipschitz}, that $F_K \left( e^{\frac{\pm \alpha}{2}
  \bar{\phi}_{\varepsilon, N}} \right) \in H^1_{\ell} (\mathbb{R}^2 \times
  \varepsilon \mathbb{Z}^2) \cap L^{\infty} (\mathbb{R}^2 \times \varepsilon
  \mathbb{Z}^2)$, that $F_K:H^1_{\ell} (\mathbb{R}^2 \times
  \varepsilon \mathbb{Z}^2) \rightarrow H^{1}_{\ell}(\mathbb{R}^2 \times
  \varepsilon \mathbb{Z}^2)$ is continuous when the target space is equipped with the weak topology, and
  \[ F_K \left( e^{\frac{\pm \alpha}{2} \bar{\phi}_{\varepsilon, N} (x, z)}
     \right) = e^{\pm \alpha \bar{\phi}_{\varepsilon, N} (x, z)}, \]
  when $e^{\frac{\alpha}{2} \bar{\phi}_{\varepsilon, N} (x, z)} \leqslant K$.
  In particular we have that $F_K \left( e^{\frac{\pm \alpha}{2}
  \bar{\phi}_{\varepsilon_n, N_{k_o}} (x, z)} \right), F_K \left( e^{\frac{\pm
  \alpha}{2} \kappa} \right) \in L^1 (\rho^{(4)}_{\ell}
  \mu^{\pm \alpha}_{\varepsilon_n})$ and, by the convergence of
  $\bar{\phi}_{\varepsilon_n, N_k}$ to $\kappa$ in $H^1_{\ell} (\mathbb{R}^2
  \times \varepsilon_n \mathbb{Z}^2)$, and since $F_K$ is Lipschitz, $F_K
  \left( e^{\frac{\pm \alpha}{2} \bar{\phi}_{\varepsilon, N_k} (x, z)}
  \right)$ converges to $F_K (\kappa)$ Lebesgue almost everywhere. This means
  that, by Fatou lemma and using the fact that $\mu_{\varepsilon_n, N_k}^{\pm
  \alpha}$ are functions converging to $\mu_{\varepsilon_n}^{\pm \alpha}$
  Lebesgue almost everywhere, we get
  \begin{align*}
       &\int_{\mathbb{R}^2 \times \varepsilon_n \mathbb{Z}^2} \rho^{(4)}_{\ell
       + \ell'} (x, z) F_K \left( e^{\pm \frac{\alpha}{2} \kappa} \right)
       \mu^{\pm\alpha}_{\varepsilon_n} (\mathd x, \mathd z)\\
       \leqslant& \liminf_{N_{k_o} \rightarrow + \infty} \int_{\mathbb{R}^2
       \times \varepsilon_n \mathbb{Z}^2} \rho^{(4)}_{\ell + \ell'} (x, z) F_K
       \left( e^{\pm \frac{\alpha}{2} \bar{\phi}_{\varepsilon, N_k} (x, z)}
       \right) \mu^{\pm\alpha}_{\varepsilon_n, N_{k_o}} (\mathd x, \mathd z)\\
       \leqslant& \liminf_{N_{k_o} \rightarrow + \infty} \int_{\mathbb{R}^2
       \times \varepsilon_n \mathbb{Z}^2} \rho^{(4)}_{\ell + \ell'} (x, z)
       e^{\pm \alpha \bar{\phi}_{\varepsilon, N_{k_o}} (x, z)}
       \mu^{\pm\alpha}_{\varepsilon_n, N_{k_o}} (\mathd x, \mathd z) \lesssim 1.
    \end{align*}
  Applying again Fatou lemma for $K \rightarrow + \infty$
  \begin{align*}
       &\int_{\mathbb{R}^2 \times \varepsilon_n \mathbb{Z}^2} \rho^{(4)}_{\ell'
       + \ell} (x, z) e^{\pm \alpha \kappa} \mu^{\pm \alpha}_{\varepsilon_n}
       (\mathd x, \mathd z) \\
       =& \int_{\mathbb{R}^2 \times \varepsilon_n \mathbb{Z}^2}
       \rho^{(4)}_{\ell' + \ell} (x, z) \lim_{K \rightarrow + \infty} \left(
       F_K \left( e^{\pm \frac{\alpha}{2} \kappa} \right) \right)
       \mu^{\pm\alpha}_{\varepsilon_n} (\mathd x, \mathd z)\\
       \leqslant& \lim_{K \rightarrow + \infty} \int_{\mathbb{R}^2 \times
       \varepsilon_n \mathbb{Z}^2} \rho^{(4)}_{\ell' + \ell} (x, z) F_K \left(
       e^{\pm \frac{\alpha}{2} \kappa} \right) \mu^{\pm\alpha}_{\varepsilon_n}
       (\mathd x, \mathd z) \lesssim 1.
     \end{align*}
  This implies that $e^{\pm \alpha \kappa} \in L^1 (\rho_{\ell}^{4}\mathd \mu^{\pm
  \alpha}_{\varepsilon_n})$. What remain to prove is that $\kappa$ solves
  (weakly) equation {\eqref{eq:maindiscrete}}.
  
  \
  
   By the inequalities
  {\eqref{eq:inequalityepsilonN}}, we have that there exists a $p > 1$ such
  that
  \[ \rho^{(4)}_{\ell} (x, z) e^{\pm \alpha \bar{\phi}_{\varepsilon, N} (x,
     z)} \in L^p (\mu^{\pm \alpha}_{\varepsilon, N} (\mathd x, \mathd z)),
      \]
  uniformly in $N$. This implies by Jensen's inequality that, for any $\sigma>0$, there is a
  $K_{\sigma} > 0$ for which, for any $\bar{K} > K_{\sigma}$ and any $N \in
  \mathbb{N}$,
  \[ \int_{e^{\pm \alpha \bar{\phi}_{\varepsilon, N}} > \bar{K}}
     \rho^{(4)}_{\ell} (x, z) e^{\pm \alpha \bar{\phi}_{\varepsilon, N} (x,
     z)} \mu^{\pm \alpha}_{\varepsilon, N} (\mathd x, \mathd z) < \sigma . \]
  Consider $g \in \mathcal{S} (\mathbb{R}^2 \times \varepsilon_{n} \mathbb{Z}^2)$,
  then by multiplying equation {\eqref{eq:mainfinitedimensional}} by $g$,
  taking the integral and using integration by parts, for any $\bar{K} >
  K_{\sigma}$, we get
  \begin{equation}
    \begin{array}{rl}
      &\left. \left| \int_{\mathbb{R}^2 \times \varepsilon_{n} \mathbb{Z}^2} [(m^2
      - \Delta_{\mathbb{R}^2 \times \varepsilon_{n} \mathbb{Z}^2}) g (x, z)]
      \phi_{\varepsilon_n, N_{k_o}} \mathd x \mathd z + \alpha
      \int_{\mathbb{R}^2 \times \varepsilon_{n} \mathbb{Z}^2} g (x, z) \left(
      F_{\bar{K}} \left( e^{\frac{\alpha}{2} \bar{\phi}_{\varepsilon_n, N_{K_o}} (x,
      z)} \right) \times \right. \right. \right.\\
      &\left. \phantom{\int_{\mathbb{R}}} \mu^{\alpha}_{\varepsilon_n, N_{K_o}}
      (\mathd x, \mathd z) \left. - F_{\bar{K}} \left( e^{- \frac{\alpha}{2}
      \bar{\phi}_{\varepsilon, N_{K_o}} (x, z)} \right) \mu^{-
      \alpha}_{\varepsilon_n, N_{K_o}} (\mathd x, \mathd z) \right) \right|
      \\
      \lesssim &\left| \int_{\mathbb{R}^2 \times \varepsilon_{n} \mathbb{Z}^2} g
      (x, z) \left. \left( F_{\bar{K}} \left( e^{\frac{\alpha}{2}
      \bar{\phi}_{\varepsilon_n, N_{K_o}} (x, z)} \right) - e^{\alpha
      \bar{\phi}_{\varepsilon_n, N_{K_o}} (x, z)} \right)
      \mu^{\alpha}_{\varepsilon_n, N_{K_o}} (\mathd x, \mathd z) \right.
      \right| \\
      &+ \left| \int_{\mathbb{R}^2 \times \varepsilon_n \mathbb{Z}^2} g (x, z)
      \left. \left( F_{\bar{K}} \left( e^{- \frac{\alpha}{2}
      \bar{\phi}_{\varepsilon_n, N_{K_o}} (x, z)} \right) - e^{- \alpha
      \bar{\phi}_{\varepsilon_n, N_{K_o}} (x, z)} \right) \mu^{-
      \alpha}_{\varepsilon_n, N_{K_o}} (\mathd x, \mathd z) \right. \right|\\
      \lesssim& \int_{e^{\pm \alpha \bar{\phi}_{\varepsilon_n, N}} > \bar{K}} |
      g (x, z) | e^{\pm \alpha \bar{\phi}_{\varepsilon_n, N} (x, z)} \mu^{\pm
      \alpha}_{\varepsilon_n, N_{K_{o}}} (\mathd x, \mathd z) \lesssim \sigma \| g
      \|_{L^{\infty}_{- \ell - \ell'} (\mathbb{R}^2 \times \varepsilon_{n}
      \mathbb{Z}^2)}.
    \end{array} \label{eq:inequalityKsigma}
  \end{equation}

  Since $\phi_{\varepsilon_n, N_{K_o}}$ and $F_{\bar{K}} \left( e^{\pm
  \frac{\alpha}{2} \bar{\phi}_{\varepsilon_n, N_{K_o}} (x, z)} \right)$ converge
  weakly (in $H^1_{\ell'} (\mathbb{R}^2 \times \varepsilon_n \mathbb{Z}^2)$)
  to $\kappa$ and $F_{\bar{K}} \left( e^{\pm \frac{\alpha}{2} \kappa} \right)$
  respectively, and, by Proposition \ref{proposition:convergencealpha}, \
  $\mu^{\pm \alpha}_{\varepsilon_n, N_{K_o}}$ converge to $\mu^{\pm
  \alpha}_{\varepsilon_n}$ strongly in $B^{1 - \delta}_{2, 2, \ell}
  (\mathbb{R}^2 \times \varepsilon_n \mathbb{Z}^2)$, we can take the limit as
  $N_{K_o} \rightarrow + \infty$ in inequality {\eqref{eq:inequalityKsigma}}
  to get
  \begin{equation}
    \begin{array}{rl}
      \left. \left| \int_{\mathbb{R}^2 \times \varepsilon_{n} \mathbb{Z}^2} [(m^2
      - \Delta_{\mathbb{R}^2 \times \varepsilon \mathbb{Z}^2}) g (x, z)]
      \kappa \mathd x \mathd z + \alpha \int_{\mathbb{R}^2 \times \varepsilon_{n}
      \mathbb{Z}^2} g (x, z) \left( F_{\bar{K}} \left( e^{\frac{\alpha}{2}
      \kappa} \right) \times \right. \right. \right.&\\
      \left. \phantom{\int_{\mathbb{R}}} \left. \mu^{\alpha}_{\varepsilon_n}
      (\mathd x, \mathd z) - F_{\bar{K}} \left( e^{- \frac{\alpha}{2} \kappa}
      \right) \mu^{\alpha}_{\varepsilon_n} (\mathd x, \mathd z) \right)
      \right| & \leqslant \sigma \| g \|_{L^{\infty}_{- \ell - \ell'}
      (\mathbb{R}^2 \times \varepsilon_{n} \mathbb{Z}^2)} .
    \end{array} \label{eq:inequalityKsigma2}
  \end{equation}
  Since $e^{\pm \alpha \bar{\phi}_{\varepsilon_{n}} (x, z)} \in L^1
  (\rho^{(4)}_{\ell' + \ell} \mathd \mu^{\pm \alpha}_{\varepsilon_n})$, by
  Lebesgue dominate convergence theorem, we can take the limit $\bar{K}
  \rightarrow + \infty$, obtaining
  \[ \left| \int_{\mathbb{R}^2 \times \varepsilon_n \mathbb{Z}^2} [(m^2 -
     \Delta_{\mathbb{R}^2 \times \varepsilon_n \mathbb{Z}^2}) g] \kappa \mathd
     x \mathd z + \alpha \int_{\mathbb{R}^2 \times \varepsilon_n \mathbb{Z}^2}
     g \; (e^{\alpha \kappa} \mu^{\alpha}_{\varepsilon_n} (\mathd x, \mathd z)
     - e^{- \alpha \kappa} \mu^{\alpha}_{\varepsilon_n} (\mathd x, \mathd z))
     \right| \lesssim \sigma. \]
  Since $\sigma$ is arbitrary and $\mathcal{S} (\mathbb{R}^2 \times
  \varepsilon_n \mathbb{Z}^2)$ is dense in the dual of $H^1_{\ell + \ell'}
  (\mathbb{R}^2 \times \varepsilon_{n} \mathbb{Z}^2)$ we have that $\kappa$ solves
  the equation {\eqref{eq:maindiscrete}}. By the uniqueness of solution to
  equation {\eqref{eq:maindiscrete}} obtained in Remark
  \ref{remark:uniqueness}, we get then $\kappa = \bar{\phi}_{\varepsilon_n}$,
  where $\bar{\phi}_{\varepsilon_n}$ being the unique solution to equation
  {\eqref{eq:maindiscrete}}.\\
  
  The fact that $e^{\pm \beta \alpha \bar{\phi}_{\varepsilon_n}} \in H^1_{\ell}
  (\mathbb{R}^2 \times \varepsilon_n \mathbb{Z}^2)$ and $e^{\pm (1 \pm 2
  \beta) \alpha \bar{\phi}_{\varepsilon_n}} \in L^1 (\rho_{\ell}^{(4)} \mathd
  \mu^{\pm \alpha}_{\varepsilon_n})$ for any $0 \leqslant \beta < 1$ follows from the uniform in $N_k$ bound  \eqref{eq:inequalityepsilonN}.
\end{proof}

\begin{theorem}
  \label{theorem:convergenceequation1}Consider $\varepsilon_n$ as in
  Proposition \ref{proposition:convergencealpha} and let
  $\bar{\phi}_{\varepsilon_n}$ be the solution to equation
  {\eqref{eq:maindiscrete}} for $\varepsilon = \varepsilon_n$. Then
  $\overline{\mathcal{E}}^{\varepsilon_n} (\bar{\phi}_{\varepsilon_n})$
  converges in $\mathcal{S}' (\mathbb{R}^4)$ to the unique solution
  $\bar{\phi}$ to equation {\eqref{eq:main5}} in $H^1_{\ell} (\mathbb{R}^4)$
  (for any $\ell \geqslant \ell_0$ and a suitable $\ell_0 > 0$).  Furthermore, $e^{\pm \beta \alpha \bar{\phi}} \in H^1_{\ell}
  (\mathbb{R}^4)$ and $e^{\pm (1 \pm 2
  \beta) \alpha \bar{\phi}} \in L^1 (\rho_{\ell}^{(4)} \mathd
  \mu^{\pm \alpha})$ for any $0 \leqslant \beta < 1$.
\end{theorem}

\begin{proof}
  The proof is similar to the one of Lemma \ref{lemma:solutiondiscrete}. We
  report here only a sketch of the proof pointing out the main differences.
  
  \
  
  By the Lemma \ref{lemma:solutiondiscrete}, there is a unique solution to \
  equation {\eqref{eq:maindiscrete}} and by Theorem \ref{theorem_apriori1},
  Theorem \ref{theorem_apriori2} and Proposition
  \ref{proposition:convergencealpha} we get for $\beta \in (- 1, 1)$ that
  \begin{equation}
    \| \bar{\phi}_{\varepsilon_n} \|_{H^1_{\ell} (\mathbb{R}^2 \times
    \varepsilon_n \mathbb{Z}^2)}, \left\| e^{\pm \beta \alpha
    \bar{\phi}_{\varepsilon_n}} \right\|_{H^1_{\ell} (\mathbb{R}^2 \times
    \varepsilon_n \mathbb{Z}^2)}, \left\| e^{\pm (1 + 2 \beta)\alpha
    \bar{\phi}_{\varepsilon_n}} \right\|_{L^1 (\rho_{2 \ell}^{(4)} \mathd
    \mu^{\pm \alpha}_{\varepsilon_n})} \lesssim 1,
    \label{eq:inequalityepsilonsmalln}
  \end{equation}
  holds, where the constants hidden in the symbol $\lesssim$ are independent
  on $\varepsilon_n$. Using the same methods of Lemma
  \ref{lemma:solutiondiscrete} there is a subsequence $\varepsilon_{n_k}
  \rightarrow 0$ and a $\kappa \in B^{\frac{1}{p}}_{p, p, \ell}
  (\mathbb{R}^4)$ (for any $1 \leqslant p \leqslant 2$) such that
  $\overline{\mathcal{E}}^{\varepsilon_{n_k}} \left(
  \bar{\phi}_{\varepsilon_{n_k}} \right) \rightarrow \kappa$,
  $\overline{\mathcal{E}}^{\varepsilon_{n_k}} \left( e^{\pm \beta \alpha
  \bar{\phi}_{\varepsilon_{n_{k}}}} \right) \rightarrow e^{\pm \alpha \beta \kappa}$
  weakly in $\tmmathbf{} B^{\frac{1}{p}}_{p, p, \ell} (\mathbb{R}^4)$ (and
  thus strongly in $B^{\frac{1}{p} - \upsilon}_{p, p, \ell} (\mathbb{R}^4)$
  for any $\upsilon > 0$). Furthermore, by the second part of Theorem
  \ref{theorem:extension3}, we also have that $\kappa, e^{\pm \beta \alpha
  \kappa} \in H^1_{\ell} (\mathbb{R}^4)$. What remains to prove is that
  $e^{\pm \alpha \kappa} \in L^1 (\rho^{(4)}_{\ell} \mathd \mu^{\pm \alpha})$ and
  $\kappa$ is a solution to equation {\eqref{eq:main5}}.
  
  \
  
  First we prove that $e^{\pm \alpha \kappa} \in L^1 (\rho^{(4)}_{\ell} \mathd
  \mu^{\pm \alpha})$. Let $\rho^{(4)}_{\ell, \varepsilon_n} : \mathbb{R}^2
  \times \varepsilon_{n} \mathbb{Z}^2 \rightarrow \mathbb{R}$ be the restriction
  of $\rho^{(4)}_{\ell}$, then, having $\rho^{(4)}_{\ell} \in B^k_{\infty,
  \infty, - \ell} (\mathbb{R}^4)$ for any $k > 0$, it is simple to prove
  (using the difference spaces $W^{s, q}_{- \ell} (\mathbb{R}^2 \times
  \varepsilon \mathbb{Z}^2)$ and their continuous immersion in $B^{s -
  \upsilon}_{q,q,\ell} (\mathbb{R}^2 \times \varepsilon \mathbb{Z}^2)$ for any
  $\upsilon > 0$), that $\rho^{(4)}_{\ell, \varepsilon_n} \in B^k_{\infty,
  \infty, - \ell} (\mathbb{R}^2 \times \varepsilon_n \mathbb{Z}^2)$ and also
  that $\overline{\mathcal{E}}^{\varepsilon_n} (\rho^{(4)}_{\ell,
  \varepsilon_n})$ converges strongly to $\rho_{\ell}^{(4)}$ in
  $B^{\frac{1}{q} - \upsilon}_{q, q, - \ell} (\mathbb{R}^4)$ for any $q \in [1,+\infty)$ and $\upsilon
  > 0$ (here $q$ is not $p$ above).\\
  
  Consider the operator $\Delta_N^{\varepsilon_n} : \mathcal{S}' (\mathbb{R}^2
  \times \varepsilon_n \mathbb{Z}^2) \rightarrow \mathcal{S}' (\mathbb{R}^2
  \times \varepsilon_n \mathbb{Z}^2)$, defined as $\Delta^{\varepsilon_n}_N f
  = f \ast_{\varepsilon_n} \rho_N$, where $f \in \mathcal{S}' (\mathbb{R}^2
  \times \varepsilon_n \mathbb{Z}^2)$ and $\mathcal{F}^{\varepsilon_n} \rho_N
  : \mathbb{R}^2 \times \mathbb{T}^2_{\frac{1}{\varepsilon_n}} \rightarrow [0, 1]$ is
  smooth and compactly supported in a ball of radius $2 N$ and
  $\mathcal{F}^{\varepsilon_n} \rho_N (\xi) = 1$ for \ $| \xi | \leqslant N$.
  The operator $\Delta^{\varepsilon_n}_N$ has a regularizing property: 
  $\Delta_N^{\varepsilon_n} : B^s_{q, q, \ell} (\mathbb{R}^2 \times
  \varepsilon_n \mathbb{Z}^2) \rightarrow B^{s'}_{q', q', \ell} (\mathbb{R}^2
  \times \varepsilon_n \mathbb{Z}^2)$ for any $s' > s$ and $q'  \geqslant q$ with a norm depending
  only on $q, q', \ell, s$ and $s'$ but not on $\varepsilon_n$.\\
  
  Furthermore, since by Theorem \ref{theorem:convergencebar} and Theorem
  \ref{theorem_convergenceepsilon} and by the proof of Proposition
  \ref{proposition:convergencealpha}, $\mathbb{N}^s_{p, \ell, \varepsilon_n}
  (\mu_{\varepsilon_n}^{\alpha}, a, \beta) \rightarrow \mathbb{N}^s_{p, \ell}
  (\mu^{\alpha}, a, \beta)$ almost surely, for any $\upsilon > 0$ there is a
  random variable $N_{\upsilon} : \Omega \rightarrow \mathbb{N}$, almost
  surely finite, such that for any $\bar{N} \geqslant N_{\upsilon}$
  \[ \| (1 - \Delta_{\bar{N}}^{\varepsilon_n}) \mu^{\pm
     \alpha}_{\varepsilon_n} \|_{B^{1 - \delta}_{2, 2, \ell} (\mathbb{R}^2
     \times \varepsilon_n \mathbb{Z}^2)} \leqslant \upsilon . \]
  Furthermore using again the proof of Proposition
  \ref{proposition:convergencealpha}, one can prove that
  $\overline{\mathcal{E}}^{\varepsilon_n} (\Delta_N^{\varepsilon_n}
  \mu^{\alpha}_{\varepsilon_n})$ converges to $\Delta_N \mu^{\pm \alpha}$ in
  $B^{\frac{1}{q} - \upsilon'}_{q, q, \ell} (\mathbb{R}^4)$ for any $\ell
  \geqslant 4$, $q\in(1,+\infty
  )$ and $\upsilon' > 0$. Consider $K > 0$ and $F_K$ as in the proof
  of Lemma \ref{lemma:solutiondiscrete}, then
  \begin{align*}   
   &\left| \int \rho^{(4)}_{2 \ell} (x, z) F_K \left( e^{\pm
     \frac{\alpha}{2} \kappa} \right) \Delta_N (\mu^{\pm \alpha}) (\mathd x,
     \mathd z) \right| \\
   =& \lim_{k \rightarrow + \infty} \left| \int
     \overline{\mathcal{E}}^{\varepsilon_k} \left( \rho^{(4)}_{2 \ell,
     \varepsilon_k} F_K \left( e^{\pm \frac{\alpha}{2}
     \bar{\phi}_{\varepsilon_k}} \right) \right)
     \overline{\mathcal{E}}^{\varepsilon_k} (\Delta_N (\mu^{\pm
     \alpha}_{\varepsilon_k}) (\mathd x, \mathd z)) \right| \\
   =& \lim_{k \rightarrow + \infty} \left| \int \rho^{(4)}_{2 \ell,
     \varepsilon_k} F_K \left( e^{\pm \frac{\alpha}{2}
     \bar{\phi}_{\varepsilon_k}} \right) \Delta_N \mu^{\pm
     \alpha}_{\varepsilon_k} (\mathd x, \mathd z) \right| \\
   \leqslant& \sup_k \left( \left| \int \rho^{(4)}_{2 \ell, \varepsilon_k}
     F_K \left( e^{\pm \frac{\alpha}{2} \bar{\phi}_{\varepsilon_k}} \right)
     \mu^{\pm \alpha}_{\varepsilon_k} (\mathd x, \mathd z) \right| + \left|
     \int \rho^{(4)}_{2 \ell, \varepsilon_k} F_K \left( e^{\pm
     \frac{\alpha}{2} \bar{\phi}_{\varepsilon_k}} \right) (1 - \Delta_N)
     (\mu^{\pm \alpha}_{\varepsilon_k}) (\mathd x, \mathd z) \right| \right)
  \\
   \lesssim& \sup_k \left( \int \rho^{(4)}_{2 \ell, \varepsilon_k} e^{\pm
     \alpha \bar{\phi}_{\varepsilon_k}} \mu_{\varepsilon_k}^{\pm \alpha}
     (\mathd x \mathd z) \right) + \upsilon \sup_k \left( \left\| F_K \left(
     e^{\pm \frac{\alpha}{2} \bar{\phi}_{\varepsilon_k}} \right)
     \right\|_{H^1_{\ell} (\mathbb{R}^2 \times \varepsilon_{n_k}
     \mathbb{Z}^2)} \right) . \end{align*}
  Since $\upsilon$ is arbitrary taking the limit $N \rightarrow + \infty$ we
  get
  \[ \left| \int \rho^{(4)}_{2 \ell} (x, z) F_K \left( e^{\pm
     \frac{\alpha}{2} \kappa} \right) \mu^{\pm \alpha} (\mathd x, \mathd z)
     \right| \lesssim \sup_k \left( \int \rho^{(4)}_{2 \ell, \varepsilon_k}
     e^{\pm \alpha \bar{\phi}_{\varepsilon_k}} \mu_{\varepsilon_k}^{\pm
     \alpha} (\mathd x \mathd z) \right) < + \infty . \]
  By Fatou lemma, the limit $K \rightarrow + \infty$ gives that $e^{\pm \alpha
  \kappa} \in L^1 (\rho^{(4)}_{2 \ell} \mathd \mu^{\pm \alpha}) .$ Using a
  similar strategy it is simple to prove that $e^{\pm 2 \alpha \kappa} \in L^1
  (\rho^{(4)}_{2 \ell} \mathd \mu^{\pm \alpha})$ as required in Proposition
  \ref{proposition:uniqueness}.
  
  \
  
  For proving that $\kappa$ is also a (weak) solution to equation
  {\eqref{eq:main5}} we shall use a combination of the technique in the proof
  of Lemma \ref{lemma:solutiondiscrete} and in the previous part of the proof.
  Consider $g \in \mathcal{S} (\mathbb{R}^4)$ and let $g_{\varepsilon} \in
  \mathcal{S} (\mathbb{R}^2 \times \varepsilon \mathbb{Z}^2)$ be the
  restriction of $g$ on $\mathbb{R}^2 \times \varepsilon \mathbb{Z}^2$. We
  have that $\overline{\mathcal{E}}^{\varepsilon} (g_{\varepsilon})$ converges
  to $g$ in $B^{\frac{1}{q}}_{q, q, - \ell} (\mathbb{R}^4)$, for any $q\in [1,+\infty]$, and also
  $\overline{\mathcal{E}}^{\varepsilon} (g_{\varepsilon} \mathd \mu^{\pm
  \alpha}_{\varepsilon}) = \overline{\mathcal{E}}^{\varepsilon}
  (g_{\varepsilon}) \mathd \bar{\mu}^{\pm \alpha}_{\varepsilon}$. Thus we obtain the following equality
  \[ \begin{array}{rl}
       0=& - \langle g_{\varepsilon_n}, \Delta_{\mathbb{R}^2 \times
       \varepsilon \mathbb{Z}^2} (\bar{\phi}_{\varepsilon_n}) \rangle + m^2
       \langle g_{\varepsilon_n}, \bar{\phi}_{\varepsilon_n} \rangle + \alpha
       \int_{\mathbb{R}^2 \times \varepsilon \mathbb{Z}^2} g_{\varepsilon_n}
       e^{\alpha \bar{\phi}_{\varepsilon_n}} \mathd
       \mu_{\varepsilon_n}^{\alpha} - \alpha \int_{\mathbb{R}^2 \times
       \varepsilon \mathbb{Z}^2} g_{\varepsilon_n} e^{- \alpha
       \bar{\phi}_{\varepsilon_n}} \mathd \mu_{\varepsilon_n}^{- \alpha} \\
       =& - \langle \Delta_{\mathbb{R}^2 \times \varepsilon \mathbb{Z}^2}
       g_{\varepsilon_n}, \bar{\phi}_{\varepsilon_n} \rangle + m^2 \langle
       g_{\varepsilon_n}, \bar{\phi}_{\varepsilon_n} \rangle + \alpha
       \int_{\mathbb{R}^2 \times \varepsilon \mathbb{Z}^2} g_{\varepsilon_n}
       e^{\alpha \bar{\phi}_{\varepsilon_n}} \mathd
       \mu_{\varepsilon_n}^{\alpha} - \alpha \int_{\mathbb{R}^2 \times
       \varepsilon \mathbb{Z}^2} g_{\varepsilon} e^{- \alpha
       \bar{\phi}_{\varepsilon}} \mathd \mu_{\varepsilon_n}^{- \alpha} \\
       =& - \langle \Delta_{\mathbb{R}^2 \times \varepsilon \mathbb{Z}^2}
       g_{\varepsilon_n}, \bar{\phi}_{\varepsilon_n} \rangle + m^2 \langle
       g_{\varepsilon_n}, \bar{\phi}_{\varepsilon_n} \rangle + \alpha
       \int_{\mathbb{R}^2 \times \varepsilon \mathbb{Z}^2} g_{\varepsilon_n}
       F_{\bar{K}} \left( e^{\alpha \bar{\phi}_{\varepsilon_n}} \right) \mathd
       \mu_{\varepsilon_n}^{\alpha} - \alpha \int_{\mathbb{R}^2 \times
       \varepsilon \mathbb{Z}^2} g_{\varepsilon} F_{\bar{K}} (e^{- \alpha
       \bar{\phi}_{\varepsilon}}) \mathd \mu_{\varepsilon_n}^{- \alpha}\\
        &+ \alpha \int_{e^{\alpha \bar{\phi}_{\varepsilon_n}} > \bar{K}}
       g_{\varepsilon_n} \left( e^{\alpha \bar{\phi}_{\varepsilon_n}} -
       \bar{K} \right) \mathd \mu_{\varepsilon_n}^{\alpha} - \alpha \int_{e^{-
       \alpha \bar{\phi}_{\varepsilon_n}} > \bar{K}} g_{\varepsilon_n} \left(
       e^{- \alpha \bar{\phi}_{\varepsilon_n}} - \bar{K} \right) \mathd
       \mu_{\varepsilon_n}^{- \alpha}\\
       =& - \int_{\mathbb{R}^4} \overline{\mathcal{E}}^{\varepsilon}
       (\Delta_{\mathbb{R}^2 \times \varepsilon \mathbb{Z}^2} g_{\varepsilon})
       \overline{\mathcal{E}}^{\varepsilon} (\bar{\phi}_{\varepsilon}) \mathd
       x \mathd z + m^2 \int_{\mathbb{R}^4}
       \overline{\mathcal{E}}^{\varepsilon} (g_{\varepsilon})
       \overline{\mathcal{E}}^{\varepsilon} (\bar{\phi}_{\varepsilon}) \mathd
       x \mathd z \\
       &+ \alpha \int_{\mathbb{R}^4}
       \overline{\mathcal{E}}^{\varepsilon} (g_{\varepsilon_n}) F_{\bar{K}}
       \left( e^{\alpha \overline{\mathcal{E}}^{\varepsilon_n}
       (\bar{\phi}_{\varepsilon_n})} \right) \mathd
       \overline{\mathcal{E}}^{\varepsilon} (\mu_{\varepsilon_n}^{\alpha})
       + \alpha \int_{\mathbb{R}^4}
       \overline{\mathcal{E}}^{\varepsilon} (g_{\varepsilon_n}) F_{\bar{K}}
       \left( e^{\alpha \overline{\mathcal{E}}^{\varepsilon_n}
       (\bar{\phi}_{\varepsilon_n})} \right) \mathd
       \overline{\mathcal{E}}^{\varepsilon} (
       \mu_{\varepsilon_n}^{\alpha})\\
       &+ \alpha \left( \int_{e^{\alpha \bar{\phi}_{\varepsilon_n}} > \bar{K}}
       g_{\varepsilon_n} \left( e^{\alpha \bar{\phi}_{\varepsilon_n}} -
       \bar{K} \right) \mathd \mu_{\varepsilon_n}^{\alpha} - \int_{e^{- \alpha
       \bar{\phi}_{\varepsilon_n}} > \bar{K}} g_{\varepsilon_n} \left( e^{-
       \alpha \bar{\phi}_{\varepsilon_n}} - \bar{K} \right) \mathd
       \mu_{\varepsilon_n}^{- \alpha} \right). \\
     \end{array} \]
  The previous equality (together with a reasoning similar to the one of
  equation {\eqref{eq:inequalityKsigma}}) implies that there exist
  $K_{\sigma}, N_{\upsilon}$ such that for $\bar{K} > K_{\sigma}$ and $N >
  N_{\upsilon}$ we get
  \begin{eqnarray*}
    &  & \left| - \int_{\mathbb{R}^4} \overline{\mathcal{E}}^{\varepsilon_k}
    (\Delta_{\mathbb{R}^2 \times \varepsilon_k \mathbb{Z}^2}
    g_{\varepsilon_k}) \overline{\mathcal{E}}^{\varepsilon_k}
    (\bar{\phi}_{\varepsilon_k}) \mathd x \mathd z + m^2 \int_{\mathbb{R}^4}
    \overline{\mathcal{E}}^{\varepsilon_k} (g_{\varepsilon_k})
    \overline{\mathcal{E}}^{\varepsilon_k} (\bar{\phi}_{\varepsilon_k}) \mathd
    x \mathd z \right. +\\
    &  & \left. \alpha \int_{\mathbb{R}^4}
    \overline{\mathcal{E}}^{\varepsilon_k} (g_{\varepsilon_k}) F_{\bar{K}}
    \left( e^{\alpha \overline{\mathcal{E}}^{\varepsilon_k}
    (\bar{\phi}_{\varepsilon_k})} \right) \mathd
    \overline{\mathcal{E}}^{\varepsilon_k} ( \mu_{\varepsilon_k}^{\alpha}) -
    \alpha \int_{\mathbb{R}^2 \times \varepsilon_k \mathbb{Z}^2}
    \overline{\mathcal{E}}^{\varepsilon_k} (g_{\varepsilon_k}) F_{\bar{K}}
    \left( e^{- \alpha \overline{\mathcal{E}}^{\varepsilon_k}
    (\bar{\phi}_{\varepsilon_k})} \right) \mathd
    \overline{\mathcal{E}}^{\varepsilon_k} ( \mu_{\varepsilon_k}^{- \alpha})
    \right|\\
    & \lesssim & \left| \alpha \left( \int_{e^{\alpha
    \bar{\phi}_{\varepsilon_k}} > \bar{K}} g_{\varepsilon_k} \left( e^{\alpha
    \bar{\phi}_{\varepsilon_k}} - \bar{K} \right) \mathd
    \mu_{\varepsilon_k}^{\alpha} - \int_{e^{- \alpha
    \bar{\phi}_{\varepsilon_k}} > \bar{K}} g_{\varepsilon_k} \left( e^{-
    \alpha \bar{\phi}_{\varepsilon_k}} - \bar{K} \right) \mathd
    \mu_{\varepsilon_k}^{- \alpha} \right) \right|\\
    &  & + \left| \alpha \int_{\mathbb{R}^2 \times \varepsilon \mathbb{Z}^2}
    g_{\varepsilon_k} F_{\bar{K}} \left( e^{\alpha \bar{\phi}_{\varepsilon_k}}
    \right) \mathd  \mu_{\varepsilon_k}^{\alpha} +
    \alpha \int_{\mathbb{R}^2 \times \varepsilon \mathbb{Z}^2}
    g_{\varepsilon_k} F_{\bar{K}} \left( e^{- \alpha
    \bar{\phi}_{\varepsilon_k}} \right) \mathd 
    \mu_{\varepsilon_k}^{- \alpha} \right|\\
    & \lesssim & 2 (\sigma + \upsilon) (\sup_k \| g_{\varepsilon_k}
    \|_{B^1_{\infty, \infty, - \ell} (\mathbb{R}^2 \times \varepsilon_k
    \mathbb{Z}^2)}) .
  \end{eqnarray*}
  Taking the limit in the order $\varepsilon_n \rightarrow + \infty$, $N
  \rightarrow + \infty$, $\bar{K} \rightarrow + \infty$, the left hand side in
  the last inequality converges to
  \[ \left| - \int_{\mathbb{R}^4} (\Delta_{\mathbb{R}^4} g) \bar{\phi} \mathd
     x \mathd z + m^2 \int_{\mathbb{R}^4} g \bar{\phi} \mathd x \mathd z +
     \alpha \int_{\mathbb{R}^4} g (e^{\alpha \bar{\phi}} \mathd \mu^{\alpha} -
     e^{- \alpha \bar{\phi}} \mathd \mu^{- \alpha}) \right| \leqslant 2
     (\sigma + \upsilon) \sup_{k \in \mathbb{N}} \| g_{\varepsilon_k} \|_{B^1_{\infty, \infty, -
     \ell}} . \]
Since $\sigma, \upsilon > 0$ and $g \in \mathcal{S}
  (\mathbb{R}^4)$ are arbitrary, we get the thesis.\\
   The fact that $e^{\pm \beta \alpha \bar{\phi}} \in H^1_{\ell}
  (\mathbb{R}^4)$ and $e^{\pm (1 + 2
  \beta) \alpha \bar{\phi}} \in L^1 (\rho_{\ell}^{(4)} \mathd
  \mu^{\pm \alpha})$ for any $0 \leqslant \beta < 1$ follows from the uniform in $\varepsilon_n$ bound \eqref{eq:inequalityepsilonsmalln}.
\end{proof}

\subsection{Improved a priori estimates and stochastic quantization}

\

In this section we want to prove the (elliptic) stochastic quantization of
the $\cosh (\beta \varphi)_2$ model. First we recall that, for any $x \in
\mathbb{R}^2$ and $f \in \mathcal{S} (\mathbb{R}^2)$, the tempered
distribution $\delta_x \otimes f \in \mathcal{S}' (\mathbb{R}^4)$ (where
$\delta_x \in \mathcal{S}' (\mathbb{R}^2)$ is the Dirac delta with unitary
mass in $x$) belongs to the Cameron-Martin space associated with the Gaussian
random distribution $W = (- \Delta_{\mathbb{R}^4 } + m^2)^{-1} (\xi)$, namely $\delta_x \otimes f \in H^{-2}
(\mathbb{R}^4)$. This means that for any $x \in \mathbb{R}^2$, the Gaussian
random distribution $W (x, \cdot)$ is well defined:
\begin{equation}
  _{\mathcal{S}' (\mathbb{R}^2)} \langle W (x, \cdot), f \rangle_{\mathcal{S}
  (\mathbb{R}^2)} \assign \langle W, \delta_x \otimes f \rangle \in L^2
  (\Omega), \label{eq:xix}
\end{equation}
where the pairing $\langle W, \delta_x \otimes f \rangle$ is understood in the
probabilistic sense of Wiener-Skorokhod integral (see, e.g.,
{\cite{Ustunel1995}}). Furthermore it is possible to build a version of the
map $x \mapsto W (x, \cdot) \in \mathcal{S}' (\mathbb{R}^2) \otimes L^2
(\Omega)$ which is continuous with respect to the natural topology of
$\mathcal{S}' (\mathbb{R}^2) \otimes L^2 (\Omega)$ (see Lemma 33 of
{\cite{AlDeGu2019}}). Hereafter we write $W (x, \cdot)$ for the continuous
version of the previous map satisfying the relation {\eqref{eq:xix}}.

\

In the remainder of this section we prove the following statement.

\begin{theorem}
  \label{theorem:stochasticquantization}For any $| \alpha | < 4 \pi$ and for
  any $x \in \mathbb{R}^2$ the random tempered distribution $\phi (x, \cdot) =
  \bar{\phi} (x, \cdot) + W (x, \cdot) \in \mathcal{S}'
  (\mathbb{R}^2)$, where $\bar{\phi}$ is the (unique) solution to equation
  {\eqref{eq:main5}} satisfying the hypotheses of Proposition
  \ref{proposition:uniqueness}, is well defined and distributed as
  $\nu_m^{\cosh, \beta}$ for $\beta = \frac{\alpha}{\sqrt{4 \pi}}$, where
  $\nu_m^{\cosh, \beta}$ is the probability measure associated with the $\cosh
  (\beta \varphi)_2$ model in the sense of Definition \ref{definition:coshmodel}.
\end{theorem}

In order to prove Theorem \ref{theorem:stochasticquantization} we use the
approximations $\bar{\phi}_{\varepsilon, N}$ and $\bar{\phi}_{\varepsilon}$
satisfying equations {\eqref{eq:mainfinitedimensional}} and
{\eqref{eq:maindiscrete}} respectively. The first link between the measure
$\nu_m^{\cosh, \beta}$ and the approximation $\bar{\phi}_{\varepsilon, N}$ is
given by the following lemma.

\begin{lemma}
  \label{lemma:approximationmeasure}There exists a unique solution
  $\bar{\phi}_{\varepsilon, N}$ to equation {\eqref{eq:mainfinitedimensional}}
  which is continuous in $x \in \mathbb{R}^2$ and grows at most exponentially
  at infinity. Furthermore for any $x \in \mathbb{R}^2$ we have
  \begin{equation}
    \overline{\mathcal{E}}^{\varepsilon} (\phi_{\varepsilon, N}) (x, \cdot)
    \sim \frac{e^{- V_N^{\cosh, \beta} (\varphi)}}{Z_{m, \varepsilon, N}}
    \nu_{m, \varepsilon, N} (\mathd \varphi) = \nu^{\cosh, \beta}_{m,
    \varepsilon, N}, \label{eq:lawfinitedimensional}
  \end{equation}
  where $\phi_{\varepsilon, N} = \bar{\phi}_{\varepsilon, N} + W_{\varepsilon,
  N}$
 
\end{lemma}

\begin{proof}
  The random field $\bar{\phi}_{\varepsilon, N}$ solves equation
  {\eqref{eq:mainfinitedimensional}} if and only if $\phi_{\varepsilon, N}
  (\cdot, z)$ solves the following (finite) set of elliptic SPDE on
  $\mathbb{R}^2$
  \begin{equation}
    (- \Delta_{\mathbb{R}^2} + m^2) \phi_{\varepsilon, N} (x, i \varepsilon) +
    \partial_{y^i} V_{\varepsilon} (\{ \phi_{\varepsilon, N} (x, i
    \varepsilon) \}) = \xi_{\varepsilon} (x, i \varepsilon) \quad i \in
    \left[ - \frac{N}{2}, \frac{N}{2} \right)^2 \label{eq:finitedimensional1},
  \end{equation}
  where $V_{\varepsilon} : \mathbb{R}^{N \times N} \rightarrow \mathbb{R}^2$
  is defined as
  \[ V_{\varepsilon} (y) = \sum_{i \in \left[ - \frac{N}{2}, \frac{N}{2}
     \right)^2} \left[ \frac{1}{2}\left( \frac{(y_i - y_{i + e_1})^2}{\varepsilon^2} +
     \frac{(y_i - y_{i + e_2})^2}{\varepsilon^2} \right) + 2 e^{- \alpha^2
     c_{\varepsilon}} \cosh (\alpha y_i) \right], \quad \{ y_i \}_{i \in
     \left[ - \frac{N}{2}, \frac{N}{2} \right)^2} \in \mathbb{R}^{N \times N},
  \]
  where $c_{\varepsilon} = \frac{1}{2} \mathbb{E} [W_{\varepsilon}^2]$. Since
  the function $V_{\varepsilon}$ is convex, by Theorem 4 of
  {\cite{AlDeGu2018}} there is a unique solution to equation
  {\eqref{eq:mainfinitedimensional}} which grows at most exponentially at
  infinity. Furthermore, for any $x \in \mathbb{R}^2$, the law of
  $\phi_{\varepsilon, N} (x, \cdot)$ satisfies relation
  {\eqref{eq:lawfinitedimensional}}.
\end{proof}

Thanks to Lemma \ref{lemma:approximationmeasure} what remains to prove is
that, for any $x \in \mathbb{R}^2$, $\bar{\phi}_{\varepsilon_n, N_k} (x,
\cdot) \rightarrow \bar{\phi}_{\varepsilon_n} (x, \cdot)$ and
$\overline{\mathcal{E}}^{\varepsilon_n} (\bar{\phi}_{\varepsilon_n}) (x,
\cdot) \rightarrow \bar{\phi} (x, \cdot)$ almost surely, as $N_k \rightarrow +
\infty$ and $\varepsilon_n \rightarrow 0$. The convergence results proved
above, namely the convergence of $\bar{\phi}_{\varepsilon_n, N_k}$ and
$\overline{\mathcal{E}}^{\varepsilon_n} (\bar{\phi}_{\varepsilon_n})$ to
$\bar{\phi}_{\varepsilon_n}$ and $\bar{\phi}$ in $H^1_{\ell} (\mathbb{R}^2
\times \varepsilon_n \mathbb{Z}^2)$ and $H^1_{\ell} (\mathbb{R}^4)$
respectively, is too weak for proving Theorem
\ref{theorem:stochasticquantization}. We need some a priori estimates stronger
than Theorem \ref{theorem_apriori1} and Theorem \ref{theorem_apriori2}.

\begin{theorem}
  \label{theorem:improvedapriori}Suppose that $\eta_{\pm} \in B^{-1 +
  \tau}_{2,2,\ell} (\mathbb{R}^2 \times \varepsilon \mathbb{Z}^2)$ (for any $\tau
  > 0$ and any $\ell > \ell_0$) and suppose that there are $r > 1$, $1
  \leqslant p' < 3$, $\theta, s > \frac{2}{p'}$ such that, writing $q' =
  \frac{p'}{p' - 1}$, we have
  \[ \mathbb{M}_{r, \ell, \varepsilon}^{\left( - \theta + \frac{2}{p'} \right)
     q', \left( - s + \frac{2}{p'} \right) q'} (\eta_{\pm}, \bar{a},
     \bar{\beta}) < + \infty \]
  (where $\bar{a} > 0$ and $0 < \bar{\beta} < 1$ are the constants appearing
  in Remarks \ref{remark:k:epsilon} and Remark \ref{remark:k:epsilon2}). If
  $\bar{\psi}$ is the unique solution to equation {\eqref{eq:apriori}}
  belonging to $H^1_{\ell} (\mathbb{R}^2 \times \varepsilon \mathbb{Z}^2)$,
  then, writing $p=r q'$, we have
  \[ \bar{\psi} \in B^{2 - \theta}_{p, p, \ell} (\mathbb{R}^2, B^{- s}_{p, p,\ell} (\varepsilon \mathbb{Z}^2)), \]
  and also, for $\tau > 0$ small enough and $\ell > 0$ big enough we have,
  \[ \begin{array}{lll}
       \| \bar{\psi} \|_{B^{2 - \theta}_{p, p, \ell} (\mathbb{R}^2, B^{-
       s}_{p, p, \ell} (\varepsilon \mathbb{Z}^2))} & \lesssim & P_{\theta,
       \alpha, r, \tau} \left. \left( \| \eta_{\pm} \|_{B^{-1 + \tau}_{\ell}
       (\mathbb{R}^2 \times \varepsilon \mathbb{Z}^2)}, \mathbb{M}_{r, \ell,
       \varepsilon}^{\left( - \theta + \frac{2}{p'} \right) q', \left( - s +
       \frac{2}{p'} \right) q'} (\eta_+, \bar{a}, \bar{\beta}) + \right.
       \right.\\
       &  & \left. \left. +\mathbb{M}_{r, \ell, \varepsilon}^{\left( - \theta
       + \frac{2}{p'} \right) q', \left( - s + \frac{2}{p'} \right) q'} (\eta_-,
       \bar{a}, \bar{\beta}) \right) \right.^{\frac{1}{q'}},
     \end{array} \]
  where $P_{\theta, \alpha, r, \tau}$ is a polynomial independent of $0 <
  \varepsilon \leqslant 1$.
\end{theorem}

\begin{proof}
  By Theorem \ref{theorem_apriori2}, we have that, for any $\delta > 0$, \[\|
  e^{\pm \alpha \psi} \|_{L^{3 - \delta}_{\ell'} (\mathd \eta_{\pm})} \leqslant P_{(1
  - \delta)} \left(\| \eta_+ \|_{B^{-1 + \tau}_{2, 2, \ell} (\mathbb{R}^2 \times
  \varepsilon \mathbb{Z}^2)}, \| \eta_- \|_{B^{-1 + \tau}_{2, 2, \ell}
  (\mathbb{R}^2 \times \varepsilon \mathbb{Z}^2)}, \int_{\mathbb{R}^2\times \varepsilon\mathbb{Z}^2}\rho_{\ell}^{(4)}\mathrm{d}\eta_{+},\int_{\mathbb{R}^2\times \varepsilon\mathbb{Z}^2}\rho_{\ell}^{(4)}\mathrm{d}\eta_{-}\right)\] for a suitable polynomial
  $P_{1 - \delta}$ (independent of $\varepsilon$).\\
  
  By Theorem \ref{thm:estimate-measure}, choosing $ p'=3-\delta$, we get that
  \[ \| e^{\pm \alpha \bar{\psi}} \eta_{\pm} \|_{B_{p, \ell,  \ell}^{-
     \theta, - s}} \lesssim \| e^{\pm \alpha \bar{\psi}} \|_{L^{3 -
     \delta}_{\ell} (\eta_{\pm})} \left( \mathbb{M}_{r, \ell,
     \varepsilon}^{\left( - \theta + \frac{2}{p'} \right) q', \left( - s +
     \frac{2}{p'} \right) q'} (\eta_{\pm}, \bar{a}, \bar{\beta}) \right)^{1 /
     q'} . \]
  Then, since
  \[ \bar{\psi} = (\Delta_{\mathbb{R}^2 \times \varepsilon \mathbb{Z}^2} +
     m^2)^{- 1} (\alpha e^{\alpha \bar{\psi}} \mathd \eta_+ - \alpha e^{-
     \alpha \bar{\psi}} \mathd \eta_-), \]
  by Theorem \ref{theorem:besov:rz:equivalent}, the thesis follows.
\end{proof}

\begin{remark}
  \label{remark:theta}Fix $\alpha \in \mathbb{R}$ such that
  $\frac{\alpha^2}{(4 \pi)^2} < 1$, then there exist $1 < p' < 3$, $\theta >
  \frac{2}{p'}$, $r > 1$, $2 > \frac{\alpha^2}{(4 \pi)^2} (r - 1)$ such that
  \begin{equation}\label{numerology-technical} 2 - \theta - \frac{2 (p' - 1)}{r p'} > 0, \quad \theta-\frac{2}{p'}<\frac{\alpha^2(r-1)}{(4\pi)^2 q'}. \end{equation}
  Indeed writing $\theta = \frac{2}{p'} + \delta$, we have $2 - \theta -
  \frac{2 (p' - 1)}{r p'} = 2 - \frac{2}{p'} - \delta - \frac{2 (p' - 1)}{r
  p'} = \left( 1 - \frac{1}{r} \right) \frac{2 (p' - 1)}{p'} - \delta$ which
  is positive if we choose $0 < \delta < \frac{2 (p' - 1)}{p'} \left( 1 -
  \frac{1}{r} \right)$. The second condition in \eqref{numerology-technical} is satisfied if we choose $\delta < \frac{\alpha^2(r-1)}{(4\pi)^2 q'}$
\end{remark}

\begin{proposition}
  \label{proposition:quantizationbounds}Let $\alpha \in \mathbb{R}$,
  $\frac{\alpha^2}{(4 \pi)^2} < 1$ and consider $p', r > 1$ and $\theta >
  \frac{2}{p'}$ satisfying the requirement of Remark \ref{remark:theta}, and,
  writing $q' = \frac{p'}{(p' - 1)}$ and $p = r q'$, and take
  \[ s_{\theta, \alpha, r} > \frac{\alpha^2(r - 1)}{(4 \pi)^2q'}  - \theta +
     \frac{4}{p'} . \]
  Then there is a sequence $\varepsilon_n = 2^{- h_n} \rightarrow 0$ as $n
  \rightarrow + \infty$ and $N_k \rightarrow + \infty$, such that, for any
  $\ell \geqslant 4$, we have
  \[ \sup_{\varepsilon_n, N_k} \left( \mathbb{M}_{r, \ell,
     \varepsilon_n}^{\left( - \theta + \frac{2}{p'} \right) q', \left( -
     s_{\theta, \alpha, r} + \frac{2}{p'} \right) q'} (\mu_{\varepsilon_n}^{\pm
     \alpha}, \bar{a}, \bar{\beta}), \mathbb{M}_{r, \ell,
     \varepsilon_n}^{\left( - \theta + \frac{2}{p'} \right) q', \left( -
     s_{\theta, \alpha, r} + \frac{2}{p'} \right) q'} (\mu_{\varepsilon_n,
     N_k}^{\pm \alpha}, \bar{a}, \bar{\beta}) \right) < + \infty \]
  almost surely.
\end{proposition}

\begin{proof}
  The proof is similar to the one of Corollary \ref{corollary-3.10}, thus we omit the details.
\end{proof}

\begin{lemma}
  \label{lemma:convergenceimproved}Let $\varepsilon_n$ and $N_k$ as in
  Proposition \ref{proposition:quantizationbounds}, then, for any
  $\varepsilon_n$, the sequence (in $N_k$) of measures $\nu^{\cosh, \beta}_{m,
  \varepsilon_n, N_k}$ on $\mathcal{S}' (\mathbb{R}^2)$ converges to a unique
  measure $\nu^{\cosh, \beta}_{m, \varepsilon_n}$. Furthermore, for any $x \in
  \mathbb{R}^2$, we have $\overline{\mathcal{E}}^{\varepsilon_n}
  (\bar{\phi}_{\varepsilon_n} + W_{\varepsilon_n}) (x, \cdot) =
  \overline{\mathcal{E}}^{\varepsilon_n} (\bar{\phi}_{\varepsilon_n}) (x,
  \cdot) + \bar{W}_{\varepsilon_n} (x, \cdot) \sim \nu^{\cosh, \beta}_{m,
  \varepsilon_n}$.
\end{lemma}

\begin{proof}
  By Theorem \ref{theorem:improvedapriori} and Proposition
  \ref{proposition:quantizationbounds}, if we choose $p, s_{\theta, \alpha,
  r}, \ell, \theta$ such as in Remark \ref{remark:theta} and Proposition
  \ref{proposition:quantizationbounds}, we have that(note that the second condition of \eqref{numerology-technical} and the choice of $s_{\theta,\alpha,r}$ from Remark \ref{remark:theta} imply $s_{\theta,\alpha,r}>\frac{2}{p'}$.) $\|
  \bar{\phi}_{\varepsilon_n, N_k} \|_{B^{2 - \theta}_{p, p, \ell}
  (\mathbb{R}^2, B^{- s_{\theta, \alpha, r}}_{p, p, \ell} (\varepsilon_n
  \mathbb{Z}^2))}, \| \bar{\phi}_{\varepsilon_n} \|_{B^{2 - \theta}_{p, p,
  \ell} (\mathbb{R}^2, B^{- s_{\theta, \alpha, r}}_{p, p, \ell} (\varepsilon_n
  \mathbb{Z}^2))} \leqslant C (\omega)$ (for some random variable $0 \leqslant
  C (\omega) < + \infty$ for almost all $\omega \in \Omega$) uniformly in
  $\varepsilon_n$ and $N_k$.
  
  This implies, by Theorem \ref{theorem:extension:lattice:negative}, that
  \begin{equation}
    \| \overline{\mathcal{E}}^{\varepsilon_n} (\bar{\phi}_{\varepsilon_n,
    N_k}) \|_{B^{2 - \theta}_{p, p, \ell} (\mathbb{R}^2, B^{- s_{\theta,
    \alpha, r} - \delta'}_{p, p, \ell} (\mathbb{R}^2))}, \|
    \overline{\mathcal{E}}^{\varepsilon_n} (\bar{\phi}_{\varepsilon_n})
    \|_{B^{2 - \theta}_{p, p, \ell} (\mathbb{R}^2, B^{- s_{\theta, \alpha, r}
    - \delta'}_{p, p, \ell} (\mathbb{R}^2))} \leqslant C' (\omega)
    \label{eq:improvedbound3}
  \end{equation}
  (for some random variable $0 \leqslant C (\omega) < + \infty$ for almost all
  $\omega \in \Omega$, and any $\delta' > 0$) uniformly in $\varepsilon_n$ and
  $N_k$. Since, by the proof of Lemma \ref{lemma:solutiondiscrete},
  $\bar{\phi}_{\varepsilon_n, N_k} \rightarrow \bar{\phi}_{\varepsilon_n}$
  almost surely in $H^{1 - \delta''}_{\ell} (\mathbb{R}^2 \times \varepsilon_n
  \mathbb{Z}^2)$, by Theorem \ref{theorem:extension3},
  $\overline{\mathcal{E}}^{\varepsilon_n} (\bar{\phi}_{\varepsilon_n, N_k})
  \rightarrow \overline{\mathcal{E}}^{\varepsilon_n}
  (\bar{\phi}_{\varepsilon_n})$ in $L^2_{\ell} (\mathbb{R}^4)$ almost surely.
  Thus, by the bound {\eqref{eq:improvedbound3}},
  $\overline{\mathcal{E}}^{\varepsilon} (\bar{\phi}_{\varepsilon_n, N_k})
  \rightarrow \overline{\mathcal{E}}^{\varepsilon_n}
  (\bar{\phi}_{\varepsilon_n})$ almost surely and (weakly) in $B^{2 -
  \theta}_{p, p, \ell} (\mathbb{R}^2, B^{- s_{\theta, \alpha, r} -
  \delta'}_{p, p, \ell} (\mathbb{R}^2))$. On the the other hand since $2 -
  \theta - \frac{2}{p} > 0$ (by the choice in Remark \ref{remark:theta}), by
  Besov embedding theorem (see Theorem 2.2.4 of {\cite{Amann2019}}), we have
  $B^{2 - \theta}_{p, p, \ell} (\mathbb{R}^2, B^{- s_{\theta, \alpha, r} -
  \delta'}_{p, p, \ell} (\mathbb{R}^2)) \subset C^0 (\mathbb{R}^2, B^{-
  s_{\theta, \alpha, r} - \delta'}_{p, p, \ell} (\mathbb{R}^2)^{})$, where
  $C^0 (\mathbb{R}^2, B^{- s_{\theta, \alpha, r} - \delta'}_{p, p, \ell}
  (\mathbb{R}^2)^{})$ denotes the Fr{\'e}chet space of (locally bounded)
  continuous functions from $\mathbb{R}^2$ into $B^{- s_{\theta, \alpha, r} -
  \delta'}_{p, p, \ell} (\mathbb{R}^2)^{}$, which implies that $\delta_x
  \otimes f \in (B^{2 - \theta}_{p, p, \ell} (\mathbb{R}^2, B^{- s_{\theta,
  \alpha, r} - \delta'}_{p, p, \ell} (\mathbb{R}^2)))^{\ast}$ for any $x \in
  \mathbb{R}^2$. From this we deduce that, for any $x \in \mathbb{R}^2$,
  $\overline{\mathcal{E}}^{\varepsilon_n} (\bar{\phi}_{\varepsilon_n, N_k})
  (x, \cdot) \rightarrow \overline{\mathcal{E}}^{\varepsilon_n}
  (\bar{\phi}_{\varepsilon_n}) (x, \cdot)$ almost surely and (weakly) in
  $\mathcal{S}' (\mathbb{R}^2)$.\\
  
  Furthermore, using a proof similar to the one of Theorem
  \ref{theorem:Lpconvergence}, it is easy to prove that
  $\overline{\mathcal{E}}^{\varepsilon_n} (W_{\varepsilon_n, N_k}) (x, \cdot)
  \rightarrow \overline{\mathcal{E}}^{\varepsilon_n} (W_{\varepsilon_n}) (x,
  \cdot)$ in probability and weakly in $\mathcal{S}' (\mathbb{R}^2)$. This
  means that, since, by Lemma \ref{lemma:approximationmeasure},
  $\overline{\mathcal{E}}^{\varepsilon_n} (W_{\varepsilon_n, N_k}) (x, \cdot)
  + \overline{\mathcal{E}}^{\varepsilon_n} (\bar{\phi}_{\varepsilon_n, N_k})
  (x, \cdot) \sim \nu^{\cosh, \beta}_{m, \varepsilon_n, N_k}$, that
  $\nu^{\cosh, \beta}_{m, \varepsilon_n, N_k}$ is tight in $\mathcal{S}'
  (\mathbb{R}^2)$ and it converges to a (unique) measure $\nu^{\cosh,
  \beta}_{m, \varepsilon_n}$ which is the probability law of the random
  distribution $\overline{\mathcal{E}}^{\varepsilon_n}
  (\bar{\phi}_{\varepsilon_n} + W_{\varepsilon_n}) (x, \cdot)$.
\end{proof}

\begin{proof*}{Proof of Theorem \ref{theorem:stochasticquantization}}
  The proof follows the same lines of the proof of Lemma
  \ref{lemma:convergenceimproved}. More precisely, by inequality
  {\eqref{eq:improvedbound3}}, since by Theorem
  \ref{theorem:convergenceequation1} $\overline{\mathcal{E}}^{\varepsilon_n}
  (\bar{\phi}_{\varepsilon_n}) \rightarrow \bar{\phi}$ almost surely in
  $\mathcal{S}' (\mathbb{R}^4)$, $\overline{\mathcal{E}}^{\varepsilon_n}
  (\bar{\phi}_{\varepsilon_n})$ converges to $\bar{\phi}$ almost surely and
  weakly in $B^{2 - \theta}_{p, p, \ell} (\mathbb{R}^2, B^{- s_{\theta,
  \alpha, r} - \delta'}_{p, p, \ell} (\mathbb{R}^2))$. Furthermore, since
  $\delta_x \otimes f \in (B^{2 - \theta}_{p, p, \ell} (\mathbb{R}^2, B^{-
  s_{\theta, \alpha, r} - \delta'}_{p, p, \ell} (\mathbb{R}^2)))^{\ast}$ for
  any $x \in \mathbb{R}^2$ and any $f \in \mathcal{S} (\mathbb{R}^2)$, the
  random distribution $\overline{\mathcal{E}}^{\varepsilon_n}
  (\bar{\phi}_{\varepsilon_n}) (x, \cdot)$ converges to $\bar{\phi} (x,
  \cdot)$ almost surely and in $\mathcal{S}' (\mathbb{R}^2)$. By Theorem
  \ref{theorem:Lpconvergence}, we have that
  $\overline{\mathcal{E}}^{\varepsilon_n} (W_{\varepsilon_n}) (x, \cdot)
  \rightarrow W (x, \cdot)$ in $L^2 (\Omega)$ and in $\mathcal{S}'
  (\mathbb{R}^2)$. By Lemma \ref{lemma:convergenceimproved} $\nu^{\cosh,
  \beta}_{m, \varepsilon_n} \sim \overline{\mathcal{E}}^{\varepsilon_n}
  (\bar{\phi}_{\varepsilon_n} + W_{\varepsilon_n}) (x, \cdot)$, which, by the
  previous convergences, implies that $\nu^{\cosh, \beta}_{m, \varepsilon_n}$
  is tight in $\mathcal{S}' (\mathbb{R}^2)$ and it converges, as $n
  \rightarrow + \infty$, to a measure $\nu^{\cosh, \beta}_m$ on $\mathcal{S}'
  (\mathbb{R}^2)$ which is the law of the random field $W (x, \cdot) +
  \bar{\phi} (x, \cdot)$. The measure $\nu^{\cosh, \beta}_m$ is unique due to
  the uniqueness of the law of the free field $W$ and the uniqueness of the
  solution to equation {\eqref{eq:main5}}.
\end{proof*}

\appendix\section{Osterwalder-Schrader Axioms using stochastic
quantization}\label{appendix:axioms}

In this appendix we will show that the measure we constructed satisfies the
Osterwalder-Schrader axioms, using the stochastic quantization equation
introduced above. The study of the properties of the measure $\nu^{\cosh,
\beta}_m$ presented here are inspired by the analogous stochastic quantization
proofs for the $\varphi^4_3$ measure in
{\cite{albeverio2021construction,GuHof2018,hairer2021phi34}}.\\

Let us recall here briefly the mentioned axioms for the convenience of the
reader: We take the version of the axioms, in the case of two dimensional
quantum fields, used, e.g., in {\cite{Glimm_Jaffe_book,Simon_phi2}} and
related to the moments of a measure rather than the original ones proposed by
Osterwalder and Schrader in
{\cite{OsterwalderSchrader1,OsterwalderSchrader2}}. Hereafter, if $\nu$ is a
probability measure on $\mathcal{S}' (\mathbb{R}^2)$, we define the
distributions $S_n \in \mathcal{S}' (\mathbb{R}^{2 n})$ as
\begin{equation}
  S_n (f_1 \otimes \cdots \otimes f_n) = \int \langle \varphi, f_1
  \rangle_{\mathcal{S}', \mathcal{S}} \cdots \langle \varphi, f_n
  \rangle_{\mathcal{S}', \mathcal{S}} \mathd \nu (\varphi), \label{eq:npoint}
\end{equation}
usually called $n$-point functions of the measure $\nu$.

\begin{axiom}[Regularity]
  \label{axiom:regularity}$S_0 = 1$ and there exists a Schwartz semi-norm $\|
  \cdot \|_s, \beta > 1, K > 0$ such that
  \[ | S_n (f_1 \otimes \ldots \otimes f_n) | \leqslant K^n (n!)^{\beta}
     \prod_{i = 1}^n \| f_i \|_s. \]
\end{axiom}

\begin{axiom}[Euclidean Invariance]
  \label{axiom:euclidean}Let the Euclidean group with $G = (R, a)$ $R \in O
  (2), a \in \mathbb{R}^2$ act on functions by
  \[ (G f) (x) = f (R x - a) . \]
  Then
  \[ S_n (G f_1 \otimes \ldots \otimes G f_n) = S_n (f_1 \otimes \ldots
     \otimes f_n) . \]
\end{axiom}

\begin{axiom}[Reflection Positivity]
  \label{axiom:reflection-pos}Let $\mathbb{R}_+^{2 n}$ be the set
  \[ \left\{ x \in \mathbb{R}^{2 n} : x = (x_1, \ldots ., x_n) \tmop{and} x_i
     = (\tau_i, y_i)  \text{ with } \tau_i > 0 \text{ and } y_i \in \mathbb{R}
     \right\}. \]
  Furthermore define the reflection $\Theta (x) = \Theta ((\tau, y)) = (-
  \tau, y)$ and its action on a function $f \in \mathcal{S}' (\mathbb{R}^{2
  n})$ by
  \[ \Theta f (x_1, \ldots, x_n) = f (\Theta x_1, \ldots ., \Theta x_n) . \]
  Now we require that for all finite families $\{ f_n \in \mathcal{S}
  (\mathbb{R}^{2 n}) \}_{n \leqslant M}$ such that $\tmop{supp} f_n \subseteq
  \mathbb{R}_+^{2 n}$ we have
  \[ \sum_{i, j = 1}^M S_{i + j} (\Theta f_i \otimes f_j) \geqslant 0_{} . \]
\end{axiom}

\begin{axiom}[Symmetry]
  \label{axiom:symmetry}Let $\pi$ be a permutation of $\{ 1, \ldots ., n \}$.
  Then
  \[ S_n (f_{\pi (1)} \otimes \ldots \otimes f_{\pi (n)}) = S_n (f_1 \otimes
     \ldots \otimes f_n). \]
\end{axiom}

\begin{axiom}[Clustering]
  \label{axiom:clustering}For any $j, n \in \mathbb{N}$, $1 \leqslant j
  \leqslant n$ \ and $a \in \mathbb{R}^2 $ such that $| a | = 1$ as $\lambda
  \rightarrow \infty$
  \[ S_n (f_1 \otimes \ldots \otimes f_j \otimes f_{j + 1} (\cdot + \lambda
     a) \otimes \ldots \otimes f_n (\cdot + \lambda a)) \rightarrow S_j (f_1
     \otimes \ldots \otimes f_j) S_{n - j} (f_{j + 1} \otimes \ldots \otimes
     f_n) . \]
  If furthermore there exists $m > 0$ such that for any family of $f_j \in
  C^{\infty}_c (\mathbb{R}^2)$ there exists a constant $C > 0$ such that
  \[ | S_2 (f_1 \otimes f_2 (\cdot + \lambda a)) - S_1 (f_1) S_1 (f_2) |
     \leqslant C e^{- m | \lambda |}, \]
  We say that the clustering is exponential.
\end{axiom}

\begin{remark}
  Exponential clustering implies clustering by Theorem VIII.36 in
  {\cite{Simon_phi2}}.
\end{remark}

In what fallows $\beta$ will always be a (positive) constant such that
$\frac{\beta^2}{(4 \pi)} = \frac{\alpha^2}{(4 \pi)^2} < 1$. Furthermore from
now on we denote by $\varepsilon_n \rightarrow 0$ and $N_r \rightarrow +
\infty$ two sequences satisfying the thesis of Proposition
\ref{proposition:convergencealpha} and of Proposition
\ref{proposition:quantizationbounds}.

\begin{theorem}
  The $n$-point functions of the probability measure $\nu_m^{\cosh, \beta}$,
  (where $\nu^{\cosh, \beta}_m$ is introduced in Definition
  \ref{definition:coshmodel} and built in Section
  \ref{sec:stochastic-quantization}) obtained by replacing $\nu$ by
  $\nu_m^{\cosh, \beta}$ in equation {\eqref{eq:npoint}}, satisfy the
  Osterwalder-Schrader axioms with exponential clustering.
\end{theorem}

\begin{proof}
  Axiom \ref{axiom:euclidean} follows from the fact that the solution to
  Equation {\eqref{eq:main5}} is unique, and therefore invariant under the
  Euclidean group, since the equation itself is invariant. Axiom
  \ref{axiom:reflection-pos} follows from the fact that $\nu_m^{\cosh, \beta}$
  is a weak limit of reflection positive measures. Axiom \ref{axiom:symmetry}
  is obvious. It remains to show Axiom \ref{axiom:regularity} and exponential
  clustering. This is the content of the next two subsections. 
\end{proof}

\subsection{Analyticity}

The aim of this section is to prove the following

\begin{theorem}
  \label{thm:analyticity}There exist $\ell, \delta, \gamma > 0$ such that
  \[ \int_{B^{- 2 \delta}_{2, 2, 2 \ell} (\mathbb{R}^2)} \exp (\gamma \|
     \varphi \|^{2}_{B^{-1}_{2, 2, \ell} (\mathbb{R}^2)}) \mathd \nu_m^{\cosh,
     \beta}  (\varphi) < \infty, \]
  where $\nu_m^{\cosh, \beta}$ is the measure built in Theorem
  \ref{theorem:stochasticquantization}.
\end{theorem}

We will follow the strategy developed in {\cite{hairer2021phi34}}, for the
case of $\varphi^4_3$ model. The proof is postponed until the end of this
subsection, while we first collect some assertions needed for it. We introduce
the notation
\begin{equation}
  \mathd \tilde{\nu}_{N, \varepsilon}^{\gamma} (\varphi) =
  \frac{1}{Z^{\varepsilon, N, \gamma}} \exp (\gamma \|
  \mathcal{E}^{\varepsilon} (\varphi) \|^2_{B_{2, 2, \ell}^{- 1}
  (\mathbb{R}^2)}) \mathd \tilde{\nu}_{m, \varepsilon, N}^{\cosh, \beta}  (\varphi),
  \label{eq:measure1}
\end{equation}
where
\begin{equation}
  Z^{\varepsilon, N, \gamma} = \int_{B^{- \delta}_{2, 2, \ell} (\mathbb{R}^2)}
  \exp (\gamma \|
  \mathcal{E}^{\varepsilon} (\varphi) \|^2_{B_{2, 2, \ell}^{- 1}
  (\mathbb{R}^2)}. ) \mathd
  \tilde{\nu}^{\cosh, \beta}_{m, \varepsilon, N} (\mathd \varphi), \label{eq:Z}
\end{equation}
the measure $\nu^{\cosh, \beta}_{m, \varepsilon, N}$ being defined in equation
{\eqref{eq:lawfinitedimensional}}, and $\delta > 0$ being small enough. We
also introduce the operator $\overline{\mathcal{D}^{\varepsilon}} : L^p_{\ell}
(\mathbb{R}^4) \rightarrow L^p_{\ell} (\mathbb{R}^2 \times \varepsilon
\mathbb{Z}^2)$ defined by
\[ \overline{\mathcal{D}^{\varepsilon}} (f) = (I \otimes
   \mathcal{D}^{\varepsilon}) (f), \quad f \in L^p_{\ell} (\mathbb{R}^4), \]
where $\mathcal{D}^{\varepsilon}$ is the operator introduced in equation
{\eqref{eq:operatorD}}. Then we have the following proposition:

\begin{proposition}
  \label{prop:dim-red-int}For $\gamma > 0$ small enough we have that
  \[ \tilde{\nu}_{N, \varepsilon}^{\gamma} = \tmop{Law} (W_{\varepsilon, N}
     (0, \cdot) + \phi_{\varepsilon, N}^{\gamma} (0, \cdot)), \]
  where $\phi_{\varepsilon, N}^{\gamma} \in B_{2, 2, \ell}^{- \delta}
  (\mathbb{R}^2 \times \varepsilon \mathbb{Z}^2)$ is the unique solution to
  the equation
  \begin{equation}
    \begin{array}{ll}
      & (- \Delta_{\mathbb{R}^2 \times \varepsilon \mathbb{Z}^2} + m^2)
      \phi_{\varepsilon, N}^{\gamma} + \alpha e^{\alpha \phi_{N,
      \varepsilon}^{\gamma}} \mathd \mu^{\alpha}_{\varepsilon, N} - \alpha
      e^{- \alpha \phi_{N, \varepsilon}^{\gamma}} \mathd \mu^{-
      \alpha}_{\varepsilon, N}\\
      = & \gamma \overline{\mathcal{D}}^{\varepsilon} (\rho_{\ell} (z) (1 -
      \Delta_{\mathbb{R}^2, z})^{- 1} (\bar{W}_{\varepsilon, N})) + \gamma
      \overline{\mathcal{D}}^{\varepsilon} (\rho_{\ell} (z) (1 -
      \Delta_{\mathbb{R}^2, z})^{- 1} (\overline{\mathcal{E}}^{\varepsilon}
      (\phi_{\varepsilon, N}^{\gamma}))),
    \end{array} \label{eq:dim-red}
  \end{equation}
  where $\Delta_{\mathbb{R}^2, z}$ is the Laplacian with respect to the second
  set of variables $z \in \mathbb{R}^2$, and $\bar{W}_{\epsilon,N}:=\mathcal{E}^{\varepsilon}W_{\epsilon,N}$.
\end{proposition}

\begin{proof}
  This is a consequence of Theorem 4 in {\cite{AlDeGu2018}} in the same way as
  Lemma \ref{lemma:approximationmeasure}, since for $\gamma$ small enough we
  have that
  \[ (\| \nabla \varphi \|^2_{L_{\ell}^2 (\varepsilon \mathbb{Z}^2)} + m^2 \|
     \varphi \|_{L^2_{\ell} (\varepsilon \mathbb{Z}^2)}^2) - \gamma \| (1 -
     \Delta_{\mathbb{R}^2})^{- 1 / 2} \mathcal{E}^{\varepsilon} (\varphi)
     \|^2_{L_{\ell}^2 (\mathbb{R}^2)} \]
  is (strictly) convex. Equation {\eqref{eq:dim-red}} follows from the fact
  that at $t=0$
  \[ \frac{\gamma}{2} \frac{\mathd}{\mathd t} \| (1 -
     \Delta_{\mathbb{R}^2})^{- 1 / 2} \mathcal{E}^{\varepsilon} (\varphi +
     t \psi) \|^2_{L_{\ell}^2 (\mathbb{R}^2)} = \gamma
     \int_{\varepsilon \mathbb{Z}^2} \mathcal{D}^{\varepsilon} \left(
     \rho_{\ell} (z) \left( 1 - \Delta_{\mathbb{R}^{^2}} \right)^{- 1}
     \mathcal{E}^{\varepsilon} (\varphi) \right) \psi (z) \mathd z. \]
  
\end{proof}

\begin{proposition}
  \label{prop:a-priopi-integrability}Let $\phi_{ \varepsilon,N}^{\gamma}$ be
  the solution to equation {\eqref{eq:dim-red}} defined by Proposition
  \ref{prop:dim-red-int}. Then for $\ell$ large enough and $\gamma$ small
  enough we have
  \[ \sup_{N_r, \varepsilon_n} \; \mathbb{E} [\| \phi_{\varepsilon_n,
     N_r}^{\gamma} \|^2_{H_{\ell}^1 (\mathbb{R}^2 \times \varepsilon_n
     \mathbb{Z}^2)}] < + \infty . \]
\end{proposition}

\begin{proof}
  We observe that, by the continuity of the operators
  $\overline{\mathcal{E}}^{\varepsilon}$ and
  $\overline{\mathcal{D}}^{\varepsilon}$ (see Lemma
  \ref{lemma:discretization:operator} and Theorem \ref{theorem:extension3}),
  \[ \gamma \int_{\mathbb{R}^2 \times \varepsilon_n \mathbb{Z}^2} \rho_{\ell}
     (x, z) \overline{\mathcal{D}}^{\varepsilon_n} (\rho_{\ell} (z) (1 -
     \Delta_{\mathbb{R}^2, z})^{- 1} \mathcal{E}^{\varepsilon_n}
     (\phi_{\varepsilon_n, N_r}^{\gamma})) \phi_{\varepsilon_n,
     N_r}^{\gamma} \mathd x \mathd z \lesssim \gamma \|
     \phi_{\varepsilon_n, N_r}^{\gamma} \|^2_{L_{\ell}^2 (\mathbb{R}^2
     \times \varepsilon_{n} \mathbb{Z}^2)} \]
  and, thus,
  \begin{eqnarray*}
    &  & \gamma \int_{\mathbb{R}^2 \times \varepsilon_{n} \mathbb{Z}^2}
    \rho_{\ell} (x, z) \overline{\mathcal{D}}^{\varepsilon_{n}} (\rho_{\ell} (z)
    \cdot (1 - \Delta_{\mathbb{R}^2, z})^{- 1}
    \overline{\mathcal{E}}^{\varepsilon_{n}} (W_{\varepsilon_n, N_r}))
    \bar{\phi}_{\varepsilon_n, N_r}^{\gamma} \mathd x \mathd z\\
    & \leqslant & \gamma \| W_{\varepsilon_n, N_r} \|_{H_{2, 2, \ell / 2}^{-
    1} (\mathbb{R}^2 \times \varepsilon_n \mathbb{Z}^2)} \|
    \phi_{\varepsilon_n, N_r}^{\gamma} \|_{B_{2, 2, \ell / 2}^1
    (\mathbb{R}^2 \times \varepsilon_{n} \mathbb{Z}^2)}\\
    & \leqslant & \gamma (\| W_{\varepsilon_n, N_r} \|^2_{H_{2, 2, \ell /
    2}^{- 1} (\mathbb{R}^2 \times \varepsilon_n \mathbb{Z}^2)} + \|
    \phi_{\varepsilon_n, N_r}^{\gamma} \|^2_{B_{2, 2, \ell / 2}^1
    (\mathbb{R}^2 \times \varepsilon_{n}\mathbb{Z}^2)}) .
  \end{eqnarray*}
  Using this observation we can repeat the proof of Theorem
  \ref{theorem_apriori1} to obtain the statement.
\end{proof}

\begin{proposition}
  \label{prop:tightness-gamma}The measures $\bar{\nu}^{\gamma}_{\varepsilon_n,
  N_r} = (\overline{\mathcal{E}}^{\varepsilon_{n}})^{\ast}
  (\tilde{\nu}^{\gamma}_{\varepsilon_n, N_r})$, where
  $\tilde{\nu}^{\gamma}_{\varepsilon_n, N_r}$ is defined by equation
  {\eqref{eq:measure1}}, are tight on $B_{2, 2, 2 \ell}^{- 2 \delta}
  (\mathbb{R}^2)$ (where $\delta > 0$).
\end{proposition}

\begin{proof}
  By Proposition \ref{prop:dim-red-int} we have
  \begin{equation}
    \begin{array}{ll}
      & \int_{B_{2, 2, 2 \ell}^{- 2 \delta} (\mathbb{R}^2)} \|
      \varphi \|^2_{B_{2, 2, \ell / 2}^{-
      \delta} (\mathbb{R}^2)} \mathd \bar{\nu}_{\varepsilon_n,
      N_r}^{\gamma}\\
      \lesssim & \int_{B_{2, 2, 2 \ell}^{- 2 \delta} (\mathbb{R}^2 \times \varepsilon_{n} \mathbb{Z}^2)} \|
      \mathcal{E}^{\varepsilon_{n}} (\varphi) \|^2_{B_{2, 2, \ell / 2}^{- \delta}
      (\mathbb{R}^2)} \mathd \tilde{\nu}_{\varepsilon_n, N_r}^{\gamma}\\
      \lesssim & \int_{B_{2, 2, 2 \ell}^{- 2 \delta} (\mathbb{R}^2 \times \varepsilon \mathbb{Z}^2)} \|
      \rho_{\ell / 2} (m^2 - \Delta_{\varepsilon_n \mathbb{Z}^2})^{- \delta/2}
      \varphi \|^2_{L^2 (\varepsilon_n \mathbb{Z}^2)} \mathd
      \tilde{\nu}_{\varepsilon_n, N_r}^{\gamma}\\
      \lesssim & \int_{\varepsilon_n \mathbb{Z}^2} \rho_{\ell} (z) \mathbb{E}
      [((m^2 - \Delta_{\varepsilon_n \mathbb{Z}^2})^{- \delta / 2}
      W_{\varepsilon_n, N_{r}} (0, z))^2] \mathd z\\
      & + \int_{\varepsilon_n \mathbb{Z}^2} \rho_{\ell} (z) \mathbb{E} [((m^2
      - \Delta_{\varepsilon_n \mathbb{Z}^2})^{- \delta / 2}
      \phi_{\varepsilon_n, N_{r}}^{\gamma} (0, z))^2] \mathd z.
    \end{array} \label{eq:Max1}
  \end{equation}
  The first term is uniformly bounded in $\varepsilon_n, N_{r}$ by well known
  properties of the Gaussian free field. For the second term we observe that
  since $\phi_{\varepsilon_n, N_{r}}^{\gamma} (0, z)$ is the unique
  solution to {\eqref{eq:dim-red}} we have by invariance of the law of
  $W_{\varepsilon_n, N_{r}} (0, \cdot)$ and of equation {\eqref{eq:dim-red}}
  under shifts in the $x$ coordinate, that, for some constant $C > 0$:
  \begin{equation}
    \begin{array}{ll}
      & \int_{\varepsilon_n \mathbb{Z}^2} \rho_{\ell} (z) \mathbb{E} [((m^2 -
      \Delta_{\varepsilon_n \mathbb{Z}^2})^{- \delta / 2}
      \phi_{\varepsilon_n, N_{r}}^{\gamma} (0, z))^2] \mathd z\\
      \leqslant & \frac{1}{\int \rho_{\ell} (x) \mathd x} \int_{\mathbb{R}^2
      \times \varepsilon_{n} \mathbb{Z}^2} \rho_{\ell} (z) \rho_{\ell} (x)
      \mathbb{E} [((m^2 - \Delta_{\varepsilon_{n} \mathbb{Z}^2})^{- \delta / 2}
      \phi_{\varepsilon_n, N_{r}}^{\gamma} (x, z))^2] \mathd x \mathd z\\
      \leqslant & C,
    \end{array} \label{eq:Max2}
  \end{equation}
  where in the last line we have used that $\rho_{\ell / 2} (z) \rho_{\ell /
  2} (x) \lesssim \rho_{\ell} (x, z)$ and Proposition
  \ref{prop:a-priopi-integrability}. This implies tightness on $B_{2, 2, 2
  \ell}^{- 2 \delta} (\mathbb{R}^2)$ by the compactness of the Besov
  embedding. 
\end{proof}

It will be helpful to have the following lemma.

\begin{lemma}
  \label{lemma:Max}Let $(\Omega, \mathcal{F})$ be a measurable space, let
  $\nu$ be a probability measure defined on $\Omega$, and let $S : \Omega
  \rightarrow \mathbb{R}_{}$ be a measurable function such that $e^{S (x)} \in
  L^1 (\Omega, \nu)$ then
  \[ \int_{\Omega} e^{S (x)} \nu (\mathd x) \leqslant e^{\int S (x) \nu_S
     (\mathd x)}, \]
  where $\nu_S$ is the probability measure defined on $\Omega$, for which
  $\nu_S (\mathd x) = \frac{e^{S (x)}}{\int_{\Omega} e^{S (x)} \nu (\mathd x)}
  \nu (\mathd x)$.
\end{lemma}

\begin{proof}
  By definition of $\nu_S$ and using the fact that $\nu$ is a probability
  measure we get
  \[ \left( \int_{\Omega} e^{S (x)} \nu (\mathd x) \right) \left(
     \int_{\Omega} e^{- S (y)} \nu_S (\mathd y) \right) = 1. \]
  Applying Jensen inequality to the term $\left( \int_{\Omega} e^{- S (y)}
  \nu_S (\mathd y) \right)$ we get the thesis.
\end{proof}

\begin{proof*}{Proof of Theorem \ref{thm:analyticity}\footnote{this argument
is due to M. Gubinelli (personal communication).}}
  Since, by Proposition \ref{prop:tightness-gamma}, the measures
  $\bar{\nu}^{\gamma}_{\varepsilon_n, N_r}$ are tight, what remains to show is
  that the constants $Z^{\varepsilon, N, \gamma}$ defined by equation
  {\eqref{eq:Z}} satisfy
  \[ \sup_{\varepsilon_n, N_r} Z^{\varepsilon_n, N_r, \gamma} < \infty . \]
  If we apply Lemma \ref{lemma:Max} to the measures
  $\bar{\nu}^{\gamma}_{\varepsilon_n, N_r}$ and $\nu_{m, \varepsilon_n,
  N_r}^{\cosh, \beta}$ (which means taking $S (\varphi) = \|
  \mathcal{E}_{\varepsilon_n} (\varphi) \|^2_{B^{- 1}_{2, 2, \ell}
  (\mathbb{R}^2)}$, we get
  \[ Z^{\varepsilon_n, N_r, \gamma} \leqslant \exp \left( \int_{B^{- 2
     \delta}_{2, 2, \ell} (\mathbb{R}^2)} \| \mathcal{E}_{\varepsilon_n}
     (\varphi) \|^2_{B^{- 1}_{2, 2, \ell} (\mathbb{R}^2)}
     \nu_{\varepsilon_n, N_r}^{\gamma} (\mathd \varphi) \right) . \]
  On the other hand by inequalities {\eqref{eq:Max1}} and {\eqref{eq:Max2}} we
  have that
  \[ \sup_{\varepsilon_n, N_r} \int_{B^{- 2 \delta}_{2, 2, \ell}
     (\mathbb{R}^2)} \| \mathcal{E}_{\varepsilon_n} (\varphi) \|^2_{B^{- 1}_{2, 2, \ell} (\mathbb{R}^2)} \nu_{\varepsilon_n, N_r}^{\gamma}
     (\mathd \varphi) \leqslant C < + \infty . \]
  From this we get the thesis.
\end{proof*}

\subsection{Clustering}

In this section we prove the exponential decay of the two point function
$S_2$. More precisely we prove the following theorem.

\begin{theorem}
  \label{clustering-exp}There exist some constants $C, \kappa > 0$ such that
  for any $\zeta, \psi \in H^1 (\mathbb{R}^2)$ such that $\tmop{supp} (\psi)
  \bigcup \tmop{supp} (\zeta) \subset B (0, 1)$ we have that
  \[ \int_{B^{- 2 \delta}_{2, 2, 2 \ell} (\mathbb{R}^2)} \langle \zeta,
     \varphi \rangle_{L^2 (\mathbb{R}^2)} \langle \psi (\cdot - x), \varphi
     \rangle_{L^2 (\mathbb{R}^2)} \mathd \nu_m^{\cosh, \beta}  (\varphi)
     \leqslant C \exp (- \kappa | x |) \| \psi \|_{H^1 (\mathbb{R}^2)} \| \zeta
     \|_{H^1 (\mathbb{R}^2)} . \]
\end{theorem}

Note that since $\cosh (\beta \varphi)$ is even and the law of the free field
is invariant with respect to reflection $\varphi \mapsto - \varphi$, we have
that
\[ \int_{B^{- 2 \delta}_{2, 2, 2 \ell} (\mathbb{R}^2)} \langle \zeta, \varphi
   \rangle_{L^2 (\mathbb{R}^2)} \mathd \nu_m^{\cosh, \beta}  (\varphi) =
   \int_{B^{- 2 \delta}_{2, 2, 2 \ell} (\mathbb{R}^2)} \langle \psi (\cdot -
   x), \varphi \rangle_{L^2 (\mathbb{R}^2)} \mathd \nu_m^{\cosh, \beta} 
   (\varphi) = 0. \]
Again we postpone the proof until the end of the section. First let us
introduce the measure
\[ \mathd \bar{\nu}_{\varepsilon, N}^{\gamma \zeta} = \frac{1}{Z^{\gamma
   \zeta}} \exp (- \gamma \langle \mathcal{D}^{\varepsilon} \zeta, \varphi
   \rangle_{L^2 (\mathbb{R}^2)}) \mathd \tilde{\nu}_{m, \varepsilon, N}^{\cosh, \beta}
   (\varphi) . \]
where
\[ Z^{\varepsilon, N, \gamma \zeta} = \int_{B^{- 2 \delta}_{2, 2, 2 \ell}
   (\varepsilon \mathbb{Z}^2)} \exp (- \gamma \langle
   \mathcal{D}^{\varepsilon} \zeta, \varphi \rangle_{L^2  (\mathbb{R}^2)})
   \mathd \tilde{\nu}_{m, \varepsilon, N}^{\cosh, \beta}  (\varphi). \]
This is a well defined measure since $\langle \mathcal{D}^{\varepsilon}
(\zeta), \cdot \rangle_{L^2  (\varepsilon \mathbb{Z}^2)}$ is continuous on
$H_{\tmop{loc}}^{- 1} (\varepsilon \mathbb{Z}^2)$ which contains the support
of $\tilde{\nu}_m^{\cosh, \beta}$. Immediately from the definitions we get the
following lemma:

\begin{lemma}
  \label{lemma:derivative}Under the previous hypotheses we have
  
  \begin{align*}
          &\int_{B^{- 2 \delta}_{2, 2, 2 \ell} (\varepsilon \mathbb{Z}^2)} \langle
    \mathcal{D}^{\varepsilon} \zeta, \varphi \rangle_{L^2 (\varepsilon
    \mathbb{Z}^2)} \langle \mathcal{D}^{\varepsilon} \psi (\cdot - x), \varphi
    \rangle_{L^2 (\varepsilon \mathbb{Z}^2)} \mathd \tilde{\nu}_{m, \varepsilon,
    N}^{\cosh, \beta}  (\varphi) \nonumber\\
    =&- \frac{\mathd}{\mathd \gamma} \left[ \int_{B^{- 2 \delta}_{2, 2, 2 \ell}
    (\varepsilon \mathbb{Z}^2)} \langle \mathcal{D}^{\varepsilon} \psi (\cdot
    - x), \varphi \rangle_{L^2 (\varepsilon \mathbb{Z}^2)} \mathd
    \bar{\nu}_{\varepsilon, N}^{\gamma \zeta} (\varphi) \right]_{\gamma = 0} .
    \nonumber
\end{align*}  
  Furthermore, if $\varepsilon_n$ and $N_r$ are the sequences in Proposition
  \ref{proposition:quantizationbounds}, we have
  \begin{eqnarray*}
    &  & \lim_{\varepsilon_n \rightarrow 0, N_r \rightarrow \infty}
    \int_{B^{- 2 \delta}_{2, 2, 2 \ell} (\varepsilon_n \mathbb{Z}^2)} \langle
    \mathcal{D}^{\varepsilon_n} \zeta, \varphi \rangle_{L^2 (\varepsilon_n
    \mathbb{Z}^2)} \langle \mathcal{D}^{\varepsilon_n} \psi (\cdot - x),
    \varphi \rangle_{L^2 (\varepsilon_n \mathbb{Z}^2)} \mathd \tilde{\nu}_{m,
    \varepsilon_n, N_r}^{\cosh, \beta}  (\varphi)\\
    & = & \int_{B^{- 2 \delta}_{2, 2, 2 \ell} (\varepsilon_n \mathbb{Z}^2)}
    \langle \zeta, \varphi \rangle_{L^2 (\mathbb{R}^2)} \langle \psi (\cdot -
    x), \varphi \rangle_{L^2 (\mathbb{R}^2)} \mathd \tilde{\nu}_m^{\cosh, \beta} 
    (\varphi).
  \end{eqnarray*}
\end{lemma}

\begin{proof}
  For proving the first assertion observe that
  \begin{eqnarray*}
    &  & \frac{\mathd}{\mathd \gamma} \int_{B^{- 2 \delta}_{2, 2, 2 \ell}
    (\varepsilon \mathbb{Z}^2)} \langle \mathcal{D}^{\varepsilon_n} \psi
    (\cdot - x), \varphi \rangle_{L^2 (\varepsilon_n \mathbb{Z}^2)} \mathd
    \bar{\nu}^{\gamma \zeta}_{\varepsilon_n, N_r}  (\phi)\\
    & = & \frac{1}{Z^{\varepsilon_n, N_r, 0}} \int_{B^{- 2 \delta}_{2, 2, 2
    \ell} (\varepsilon_n \mathbb{Z}^2)} \langle \mathcal{D}^{\varepsilon_n}
    \psi (\cdot - x), \varphi \rangle_{L^2 (\varepsilon_n \mathbb{Z}^2)}
    \frac{\mathd}{\mathd \gamma} [\exp (- \gamma \langle
    \mathcal{D}^{\varepsilon_n} \zeta, \varphi \rangle_{L^2  (\varepsilon_n
    \mathbb{Z}^2)})]_{\gamma = 0} \mathd \tilde{\nu}_{m, \varepsilon_n, N_r}^{\cosh,
    \beta}  (\varphi)\\
    &  & + \frac{\mathd}{\mathd \gamma} \left[ \frac{1}{Z^{\varepsilon_n,
    N_r, \gamma \zeta}} \right]_{\gamma = 0} \int_{B^{- 2 \delta}_{2, 2, 2
    \ell} (\varepsilon_n \mathbb{Z}^2)} \langle \mathcal{D}^{\varepsilon_n}
    \psi (\cdot - x), \varphi \rangle_{L^2 (\varepsilon_n \mathbb{Z}^2)}
    \mathd \tilde{\nu}_{m, \varepsilon_n, N_r}^{\cosh, \beta} (\varphi)\\
    & = & - \int_{B^{- 2 \delta}_{2, 2, 2 \ell} (\varepsilon_n \mathbb{Z}^2)}
    \langle \mathcal{D}^{\varepsilon_n} \psi (\cdot - x), \varphi \rangle_{L^2
    (\varepsilon_n \mathbb{Z}^2)} \langle \mathcal{D}^{\varepsilon} \zeta,
    \varphi \rangle_{L^2  (\varepsilon_n \mathbb{Z}^2)} \mathd \tilde{\nu}_{m,
    \varepsilon_n, N_r}^{\cosh, \beta}  (\varphi),
  \end{eqnarray*}
  where we used the fact $\exp (- \gamma \langle \mathcal{D}^{\varepsilon_n}
  \zeta, \varphi \rangle_{L^2  (\varepsilon_n \mathbb{Z}^2)})$ is, by Theorem
  \ref{thm:analyticity} and its proof, uniformly in $L^p$ with respect to the
  measure $\nu_m^{\cosh, \beta}  (\varphi)$, and so we can exchange the
  derivative operation with the integral, and in the last line, we exploit the
  property that
  \[ \int_{B^{- 2 \delta}_{2, 2, 2 \ell} (\varepsilon_n \mathbb{Z}^2)}
     \langle \mathcal{D}^{\varepsilon} \psi (\cdot - x), \varphi \rangle_{L^2
     (\varepsilon_n \mathbb{Z}^2)} \mathd  \tilde{\nu}_{m,
    \varepsilon_n, N_r}^{\cosh, \beta}  (\varphi) 
     = 0, \]
  and that $\frac{1}{Z^{\varepsilon_n, N_r, \gamma \zeta}}$ is differentiable
  in $\gamma$. For the second assertion it is enough to observe that
  \begin{eqnarray*}
    &  & \int_{B^{- 2 \delta}_{2, 2, 2 \ell} (\varepsilon_n \mathbb{Z}^2)}
    \langle \mathcal{D}^{\varepsilon} \zeta, \phi \rangle_{L^2 (\varepsilon_n
    \mathbb{Z}^2)} \langle \mathcal{D}^{\varepsilon} \psi (\cdot - x), \varphi
    \rangle_{L^2 (\varepsilon_n \mathbb{Z}^2)} \mathd \tilde{\nu}_{m, \varepsilon_n,
    N_r}^{\cosh, \beta}  (\varphi)\\
    & = & \int_{B^{- 2 \delta}_{2, 2, 2 \ell} (\varepsilon_n \mathbb{Z}^2)}
    \langle \zeta, \mathcal{E}^{\varepsilon_n} \varphi \rangle_{L^2
    (\mathbb{R}^2)} \langle \psi (\cdot - x), \mathcal{E}^{\varepsilon_n}
    \varphi \rangle_{L^2 (\mathbb{R}^2)} \mathd \tilde{\nu}_{m, \varepsilon_n,
    N_r}^{\cosh, \beta}  (\varphi)
  \end{eqnarray*}
  and we have established \ in Theorem \ref{theorem:stochasticquantization}
  that $\lim_{\varepsilon_n \rightarrow 0, N_r \rightarrow \infty,}
  (\mathcal{E}^{\varepsilon_n})^{\ast} (\tilde{\nu}_{m, \varepsilon_n, N_r}^{\cosh, \beta}) =
  \nu_m^{\cosh, \beta}$ weakly on $B_{2, 2, 2 \ell}^{- 2 \delta}
  (\mathbb{R}^2)$, for any $\delta > 0$. The weak convergence together with
  the bound
  \begin{eqnarray*}
    &  & \mathbbm{1}_{\{  \| (1 - \Delta_{\varepsilon \mathbb{Z}^2})^{- 1 /
    2} (\varphi) \|^2_{L_{\ell}^2 (\varepsilon_n \mathbb{Z}^2)}  \geqslant N
    \}} \langle \zeta, \mathcal{E}^{\varepsilon_n} \varphi \rangle_{L^2
    (\mathbb{R}^2)} \langle \psi (\cdot - x), \mathcal{E}^{\varepsilon_n}
    \varphi \rangle_{L^2 (\mathbb{R}^2)}\\
    & \leqslant & \exp (- \gamma N / 2) \exp (\gamma \| (1 -
    \Delta_{\varepsilon \mathbb{Z}^2})^{- 1 / 2} (\varphi) \|^2_{L_{\ell}^2
    (\varepsilon_n \mathbb{Z}^2)})
  \end{eqnarray*}
  proves the thesis.
\end{proof}

\begin{proposition}
  For $\gamma > 0$ small enough we have that
  \[ \bar{\nu}_{N, \varepsilon}^{\gamma} = \tmop{Law} (W_{\varepsilon, N} (0,
     \cdot) + \bar{\phi}_{\varepsilon, N}^{\gamma \zeta} (0, \cdot)), \]
  where $\bar{\phi}_{\varepsilon, N}^{\gamma \zeta} \in B_{2, 2, l}^{- \delta}
  (\mathbb{R}^2 \times \varepsilon \mathbb{Z}^2)$ is the unique solution to
  the equation
  \begin{equation}
    \begin{array}{lll}
      &  & (- \Delta_{\mathbb{R}^2 + \varepsilon \mathbb{Z}^2} + m^2)
      \bar{\phi}_{\varepsilon, N}^{\gamma \zeta} + \alpha e^{\alpha
      \bar{\phi}_{\varepsilon, N}^{\gamma \zeta}} \mu^{\alpha}_{\varepsilon,
      N} - \alpha e^{- \alpha \bar{\phi}_{\varepsilon, N}^{\gamma \zeta}}
      \mu^{- \alpha}_{\varepsilon, N}\\
      & = & \gamma \overline{\mathcal{D}^{\varepsilon}} \zeta,
    \end{array} \label{eq:dim-red-2}
  \end{equation}
  where $\overline{\mathcal{D}^{\varepsilon}} \zeta : \mathbb{R}^2 \times
  \varepsilon \mathbb{Z}^2 \rightarrow \mathbb{R}$ is defined by
  \[ \overline{\mathcal{D}^{\varepsilon}} \zeta^{} (x, z) =
     \mathcal{D}^{\varepsilon} \zeta (z) . \]
\end{proposition}

\begin{proof}
  The proof is analogous to the one of Proposition \ref{prop:dim-red-int}.
\end{proof}

Let us introduce the notation
\[ w (x, z) : \mathbb{R}^2 \times \varepsilon \mathbb{Z}^2 \rightarrow (1 +
   \kappa (1 + | x |^2)^{1 / 2})^{- n} \exp (\kappa (1 + | z |^2)^{1 / 2}) .
\]
\begin{proposition}
  \label{prop:exp-decay}Let $\bar{\phi}_{\varepsilon, N}^{\gamma \zeta}$ be
  the solution to equation {\eqref{eq:dim-red-2}}. Then provided $\kappa$ is
  small enough there exists a constant $C > 0$ such that
  \[ \mathbb{E} \left[ \int_{\mathbb{R}^2 \times \varepsilon_n \mathbb{Z}^2}
     w (x, z) (\bar{\phi}_{\varepsilon_n, N_r}^{\gamma \zeta} -
     \bar{\phi}_{\varepsilon_n, N_r}^0)^2 \mathd x \mathd z \right]^{1 / 2}
     \leqslant C \gamma \| \zeta \|_{L^2 (\mathbb{R}^2)}. \]
\end{proposition}

\begin{proof}
  By the definitions of $\bar{\phi}_{\varepsilon_n, N_r}^{\gamma \zeta},
  \bar{\phi}_{\varepsilon_n, N_r}^0$ we have
  \begin{eqnarray*}
    &  & (- \Delta_{\mathbb{R}^2 + \varepsilon_n \mathbb{Z}^2} + m^2)
    (\bar{\phi}_{\varepsilon_n, N_r}^{\gamma \zeta} -
    \bar{\phi}_{\varepsilon_n, N_r}^0)\\
    &  & + \alpha c_{\varepsilon, \alpha} (\sinh (\alpha (W_{\varepsilon_n,
    N_r} + \bar{\phi}_{\varepsilon_n, N_r}^{\gamma \zeta})) - \sinh (\alpha
    (W_{\varepsilon_n, N_r} + \bar{\phi}_{\varepsilon_n, N_r}^0)))\\
    & = & \gamma \overline{\mathcal{D}}^{\varepsilon_n} \zeta,
  \end{eqnarray*}
  where
  \[ c_{\varepsilon, \alpha} = e^{- \frac{\alpha^2}{2} \mathbb{E}
     [W_{\varepsilon_n, N_r}^2]} . \]
  Testing the previous equation with the function $w (x, z)
  (\bar{\phi}_{\varepsilon_n, N_r}^{\gamma \zeta} - \bar{\phi}_{\varepsilon_n,
  N_r}^0)$ we get
  \[ \begin{array}{c}
       \int_{\mathbb{R}^2 \times \varepsilon_n \mathbb{Z}^2} w (x, z) (-
       \Delta_{\mathbb{R}^2 \times \varepsilon_n \mathbb{Z}^2} + m^2)
       (\bar{\phi}_{\varepsilon_n, N_r}^{\gamma \zeta} -
       \bar{\phi}_{\varepsilon_n, N_r}^0) (\bar{\phi}_{\varepsilon_n,
       N_r}^{\gamma \zeta} - \bar{\phi}_{\varepsilon_n, N_r}^0) \mathd x
       \mathd z \\
       + \alpha \int_{\mathbb{R}^2 \times \varepsilon_n \mathbb{Z}^2} 2 w (x,
       z) c_{\varepsilon, \alpha} (\sinh (\alpha (W_{\varepsilon_n, N_r} +
       \bar{\phi}_{\varepsilon_n, N_r}^{\gamma \zeta})) - \sinh (\alpha
       (W_{\varepsilon_n, N_r} + \bar{\phi}_{\varepsilon_n, N_r}^0))) \\
       \times (\bar{\phi}_{\varepsilon_n, N_r}^{\gamma \zeta} -
       \bar{\phi}_{\varepsilon_n, N_r}^0) \mathd x \mathd z = \gamma
       \int_{\mathbb{R}^2 \times \varepsilon_n \mathbb{Z}^2} w (x, z)
       \overline{\mathcal{D}^{\varepsilon}} \zeta (\phi_{N,
       \varepsilon}^{\gamma \zeta} - \phi_{N, \varepsilon}^0) \mathd x \mathd
       z.
     \end{array} \]
  Now since $\cosh$ is convex
  \[ \alpha [\sinh (\alpha (W_{\varepsilon_n, N_r} +
     \bar{\phi}_{\varepsilon_n, N_r}^{\gamma \zeta})) - \sinh (\alpha
     (W_{\varepsilon_n, N_r} + \bar{\phi}_{\varepsilon_n, N_r}^0))]
     (\bar{\phi}_{\varepsilon_n, N_r}^{\gamma \zeta} -
     \bar{\phi}_{\varepsilon_n, N_r}^0) \geqslant 0. \]
  So we have
  \begin{align*}
    &\int_{\mathbb{R}^2 \times \varepsilon_n \mathbb{Z}^2} w (x, z) (-
    \Delta_{\mathbb{R}^2 \times \varepsilon_n \mathbb{Z}^2} + m^2)
    (\bar{\phi}_{\varepsilon_n, N_r}^{\gamma \zeta} -
    \bar{\phi}_{\varepsilon_n, N_r}^0) (\bar{\phi}_{\varepsilon_n,
    N_r}^{\gamma \zeta} - \bar{\phi}_{\varepsilon_n, N_r}^0) \mathd x \mathd z\\
    \leqslant & 
    \gamma \int_{\mathbb{R}^2 \times \varepsilon_n \mathbb{Z}^2} w (x, z)
    \overline{\mathcal{D}^{\varepsilon_n}} \zeta (\bar{\phi}_{\varepsilon_n,
    N_r}^{\gamma \zeta} - \bar{\phi}_{\varepsilon_n, N_r}^0) \mathd x \mathd
    y. \nonumber
  \end{align*}
  On the other hand
  \begin{eqnarray*}
    &  & \int_{\mathbb{R}^2 \times \varepsilon_n \mathbb{Z}^2} w (x, z) (-
    \Delta_{\mathbb{R}^2 \times \varepsilon_n \mathbb{Z}^2} + m^2)
    (\bar{\phi}_{\varepsilon_n, N_r}^{\gamma \zeta} -
    \bar{\phi}_{\varepsilon_n, N_r}^0) (\bar{\phi}_{\varepsilon_n,
    N_r}^{\gamma \zeta} - \bar{\phi}_{\varepsilon_n, N_r}^0) \mathd x \mathd
    z\\
    & = & m^2 \int_{\mathbb{R}^2 \times \varepsilon_n \mathbb{Z}^2} w (x, z)
    (\bar{\phi}_{\varepsilon_n, N_r}^{\gamma \zeta} -
    \bar{\phi}_{\varepsilon_n, N_r}^0)^2 \mathd x \mathd z \\
    &  & + \int_{\mathbb{R}^2 \times \varepsilon_n \mathbb{Z}^2} w (x, z)
    (\nabla_{\mathbb{R}^2 \times \varepsilon_n \mathbb{Z}^2}
    (\bar{\phi}_{\varepsilon_n, N_r}^{\gamma \zeta} -
    \bar{\phi}_{\varepsilon_n, N_r}^0))^2 \mathd x \mathd z \\
    &  & + \int_{\mathbb{R}^2 \times \varepsilon_n \mathbb{Z}^2}
    \nabla_{\mathbb{R}^2 \times \varepsilon \mathbb{Z}^2}
    (\bar{\phi}_{\varepsilon_n, N_r}^{\gamma \zeta} -
    \bar{\phi}_{\varepsilon_n, N_r}^0) (\bar{\phi}_{\varepsilon_n,
    N_r}^{\gamma \zeta} - \bar{\phi}_{\varepsilon_n, N_r}^0)
    (\nabla_{\mathbb{R}^2 \times \varepsilon \mathbb{Z}^2} w) \mathd x \mathd
    z.
  \end{eqnarray*}
  It is easy to prove that
  \[ | \nabla_{\mathbb{R}^2 \times \varepsilon_n \mathbb{Z}^2} w | \leqslant
     \kappa | w | \]
  for $\varepsilon_n$ and $\kappa$ small enough. This and Cauchy-Schwarz
  inequality give us \
  \begin{eqnarray*}
    &  & \int_{\mathbb{R}^2 \times \varepsilon_n \mathbb{Z}^2} w (x, z) (-
    \Delta_{\mathbb{R}^2 \times \varepsilon_n \mathbb{Z}^2} + m^2)
    (\bar{\phi}_{\varepsilon_n, N_r}^{\gamma \zeta} -
    \bar{\phi}_{\varepsilon_n, N_r}^0) (\bar{\phi}_{\varepsilon_n,
    N_r}^{\gamma \zeta} - \bar{\phi}_{\varepsilon_n, N_r}^0) \mathd x \mathd
    z\\
    & \geqslant & \frac{1}{2} m^2 \int_{\mathbb{R}^2 \times \varepsilon_n
    \mathbb{Z}^2} w (x, z) (\bar{\phi}_{\varepsilon_n, N_r}^{\gamma \zeta} -
    \bar{\phi}_{\varepsilon_n, N_r}^0)^2 \mathd x \mathd z\\
    &  & + \frac{1}{2} \int_{\mathbb{R}^2 \times \varepsilon_n \mathbb{Z}^2}
    w (x, z) (\nabla_{\mathbb{R}^2 \times \varepsilon_n \mathbb{Z}^2}
    (\bar{\phi}_{\varepsilon_n, N_r}^{\gamma \zeta} -
    \bar{\phi}_{\varepsilon_n, N_r}^0))^2 \mathd x \mathd z .
  \end{eqnarray*}
  So putting everything together we get
  \begin{eqnarray*}
    &  & \int_{\mathbb{R}^2 \times \varepsilon_n \mathbb{Z}^2} w (x, z) (m^2
    (\bar{\phi}_{\varepsilon_n, N_r}^{\gamma \zeta} -
    \bar{\phi}_{\varepsilon_n, N_r}^0)^2 + (\nabla_{\mathbb{R}^2 \times
    \varepsilon_n \mathbb{Z}^2} (\bar{\phi}_{\varepsilon_n, N_r}^{\gamma
    \zeta} - \bar{\phi}_{\varepsilon_n, N_r}^0))^2) \mathd x \mathd z\\
    & \leqslant & 2 \gamma \int_{\mathbb{R}^2 \times \varepsilon_n \mathbb{Z}^2} w
    (x, z) \overline{\mathcal{D}^{\varepsilon_n}} \zeta
    (\bar{\phi}_{\varepsilon_n, N_r}^{\gamma \zeta} -
    \bar{\phi}_{\varepsilon_n, N_r}^0) \mathd x \mathd z.
  \end{eqnarray*}
  Now since by Young's inequality
  
  \begin{align*}
    &2 \int_{\mathbb{R}^2 \times \varepsilon_n \mathbb{Z}^2} w (x, z)
    \overline{\mathcal{D}^{\varepsilon_n}} \zeta (\bar{\phi}_{\varepsilon_n,
    N_r}^{\gamma \zeta} - \bar{\phi}_{\varepsilon_n, N_r}^0) \mathd x \mathd z\\
    \leqslant& \nonumber
    \gamma \int_{\mathbb{R}^2 \times \varepsilon_n \mathbb{Z}^2} w (x, z) \left(
    \overline{\mathcal{D}^{\varepsilon_n}} \zeta \right)^2 \mathd x \mathd z +
    \int_{\mathbb{R}^2 \times \varepsilon_n \mathbb{Z}^2} w (x, z)
    (\bar{\phi}_{\varepsilon_n, N_r}^{\gamma \zeta} -
    \bar{\phi}_{\varepsilon_n, N_r}^0)^2 \mathd x \mathd y \nonumber
  \end{align*}
  and since
  \[ \int_{\mathbb{R}^2 \times \varepsilon_n \mathbb{Z}^2} w (x, z) \left(
     \overline{\mathcal{D}^{\varepsilon_n}} \zeta \right)^2 \mathd x \mathd z
     \leqslant C \| \zeta \|_{L^2 (\mathbb{R}^2)}^2, \]
  for some constant $C > 0$ independent on $\zeta$ (due to the fact that we
  have assumed $\zeta$ to be supported in $B (0, 1)$), we can conclude the
  proof of Proposition \ref{prop:exp-decay}.
\end{proof}

\begin{proof*}{Proof of Theorem \ref{clustering-exp}}
  By Lemma \ref{lemma:derivative} we have that
  \[ \begin{array}{rl}
       &\sup_{\varepsilon_n, N_r} \left| \int_{B^{- 2 \delta}_{2, 2, 2 \ell}
       (\varepsilon_n \mathbb{Z}^2)} \langle \mathcal{D}^{\varepsilon_n}
       \zeta, \varphi \rangle_{L^2 (\varepsilon_n \mathbb{Z}^2)} \langle
       \mathcal{D}^{\varepsilon_n} \psi (\cdot - x), \varphi \rangle_{L^2
       (\varepsilon_n \mathbb{Z}^2)} \mathd \nu_{m, \varepsilon_n,
       N_r}^{\cosh, \beta}  (\varphi) \right|\\
       \leqslant & \sup_{\varepsilon_n, N_r} \sup_{\gamma > 0} \frac{1}{\gamma}\left| \int_{B^{-
       2 \delta}_{2, 2, 2 \ell} (\varepsilon_n \mathbb{Z}^2)} \langle
       \mathcal{D}^{\varepsilon_n} \psi (\cdot - x), \varphi \rangle_{L^2
       (\varepsilon_n \mathbb{Z}^2)} (\mathd \bar{\nu}^{\gamma \zeta}_{m,
       \varepsilon_n, N_r}  (\varphi) - \mathd \nu_{m, \varepsilon_n,
       N_r}^{\cosh, \beta} (\varphi)) \right|.
     \end{array} \]
  On the other hand, by the invariance of the law of
  $\bar{\phi}_{\varepsilon_n, N_r}^{\gamma \zeta} (x, \cdot)$ and
  $\bar{\phi}_{\varepsilon_n, N_r}^0 (x \comma \cdot)$ with respect to the
  translation of $x$ coordinates, we get
  \begin{eqnarray*}
    &  & \left| \int_{B^{- 2 \delta}_{2, 2, 2 \ell} (\varepsilon_n
    \mathbb{Z}^2)} \langle \mathcal{D}^{\varepsilon_n} \psi (\cdot - y),
    \varphi \rangle_{L^2 (\varepsilon_n \mathbb{Z}^2)} (\mathd
    \bar{\nu}^{\gamma \zeta}_{m, \varepsilon_n, N_r}  (\varphi) - \mathd
    \nu_{m, \varepsilon_n, N_r}^{\cosh, \beta} (\varphi)) \right|\\
    & = & | \mathbb{E} [\langle \mathcal{D}^{\varepsilon_n} \psi (\cdot - y),
    (\bar{\phi}_{\varepsilon_n, N_r}^{\gamma \zeta} (0, \cdot) -
    \bar{\phi}_{\varepsilon_n, N_r}^0 (0 \comma \cdot)) \rangle] |\\
    & = & \int_{\varepsilon_n \mathbb{Z}^2} (\mathcal{D}^{\varepsilon_n} \psi
    (z - y)) (\mathbb{E} [\bar{\phi}_{\varepsilon_n, N_r}^{\gamma \zeta} (0,
    z) - \bar{\phi}_{\varepsilon_n, N_r}^0 (0 \comma z)]) \mathd z\\
    & = & C \int_{\mathbb{R}^2 \times \varepsilon_n \mathbb{Z}^2} (1 + | x
    |)^{- n} \left( \overline{\mathcal{D}^{\varepsilon_n}} \psi (x, z - y)
    \right) (\mathbb{E} [\bar{\phi}_{\varepsilon_n, N_r}^{\gamma \zeta} (x, z)
    - \bar{\phi}_{\varepsilon_n, N_r}^0 (x \comma z)]) \mathd x \mathd z\\
    & \leqslant & C \exp (- \kappa | y |) \int_{\mathbb{R}^2 \times
    \varepsilon_n \mathbb{Z}^2} w (x, z) \left(
    \overline{\mathcal{D}^{\varepsilon_n}} \psi (x, z - y) \right) (\mathbb{E}
    [\bar{\phi}_{\varepsilon_n, N_r}^{\gamma \zeta} (x, z) -
    \bar{\phi}_{\varepsilon_n, N_r}^0 (x \comma z)]) \mathd x \mathd z,
  \end{eqnarray*}
  where again
  \[ w (x, z) = (1 + \kappa (1 + | x |^2)^{1 / 2})^{- n} \exp (\kappa (1 + |
     z |^2)^{1 / 2}), \]
  and we have used that $\psi$ is supported in $B (0, 1)$. From this we can
  conclude the proof of Theorem \ref{clustering-exp} using Proposition
  \ref{prop:exp-decay} and the Cauchy-Schwarz inequality.
\end{proof*}

\bibliographystyle{plain}
\bibliography{exponential}

\end{document}